\documentclass[11pt,a4paper]{article}
\usepackage[onehalfspacing]{setspace}
 
\usepackage[utf8]{inputenc}
\usepackage[T1]{fontenc}
\usepackage{fix-cm}
\usepackage{xpatch}

\makeatletter
\xpatchcmd{\titlepage}{\@restonecolfalse\newpage}{\@restonecolfalse}{}{}
\xpatchcmd{\endtitlepage}{\if@restonecol\twocolumn \else \newpage \fi}{\if@restonecol\twocolumn \else  \fi}{\typeout{success}}{\typeout{fail}}
\makeatother

\usepackage{hhline}
\usepackage{graphicx}
\usepackage[left = 2.5cm, right=2.5cm, top = 2.5cm, bottom = 2.5cm]{geometry}
\usepackage{braket}
\usepackage{mathptmx}  
\usepackage[citestyle=numeric-comp, bibstyle=numeric, backend=biber, firstinits]{biblatex}

\usepackage[bbgreekl]{mathbbol}

\usepackage{pict2e,picture}

\makeatletter
\DeclareRobustCommand{\bbDelta}{{\mathpalette\bb@Delta\relax}}
\newcommand{\bb@Delta}[2]{%
  \begingroup
  \sbox\z@{$\m@th#1\Delta$}%
  \dimendef\Dht=6 \dimendef\Dwd=8
  \setlength{\Dwd}{\wd\z@}%
  \setlength{\Dht}{\ht\z@}%
  \begin{picture}(\Dwd,\Dht)
  \put(0,0){$\m@th#1\Delta$}
  \put(.42\Dwd,.7\Dht){\line(10,-26){.25\Dht}}
  \end{picture}%
  \endgroup
}
\makeatother
\usepackage{tikz-cd}
\usepackage{tikz}
\usepackage{tikz-3dplot}
\usepackage[symbol]{footmisc}
\usepackage{xcolor}
\usepackage{extarrows}
\usepackage{stmaryrd}
\usepackage{color}
\usepackage{multirow}
\usepackage{adjustbox}
\usepackage{amssymb}
\usepackage{chngcntr}
\usepackage{amsmath,bm}
\usepackage{physics}
\usepackage[amsthm, thmmarks, amsmath]{ntheorem}
\usepackage[english]{babel}
\usepackage{amscd}
\definecolor{darkgreen}{RGB}{33,134,115}
\definecolor{darkred}{RGB}{163.2,8.4,46.8}
\usepackage[colorlinks,linktocpage=true,
urlcolor = blue,
linkcolor = darkred,
citecolor = darkgreen
]{hyperref}
\hypersetup{
colorlinks = true
}
\usepackage{caption}
\usepackage{newtxmath}
\usepackage{subcaption}
\usepackage{mwe}
\usepackage{wasysym, stackengine, makebox, scalerel}
\usepackage{float}
\usepackage{newtxmath}
\usepackage{xparse,aliascnt,bookmark}
\usepackage{mathdots}
\usepackage{cleveref}
\crefrangeformat{equation}{(#3#1#4) to~(#5#2#6)}
\usepackage{comment}
\usepackage{bold-extra}
\DeclareFieldFormat[article,periodical]{volume}{\mkbibbold{#1}}

\stackMath
\newcommand\reallywidehat[1]{%
\savestack{\tmpbox}{\stretchto{%
  \scaleto{%
    \scalerel*[\widthof{\ensuremath{#1}}]{\kern-.6pt\bigwedge\kern-.6pt}%
    {\rule[-\textheight/2]{1ex}{\textheight}}
  }{\textheight}%
}{0.5ex}}%
\stackon[1pt]{#1}{\tmpbox}%
}
\renewcommand{\l}{\left}
\renewcommand{\r}{\right}

\makeatletter
\newcommand{\pushright}[1]{\ifmeasuring@#1\else\omit\hfill$\displaystyle#1$\fi\ignorespaces}
\newcommand{\pushleft}[1]{\ifmeasuring@#1\else\omit$\displaystyle#1$\hfill\fi\ignorespaces}
\makeatother

\counterwithin{figure}{section}
\counterwithin{table}{section}

\NewDocumentCommand{\xnewtheorem}{m o m}
 {%
  \IfNoValueTF{#2}
   {\newtheorem{#1}{#3}}
   {%
    \newaliascnt{#1}{#2}%
    \newtheorem{#1}[#1]{#3}%
    \aliascntresetthe{#1}%
    \expandafter\newcommand\csname #1autorefname\endcsname{#3}%
   }%
 }

\numberwithin{equation}{section}

\newcommand{\W}{\mathcal W}
\newcommand{\E}{\mathbb E}
\newcommand{\ini}{\mathrm{in}}
\newcommand{\R}{\mathbb{R}}
\newcommand{\Z}{\mathbb{Z}}

\newcommand{\res}{\mathrm{Res}}
\newcommand{\rmi}{\mathrm{i}}

\newcommand{\C}{\mathbb{C}}
\newcommand{\Q}{\mathcal{H}}

\newcommand{\p}{\partial}
\newcommand{\G}{\mathcal{G}}
\newcommand{\D}{\mathcal{D}}

\newcommand{\rmd}{\mathrm{d}}

\newcommand{\Id}{\mathrm{Id}}

\newcommand{\stem}{\mathcal S}
\newcommand{\node}{\mathscr n}
\newcommand{\Node}{\mathcal N}
\newcommand{\leaf}{\mathscr l}
\newcommand{\Leaf}{\mathcal L}
\newcommand{\N}{\mathbb N}
\newcommand{\J}{\mathcal J}
\newcommand{\flower}{\mathscr f}
\newcommand{\T}{\mathcal T}
\newcommand{\M}{\mathcal M}

\renewcommand{\root}{\mathscr r}
\newcommand{\lege}{\mathrel{\mathpalette\gele@{\le\ge}}}

\theoremstyle{plain}
\xnewtheorem{theorem}{Theorem}[section]
\xnewtheorem{corollary}[theorem]{Corollary}
\xnewtheorem{lemma}[theorem]{Lemma}
\xnewtheorem{proposition}[theorem]{Proposition}

\newtheorem*{notation}{Notation}

\theoremstyle{plain}
\newtheorem{assumption}{Assumption}

\theoremstyle{plain}

\theoremstyle{definition}
\xnewtheorem{definition}[theorem]{Definition}
\xnewtheorem{example}[theorem]{Example}

\theoremstyle{remark}
\xnewtheorem{remark}[theorem]{Remark}

\renewbibmacro{in:}{}
\DeclareFieldFormat[article]{citetitle}{#1}
\DeclareFieldFormat[article]{title}{#1}
\DeclareFieldFormat{pages}{#1}

\makeatletter
\newcounter{savesection}
\newcounter{apdxsection}
\renewcommand\appendix{\par
  \setcounter{savesection}{\value{section}}%
  \setcounter{section}{\value{apdxsection}}%
  \setcounter{subsection}{0}%
  \gdef\thesection{\@Alph\c@section}}
\newcommand\unappendix{\par
  \setcounter{apdxsection}{\value{section}}%
  \setcounter{section}{\value{savesection}}%
  \setcounter{subsection}{0}%
  \gdef\thesection{\@arabic\c@section}}
\makeatother
\begin{document}

\begin{center}
    \LARGE{\textsc{\textbf{Discrete wave turbulence for a coupled system of quintic Schrödinger equations}}}\\~\\

    \Large \textsc{Shayan Zahedi}
    
\end{center}

\begin{abstract} 

We derive rigorously the non-linear macroscopic system associated to a microscopic system of coupled quintic Schrödinger equations in the framework of discrete wave turbulence under a particular scaling law that describes the limiting process. Our system evolves from a pair of well-prepared random initial data. More precisely, in dimensions $d\geq2$, we set up our microscopic system on a large box of size $L$ with weak non-linearity of strength $\epsilon$. In the limit $L\to\infty$ and $\epsilon\to0$, under the scaling law $\epsilon L^{\frac{1}{\beta}}=1$ with $\beta\in(1,\infty)$, we prove that the long-time behaviour of our microscopic system is statistically described up to times $\delta\epsilon^{-1}$ by a non-linear resonant system whose dynamics are driven by exact resonances, where $\delta$ is independent of $L$ and $\epsilon$. Our system does not display generic symmetries, in particular not mass conservation. In such systems with fewer invariances, exact resonances contribute significantly compared to quasi-resonances and are essentially responsible for the effective dynamics in the large-box limit. We justify the emergence of discrete wave turbulence for our microscopic model.  
 
\end{abstract}

{\let\clearpage\relax \tableofcontents}



\renewcommand*{\thefootnote}{\arabic{footnote}}

\section{Introduction}

\subsection{Background and motivation}
Wave turbulence theory describes the nonequilibrium statistical behaviour of a system with interacting waves in the thermodynamic limit where the size of the box $L$ tends to infinity. It is the wave analogue of the classical kinetic theory in which the number of particles $N$ tends to infinity. The degrees of freedom $L$ and $N$ are analogs of each other. Wave turbulence theory made its appearance first in the physics literature \cite{Peierls1929,BROUT1956621,Vedenov1967,Zaslavskii,Hasselmann_1962,Hasselmann_1963,benney_saffman_1966}. Mathematically, one is interested under which conditions one could derive rigorously a macroscopic system from a microscopic system of non-linear partial differential equations with well-prepared random initial data at an appropriate time scale. General well-posedness of the macroscopic system is also of interest. 

The physics literature foresaw that an expansion in Feynman diagrams is the correct approach to derive the macroscopic system. However, the main obstacle is to make this expansion rigorous and prove convergence, while having the correct scaling law between $L$ and the degree of non-linearity $\epsilon$ in mind. Lukkarinen and Spohn \cite{Lukkarinen_2010} were the first to succeed in proving convergence of the diagrammatic expansion in the context of a cubic Schrödinger equation on a lattice. They studied the time correlations of the invariant Gibbs measure in the thermodynamic limit. 
Ever since, many works on the wave turbulence of the cubic Schrödinger equation, in particular on the Fourier spectrum of its solutions \cite{Collotdietert}, and the rigorous derivation of its kinetic wave equation \cite{Collot2025,COLLOT2026111179,Deng_2021,Dymov2,Dymov3} have emerged. The derivation of the wave kinetic equation for the quintic Schrödinger equation was studied in \cite{ASDS}. In \cite{staffilani2024waveturbulencetheorystochastic}, a 3-wave kinetic equation is derived from a KdV-type equation on a hypercubic lattice. In \cite{Vassilev_2025} the wave kinetic equation for a one-dimensional MMT model is derived. In \cite{https://doi.org/10.1002/cpa.22224,deng2025waveturbulencetheory2d} the wave turbulence of irrotational gravity water waves in 2$D$ is studied.  
Deng and Hani proved the derivation of the wave kinetic equation for the cubic Schrödinger equation up to the kinetic time scale \cite{Deng2021FullDO,deng2023derivationwavekineticequation}. The ideas in the last two references were then adapted to the long-time derivation of the Boltzmann equation \cite{deng2025longtimederivationboltzmann} and other rigorous derivations of fundamental partial differential equations of fluid mechanics \cite{deng2025hilbertssixthproblemderivation}. 
To summarize, the kinetic description requires the non-linearity to be weak. However, in the complementary regime of very weak nonlinearity, the exact resonances of the underlying microscopic system dominate over quasi-resonances. In this regime, different effective equations arise \cite{desuzzoni2025waveturbulencesemilinearkleingordon,faou2013weaklynonlinearlargebox,https://doi.org/10.1002/cpa.21749}. This regime is referred to as discrete wave turbulence \cite{Nazarenkodiscrete,PhysRevLett.98.214502,nazarenko2011wave}. Those regimes are characterized by the absence of symmetries that could allow for a kinetic description of the underlying microscopic model. The discrete wave turbulence of a coupled system of quadratic Klein-Gordon equations was studied in \cite{desuzzoni2025waveturbulencesemilinearkleingordon} and its associated macroscopic system was derived rigorously. In this paper , we are interested in the discrete wave turbulence of a coupled system of quintic Schrödinger equations.  

We consider a system of quintic Schrödinger equations with weak nonlinearity $\epsilon Q^\eta\l(f^\eta,\overline{f^\eta},\overline{f^\eta},f^{\overline\eta},f^{\overline\eta}\r)$ on a box of size $L$, where $\eta\in\{0,1\}$ and $\overline\eta=1-\eta$. We specifically consider those $Q^\eta$ that break certain invariances, in particular mass conservation. The absence of these symmetries implies that exact resonances will drive the dynamics of the non-linear macroscopic system \eqref{dee} in the large-box limit $L\to\infty$, rendering it effectively of discrete type under the scaling law $\epsilon L^{\frac{1}{\beta}}=1$ for $\beta\in(1,\infty)$. We assume the initial data to be random and well-prepared. We expect the dynamics of the correlations of the microscopic system \eqref{system} to be governed by \eqref{dee}. 

\subsection{Statement of the main result}
\subsubsection{The microscopic system \eqref{system}}

We consider a coupled system of quintic Schrödinger equations in space dimension $d\geq2$ 
\begin{equation}\label{system}\tag{qNLS}
    \begin{cases}
        (\rmi\p_t+\Delta)f^0 = \epsilon Q^0\l(f^0,\overline{f^0,}\overline{f^0},f^1,f^1\r)&\text{ in }\R_\geq\times \mathbb T_L^d,\\
        (\rmi\p_t+\Delta)f^1 = \epsilon Q^1\l(f^1,\overline{f^1},\overline{f^1},f^0,f^0\r)&\text{ in }\R_\geq\times \mathbb T_L^d,\\
        \l(f^0(0),f^1(0)\r) = \l(f^0_{\ini},f^1_{\ini}\r)&\text{ on } \mathbb T_L^d,
    \end{cases}
\end{equation}
where
\begin{equation}
    \hat f(t,k)=\frac{1}{L^{\frac{d}{2}}}\int_{\mathbb T_L^d}f(t,x)e^{\rmi k\cdot x}\rmd x,\quad f(t,x)=\frac{1}{L^{\frac{d}{2}}}\sum_{k\in\Z_L^d}\hat f(t,k)e^{\rmi k\cdot x}
\end{equation}
such that the Fourier transform of a convolution takes the form ($n>1$)
\begin{equation}
    \widehat{f_1\cdots f_n}(k) = \frac{1}{L^{d(n-1)/2}}\sum_{k_1+\cdots+k_n = k}\prod_{i=1}^n\widehat{f_i}(k_i)
\end{equation}
and for $\eta\in\{0,1\}$,
\begin{equation}
    \reallywidehat{Q^\eta\l(f,g,h,u,w\r)}(k)\coloneqq\frac{1}{L^{2d}}\sum_{\sum_{i=1}^5k_i=k} Q^\eta(k_1,k_2,k_3,k_4,k_5)\hat{f}(k_1)\hat g(k_2)\hat h(k_3)\hat u(k_4)\hat w(k_5).
\end{equation}
We assume $Q^\eta,\abs{Q^\eta}\in\l(W^{1,1}\cap W^{1,\infty}\r)\l(\l(\R^d\r)^5,\C\r)$ for both $\eta\in\{0,1\}$.

\begin{remark}\label{breaking mass conservation}
    If $f\coloneqq f^0=f^1$ and $Q^\eta(k_1,k_2,k_3,k_4,k_5)=1$ for both $\eta\in\{0,1\}$, then we would have the quintic Schrödinger equation \begin{equation}
    \begin{cases}
        (\rmi\p_t+\Delta)f = \epsilon\abs{f}^4f&\text{ in }\R_\geq\times \mathbb T_L^d,\\
        f(0) = f_\ini&\text{ on }\mathbb T_L^d
    \end{cases}
    \end{equation}
    with well-prepared random initial data $f_\ini=f_\ini^0=f_\ini^1$ and periodic boundary conditions in the space variables. In this case, we will not observe non-linear effective dynamics in the discrete wave turbulence regime as can be seen by \eqref{dee}. We also have time conservation of mass:
    \begin{equation}
        \p_t\norm{f(t)}_{L^2(T_L^d)}^2 = \int_{T_L^d}\l(\p_t f(t)\overline f(t)+f(t)\p_t\overline f(t)\r)\rmd x = \rmi\int_{T_L^d}\l(\overline f(t)\Delta f(t)-f(t)\Delta\overline f(t)\r)\rmd x = 0.
    \end{equation}
    The system \eqref{system} is devised in a way to break this invariance by introducing $Q^\eta$. We generalize to a coupled system, which will lead to the introduction of colour in the diagrammatic part of the analysis, as was done in \cite{desuzzoni2025waveturbulencesemilinearkleingordon} for a system of coupled quadratic Klein-Gordon equations. 
\end{remark}

Let $F^\eta\coloneqq e^{-\rmi t\Delta}f^\eta$ and $\overline\eta\coloneqq 1-\eta$ so that we arrive at the following equation for the Fourier modes: 

\begin{equation}\label{reformulation in momentum space}
    \widehat{F^\eta}(t,k) = \mu^\eta_k-\frac{\rmi\epsilon}{L^{2d}}\int_0^t\sum_{\sum_{i=1}^5k_i=k}Q^\eta\l(k_1,\ldots,k_5\r)e^{\rmi t\Omega}\widehat{F^\eta}(\tau,k_1)\widehat{\overline{F^\eta}}(\tau,k_2)\widehat{\overline{F^\eta}}(\tau,k_3)\widehat{F^{\overline\eta}}(\tau,k_4)\widehat{F^{\overline\eta}}(\tau,k_5)\rmd\tau, 
\end{equation}
where 
\begin{equation}
    \Omega=\Omega\l(k,k_1,k_2,k_3,k_4,k_5\r)\coloneqq\abs{k}^2-\abs{k_1}^2+\abs{k_2}^2+\abs{k_3}^2-\abs{k_4}^2-\abs{k_5}^2
\end{equation}
denotes the \textit{resonance factor}. 

In the large-box limit $L\to\infty$, we are seeking the dynamics of the three correlations $\E\l(\abs{\widehat{f^0}(t,k)}^2\r)$, $\E\l(\abs{\widehat{f^1}(t,k)}^2\r)$ and $\E\l({\widehat{f^0}(t,k)}\overline{\widehat{f^1}(t,k)}\r)$ and the averaging happens over the random distribution of the initial data. For a complex number $z\in\C$, we denote $z^+=z$ and $z^-=\overline z$ and assume the initial data $\l(\widehat{f^0}(k),\widehat{f^1}(k)\r)_{k\in\Z_L^d}$ to be a family of independent Gaussian variables in $\C^2$ such that 
\begin{equation}\label{random initial data}
    \E\l(\mu_k^\eta\r)=0\text{ and }\E\l(\mu_k^{\eta,\iota}{\mu_{k'}^{\eta',\iota'}}\r)=\delta_{\iota+\iota'}\delta_{k-k'}M^{\eta,\eta'}(k)^\iota
\end{equation}
for all $k,k'\in\Z_L^d$, $\eta,\eta'\in\{0,1\}$ and where $M^{\eta,\eta'}\in\l(W^{1,\infty}\cap W^{1,1}\r)\l(\R^d\r)$. The assumption \eqref{random initial data} implies necessarily $\overline{M^{\eta,\eta'}}=M^{\eta',\eta}$. We also assume that $M^{\eta,\eta'}$ are supported in a ball of fixed radius $R>0$ around the origin.

\begin{assumption}\label{the assumption i have to use}
    One shall assume for all $n\leq\abs{\log\epsilon}$ and $i\in\llbracket1,4\rrbracket$ that 
    \begin{equation}
        \sum_{\substack{k_1,\ldots,k_i\in B_{nR}^{\Z_L^d}(0)}}\sum_{\sigma\in\mathcal S_5}\abs{Q_\iota^\eta\l(k_{\sigma(1)},\ldots,k_{\sigma(5)}\r)}\lesssim L^{id}
    \end{equation}
    for all $k_{i+1},\ldots,k_5\in\R^d$, where the implicit constant is independent of $n$ and $L$ and $\mathcal S_5$ denotes the symmetric group of degree $5$. 
\end{assumption}

The time scale at which we exhibit the macroscopic system to display the non-linear effective dynamics is $\delta T$, where 

\begin{equation}
    T \coloneqq \epsilon^{-1},
\end{equation}
$\epsilon=L^{-\frac{1}{\beta}}$, $\beta\in(1,\infty)$, and $\delta$ is a constant that is independent of $L$ and $\epsilon$.

\subsubsection{The macroscopic system \eqref{dee}}
Under the assumption of independence on the initial data $\widehat{f^\eta_{\ini}}$ the non-linear macroscopic system is given by
\begin{equation}\tag{MS}\label{dee}
    \begin{cases}
        \p_t\rho^\eta(t,\xi) 
    = 2\Bigg[\rho^\eta(t,\xi)\int_{\R^d}\int_{\R^d}\Im\l(Q^\eta(\xi,\xi_1,\xi_2,-\xi_1,-\xi_2)\overline{\rho^{\times,\eta}(t,-\xi_1)}\overline{\rho^{\times,\eta}(t,-\xi_2)}\r) \rmd\xi_1\rmd\xi_2\\+ \rho^\eta(t,\xi)\int_{\R^d}\int_{\R^d}\Im\l(Q^\eta(\xi,\xi_1,\xi_2,-\xi_2,-\xi_1)\overline{\rho^{\times,\eta}(t,-\xi_1)}\overline{\rho^{\times,\eta}(t,-\xi_2)}\r)\rmd\xi_1\rmd\xi_2\\+\int_{\R^d}\int_{\R^d}\Im\l(\overline{\rho^\times(t,\xi)}{Q^\eta(\xi_1,-\xi_1,\xi_2,\xi,-\xi_2)}\rho^\eta(t,\xi_1)\overline{\rho^{\times,\eta}(t,-\xi_2)}\r)\rmd\xi_1\rmd\xi_2\\+\int_{\R^d}\int_{\R^d}\Im\l(\overline{\rho^\times(t,\xi)}{Q^\eta(\xi_1,\xi_2,-\xi_1,\xi,-\xi_2)}\rho^\eta(t,\xi_1)\overline{\rho^{\times,\eta}(t,-\xi_2)}\r)\rmd\xi_1\rmd\xi_2\\+\int_{\R^d}\int_{\R^d}\Im\l(\overline{\rho^\times(t,\xi)}Q^\eta(\xi_1,-\xi_1,\xi_2,-\xi_2,\xi)\rho^\eta(t,\xi_1)\overline{\rho^{\times,\eta}(t,-\xi_2)}\r)\rmd\xi_1\rmd\xi_2\\+\int_{\R^d}\int_{\R^d}\Im\l(\overline{\rho^\times(t,\xi)}Q^\eta(\xi_1,\xi_2,-\xi_1,-\xi_2,\xi)\rho^\eta(t,\xi_1)\overline{\rho^{\times,\eta}(t,-\xi_2)}\r)\rmd\xi_1\rmd\xi_2\Bigg]\\
    \p_t\rho^\times(t,\xi)=
        -\rmi\Bigg[\rho^\times(t,\xi)\int_{\R^d}\int_{\R^d}\l(Q^0(\xi,\xi_1,\xi_2,-\xi_1,-\xi_2)-\overline{Q^1(\xi,\xi_1,\xi_2,-\xi_1,-\xi_2)}\r)\overline{\rho^{\times}(t,-\xi_1)}\overline{\rho^{\times}(t,-\xi_2)}\rmd\xi_1\rmd\xi_2\\+\rho^\times(t,\xi)\int_{\R^d}\int_{\R^d}\l(Q^0(\xi,\xi_1,\xi_2,-\xi_2,-\xi_1)-\overline{Q^1(\xi,\xi_1,\xi_2,-\xi_2,-\xi_1)}\r)\overline{\rho^{\times}(t,-\xi_1)}\overline{\rho^{\times}(t,-\xi_2)}\rmd\xi_1\rmd\xi_2\\+\int_{\R^d}\int_{\R^d}\l(\rho^1(t,\xi)Q^0(\xi_1,-\xi_1,\xi_2,\xi,-\xi_2)\rho^0(t,\xi_1)-\rho^0(t,\xi)\overline{Q^1(\xi_1,-\xi_1,\xi_2,\xi,-\xi_2)}\rho^1(t,\xi_1)\r)\overline{\rho^{\times}(t,-\xi_2)}\rmd\xi_1\rmd\xi_2\\+\int_{\R^d}\int_{\R^d}\l(\rho^1(t,\xi)Q^0(\xi_1,\xi_2,-\xi_1,\xi,-\xi_2)\rho^0(t,\xi_1)-\rho^0(t,\xi)\overline{Q^1(\xi_1,\xi_2,-\xi_1,\xi,-\xi_2)}\rho^1(t,\xi_1)\r)\overline{\rho^{\times}(t,-\xi_2)}\rmd\xi_1\rmd\xi_2\\+\int_{\R^d}\int_{\R^d}\l(\rho^1(t,\xi)Q^0(\xi_1,-\xi_1,\xi_2,-\xi_2,\xi)\rho^0(t,\xi_1)-\rho^0(t,\xi)\overline{Q^1(\xi_1,-\xi_1,\xi_2,-\xi_2,\xi)}\rho^1(t,\xi_1)\r)\overline{\rho^{\times}(t,-\xi_2)}\rmd\xi_1\rmd\xi_2\\+\int_{\R^d}\int_{\R^d}\l(\rho^1(t,\xi)Q^0(\xi_1,\xi_2,-\xi_1,-\xi_2,\xi)\rho^0(t,\xi_1)-\rho^0(t,\xi)\overline{Q^1(\xi_1,\xi_2,-\xi_1,-\xi_2,\xi)}\rho^1(t,\xi_1)\r)\overline{\rho^{\times}(t,-\xi_2)}\rmd\xi_1\rmd\xi_2\Bigg],\\
        \l(\rho^0(0,\xi),\rho^1(0,\xi),\rho^\times(0,\xi)\r) = \l(M^{0,0}(\xi),M^{1,1}(\xi),M^{0,1}(\xi)\r).
    \end{cases}
\end{equation}
We will prove in \cref{where local well-posedness is proven} that there exists a small enough $\delta>0$, depending on $M^{\eta,\eta'}$, such that there exists a unique local solution $\l(\rho^\eta,\rho^\times\r)\in\mathcal C\l([0,\delta],\l(W^{1,\infty}\cap W^{1,1}\r)\l(\R^d\r)\r)^2$ of \eqref{dee} on the time interval $[0,\delta]$.

\subsubsection{The main result}
The main result of this manuscript is the rigorous derivation of \eqref{dee} over the existence interval $[0,\delta]$ as the limit of the averaged \eqref{system} dynamics under the scaling law $\epsilon=L^{-\frac{1}{\beta}}$ for $\frac{1}{\beta}\in(0,1)$.

\begin{theorem}\label{this is the main theorem}
    Let $d\geq2$, $s>\frac{d}{2}$ and $\beta\in(1,\infty)$. There exist $\delta,L_0,A_0>0$ such that for all $L\geq L_0$ and $A\geq A_0$:
    \begin{itemize}
        \item[(i)] There exists a set $\mathcal E_{L,A}$ of probability greater or equal to $1-L^{-A}$ such that if the initial data $f_{\ini}^\eta$ is taken from $\mathcal E_{L,A}$, \eqref{system} has a unique solution $\l(f^0,f^1\r)\in\mathcal C\l(\l[0,\delta L^{\frac{1}{\beta}}\r],H^s\l(\mathbb T_L^d\r)\r)^2$.
        \item[(ii)] There exists a unique solution $\l(\rho^0,\rho^1,\rho^\times\r)\in\mathcal C\l([0,\delta],\l(W^{1,\infty}\cap W^{1,1}\r)\l(\R^d\r)\r)^3$ of \eqref{dee}.
        \item[(iii)] We have for all $\eta\in\{0,1\}$,
        \begin{align}\label{first convergence}            \lim_{L\to\infty}\sup_{t\in[0,\delta]}\sup_{k\in\Z_L^d}\abs{\E\l(\vmathbb{1}_{\mathcal E_{L,A}}\abs{\widehat{f^\eta}\l(L^{\frac{1}{\beta}}t,k\r)}^2\r)-\rho^\eta(t,k)}&=0,\\           \lim_{L\to\infty}\sup_{t\in[0,\delta]}\sup_{k\in\Z_L^d}\abs{\E\l(\vmathbb1_{\mathcal E_{L,A}}\widehat{f^0}\l(L^{\frac{1}{\beta}}t,k\r)\overline{\widehat{f^1}\l(L^{\frac{1}{\beta}}t,k\r)}\r)-\rho^\times(t,k)}&=0.\label{second convergence}
        \end{align}
    \end{itemize}
\end{theorem}

\begin{remark}
    The convergence of \cref{first convergence,second convergence} is quantitative in the sence that there exists $C,\nu>0$ independent of $L$ such that 
    \begin{equation}
    \begin{gathered}
        \sup_{t\in[0,\delta]}\sup_{k\in\Z_L^d}\abs{\E\l(\vmathbb1_{\mathcal E_{L,A}}\widehat{f^0}\l(L^{\frac{1}{\beta}}t,k\r)\overline{\widehat{f^1}\l(L^{\frac{1}{\beta}}t,k\r)}\r)-\rho^\times(t,k)}\\+\sup_{t\in[0,\delta]}\sup_{k\in\Z_L^d}\abs{\E\l(\vmathbb{1}_{\mathcal E_{L,A}}\abs{\widehat{f^\eta}\l(L^{\frac{1}{\beta}}t,k\r)}^2\r)-\rho^\eta(t,k)}\leq CL^{-\nu}.
    \end{gathered}
    \end{equation}  
\end{remark}

\subsubsection{Heuristic derivation of \eqref{dee}}\label{the heuristics}

We define $\widehat{G^\eta}(t,k)\coloneqq\widehat{F^\eta}\l(\epsilon^{-1}t,k\r)$ and 
\begin{align}
    \rho^\eta(t,k)&\coloneqq\E\l(\abs{\widehat{G^\eta}(t,k)}^2\r)=\E\l(\abs{\widehat{f^\eta}\l(\epsilon^{-1}t,k\r)}^2\r),\label{the first limit}\\
    \rho^{\times,\eta}(t,k)&\coloneqq\E\l(\widehat{G^\eta}(t,k)\overline{\widehat{G^{\overline\eta}}(t,k)}\r)=\E\l(\widehat{f^\eta}\l(\epsilon^{-1}t,k\r)\overline{\widehat{f^{\overline\eta}}\l(\epsilon^{-1}t,k\r)}\r).\label{the second limit}
\end{align}

Note, $\rho^{\times,\overline\eta}=\overline{\rho^{\times,\eta}}$.

By definition 
\begin{equation}
    \widehat{G^\eta}(t,k) = \widehat{F^\eta}_\ini(k)-\frac{\rmi}{L^{2d}}\int_0^t\sum_{\sum_{i=1}^5=k}Q^\eta(k_1,\ldots,k_5)e^{\rmi\epsilon^{-1}\tau\Omega}\widehat{G^\eta}(\tau,k_1)\widehat{\overline{G^\eta}}(\tau,k_2)\widehat{\overline{G^\eta}}(\tau,k_3)\widehat{G^{\overline\eta}}(\tau,k_4)\widehat{G^{\overline\eta}}(\tau,k_5)\rmd\tau.
\end{equation}
Since we are interested in the large-box limit $L\to\infty$ of \cref{the first limit,the second limit}, the dynamics of $\widehat{G^\eta}$ is led by the exact resonances $\{\Omega=0\}$. We may assume heuristically 
\begin{equation}
    \p_t\widehat{G^\eta}(t,k) = -\frac{\rmi}{L^{2d}}\sum_{\substack{\sum_{i=1}^5=k\\\Omega=0}}Q^\eta(k_1,\ldots,k_5)\widehat{G^\eta}(t,k_1)\widehat{\overline{G^\eta}}(t,k_2)\widehat{\overline{G^\eta}}(t,k_3)\widehat{G^{\overline\eta}}(t,k_4)\widehat{G^{\overline\eta}}(t,k_5).
\end{equation}
Trivial resonances are (these six sets correspond to the six drawings in \cref{all possible resonant roots}) 
$\{k_2+k_4=k_3+k_5=0\text{ and }k_1=k\}\cup\{k_2+k_5=k_3+k_4=0\text{ and }k_1=k\}\cup\{k_1+k_2=k_3+k_5=0\text{ and }k_4=k\}\cup\{k_1+k_3=k_2+k_5=0\text{ and }k_4=k\}\cup\{k_1+k_2=k_3+k_4=0\text{ and }k_5=k\}\cup\{k_1+k_3=k_2+k_4=0\text{ and }k_5=k\}$ so that 
\begin{equation}
\begin{gathered}
    \p_t\widehat{G^\eta}(t,k) = -\frac{\rmi}{L^{2d}}\Bigg[\widehat{G^{\eta}}(t,k)\sum_{k_1,k_2\in\Z_L^d}Q^\eta(k,k_1,k_2,-k_1,-k_2)\widehat{\overline{G^\eta}}(t,k_1)\widehat{\overline{G^\eta}}(t,k_2)\widehat{G^{\overline\eta}}(t,-k_1)\widehat{G^{\overline\eta}}(t,-k_2)\\+\widehat{G^{\eta}}(t,k)\sum_{k_1,k_2\in\Z_L^d}Q^\eta(k,k_1,k_2,-k_2,-k_1)\widehat{\overline{G^\eta}}(t,k_1)\widehat{\overline{G^\eta}}(t,k_2)\widehat{G^{\overline\eta}}(t,-k_2)\widehat{G^{\overline\eta}}(t,-k_1)\\+\widehat{G^{\overline\eta}}(t,k)\sum_{k_1,k_2\in\Z_L^d}Q^\eta(k_1,-k_1,k_2,k,-k_2)\widehat{{G^\eta}}(t,k_1)\widehat{\overline{G^\eta}}(t,-k_1)\widehat{\overline{G^{\eta}}}(t,k_2)\widehat{G^{\overline\eta}}(t,-k_2)\\+\widehat{G^{\overline\eta}}(t,k)\sum_{k_1,k_2\in\Z_L^d}Q^\eta(k_1,k_2,-k_1,k,-k_2)\widehat{{G^\eta}}(t,k_1)\widehat{\overline{G^\eta}}(t,k_2)\widehat{\overline{G^{\eta}}}(t,-k_1)\widehat{G^{\overline\eta}}(t,-k_2)\\+\widehat{G^{\overline\eta}}(t,k)\sum_{k_1,k_2\in\Z_L^d}Q^\eta(k_1,-k_1,k_2,-k_2,k)\widehat{{G^\eta}}(t,k_1)\widehat{\overline{G^\eta}}(t,-k_1)\widehat{\overline{G^{\eta}}}(t,k_2)\widehat{G^{\overline\eta}}(t,-k_2)\\+\widehat{G^{\overline\eta}}(t,k)\sum_{k_1,k_2\in\Z_L^d}Q^\eta(k_1,k_2,-k_1,-k_2,k)\widehat{{G^\eta}}(t,k_1)\widehat{\overline{G^\eta}}(t,k_2)\widehat{\overline{G^{\eta}}}(t,-k_1)\widehat{G^{\overline\eta}}(t,-k_2)\Bigg]
\end{gathered}
\end{equation}
Now using Isserlis' theorem (see \cref{Isserlis}) and the fact that in these heuristic calculations we have the large-box limit $L\to\infty$ in mind, we find 
\begin{equation}
\begin{aligned}\label{pure effective}
    \p_t\rho^\eta(t,k) &= 2\Re\E\l(\p_t\widehat{G^\eta}(t,k)\overline{\widehat{G^\eta}(t,k)}\r),\\
    &= \frac{2}{L^{2d}}\Bigg[\rho^\eta(t,k)\sum_{k_1,k_2\in\Z_L^d}\Im\l(Q^\eta(k,k_1,k_2,-k_1,-k_2)\overline{\rho^{\times,\eta}(t,-k_1)}\overline{\rho^{\times,\eta}(t,-k_2)}\r) \\&+ \rho^\eta(t,k)\sum_{k_1,k_2\in\Z_L^d}\Im\l(Q^\eta(k,k_1,k_2,-k_2,-k_1)\overline{\rho^{\times,\eta}(t,-k_1)}\overline{\rho^{\times,\eta}(t,-k_2)}\r)\\&+\sum_{k_1,k_2\in\Z_L^d}\Im\l(\overline{\rho^{\times,\eta}(t,k)}{Q^\eta(k_1,-k_1,k_2,k,-k_2)}\rho^\eta(t,k_1)\overline{\rho^{\times,\eta}(t,-k_2)}\r)\\&+\sum_{k_1,k_2\in\Z_L^d}\Im\l(\overline{\rho^{\times,\eta}(t,k)}{Q^\eta(k_1,k_2,-k_1,k,-k_2)}\rho^\eta(t,k_1)\overline{\rho^{\times,\eta}(t,-k_2)}\r)\\&+\sum_{k_1,k_2\in\Z_L^d}\Im\l(\overline{\rho^{\times,\eta}(t,k)}Q^\eta(k_1,-k_1,k_2,-k_2,k)\rho^\eta(t,k_1)\overline{\rho^{\times,\eta}(t,-k_2)}\r)\\&+\sum_{k_1,k_2\in\Z_L^d}\Im\l(\overline{\rho^{\times,\eta}(t,k)}Q^\eta(k_1,k_2,-k_1,-k_2,k)\rho^\eta(t,k_1)\overline{\rho^{\times,\eta}(t,-k_2)}\r)\Bigg].
\end{aligned}
\end{equation}
Similarly

\begin{equation}\label{mixed effective}
    \begin{gathered}
        \p_t\rho^\times(t,k)=\E\l(\p_t\widehat{G^0}(t,k)\overline{\widehat{G^1}(t,k)}\r)+\E\l(\widehat{G^0}(t,k)\p_t\overline{\widehat{G^1}(t,k)}\r)\\
        -\frac{\rmi}{L^{2d}}\Bigg[\rho^\times(t,k)\sum_{k_1,k_2\in\Z_L^d}\l(Q^0(k,k_1,k_2,-k_1,-k_2)-\overline{Q^1(k,k_1,k_2,-k_1,-k_2)}\r)\overline{\rho^{\times}(t,-k_1)}\overline{\rho^{\times}(t,-k_2)}\\+\rho^\times(t,k)\sum_{k_1,k_2\in\Z_L^d}\l(Q^0(k,k_1,k_2,-k_2,-k_1)-\overline{Q^1(k,k_1,k_2,-k_2,-k_1)}\r)\overline{\rho^{\times}(t,-k_1)}\overline{\rho^{\times}(t,-k_2)}\\+\sum_{k_1,k_2\in\Z_L^d}\l(\rho^1(t,k)Q^0(k_1,-k_1,k_2,k,-k_2)\rho^0(t,k_1)-\rho^0(t,k)\overline{Q^1(k_1,-k_1,k_2,k,-k_2)}\rho^1(t,k_1)\r)\overline{\rho^{\times}(t,-k_2)}\\\+\sum_{k_1,k_2\in\Z_L^d}\l(\rho^1(t,k)Q^0(k_1,k_2,-k_1,k,-k_2)\rho^0(t,k_1)-\rho^0(t,k)\overline{Q^1(k_1,k_2,-k_1,k,-k_2)}\rho^1(t,k_1)\r)\overline{\rho^{\times}(t,-k_2)}\\+\sum_{k_1,k_2\in\Z_L^d}\l(\rho^1(t,k)Q^0(k_1,-k_1,k_2,-k_2,k)\rho^0(t,k_1)-\rho^0(t,k)\overline{Q^1(k_1,-k_1,k_2,-k_2,k)}\rho^1(t,k_1)\r)\overline{\rho^{\times}(t,-k_2)}\\+\sum_{k_1,k_2\in\Z_L^d}\l(\rho^1(t,k)Q^0(k_1,k_2,-k_1,-k_2,k)\rho^0(t,k_1)-\rho^0(t,k)\overline{Q^1(k_1,k_2,-k_1,-k_2,k)}\rho^1(t,k_1)\r)\overline{\rho^{\times}(t,-k_2)}\Bigg]
    \end{gathered}
\end{equation}

We take the limit
\begin{equation}
    \lim_{L\to\infty}\frac{1}{L^{2d}}\sum_{k_1,k_2\in\Z_L^d} = \int_{\R^d}\int_{\R^d}\rmd\xi_1\rmd\xi_2
\end{equation}
in \cref{pure effective,mixed effective} and justify its validity in \cref{the large box limit section}. 

\begin{equation}\label{the expected heuristic equation}
    \begin{gathered}
        \p_t\rho^\eta(t,\xi) 
    = 2\Bigg[\rho^\eta(t,\xi)\int_{\R^d}\int_{\R^d}\Im\l(Q^\eta(\xi,\xi_1,\xi_2,-\xi_1,-\xi_2)\overline{\rho^{\times,\eta}(t,-\xi_1)}\overline{\rho^{\times,\eta}(t,-\xi_2)}\r) \rmd\xi_1\rmd\xi_2\\+ \rho^\eta(t,\xi)\int_{\R^d}\int_{\R^d}\Im\l(Q^\eta(\xi,\xi_1,\xi_2,-\xi_2,-\xi_1)\overline{\rho^{\times,\eta}(t,-\xi_1)}\overline{\rho^{\times,\eta}(t,-\xi_2)}\r)\rmd\xi_1\rmd\xi_2\\+\int_{\R^d}\int_{\R^d}\Im\l(\overline{\rho^\times(t,\xi)}{Q^\eta(\xi_1,-\xi_1,\xi_2,\xi,-\xi_2)}\overline{\rho^{\times,\eta}(t,-\xi_2)}\r)\rho^\eta(t,\xi_1)\rmd\xi_1\rmd\xi_2\\+\int_{\R^d}\int_{\R^d}\Im\l(\overline{\rho^\times(t,\xi)}{Q^\eta(\xi_1,\xi_2,-\xi_1,\xi,-\xi_2)}\overline{\rho^{\times,\eta}(t,-\xi_2)}\r)\rho^\eta(t,\xi_1)\rmd\xi_1\rmd\xi_2\\+\int_{\R^d}\int_{\R^d}\Im\l(\overline{\rho^\times(t,\xi)}Q^\eta(\xi_1,-\xi_1,\xi_2,-\xi_2,\xi)\overline{\rho^{\times,\eta}(t,-\xi_2)}\r)\rho^\eta(t,\xi_1)\rmd\xi_1\rmd\xi_2\\+\int_{\R^d}\int_{\R^d}\Im\l(\overline{\rho^\times(t,\xi)}Q^\eta(\xi_1,\xi_2,-\xi_1,-\xi_2,\xi)\overline{\rho^{\times,\eta}(t,-\xi_2)}\r)\rho^\eta(t,\xi_1)\rmd\xi_1\rmd\xi_2\Bigg],\\
    \p_t\rho^\times(t,\xi)=
        -\rmi\Bigg[\rho^\times(t,\xi)\int_{\R^d}\int_{\R^d}\l(Q^0(\xi,\xi_1,\xi_2,-\xi_1,-\xi_2)-\overline{Q^1(\xi,\xi_1,\xi_2,-\xi_1,-\xi_2)}\r)\overline{\rho^{\times}(t,-\xi_1)}\overline{\rho^{\times}(t,-\xi_2)}\rmd\xi_1\rmd\xi_2\\+\rho^\times(t,\xi)\int_{\R^d}\int_{\R^d}\l(Q^0(\xi,\xi_1,\xi_2,-\xi_2,-\xi_1)-\overline{Q^1(\xi,\xi_1,\xi_2,-\xi_2,-\xi_1)}\r)\overline{\rho^{\times}(t,-\xi_1)}\overline{\rho^{\times}(t,-\xi_2)}\rmd\xi_1\rmd\xi_2\\+\int_{\R^d}\int_{\R^d}\l(\rho^1(t,\xi)Q^0(\xi_1,-\xi_1,\xi_2,\xi,-\xi_2)\rho^0(t,\xi_1)-\rho^0(t,\xi)\overline{Q^1(\xi_1,-\xi_1,\xi_2,\xi,-\xi_2)}\rho^1(t,\xi_1)\r)\overline{\rho^{\times}(t,-\xi_2)}\rmd\xi_1\rmd\xi_2\\+\int_{\R^d}\int_{\R^d}\l(\rho^1(t,\xi)Q^0(\xi_1,\xi_2,-\xi_1,\xi,-\xi_2)\rho^0(t,\xi_1)-\rho^0(t,\xi)\overline{Q^1(\xi_1,\xi_2,-\xi_1,\xi,-\xi_2)}\rho^1(t,\xi_1)\r)\overline{\rho^{\times}(t,-\xi_2)}\rmd\xi_1\rmd\xi_2\\+\int_{\R^d}\int_{\R^d}\l(\rho^1(t,\xi)Q^0(\xi_1,-\xi_1,\xi_2,-\xi_2,\xi)\rho^0(t,\xi_1)-\rho^0(t,\xi)\overline{Q^1(\xi_1,-\xi_1,\xi_2,-\xi_2,\xi)}\rho^1(t,\xi_1)\r)\overline{\rho^{\times}(t,-\xi_2)}\rmd\xi_1\rmd\xi_2\\+\int_{\R^d}\int_{\R^d}\l(\rho^1(t,\xi)Q^0(\xi_1,\xi_2,-\xi_1,-\xi_2,\xi)\rho^0(t,\xi_1)-\rho^0(t,\xi)\overline{Q^1(\xi_1,\xi_2,-\xi_1,-\xi_2,\xi)}\rho^1(t,\xi_1)\r)\overline{\rho^{\times}(t,-\xi_2)}\rmd\xi_1\rmd\xi_2\Bigg]
    \end{gathered}
\end{equation}

\subsection{Ingredients of the proof}

Our proof of \cref{this is the main theorem} combines the strategies in \cite{desuzzoni2025waveturbulencesemilinearkleingordon} and \cite{Collot2025} and is essentially structured in one preparatory and three crucial steps. The idea is to take the ansatz for a solution to \eqref{system} as a Dyson series expansion. More precisely, 
\begin{equation}\label{ansatz}
    f^\eta = f_0^\eta+\cdots f_{N(L)}^\eta+v^\eta
\end{equation}
and let $N(L)=\lfloor\log(L)\rfloor$ diverge as $L\to\infty$. The main issue is to prove the existence and uniqueness of the remainder term $v^\eta$ for sufficiently large $L$. Each element $f_n$ of the Dyson series is a finite sum over Feynman diagrams. This is shown in \cref{second section}. The correlations of the dyson iterates are then reformulated as finite sums over the set of couples whose cardinality scales factorially in the number of branching nodes. This factorial dependence is the main obstruction in proving the existence and uniqueness of a solution to \eqref{system} on the time interval $\l[0,\delta L^{\frac{1}{\beta}}\r]$ and to prove the convergence of the correlations $\E\l(\vmathbb1_{\mathcal E_{L,A}}\abs{\widehat{f^\eta}}^2\r)$ and $\E\l(\vmathbb1_{\mathcal E_{L,A}}{\widehat{f^\eta}}\overline{\widehat{f^{\overline\eta}}}\r)$ to the solution $\l(\rho^0,\rho^1,\rho^\times\r)$ of the resonant system \eqref{dee} as $L\to\infty$.

In \cref{third section} we make the ansatz \eqref{ansatz} rigorous and prove that for a large enough $L$, the iterates of the Dyson series $f_n^\eta$ up to $N(L)$ and the remainder $v^\eta$ satisfy bounds so that $v^\eta$ may be obtained as a fixed point of a contraction map in a closed ball inside the Banach space $\mathcal C\l([0,\delta],H^s\l(\mathbb T_L^d\r)\r)^2$. This proves the existence and uniqueness of a solution to \eqref{system} on $[0,\delta L^{\frac{1}{\beta}}]$ for any $\beta\in(1,\infty)$. The main difficulty of this section is to obtain smallness in the form of positive powers of $\epsilon$. The idea is to adapt a key result of Lukkarinen and Spohn \cite{Lukkarinen_2010} to $5$-ary trees. This adaptation allows us to formulate a coordinate transformation with which we may restate a sum over all $k$-decorations of a product over all nodes as a product of sums that are structured according to a particular order relation. The sums are estimated by integrals over bounded domains. These integrals enable us to access positive powers of $\epsilon$ and gain smallness.

Motivated by the heuristic derivation of the macroscopic system in \cref{the heuristics}, the local time well-posedness of \eqref{the expected heuristic equation} is almost immediately given and delivers in \cref{fourth section} a $\delta>0$ so that a solution exists uniquely on $[0,\delta]$. We then construct a sequence of functions $\rho_n$ and prove that the series $\l(\sum_{n\leq m}\rho_n\r)_{m\in\N}$ converges to the solution $\rho=\l(\rho^0,\rho^1,\rho^\times\r)\in\mathcal C\l([0,\delta],\l(W^{1,\infty}\cap W^{1,1}\r)\l(\R^d\r)\r)^3$. 

\textit{Resonant nodes} are nodes at which the resonance factor vanishes identically.
In \cref{fifth section}, we prove that in the large-box limit $L\to\infty$, the only contributions to the correlations of the microscopic system \eqref{system} come from \textit{resonant couples}, which are couples whose nodes are all resonant, and we identify the resonant structure through ternary trees. This allows us to make the convergence of the correlations of \eqref{system} to the solution of the macroscopic system \eqref{dee} systematic, that is, iterate by iterate. 

\begin{remark}
    A comment on different non-linearities is in order. If the non-linearity was of degree $2k+1$ and one could harmonize the sign rule of trees with what it means to be a {resonant node}, the right description to capture the recursive structure of the resonant system would be via $(k+1)$-ary trees (see \cref{for general non-linearity and resonant structure}). This statement foreshadows the content of \cref{recursive structure for subcouples} for $k=2$.

    The following analysis works for any coupled Schrödinger equation with odd non-linearity by adding more cases to \cref{main technical necessary theorem}. In the case of a cubic non-linearity, only the cases $i\in\{1,2\}$ are relevant. 
\end{remark}

\section{Feynman diagrams and couples}\label{second section}

We iteratively define a series $\l(F_n^\eta\r)_{n\in\N}$ by  
\begin{equation}
    \begin{aligned}\label{Dyson series}
        \widehat{F_0^\eta}&\coloneqq\mu^\eta_k,\\
        \widehat{F_{n+1}^\eta}(t,k)&\coloneqq-\frac{\rmi\epsilon}{L^{2d}}\int_0^t\sum_{\substack{\sum_{i=1}^5n_i=n\\\sum_{i=1}^5k_i=k}}e^{\rmi\tau\Omega}Q^\eta(k_1,\ldots,k_5)\widehat{F_{n_1}^\eta}(\tau,k_1)\widehat{\overline{F_{n_2}^\eta}}(\tau,k_2)\widehat{\overline{F_{n_2}^\eta}}(\tau,k_3)\widehat{F_{n_4}^{\overline\eta}}(\tau,k_4)\widehat{F_{n_5}^{\overline\eta}}(\tau,k_5)\rmd\tau
    \end{aligned}
\end{equation}
for $n>0$.

\begin{remark}
    One can show that \begin{equation}
        \sum_{n\geq0}F_n^\eta
    \end{equation}formally solves \eqref{reformulation in momentum space}.  
\end{remark}

More rigorously, we make the Ansatz
    \begin{equation}
        F^\eta=F^\eta_{\leq N}+v^\eta,
    \end{equation}
    where $F^\eta_{\leq N}\coloneqq\sum_{n\leq N}F_n^\eta$ and the existence and uniqueness of $v^\eta$ will be dealt with a fixed point argument later in \cref{the fixed point argument}.

\begin{definition}
    We define for $\iota\in\{\pm\}$ the operator
    \begin{equation}
        \reallywidehat{C^{\iota}(t,f_1,\ldots,f_5)}(k)\coloneqq-\frac{\rmi\iota\epsilon}{L^{2d}}\sum_{\sum_{i=1}^5k_i=k}e^{\rmi\iota t\Omega}Q^\eta_\iota(k_1,\ldots,k_5)\prod_{i=1}^5\hat f_i(k_i).
    \end{equation}
\end{definition}

\begin{remark} We make the following observations. 
    \begin{itemize}
        \item First,
        \begin{equation}\label{useless property}
            C^{-}(t,f_1,\ldots,f_5) = \overline{C^{+}\l(t,\overline f_1,\ldots,\overline{f_5}\r)}.
        \end{equation}
        \item Second, the fixed point equation \eqref{reformulation in momentum space} that we would like to solve rewrites as 
        \begin{equation}
            F^\eta(t) = F^\eta_\ini + \int_0^t C^{+}\l(\tau,F^\eta(\tau),\overline{F^\eta}(\tau),\overline{F^\eta}(\tau),F^{\overline\eta}(\tau),F^{\overline\eta}(\tau)\r)\rmd\tau.
        \end{equation}
        
        The elements $F_{n+1}^\eta$ of the Dyson series \eqref{Dyson series} for $n\geq0$ can be rewritten as 
        \begin{equation}
            F_{n+1}^\eta(t) = \int_0^t\sum_{\substack{\sum_{i=1}^5n_i = n}}C^{+}\l(\tau,F_{n_1}^\eta(\tau),\overline{F_{n_2}^\eta}(\tau),\overline{F_{n_3}^\eta}(\tau),F_{n_4}^{\overline\eta}(\tau),F_{n_5}^{\overline\eta}(\tau)\r)\rmd\tau.
        \end{equation}
    \end{itemize}
\end{remark}

\subsection{Reformulation in terms of signed and coloured trees}
\begin{definition}
    We define \textbf{recursively rooted 5-ary trees} $T$ with \textbf{sign} $\iota$ and \textbf{colour} $\eta$ and denote by $(\perp,\iota,\eta)$ the trivial tree if the tree only consists of its root. For this trivial tree, the notions of root and leaf coincide. We define iteratively for all $n\in\N$,
    \begin{equation}
    \begin{aligned}
        \mathcal T_0^{\iota,\eta}&\coloneqq\{(\perp,\iota,\eta)\},\\
        \mathcal T_{n+1}^{\iota,\eta}&\coloneqq\l\{\bullet(T_1,T_2,T_3,T_4,T_5)\mid\l(T_i\r)_{i\in\llbracket1,5\rrbracket}\in\mathcal T_{n_1}^{\iota,\eta}\times \mathcal T_{n_2}^{-\iota,\eta}\times \mathcal T_{n_3}^{-\iota,\eta}\times \mathcal T_{n_4}^{\iota,\overline\eta}\times \mathcal T_{n_5}^{\iota,\overline\eta}\text{ and }\sum_{i=1}^5n_i=n\r\},
    \end{aligned}
    \end{equation}    
    where $\bullet(T_1,T_2,T_3,T_4,T_5)$ defines the operation where the individual roots of the five rooted trees $T_i$ are connected to a common root $\bullet$. For $T\in\mathcal T_n^{\iota,\eta}$, we denote $\mathcal N(T)$ as the set of (branching) nodes of $T$ and $\mathcal L(T)$ as the set of leaves of $T$. 
\end{definition}

\begin{figure}[H]
    \centering
    \includegraphics[width=0.5\linewidth]{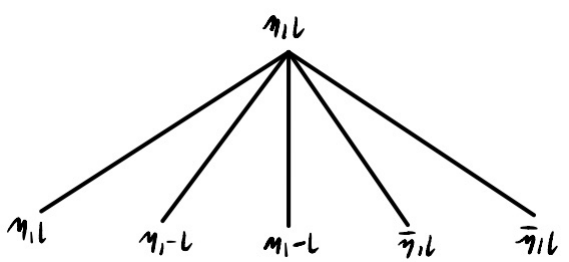}
    \caption{Convention for the nonlinearity in \eqref{system}}
    \label{fig:placeholder}
\end{figure}

\begin{remark}
    The rules for signs and colours of a branching node $\node$ that has sign $\iota$ and colour $\eta$ is depicted in figure \ref{fig:placeholder}.
    If $T\in\mathcal \bigsqcup_{\substack{\iota\in\{\pm\}\\\eta\in\{0,1\}}}T_n^{\iota,\eta}$, then one proves inductively $\abs{\mathcal L(T)} = 4n+1$. We have $5n+1$ vertices in total, so that
    \begin{equation}
        \abs{\bigsqcup_{\substack{\iota\in\{\pm\}\\\eta\in\{0,1\}}}T_n^{\iota,\eta}}\leq \binom{5n+1}{n}\leq(10e)^n.
    \end{equation}
    That is, the cardinality of $\T_n^{\iota,\eta}$ grows as $\Lambda^n$ for some $\Lambda>0$.
\end{remark}

\begin{definition}
    For $T\in\mathcal T_n^{\iota,\eta}$ and $\mathscr n\in\mathcal N(T)$, we define $C(\mathscr n)$ to be the set of the children of $\mathscr n$. 
    For $\mathscr n\in(\mathcal N(T)\cup\mathcal L(T))\setminus\{r_T\}$, where $r_T$ denotes the root of $T$, we define $P(\mathscr n)$ to be the parent of $\mathscr n$. We denote the siblings set as $S(\mathscr n)\coloneqq C(P(\mathscr n))$.
\end{definition}

\begin{definition}\label{parentality order}
    Let $T\in\mathcal T_n^{\iota,\eta}$. For $\mathscr n,\mathscr n'\in\mathcal N(T)$ we define the \textbf{parentality partial order} $\mathscr n<\mathscr n'$ if and only if there exists a finite sequence of nodes $(\mathscr n_k)_{k=1}^m$ with $\mathscr n_1 = \mathscr n$, $\mathscr n_m = \mathscr n'$ and $P(\node_{k-1}) =\node_k$ for all $k\in\llbracket2,m\rrbracket$. If $\mathscr l\in\mathcal L(T)$, we say $\leaf<\node$ if and only if $P(\leaf)\leq\node$. 
\end{definition}

\begin{definition}\label{representation in terms of trees}
    For each $T\in \mathcal T_n^{\iota,\eta}$, we iteratively define the function 
    \begin{equation}
        \J_{T}(t)\coloneqq\begin{cases}
            f_\ini^{\eta,\iota}&\text{ for }n=0,\\
            \int_0^t C^{\iota}\l(\tau,\J_{T_1}(\tau),\J_{T_2}(\tau),\J_{T_3}(\tau),\J_{T_4}(\tau),\J_{T_5}(\tau)\r)\rmd\tau&\text{ for }n\geq1,
        \end{cases}
    \end{equation}
    where $T_1\in\mathcal T_{n_1}^{\iota,\eta}$, $T_2\in\mathcal T_{n_2}^{-\iota,\eta}$, $T_3\in\mathcal T_{n_3}^{-\iota,\eta}$, $T_4\in\mathcal T_{n_4}^{\iota,\overline\eta}$ and $T_5\in\mathcal T_{n_5}^{\iota,\overline\eta}$, $n_1+\cdots+n_5=n-1$ and $\bullet(T_1,T_2,T_3,T_4,T_5)=T$ in the case $n>0$.
\end{definition}

\begin{lemma}
    Denoting $T\in \mathcal T_n^{\iota,\eta}$ more explicitly $(T,\iota,\eta)$, where $\iota$ and $\eta$ denote the sign and colour of the root of $T$, then for all $\iota'\in\{\pm\}$, we have $\J_{(T,\iota,\eta)}^{\iota'} = \J_{\l(T',\iota'\iota,\eta\r)}$.
\end{lemma}
\begin{proof}
    The statement can be proven quiet quickly via induction over the scale $n$ of $T$. If $n=0$ the statement is clear by definition \cref{representation in terms of trees}. Now suppose the statement holds for some $n\geq0$ and take $(T,\iota,\eta)\in\mathcal T_{n+1}^{\iota,\eta}$. If $\iota'=+$ there is nothing to prove so assume $\iota'=-$. Using the property \eqref{useless property}, we find
    \begin{equation}
    \begin{aligned}
        \overline{\J_{(T,\iota,\eta)}(t)}&=\int_0^t \overline{C^{\iota}\l(\tau,\J_{(T_1,\iota,\eta)}(\tau),\J_{(T_2,-\iota,\eta)}(\tau),\J_{(T_3,-\iota,\eta)}(\tau),\J_{(T_4,\iota,\overline\eta)}(\tau),\J_{(T_5,\iota,\overline\eta)}(\tau)\r)}\rmd\tau \\
        &=\int_0^t C^{-\iota}\l(\tau,\overline{\J_{(T_1,\iota,\eta)}(\tau)},\overline{\J_{(T_2,-\iota,\eta)}(\tau)},\overline{\J_{(T_3,-\iota,\eta)}(\tau)},\overline{\J_{(T_4,\iota,\overline\eta)}(\tau)},\overline{\J_{(T_5,\iota,\overline\eta)}(\tau)}\r)\rmd\tau\\&=\int_0^t C^{-\iota}\l(\tau,\J_{(T_1,-\iota,\eta)}(\tau),\J_{(T_2,\iota,\eta)}(\tau),\J_{(T_3,\iota,\eta)}(\tau),\J_{(T_4,-\iota,\overline\eta)}(\tau),\J_{(T_5,-\iota,\overline\eta)}(\tau)\r)\rmd\tau \\&= \J_{(T,-\iota,\eta)}(t),
    \end{aligned}
    \end{equation}
    where we used the induction step in the third equality sign. 
\end{proof}
The elements $F_n^{\eta,\iota}$ of the Dyson series can be written as fintie sums over signed and coloured trees in the following way.
\begin{lemma}
    The elements of the Dyson series $F_n^{\eta,\iota}$ can be represented by a finite sum over signed and coloured trees:
    \begin{equation}
        F_n^{\eta,\iota} = \sum_{T\in\mathcal T_n^{\eta,\iota}}\J_{T}.
    \end{equation}
\end{lemma}
\begin{proof}
    We prove the assertion by induction over the scale $n$ of signed and coloured trees. If $n=0$ the statement is obvious. To abbreviate notation, we set $\T_1=\T^{\eta,\iota}$, $\T_2=\T^{\eta,-\iota}$, $\T_3=\T^{\eta,-\iota}$, $\T_4=\T^{\overline\eta,\iota}$ and $\T_5=\T^{\overline\eta,\iota}$ and assume the claim holds for some $n\geq0$. Using the induction hypothesis, 
    \begin{equation}
    \begin{aligned}
        F_{n+1}^{\eta,\iota}(t)&=\int_0^t\sum_{\substack{\sum_{i=1}n_i=n}}C^{\iota}\l(\tau,F_{n_1}^{\eta,\iota}(\tau),F_{n_2}^{\eta,-\iota}(\tau),F_{n_3}^{\eta,-\iota}(\tau),F_{n_4}^{\overline\eta,\iota}(\tau),F_{n_5}^{\overline\eta,\iota}(\tau)\r)\rmd\tau
        \\&=\sum_{\sum_{i=1}n_i=n}\sum_{\substack{\l(T_i\r)_{i=1}^5\in\bigtimes_{i=1}^5\l(\T_i\r)_{n_i}}}\int_0^tC^{\iota}\l(\tau,\J_{T_1}(\tau),\J_{T_2}(\tau),\J_{T_3}(\tau),\J_{T_4}(\tau),\J_{T_5}(\tau)\r)\rmd\tau\\
        &=\sum_{\sum_{i=1}n_i=n}\sum_{\substack{\l(T_i\r)_{i=1}^5\in\bigtimes_{i=1}^5\l(\T_i\r)_{n_i}}}\J_{\bullet\l(T_1,T_2,T_3,T_4,T_5\r)}(t) \\&= \sum_{T\in\mathcal T_{n+1}^{\eta,\iota}}\J_{T}(t),
    \end{aligned}
    \end{equation}
    where in the last equality we used the fact that if $G$ is any function on $\mathcal T_{n+1}^{\eta,\iota}$, we have
    \begin{equation}
        \sum_{\sum_{i=1}^5n_i=n}\sum_{\substack{\l(T_i\r)_{i=1}^5\in\bigtimes_{i=1}^5\l(\T_i\r)_{n_i}}} G\l(\bullet\l(T_1,T_2,T_3,T_4,T_5\r)\r) = \sum_{T\in\mathcal T_{n+1}^{\eta,\iota}}G(T).
    \end{equation}
\end{proof}

\begin{proposition}\label{Fourier transform of J}
    For all $\iota\in\{\pm\}$, $\eta\in\{0,1\}$, $T\in\T_n^{\eta,\iota}$, $t\in\R$ and $k\in\Z_L^d$, the Fourier transform of $\J_{T}$ reads
    \begin{equation}
    \begin{gathered}
        \widehat{\J_{T}}(t,k) = \l(-\frac{\rmi\epsilon}{L^{2d}}\r)^n\prod_{\node\in\Node(T)}\iota_\node\sum_{\kappa\in\mathcal D_k(T)}\prod_{\node\in\Node(T)}Q_\node^T(\kappa)\int_{I_T(t)}\prod_{\node\in\Node(T)}e^{\rmi\iota_\node\Omega^T_\node(\kappa)t_\node}\rmd t_\node\prod_{\leaf\in\Leaf(T)}\mu^{\eta_\leaf,\iota_\leaf}_{\kappa(\leaf)},
    \end{gathered}
    \end{equation}
    where \begin{equation}
        \begin{aligned}
            I_T(t)&\coloneqq\l\{\l(t_\node\r)_{\node\in\Node(T)}\in[0,t]^n\mid t_\node\leq t_{\node'}\text{ if }\node\leq\node'\r\},\\
            Q_\node^T(\kappa)&\coloneqq Q_{\iota_\node}^{\eta_\node}\l(K_C\l(\node_1\r)(\kappa),\ldots,K_C\l(\node_5\r)(\kappa)\r)
        \end{aligned}
    \end{equation}
    for $C(\node)=\{\node_1,\ldots,\node_5\}$ (ordered children from left to right).
\end{proposition}
\begin{proof}
    The proof goes by induction over the scale $n$. For $n=0$, there is nothing to prove. We assume the statement to be true for $n-1\in\N$ and assume $T\in\T_n^{\eta,\iota}$ so that $T=\bullet\l(T_1,\ldots,T_5\r)$, $T_1\in\T_{n_1}^{\iota,\eta}$, $T_2\in\T_{n_2}^{-\iota,\eta}$, $T_3\in\T_{n_3}^{-\iota,\eta}$, $T_4\in\T_{n_4}^{\iota,\overline\eta}$ and $T_5\in\T_{n_5}^{\iota,\overline\eta}$ with $\sum_{i=1}^5n_i=n-1$. We get 
    \begin{equation}
    \begin{aligned}
        \widehat{\J_{T}}(t,k)&=-\frac{\rmi\iota\epsilon}{L^{2d}}\int_0^t\sum_{\sum_{i=1}^5k_i=k}e^{\rmi\iota\tau\Omega}Q^\eta_\iota(k_1,\ldots,k_5)\prod_{i=1}^5\widehat{\J_{T_i}}(\tau,k_i)\rmd\tau\\
        &=\l(-\frac{\rmi\epsilon}{L^{2d}}\r)^n\prod_{\node\in\Node(T)}\iota_\node\sum_{\substack{\sum_{i=1}^5k_i=k\\\kappa_i\mathcal D_{k_i}\forall i\in\llbracket1,5\rrbracket}}Q^\eta_\iota(k_1,\ldots,k_5)\\&\cdot\int_0^te^{\rmi\iota\tau\Omega}\l[\int_{\bigtimes_{i=1}^5I_{T_i}(\tau)}\prod_{i=1}^5\prod_{\node_i\in\Node(T_i)}e^{\rmi\iota_{\node_i}\Omega^{T_i}_{\node_i}(\kappa_i)t_{\node_i}}\rmd t_{\node_i}\r]\rmd\tau\\&\cdot\prod_{i=1}^5\prod_{\node_i\in\Node(T_i)}Q_{\node_i}^{T_i}(\kappa_i)\cdot\prod_{i=1}^5\prod_{\leaf_i\in\Leaf(T_i)}\mu^{\eta_{\leaf_i},\iota_{\leaf_i}}_{\kappa_i(\leaf_i)}.
    \end{aligned}
    \end{equation}
    Given $\kappa\in\mathcal D_K(T)$, $\left.\kappa\right|_{T_i}$ are $k_i$-decorations on $T_i$ where by $\left.\kappa\right|_{T_i}$ we mean the restriction of $\kappa$ to the subset of leaves of the subtree $T_i$. We use the canonical bijection 
    \begin{equation}
        \D_k(T)\cong\l\{(k_1,\ldots,k_5,\kappa_1,\ldots,\kappa_5)\mid k_1,\ldots,k_5\in\Z_L^d\text{ and }\sum_{i=1}^5k_i=k\text{ and }\kappa_i\in\D_{k_i}(T_i)\forall i\in\llbracket1,5\rrbracket\r\},
    \end{equation}
    the fact that $\Node(T)=\{\root_T\}\sqcup\bigsqcup_{i=1}^5\Node(T_i)$, and $\Leaf(T)=\bigsqcup_{i=1}^5\Leaf(T_i)$ to deduce  
    \begin{equation}
    \begin{aligned}
        \sum_{\substack{\sum_{i=1}^5k_i=k\\\kappa_i\in\mathcal D_{k_i}(T_i)\forall i\in\llbracket1,5\rrbracket}}&=\sum_{\kappa\in\mathcal D_k(T)},\\Q^\eta_\iota(k_1,\ldots,k_5)\prod_{i=1}^5\prod_{\node_i\in\Node(T_i)}Q_{\node_i}^{T_i}(\kappa_i) &= \prod_{\node\in\Node(T)}Q_\node^T(\kappa)\\\prod_{i=1}^5\prod_{\leaf_i\in\Leaf(T_i)} &= \prod_{\leaf\in\Leaf(T)}.
    \end{aligned}
    \end{equation}
    One can show by induction  
    \begin{equation}
        I_T(T) = \l\{\l(t_\node\r)_{\node\in\Node(T)}\in[0,t]^n\biggm\vert t_{\root_T}\in[0,t]\text{ and }\l(\l(t_\node\r)_{\node\in\Node(T_i)}\r)_{i=1}^5\in\bigtimes_{i=1}^5I_{T_i}\l(t_{\root_T}\r)\r\},
    \end{equation}
    so that 
    \begin{equation}
        \int_0^te^{\rmi\iota\Omega^T_{\root_T}(\kappa)t_{\root_T}}\int_{\bigtimes_{i=1}^5I_{T_i}(t_\root)}\prod_{i=1}^5\prod_{\node_i\in\Node(T_i)}e^{\rmi\iota_{\node_i}\Omega_{\node_i}^{T_i}(\kappa_i)t_{\node_i}}\rmd t_{\node_i}\rmd t_\root=\int_{I_T(t)}\prod_{\node\in\Node(T)}e^{\rmi\iota_\node\Omega^T_\node(\kappa)t_\node}\rmd t_\node 
    \end{equation}
    and thus
    \begin{equation}
    \begin{gathered}
        \widehat{\J_{T}}(t,k) = \l(-\frac{\rmi\epsilon}{L^{2d}}\r)^n\prod_{\node\in\Node(T)}\iota_\node\sum_{\kappa\in\mathcal D_k(T)}\prod_{\node\in\Node(T)}Q_\node^T(\kappa)\int_{I_T(t)}\prod_{\node\in\Node(T)}e^{\rmi\iota_\node\Omega^T_\node(\kappa)t_\node}\rmd t_\node\\\cdot\prod_{\leaf\in\Leaf(T)}\mu^{\iota_\leaf,\eta_\leaf}_{\kappa(\leaf)}.
    \end{gathered}
    \end{equation}
\end{proof}

\begin{remark}
    Rescaling in time to cancel the $\epsilon^n$ prefactor leads to
    \begin{equation}
    \begin{gathered}
        \widehat{\J_{T}}\l(\epsilon^{-1}t,k\r) = \l(-\frac{\rmi}{L^{2d}}\r)^n\prod_{\node\in\Node(T)}\iota_\node\sum_{\kappa\in\mathcal D_k(T)}\prod_{\node\in\Node(T)}Q_\node^T(\kappa)\\\cdot\int_{I_T(t)}\prod_{\node\in\Node(T)}e^{\rmi\epsilon^{-1}\iota_\node\Omega^T_\node(\kappa)t_\node}\rmd t_\node\prod_{\leaf\in\Leaf(T)}\mu^{\iota_\leaf,\eta_\leaf}_{\kappa(\leaf)}.
    \end{gathered}
    \end{equation}
\end{remark}

\subsection{Decorations and averaging over rooted trees}\label{some notation is introduced here}
To associate wave numbers to the nodes and leaves of a ternary tree and take subsequent expectations, we must introduce the notion of a coupling. 

\begin{definition}
    For each $T\in\T_n^{\iota,\eta}$, we call any element of $\l(\R^d\r)^{\Leaf(T)}$ a \textbf{decoration} of the tree $T$ and associate wave numbers to each node and leaf in the following way. Let $K_T\colon\Node(T)\cup\Leaf(T)\to L\l(\l(\R^d\r)^{\Leaf(T)},\R^d\r)$ be the map defined by 
    \begin{equation}
        K_T(\node)(\kappa)\coloneqq
            \sum_{\Leaf(T)\ni\leaf\leq\node}\kappa(\leaf).
    \end{equation}
    Given a decoration $\kappa\in\l(\R^d\r)^{\Leaf(T)}$ of the tree $T$, we say that the node or leaf $\node\in\Node(T)\cup\Leaf(T)$ has \textbf{wave number} $K_T(\node)(\kappa)$.
    For a fixed $k\in\Z_L^d$, we call a decoration $\kappa\in\l(\R^d\r)^{\Leaf(T)}$ a \textbf{$k$-decoration} if $\kappa(\Leaf(T))\subseteq\Z_L^d$ and $K(\root_T)(\kappa)=k$ where $\root_T$ denotes the root of $T$. We denote the subset of $k$-decorations by $\mathcal D_k(T)$.
    For each tree $T\in\T_n^{\eta,\iota}$ and node $\node\in\Node(T)$ we associate the resonance factor $\Omega^T_\node\colon\l(\R^d\r)^{\Leaf(T)}\to\R$ defined by
    \begin{equation}
        \Omega_\node^T(\kappa)\coloneqq \iota_\node\abs{K_T(\node)(\kappa)}^2-\sum_{\node'\in C(\node)}\iota_{\node'}\abs{K_T(\node')(\kappa)}^2.
    \end{equation}
    We also set \begin{align}
        Q_\node^T(\kappa)&\coloneqq Q^{\eta_\node}_{\iota_\node}\l(K_T(\node_1)(\kappa),\ldots,K_T(\node_5)(\kappa)\r),
    \end{align}
    where $\node_1,\ldots,\node_5$ denote the children of $\node$ from left to right. 
\end{definition}

\begin{definition}
    For any $T\in\T_n^{\eta,\iota}$ and $T'\in\T_{n'}^{\eta',\iota'}$, if \begin{equation}
        \abs{\{\ell\in\Leaf(T)\mid\iota_\ell=-\}\cup\{\ell\in\Leaf(T')\mid\iota'_\ell=-\}}=\abs{\{\ell\in\Leaf(T)\mid\iota_\ell=+\}\cup\{\ell\in\Leaf(T')\mid\iota'_\ell=+\}}\label{postive and negative leaves}
    \end{equation} and if there exists an involution $\sigma\colon\Leaf(T)\cup\Leaf(T')\to\Leaf(T)\cup\Leaf(T')$, called \textbf{coupling map}, such that $\sigma$ restricts on $\{\ell\in\Leaf(T)\mid\iota_\ell=-\}\cup\{\ell\in\Leaf(T')\mid\iota_\ell=-\}$ to an involution onto $\{\ell\in\Leaf(T)\mid\iota_\ell=+\}\cup\{\ell\in\Leaf(T')\mid\iota_\ell=+\}$, we call the triple $C\coloneqq(T,T',\sigma)$ a \textbf{couple}. We denote the set of the right- respectively left-hand side in \eqref{postive and negative leaves} by $\Leaf(C)_\pm$ and further denote $\Node(C)\coloneqq\Node(T)\cup\Node(T')$ and $\Leaf(C)\coloneqq\Leaf(T)\cup\Leaf(T')$. Also,
    \begin{align}
        \mathcal C_{n,n'}^{\eta,\eta',\iota,\iota'}&\coloneqq\l\{(T,T',\sigma)\text{ couple }\mid T\in\T_n^{\eta,\iota},T'\in\T_{n'}^{\eta',\iota'}\r\},\\
        \mathcal C_{n,n'}&\coloneqq\bigsqcup_{\substack{\eta,\eta'\{0,1\}\\\iota,\iota'\in\{\pm\}}}\mathcal C_{n,n'}^{\eta,\eta',\iota,\iota'}.
    \end{align}
\end{definition}

\begin{definition}\label{Definition of decoration of coupling}
Let $C\in\mathcal C_{n,n'}^{\eta,\eta',\iota,\iota'}$ be a couple and $n(C)\coloneqq\abs{\Node(C)}$. Any element of $\l(\R^d\r)^{\Leaf(C)_+}$ is called a \textbf{decoration} of the couple $C$. We define $K_C\colon\Node(C)\cup\Leaf(C)\to L\l(\l(\R^d\r)^{\Leaf(C)_+},\R^d\r)$ to be the evaluation map at positive leaves and everywhere else defined by
\begin{equation}
    K_C(\node)\coloneqq\begin{cases}
        -K_C(\sigma(\node))&\text{ if }\node\in\Leaf(C)_-,\\
        \sum\limits_{\Leaf(C)\ni\leaf<\node}K_C(\leaf)&\text{ if }\node\in\Node(C).
    \end{cases}
\end{equation}
For a particular decoration $\kappa$ of the couple $C$, the map $K_C$ associates wave numbers to each node and leaf of the couple. We define the vector subspace $V(C)\coloneqq\mathrm{span}\l\{K(\leaf)\mid\leaf\in\Leaf(C)_+\r\}$. It should be noted that $K_C(\leaf)_{\leaf\in\Leaf(C)_+}$ defines a basis for $V(C)$.
For any fixed $k\in\Z_L^d$, we say that $\kappa\in\l(\R^d\r)^{\Leaf(C)_+}$ is a \textbf{$k$-decoration} of $C$ if $\kappa\l(\Leaf(C)_+\r)\subseteq\Z_L^d$ and $K_C(\root_C)(\kappa)=k$, where $\root_C$ denotes the root of the left tree of $C$. We denote the subset of $k$-decorations of the couple $C$ by $\mathcal D_k(C)$. We call a subset of leaves $\mathcal F\subseteq\Leaf(C)$ self-coupled if $\sigma(\mathcal F)=\mathcal F$.
\end{definition}

\begin{remark}
    \Cref{Definition of decoration of coupling} implies in particular 
    \begin{equation}
        K_C(\leaf)+K_C(\sigma(\leaf)) = 0\text{ for all }\leaf\in\Leaf(C).
    \end{equation}
    If $\root_C$ and $\root'_C$ denote the roots of the left respectively right tree in $C$, then 
    \begin{equation}
        K_C\l(\root_C\r)+K_C\l(\root_C'\r)=\sum_{\leaf\in\Leaf(C)}K_C(\leaf) = \sum_{\leaf\in\Leaf(C)_+}K_C(\leaf)-\sum_{\leaf\in\Leaf(C)_-}K_C(\sigma(\leaf)) = 0
    \end{equation}
    where in the last equality sign we used the fact that $\sigma$ restricts to an involution from $\Leaf(C)_-$ onto $\Leaf(C)_+$.
Let $\mathcal F\subseteq\Leaf(C)$. Then 
        \begin{equation}
            \sum_{\leaf\in\mathcal F}K_C(\leaf)=0\Leftrightarrow\sigma\l(\mathcal F\r)=\mathcal F.
        \end{equation}
        Thence, $K_C$ can be used to measure self-coupledness of a subset of leaves of a couple $C$. 
\end{remark}

\begin{proposition}
    For any $\eta_1,\eta_2\in\{0,1\}$, $\iota_1,\iota_2\in\{\pm\}$, $n_1,n_2\in\N$ and $k\in\Z_L^d$ we have 
    \begin{equation}
        \E\l(\widehat{\J^{\eta_1,\iota_1}_{n_1}}\l(\epsilon^{-1}t,k\r)\widehat{\J^{\eta_2,\iota_2}_{n_2}}\l(\epsilon^{-1}t,k\r)\r) = \sum_{C\in\mathcal C_{n_1,n_2}^{\eta_1,\eta_2,\iota_1,\iota_2}}{\J_C}\l(\epsilon^{-1}t,k\r),
    \end{equation}
    where the function $\J_C$ is defined by
    \begin{equation}\label{crucial object}
    \begin{gathered}
        {\J_C}\l(\epsilon^{-1}t,k\r)\coloneqq\l(\frac{-\rmi}{L^{2d}}\r)^{n(C)}\prod_{n\in\Node(C)}\iota_\node\sum_{\kappa\in\D_k(C)}\prod_{\node\in\Node(C)} Q_\node\\\cdot\int_{I_C(t)}\prod_{\node\in\Node(C)}e^{\rmi\iota_\node\epsilon^{-1}\Omega_\node t_\node}\rmd t_\node\prod_{\leaf\in\Leaf(C)_+}M^{\eta_\leaf,\eta_{\sigma(\leaf)}}(\kappa(\leaf))^{\iota_\leaf},
    \end{gathered}
    \end{equation}
    where $Q_\node\coloneqq Q_\node^{T_i}(\kappa)$ and $\Omega_\node\coloneqq\Omega_n^{T_i}(\kappa)$ for $\node\in\Node(T_i)$ and $I_C(t)\coloneqq I_{T_1}(t)\times I_{T_2}(t)$. If $\Node(C)=\emptyset$, we have $I_C(t)=\emptyset$ and in that case the integral over $I_C(t)$ is conventionally understood to equal to $1$.   
\end{proposition}

\begin{proof}
    According to \cref{Fourier transform of J}, we have 
    \begin{equation}\label{initial calculation}
        \begin{gathered}           \E\l(\widehat{F^{\eta_1,\iota_1}_{n_1}}\l(\epsilon^{-1}t,k\r)\widehat{F^{\eta_2,\iota_2}_{n_2}}\l(\epsilon^{-1}t,k\r)\r)=\sum_{\substack{T_1\in\T_{n_1}^{\eta_1,\iota_1}\\T_2\in\T_{n_2}^{\eta_2,\iota_2}}}\E\l(\widehat{\J_{T_1}}\l(\epsilon^{-1}t,k\r)\widehat{\J_{T_2}}\l(\epsilon^{-1}t,k\r)\r)\\
    =\l(\frac{-\rmi}{L^{2d}}\r)^{n(C)}\prod_{\node\in\Node(C)}\iota_\node\sum_{\substack{T_1\in\T_{n_1}^{\eta_1,\iota_1}\\T_2\in\T_{n_2}^{\eta_2,\iota_2}}}\sum_{\substack{\kappa_1\in\mathcal D_k(T_1)\\\kappa_2\in\mathcal D_{-k}(T_2)}}\prod_{\node\in\Node(C)} Q_\node\\\cdot\int_{I_C(t)}\prod_{\node\in\Node(C)}e^{\rmi\iota_\node\Omega_\node t_\node}\rmd t_\node\E\l(\prod_{\leaf\in\Leaf(C)}\mu^{\eta_\leaf,\iota_\leaf}_{\kappa(\leaf)}\r),
    \end{gathered}
    \end{equation}
    where we have implicitly defined $\left.\kappa\right|_{\Leaf(T_i)}\coloneqq\kappa_i$.
    With Isserlis' theorem (see \cref{Isserlis}), we find 
    \begin{equation}
    \begin{aligned}
        \E\l(\prod_{\leaf\in\Leaf(C)}\mu^{\eta_\leaf,\iota_\leaf}_{\kappa(\leaf)}\r) &= \sum_{\text{Pairings }\sigma\text{ of }\Leaf(C)}\prod_{\leaf\in \mathcal S_\sigma}\E\l(\mu^{\eta_\leaf,\iota_\leaf}_{\kappa(\leaf)}\mu^{\eta_{\sigma(\leaf)}\iota_{\sigma(\leaf)}}_{\kappa(\sigma(\leaf))}\r)\\&=\sum_{\text{Pairings }\sigma\text{ of }\Leaf(C)}\prod_{\leaf\in \mathcal S_\sigma}\delta_{\iota_\leaf+\iota_{\sigma(\leaf)}}\delta_{\kappa(\leaf)+\kappa(\sigma(\leaf))}M^{\eta_\leaf,\eta_{\sigma(\leaf)}}(\kappa(\leaf))^{\iota_\leaf}
    \end{aligned}
    \end{equation}
    For each paring $\sigma$ of $\Leaf(C)$ there is a subset of leaves $\mathcal S_\sigma\in\Leaf(C)$ such that $\mathcal S_\sigma\sqcup\sigma(S_\sigma)=\Leaf(C)$. Now, in order to get a non-trivial contribution, the pairings must satisfy $\iota_\leaf+\iota_{\sigma(\leaf)}=0$ and $\kappa(\leaf)+\kappa(\sigma(\leaf))=0$ for all $\leaf\in S_\sigma$. The first requirement implies $S_\sigma=\Leaf(C)_+$ and the second requirement gives
    \begin{equation}
        \sum_{\substack{T_1\in\T_{n_1}^{\eta_1,\iota_1}\\T_2\in\T_{n_2}^{\eta_2,\iota_2}}}\sum_{\substack{\kappa_1\in\mathcal D_k(T_1)\\\kappa_2\in\mathcal D_{-k}(T_2)}}\sum_{\text{Pairings }\sigma\text{ of }\Leaf(C)}\prod_{\leaf\in\Leaf(C)_+}\delta_{\kappa(\leaf)+\kappa(\sigma(\leaf))} = \sum_{C=(T_1,T_2,\sigma)\in\mathcal C_{n_1,n_2}^{\eta_1,\eta_2,\iota_1,\iota_2}}\sum_{\kappa\in\mathcal D_k(C)}.
    \end{equation}
    Plugging these observations into \eqref{initial calculation}, we find exactly 
    \begin{equation} 
    \begin{gathered}
\E\l(\widehat{\J_{T_1}}\l(\epsilon^{-1}t,k\r)\widehat{\J_{T_2}}\l(\epsilon^{-1}t,k\r)\r) = \l(\frac{-\rmi}{L^{2d}}\r)^{n(C)}\sum_{C\in\mathcal C_{n_1,n_2}^{\eta_1,\eta_2,\iota_1,\iota_2}}\sum_{\kappa\in\mathcal D_k(C)}\prod_{\node\in\Node(C)}\iota_\node\prod_{\node\in\Node(C)} Q_\node\\\cdot\int_{I_C(t)}\prod_{\node\in\Node(C)}e^{\rmi\iota_\node\epsilon^{-1}\Omega_\node t_\node}\rmd t_\node\prod_{\leaf\in\Leaf(C)_+}M^{\eta_\leaf,\eta_{\sigma(\leaf)}}(\kappa(\leaf))^{\iota_\leaf}.  
    \end{gathered}
    \end{equation}
\end{proof}

\section{First part of the proof of \cref{this is the main theorem}}\label{third section}
In this section, we prove the existence and uniqueness of a solution to the microscopic system \eqref{system} on the time interval $\l[0,\delta L^{\frac{1}{\beta}}\r]$.

\subsection{Initial estimates and preparational tools}

\begin{definition}
    Let $<$ and $<'$ be two order relations on a given set. We call $<'$ \textbf{compatible} with $<$ if $a<b$ implies $a<'b$ for all elmeents $a$ and $b$ of the underlying set. For a totally ordered finite set $\mathcal S =\{s_1<'\cdots s_{\abs{\mathcal S}}\}$ it is clear that its total order relation $<'$ can be identified with a bijection $f_{<'}\colon\llbracket1,\abs{\mathcal S}\rrbracket\to S$ via $f_{<'}^{-1}(s_i)<f_{<'}^{-1}(s_j)$ if and only if $s_i<'s_j$. The underlying bijection $f_{<'}$ enumerates the elements of $\mathcal S$. Given a partial order relation $<_{\mathcal S}$, we denote the set of all total order relations on $\mathcal S$ that are compatible with $<_{\mathcal S}$ by $\mathcal M(\mathcal S,<_{\mathcal S})$. If $S$ is a subset of leaves and nodes, we will always consider the parentality order $<$ as the partial order relation and simply denote $\mathcal M(S)=\mathcal M(S,<)$. 
\end{definition}

\begin{lemma}\label{trivial bound}

    There exists $\Lambda(d,R)>0$ such that 
    \begin{equation}
        \abs{{\J_C}\l(\epsilon^{-1}t,k\r)}\leq\Lambda\abs{\M(\Node(C))}\frac{(\Lambda t)^n}{n!}
    \end{equation}
    for all $t\in[0,\delta]$ and $k\in\Z_L^d$.
    \begin{proof}
        We simply bound the integral in \eqref{crucial object} by
        \begin{equation}
            \abs{\int_{I_C(t)}\prod_{\node\in\Node(C)}e^{\rmi\iota_\node\epsilon^{-1}\Omega_\node t_\node\rmd t_\node}}\leq\int_{I_C(t)}\prod_{\node\in\Node(C)}\rmd t_\node = \abs{\M(\Node(C))}\frac{t^n}{n!}
        \end{equation} 
        and also exploit the fact that the functions $Q_\node$ and $M^{\eta,\eta'}$ are bounded and in addition that $M^{\eta,\eta'}$ are and compactly supported in $B_R(0)$. Recall $\abs{\Leaf(C)_+}=2n+1$. In this regard,
        \begin{equation}
        \begin{aligned}
        \abs{\widehat{\J_C}\l(\epsilon^{-1}t,k\r)}&\leq\Lambda^{2n+1} L^{-2nd}\abs{\M(\Node(C))}\frac{t^n}{n!}\sum_{\substack{\kappa\in\mathcal D_k(C)\\\kappa(\Leaf(C)_+)\subseteq B_R(0)}}1\\&=\Lambda^{2n+1}L^{-2nd}\abs{\M(\Node(C))}\frac{t^n}{n!}\abs{\l\{\kappa\in\l(\Z_L^d\cap B_0(R)\r)^{\Leaf(C)_+}\Bigm\vert K_C(\root)(\kappa)=k\r\}}\\&\leq\Lambda^{2n+1} L^{-2nd}\abs{\M(\Node(C))}\frac{t^n}{n!}\l(L^dR^d\r)^{2n}\\&\leq\Lambda\abs{\M(\Node(C))}\frac{(\Lambda t)^n}{n!}
        \end{aligned}
        \end{equation}
        by increasing $\Lambda$.
    \end{proof}
\end{lemma}

The following technical tool is an elementary but essential step in obtaining positive powers of $\epsilon$. 
\begin{lemma}[Resolvent identity]\label{resolvent identity}
Let $n\in\N$, $\alpha_1,\ldots,\alpha_n\in\R$ and $\nu>0$. Then 
    \begin{equation}
        \int_{0\leq t_1<\cdots<t_n\leq t}\prod_{i=1}^ne^{\rmi\alpha_it_i}\rmd t_i = \frac{e^{\nu t}}{2\pi}\int_\R\frac{e^{\rmi\xi t}}{\nu-\rmi\xi}\prod_{i=1}^n\frac{1}{\nu-\rmi\l(\xi+\sum_{k=i}^n\alpha_k\r)}\rmd\xi.
    \end{equation}    
\end{lemma}
\begin{proof}

For integrable functions $\l(g_i\r)_{i=1}^n$ in one variable, one can prove inductively quite quickly, 
\begin{equation}\label{general formula}
\begin{gathered}
    \int_0^tg_n(s_n)\int_{s_n}^tg_{n-1}(s_{n-1})\cdots\int_{s_{2}}^tg_{1}(s_1)\rmd s_1\cdots\rmd s_n \\= \int_0^tg_n(s_n)\int_0^{t-s_n}g_2(s_n+s_{n-1})\cdots\int_0^{t-\sum_{i=2}^{n}s_i}g_1\l(\sum_{i=1}^ns_i\r)\rmd s_1\cdots\rmd s_n\\=\int_{\R_+^{n+1}}g_n(s_n)g_{n-1}(s_n+s_{n-1})\cdots g_1\l(\sum_{i=1}^ns_i\r)\delta\l(\sum_{i=0}^{n}s_i-t\r)\rmd s_{0}\cdots\rmd s_{n}.
\end{gathered}
\end{equation}
We apply \cref{general formula} by setting $g_i(t_i)\coloneqq e^{\rmi\alpha_i t_i}$ and thus obtain 
\begin{equation}\label{coordinate transformed}
    \int_{0\leq t_1<\cdots<t_n\leq t}\prod_{i=1}^ne^{\rmi\alpha_it_i}\rmd t_i=\int_{\R_+^{n+1}}\prod_{k=1}^ne^{\rmi t_k\sum_{i=k}^n\alpha_i}\delta\l(\sum_{i=1}^nt_{i}+t'-t\r)\rmd t'\prod_{i=1}^n\rmd t_i.
\end{equation}
We use the identity 
\begin{equation}
    \delta(x) = \frac{1}{2\pi}\int_\R e^{\rmi\xi x}\rmd\xi
\end{equation}
and add $e^{\eta\l(t-t'-\sum_{i=1}^nt_{f_{<'}(i)}\r)}$ as a factor to the integrand of the integral in \eqref{coordinate transformed} since its value on the support of that integral is $1$. We obtain 
\begin{equation}
\begin{aligned}
    \int_{0\leq t_1<\cdots<t_n\leq t}\prod_{i=1}^ne^{\rmi\alpha_it_i}\rmd t_i&=\frac{e^{\eta t}}{2\pi}\int_\R e^{-\rmi\xi t}\int_{\R_+}e^{(-\eta+\rmi\xi)t'}\rmd t'\int_{\R_+^n}\prod_{k=1}^ne^{\l(-\eta+\rmi(\xi+\epsilon^{-1}\omega_k\r)t_k}\rmd t_k\rmd\xi\\
    &=\frac{e^{\eta t}}{2\pi}\int_\R\frac{e^{-\rmi\xi t}}{\eta-\rmi\xi}\prod_{k=1}^n\frac{1}{\eta-\rmi\l(\xi+\epsilon^{-1}\omega_k\r)}\rmd\xi,
\end{aligned}
\end{equation}
where we denote 
\begin{equation}
    \omega_k\coloneqq \sum_{i=k}^n\alpha_i.
\end{equation}
\end{proof}

\begin{corollary}
    It holds that 
    \begin{equation}
        \int_{I_C(t)}\prod_{\node\in\Node(C)}e^{\rmi\iota_\node\epsilon^{-1}\Omega_\node t_\node}\rmd t_\node = \frac{e^{nt/\delta}}{2\pi}\sum_{<'\in\M(\Node(C))}\int_\R\frac{e^{-\rmi\xi t}}{n\delta^{-1}-\rmi\xi}\prod_{\node\in\Node(C)}\frac{1}{{n\delta^{-1}-\rmi\l(\xi+\epsilon^{-1}\omega_\node(<')\r)}}\rmd\xi,
    \end{equation}
    where 
    \begin{equation}
        \omega_\node(<')\coloneqq\sum_{\node\leq'\node'\in\Node(C)\setminus\Node_{\mathrm{res}}(C)}\iota_{\node'}\Omega_{\node'}.
    \end{equation}
\end{corollary}
\begin{proof}
    This is merely an application of the resolvent identity (\cref{resolvent identity}) with $\nu=n\delta^{-1}$ and by recognizing that 
    \begin{equation}
        \int_{I_C(t)}f\l(\l(t_\node\r)_{\node\in\Node(C)}\r)\prod_{\node\in\Node(C)}\rmd t_\node = \sum_{<'\in\M(\Node(C))}\int_{0\leq t_{f_{<'}(1)}<\cdots<t_{f_{<'}(n)}\leq t}f\l(\l(t_{f_{<'}(i)}\r)_{i=1}^n\r)\prod_{i=1}^n\rmd t_{f_{<'}(i)}
    \end{equation}
    since the complement of 
    \begin{equation}
    \bigsqcup_{<'\in\M(\Node(C))}\l\{\l(t_\node\r)_{\node\in\Node(C)}\in[0,t]^n\mid0\leq t_{f_{<'}(1)}<\cdots t_{f_{<'}(n)}\leq t\r\}
\end{equation}
inside $I_C(t)$ has measure zero.
\end{proof}

We may decompose 

\begin{equation}
    {\J_C}\l(\epsilon^{-1}t,k\r) = \sum_{<'\in\M(\Node(C))}\mathcal{J_C^{<'}}\l(\epsilon^{-1}t,k\r),
\end{equation}
where \begin{equation}
\begin{gathered}
    {\J_C^{<'}}\l(\epsilon^{-1}t,k\r)\coloneqq\l(\frac{-\rmi}{L^{2d}}\r)^n\frac{e^{nt/\delta}}{2\pi}\prod_{\node\in\Node(C)}\iota_\node\sum_{\kappa\in\mathcal D_k(C)}\prod_{\node\in\Node(C)} Q_\node\\\cdot\int_\R\frac{e^{-\rmi\xi t}}{n\delta^{-1}-\rmi\xi}\prod_{\node\in\Node(C)}\frac{1}{{n\delta^{-1}-\rmi\l(\xi+\epsilon^{-1}\omega_\node(<')\r)}}\rmd\xi\prod_{\leaf\in\Leaf(C)_+}M^{\eta_\leaf,\eta_{\sigma(\leaf)}}(\kappa(\leaf))^{\iota_\leaf}.
\end{gathered}
\end{equation}
We apply again the fact that the functions $M^{\eta,\eta'}$ are compactly supported and initially estimate 
\begin{equation}\label{initial estimate}
    \sup_{t\in[0,\delta]}\abs{{\J_C^{<'}}\l(\epsilon^{-1}t,k\r)}\leq\Lambda^{n+1}L^{-{2dn}}\int_\R\frac{\sum\limits_{\substack{\kappa\in\mathcal D_k(C)\\\kappa\l(\Leaf(C)_+\r)\subseteq B_R(0)}}\prod\limits_{\node\in\Node(C)}\frac{\abs{Q_\node}}{\abs{n\delta^{-1}-\rmi\l(\xi+\epsilon^{-1}\omega_\node(<')\r)}}}{\abs{n\delta^{-1}-\rmi\xi}}\rmd\xi.
\end{equation}

\begin{definition}
    Let $C$ be a couple and $\mathcal S\subseteq\Node(C)\sqcup\Leaf(C)$. An element $K\in V(C)$ is called a \textbf{signed combination} of the family $(K_C(\node))_{\node\in\mathcal S}$ if $K$ can be decomposed as \begin{equation}
        K = \sum\limits_{\node\in\mathcal S}s_\node K_C(\node)
    \end{equation} with $s_\node\in\{-1,0,1\}$ for all $\node\in\mathcal S$. This turns $K_C(\node)$ is a signed combination of $(K_C(\leaf))_{\leaf\in\Leaf(C)_+}$ for all $\node\in\Node(C)\sqcup\Leaf(C)$.
\end{definition}

The following result delivers a coordinate transformation with which the sum over $k$-decorations may be restructured.
\begin{proposition}\label{change of variables}    
    Let $C\in\mathcal C_{n_1,n_2}$ and consider any total order relation $<_f$ on all vertices of $C$ that is compatible with parentality, where $f\colon\llbracket1,5n(c)+1\rrbracket\to\Node(C)\sqcup\Leaf(C)$ denotes the corresponding bijection that enumerates $\Node(C)\sqcup\Leaf(C)$. Then there exists a subset of vertices $\Node_{<_f}(C)\subseteq\Node(C)\sqcup\Leaf(C)$ such that 
    \begin{itemize}
        \item its cardinality is $\abs{\Node_{<_f}(C)}=2n(C)+1$,
        \item the family $(K_C(\node))_{\node\in\Node_{<_f}(C)}$ constitutes a basis for $V(C)$, 
        \item for all $\node\in\Node(C)\sqcup\Leaf(C)$, $K_C(\node)$ is a signed combination of $(K_C(\node))_{\node\leq_f\node'\in\Node_{<_f}(C)}$.
    \end{itemize}
\end{proposition}
\begin{proof}
    This is an adaptation of section 5 in \cite{Lukkarinen_2010} from ternary to $5$-ary trees.
\end{proof}

Fixing any total order relation on the leaves $\Leaf(C)$, we may combine it with any $<'\in\M(\Node(C))$ to obtain an element $<''\in\M\l(\Leaf(C)\sqcup\Node(C)\r)$. More precisely, we define $\l.<''\r|_{\Node(C)}\coloneqq <'$ and $\l.<''\r|_{\Leaf(C)}$ to be the fixed total order relation on the leaves. We apply \cref{change of variables} to $<''$ and obtain $\Node_{<''}(C)\subseteq\Node(C)\sqcup\Leaf(C)$. We denote by $\mathcal R(C)$ the set of roots of both trees in the couple $C$. 
\begin{definition}
Let $C$ be a couple. We define
\begin{equation}
    \mathcal P_i\coloneqq\l\{\node\in\Node(C)\mid\abs{C(\node)\cap\Node_{<''}(C)}=i\r\}.
\end{equation}
    Let $\node\in\Node(C)$ and define the \textbf{degree} of $\node$ (with respect to $<''$) to be 
    \begin{equation}
        \deg_{<''}(\node)\coloneqq\abs{C(\node)\cap\Node_{<''}(C)}.
    \end{equation}
\end{definition}
\begin{remark}
    The decomposition 
    \begin{equation}
        \Node(C)=\bigsqcup_{i=0}^5\mathcal P_i
    \end{equation}
    groups nodes by their respective degree, so that $\mathcal P_i$ contains all $5$-ary nodes of degree $i$, and 
    \begin{equation}
        \sum_{i=0}^5\abs{\mathcal P_i}=n(C).
    \end{equation}
\end{remark}

\begin{lemma}\label{k=0 is included lol}
    There exists exactly one $r\in\mathcal R(C)$ such that $r\in\Node_{<''}(C)$. Furthermore, $\mathcal P_5=\emptyset$ and $\abs{\mathcal P_1}+2\abs{\mathcal P_2}+3\abs{\mathcal P_3}+4\abs{\mathcal P_4} = 2n(C)$.
\end{lemma}
\begin{proof}
    For $C=\l(T_1,T_2,\sigma\r)$ we denote $\mathcal R(C)=\{\root,\root'\}$, where $\root$ denotes the root of the left and $\root'$ denotes the root of the right tree. 
    From the proof of \cite[Proposition 3.4]{desuzzoni2025waveturbulencesemilinearkleingordon} and the fact that the root vertex of the momentum graph is part of the spanning tree, we already have $\mathcal R(C)\cap\Node_{<''}(C)\neq\emptyset$.
    From now on, whenever we speak of a path or a walk, we mean a path or walk in the vertex and edge set in some iteration step of the spanning tree construction. We may very well assume that we reached the iteration step at which $v_{\mathrm{root}}$ is added to the vertex set of the spanning tree. Since the number of leaves of $5$-ary trees is odd, there exists a fusion vertex $v_\leaf$ such that there is a path $v_\leaf\to\leaf\to\root\to v_{\mathrm{root}}$ satisfying $\leaf\in\Leaf(T_1)$ and $\sigma(\leaf)\in\Leaf(T_2)$. Now suppose there exists $\node\in\Leaf(T_2)\sqcup\Node(T_2)$ such that there is a path $\sigma(\leaf)\to\node$ and suppose that adding $\{\node,P(\node)\}$ creates a cycle. This implies that there exists $\node'\in S(\node)\setminus\{\node\}$ such that $\{\node',P(\node)\}$ was added before and there exists a path from $\node'$ to some $\leaf'\in\Leaf(T_2)$. We distinguish the cases $\sigma(\leaf')\in\Leaf(T_2)$ and $\sigma(\leaf')\in\Leaf(T_1)$. In the first case, there must exist some $\leaf''\in\Leaf(T_2)$ such that there is a path $\sigma(\leaf')\to\leaf''$ and a path $\leaf''\to\node$ which closes the cycle. In this case, we can proceed to $P(\node)$ and have gained in height by $1$. In the second case, there exists $\leaf'''\in\Leaf(T_1)$ with $\sigma(\leaf''')\in\Leaf(T_2)$ such that there exists a path $\sigma(\leaf')\to\leaf'''$ and a path from $\sigma(\leaf''')$ to some $\tilde\leaf\in\Leaf(T_2)$ and ultimately a path $\tilde\leaf\to P(\node)$, and thus a path $\sigma(\leaf''')\to P(\node)$. In this case, we switch from $v_\leaf$ to $v_{\leaf'}$. This iteration can be repeated until a fusion vertex $v_{\leaf_{\mathrm{final}}}$ is considered, and the algorithm considers the other root $\root'$. We can see that adding the edge $\{\root',v_{\mathrm{root}}\}$ will create a cycle that contains the root vertex and $v_{\leaf_{\mathrm{final}}}$. 

    Suppose there exists $\node\in\mathcal P_5$ which is equivalent to stating $C(\node)\subseteq\Node_{<''}(C)$. Due to the fact that 
    \begin{equation}\label{first}
        K_C(\node) = \sum_{\node'\in C(\node)}K_C(\node'),
    \end{equation}
    we have $\node\notin\Node_{<''}(C)$, otherwise the linear independence of $\l(K_C(\node')\r)_{\node'\in\Node_{<''}(C)}$ would be violated. \Cref{coordinate transformed} implies that 
    \begin{equation}\label{seq}
        K_C(\node) = \sum_{\node<''\node'\in\Node_{<''}(C)}s_{\node'}K_C(\node')
    \end{equation}
    for some $s_{\node'}\in\{-1,0,1\}$. By construction, $<''$ is compatible with the parentality order. But this compatibility implies $C(\node)\cap\{\node'\in\Node_{<''}(C)\mid\node<''\node'\}=\emptyset$ and subtracting \cref{first} from \cref{seq} implies 
    \begin{equation}
        \sum_{n<''\node'\in\Node_{<''}(C)}s_{\node'}K_C(\node')-\sum_{\node'\in C(\node)}K_C(\node')=0
    \end{equation}
    which implies linear dependence of $\l(K_C(\node')\r)_{\node'\in\Node_{<''}(C)}$, a contradiction. It follows $\mathcal P_5=\emptyset$.

    It is a general fact that if $\mathcal S\subseteq\l(\Leaf(C)\sqcup\Node(C)\r)\setminus\mathcal R(C)$ is any subset of leaves and nodes, and if we denote $\mathcal P_i(\mathcal S)\coloneqq\l\{\node\in\Node(C)\mid\abs{C(\node)\cap\mathcal S}=i\r\}$, we obtain $\abs{S}=\sum_{i=1}^4i\abs{\mathcal P_i(\mathcal S)}+1$ and $P(\mathcal S\setminus\{\root\})=\bigsqcup_{i=1}^5\mathcal P_i(\mathcal S)$, so that
    \begin{equation}
        P\l(\Node_{<''}(C)\setminus\{r\}\r) = \bigsqcup_{i=1}^4\mathcal P_i
    \end{equation}
    and since $\abs{\Node_{<''}(C)\setminus\{r\}}=2n(C)$, 
    \begin{equation}\label{general shit}
        \sum_{i=1}^4\abs{\mathcal P_i}i=2n(C).
    \end{equation}
\end{proof}

\subsection{Passing from sums to integrals}

\begin{notation}
    For $x\in\R^d$ and $r>0$, we denote $Q_{x,r}\coloneqq\prod_{i=1}^d[x_i-{r},x_i+{r}]$, $Q_{x,r,L}\coloneqq Q_{x,r}\cap\Z_L^d$ and define $m=m(R,L)$ to be the unique element in $\N$ that satisfies $\frac{m+1}{L}> R\geq\frac{m}{L}$ and finally the special cube 
    \begin{equation}
        Q_{m}\coloneqq\l[-\frac{1}{2L}-\frac{m}{L},\frac{1}{2L}+\frac{m}{L}\r]^d
    \end{equation}
    and its discrete version $Q_{m}^{\Z_L^d}\coloneqq Q_{m}\cap\Z_L^d$.
\end{notation}

\begin{lemma}\label{case distinction}
    If $m/L+1/(2L)\geq R$, then $\l(Q_{m}\setminus Q_{0,R}\r)\cap\Z_L^d=\emptyset$ and if $m/L+1/(2L)<R$, then $\l(Q_{0,R}\setminus Q_{m}\r)\cap\Z_L^d=\emptyset$.
\end{lemma}
\begin{proof}
    Suppose the first case $m/L+1/(2L)\geq R$ and $x\in\l(Q_{m}\setminus Q_{0,R}\r)\cap\Z_L^d$. This is equivalent to saying that $R<\abs{x_i}\leq m/L+1/(2L)$ and by setting $y\coloneqq Lx\in\Z^d$, this implies 
    \begin{equation}
        m<\abs{y_i}\leq m+\frac{1}{2}
    \end{equation}
    which is a contradiction. Now suppose we are in the second case when $m/L+1/(2L)<R$ and assume there exists $x\in\l(Q_{0,R}\setminus Q_{m}\r)\cap\Z_L^d$. Similarly, by setting $y\coloneqq Lx$, we find 
    \begin{equation}
        m+\frac{1}{2}<\abs{y_i}< m+1
    \end{equation}
    which also contradicts $y\in\Z^d$.
\end{proof}

\begin{lemma}\label{sums into integrals}

    Suppose $g_i\colon\l(\R^d\r)^i\to\R_>$ are positive $\mathcal C^1$-functions for every $i\in\{1,2,3,4,5\}$ and $R>0$. Then there exists $\Lambda_i=\Lambda_i(d)>0$ such that 
    \begin{align}
        \sum_{x_1,\ldots,x_i\in Q_{0,R}^{\Z_L^d}}g_i\l(x_1,\ldots,x_i\r)\leq\Lambda_i\l(\frac{1}{L}\sum_{x_1,\ldots,x_i\in Q_{0,2R}^{\Z_L^d}}\sup_{\bigtimes_{j=1}^iQ_{x_j,\frac{1}{2L}}}\abs{\nabla g_i}+L^{id}\int_{Q_{0,2R}^i}g_i\l(x_1,\ldots,x_i\r)\rmd \l(x_1,\ldots, x_i\r)\r).
    \end{align}
    For $\mathcal C^1$ functions $f_i\colon\l(\R^d\r)^i\to\C$ there exist dimensional constants $\Lambda_i>0$ such that  
    \begin{equation}
        \abs{\frac{1}{L^{id}}\sum_{(x_j)_{j=1}^i\in\l(\Z_L^d\r)^i}f\l((x_j)_{j=1}^i\r)-\int_{\l(\R^d\r)^i}f\l((x_j)_{j=1}^i\r)\rmd \l(x_1,\ldots,x_i\r)}\leq\frac{\Lambda_i}{L^{id+1}}\sum_{\l(x_j\r)_{j=1}^i\in\l(\Z_L^d\r)^i}\sup_{\bigtimes_{j=1}^iQ_{x_j,\frac{1}{2L}}}\abs{\nabla_{\l(x_j\r)_{j=1}^i}f}
    \end{equation}
\end{lemma}
\begin{proof}
    We make a case distinction as in \cref{case distinction}. Suppose first that $m/L+1/(2L)<R$ and apply \cref{case distinction} which allows to write 
    \begin{equation}
        \sum_{x\in Q_{0,R}^{\Z_L^d}}g(x) = \sum_{x\in Q_{m}^{\Z_L^d}}g(x)\text{ and }\int_{Q_{0,R}}g(y)\rmd y = \sum_{x\in Q_{m}^{\Z_L^d}}\int_{Q_{x,\frac{1}{2L}}}g(y)\rmd y + \int_{Q_{0,R}\setminus Q_{m}}g(y)\rmd y.
    \end{equation}
    Using the fact that ${1}/{L^d}=\int_{Q_{x,\frac{1}{2L}}}1\rmd y$ allows to write further 
    \begin{equation}
    \begin{aligned}\label{create a contrast}
        \abs{\frac{1}{L^d}\sum_{x\in Q_{0,R}^{\Z_L^d}}g(x)-\int_{Q_{0,R}}g(y)\rmd y}&\leq\sum_{x\in Q_{m}^{\Z_L^d}}\int_{Q_{x,\frac{1}{2L}}}\abs{g(x)-g(y)}\rmd y+\int_{Q_{0,R}}g(y)\rmd y\\&\leq\sum_{x\in Q_{0,R}^{\Z_L^d}}\sup_{Q_{x,\frac{1}{2L}}}\abs{\nabla g}\int_{Q_{x,\frac{1}{2L}}}\abs{x-y}\rmd y + \int_{Q_{0,R}}g(y)\rmd y.
    \end{aligned}
    \end{equation}
    The geometric fact $\abs{x-y}\leq\frac{\sqrt{d}}{2L}$ leads the refined estimate to 
    \begin{equation}\label{right before assertion}
        \abs{\frac{1}{L^d}\sum_{x\in Q_{0,R}^{\Z_L^d}}g(x)-\int_{Q_{0,R}}g(y)\rmd y}\leq \frac{\sqrt{d}}{2L^{d+1}}\sum_{x\in Q_{0,R}^{\Z_L^d}}\sup_{Q_{x,\frac{1}{2L}}}\abs{\nabla g}+\int_{Q_{0,R}}g(y)\rmd y
    \end{equation}
    which delivers 
    \begin{equation}
        \sum_{x\in Q_{0,R}^{\Z_L^d}}g(x)\leq\Lambda\l(\frac{1}{L}\sum_{x\in Q_{0,R}^{\Z_L^d}}\sup_{Q_{x,\frac{1}{2L}}}\abs{\nabla g}+L^d\int_{Q_{0,R}}g(x)\rmd x\r)
    \end{equation}
    for $\Lambda_1\coloneqq\max\l(2,\frac{\sqrt{d}}{2}\r)$.

    Now suppose $m/L+1/(2L)\geq R$. In that case $Q_{0,R}\subseteq Q_m\subseteq Q_{0,2R}$. It is similarly true that 
    \begin{equation}
        \sum_{x\in Q_{0,R}^{\Z_L^d}}g(x) = \sum_{x\in Q_{m}^{\Z_L^d}}g(x)\text{ and }\int_{Q_{0,R}}g(y)\rmd y = \sum_{x\in Q_{m}^{\Z_L^d}}\int_{Q_{x,\frac{1}{2L}}}g(y)\rmd y - \int_{Q_{m}\setminus Q_{0,R}}g(y)\rmd y.
    \end{equation}
    In contrast to \cref{create a contrast}, the estimate is now 
    \begin{equation}
    \begin{aligned}
        \abs{\frac{1}{L^d}\sum_{x\in Q_{0,R}^{\Z_L^d}}g(x)-\int_{Q_{0,R}}g(y)\rmd y}&\leq\sum_{x\in Q_{m}^{\Z_L^d}}\int_{Q_{x,\frac{1}{2L}}}\abs{g(x)-g(y)}\rmd y+\int_{Q_{m}\setminus Q_{0,R}}g(y)\rmd y\\&\leq \frac{\sqrt{d}}{2L^{d+1}}\sum_{x\in Q_{0,2R}^{\Z_L^d}}\sup_{Q_{x,\frac{1}{2L}}}\abs{\nabla g}+\int_{Q_{0,2R}}g(y)\rmd y
    \end{aligned}
    \end{equation}
    but $\Lambda_1$ may be chosen as before. The remaining inequalities can be proven in a similar way with $\Lambda_i\coloneqq\max\l(2,\frac{\sqrt{id}}{2}\r)$.

    We can obviously rewrite 
    \begin{equation}
        \int_{\l(\R^d\r)^i}f_i\l((x_j)_{j=1}^i\r)\rmd\l(x_1,\dots,x_i\r) = \sum_{(x_j)_{j=1}^i\in\l(\Z_L^d\r)^i}\int_{\bigtimes_{j=1}^iQ_{x_j,\frac{1}{2L}}}f_i\l((x_j)_{j=1}^i\r)\rmd\l(x_1,\ldots,x_i\r)
    \end{equation}
    and recall that 
    \begin{equation}
        \frac{1}{L^{id}}=\int_{\bigtimes_{j=1}^iQ_{x_j,\frac{1}{2L}}}1\rmd\l(x_1,\ldots,x_i\r)
    \end{equation}
    so that 

    \begin{equation}
        \begin{aligned}
            \Bigg|\frac{1}{L^{id}}&\sum_{(x_j)_{j=1}^i\in\l(\Z_L^d\r)^i}f\l((x_j)_{j=1}^i\r)-\int_{\l(\R^d\r)^i}f\l((x_j)_{j=1}^i\r)\rmd \l(x_1,\ldots,x_i\r)\Bigg|\\&\leq\sum_{\l(x_j\r)_{j=1}^i\in\l(\Z_L^d\r)^i}{\int_{\bigtimes_{j=1}^iQ_{x_j,\frac{1}{2L}}}\abs{f\l(\l(x_j\r)_{j=1}^i\r)-f\l(\l(y_j\r)_{j=1}^i\r)}\rmd\l(y_1,\ldots,y_i\r)}\\&\leq \sum_{\l(x_j\r)_{j=1}^i\in\l(\Z_L^d\r)^i}\sup_{\bigtimes_{j=1}^iQ_{x_j,\frac{1}{2L}}}\abs{\nabla_{\l(x_j\r)_{j=1}^i}f}{\int_{\bigtimes_{j=1}^iQ_{x_j,\frac{1}{2L}}}\abs{\l(x_j\r)_{j=1}^i-\l(y_j\r)_{j=1}^i}\rmd\l(y_1,\ldots,y_i\r)}.
        \end{aligned}
    \end{equation}
    We now use the fact that 
    \begin{equation}
        \abs{\l(x_j\r)_{j=1}^i-\l(y_j\r)_{j=1}^i}^2=\sum_{j=1}^i\abs{x_j-y_j}^2\leq\sum_{j=1}^i\frac{d}{(2L)^2}=\frac{id}{(2L)^2}.
    \end{equation}
    so that $\Lambda_i$ from before delivers the result.
\end{proof}

\begin{lemma}\label{all possibilities} Let $C$ be a couple and consider a node $\node\in\Node(C)$.
    \begin{itemize}
        \item[1.)] If $\node$ has degree $1$, denote $\node'\in C(\node)$ the vertex that belongs to $\Node_{<''}(C)$. Then there exists $\node''\in C(\node)\setminus\{\node'\}$ such that $K_C(\node'')$ is a signed combination of $\l(K_C(\node''')\r)_{\node'<''\node'''\in\Node_{<''}(C)}$.
        \item[2.)] If $\node$ has degree $2$, and we denote by $\node_1$ and $\node_2$ the children of $\node$ that belong to $\Node_{<''}(C)$ and by $\node_1'$, $\node_2'$ and $\node_3'$ the children of $\node$ that do not belong to $\Node_{<''}(C)$, we have the following case distinction: 
        \begin{align}
            &\begin{cases}
                K_C(\node_i')&=-K_C(\node_1)-K_C(\node_2)+G,\\
                K_C(\node_j')&= G,\\
                K_C(\node_k')&=G,
            \end{cases}
            \text{ or }\begin{cases}
                K_C(\node_i')&=-K_C(\node_1)+G,\\
                K_C(\node_j')&=-K_C(\node_2)+G,\\
                K_C(\node_k')&=G
            \end{cases}\text{ or }\\
            &\begin{cases}
                K_C(\node_i')&=\iota_1K_C(\node_1)+\iota_2K_C(\node_2)+G,\\
                K_C(\node_j')&=\iota_3K_C(\node_2)+G,\\
                K_C(\node_k')&=\iota_4K_C(\node_2)+G,
            \end{cases}\text{ or }\begin{cases}
                K_C(\node_i')&=\iota_1K_C(\node_1)+\iota_2K_C(\node_2)+G,\\
                K_C(\node_j')&=\iota_3K_C(\node_1)+G,\\
                K_C(\node_k')&=\iota_4K_C(\node_1)+G,
            \end{cases}\text{ or }\\&\begin{cases}
                K_C(\node_i')&=\iota_1K_C(\node_1)+\iota_1K_C(\node_2)+G,\\
                K_C(\node_j')&=\iota_2K_C(\node_1)+\iota_3K_C(\node_2)+G,\\
                K_C(\node_k')&=\iota_4K_C(\node_1)+\iota_5K_C(\node_2)+G,
            \end{cases}
        \end{align}
        where in each case, $\iota_i\in\{\pm\}$ are such that adding the three lines on the right-hand side of each case delivers exactly $-K_C(\node_1)-K_C(\node_2)+G$. In the notation, we denote by $G$ any signed combination of $\l(K_C(\node'')\r)_{\node_1,\node_2<''\node''\in\Node_{<''}(C)}$.
        \item[3.)] If $\node$ is of degree $3$, $\node_1,\node_2,\node_3\in C(\node)\cap\Node_{<''}(C)$ and $\node_1',\node_2'\in\l(\Node(C)\setminus\Node_{<''}(C)\r)$, then
        \begin{equation}
            \begin{aligned}
                K_C(\node_1')&=G\wedge K_C(\node_2')=-K_C(\node_1)-K_C(\node_2)-K_C(\node_3)+G\text{ or }\\K_C(\node_1')&=-K_C(\node_3)+G\wedge K_C(\node_2')=-K_C(\node_1)-K_C(\node_2)+G\text{ or }\\K_C(\node_1')&=-K_C(\node_2)+G\wedge K_C(\node_2')=-K_C(\node_1)-K_C(\node_3)+G\text{ or }\\K_C(\node_1')&=-K_C(\node_1)+G\wedge K_C(\node_2')=-K_C(\node_2)-K_C(\node_3)+G\text{ or }\\
                K_C(\node_1')&=-K_C(\node_2)-K_C(\node_3)+G\wedge K_C(\node_2')=-K_C(\node_1)+G\text{ or }\\K_C(\node_1')&=-K_C(\node_1)-K_C(\node_3)+G\wedge K_C(\node_2')=-K_C(\node_2)+G\text{ or }\\
                K_C(\node_1')&=-K_C(\node_1)-K_C(\node_2)+G\wedge K_C(\node_2')=-K_C(\node_3)+G\text{ or }\\K_C(\node_1')&=-K_C(\node_1)-K_C(\node_2)-K_C(\node_3)+G\wedge K_C(\node_2')=G,
        \end{aligned}
        \end{equation}
        and $G$ denotes a signed combination of $\l(K_C(\node'')\r)_{\node_1,\node_2,\node_3<''\node''\in\Node_{<''}(C)}$.
        \item[4.)] If $\node$ has degree $4$ and if we denote $\node_1,\ldots,\node_4\in\Node_{<''}(C)$ the children of $\node$ that belong to $\Node_{<''}(C)$, the remaining vertex $\node''\in C(\node)\setminus\{\node_i\}_{i=1}^4$ satisfies 
        \begin{equation}
            K_C(\node'')=-\sum_{i=1}^4K_C(\node_i)+K_C(\node),
        \end{equation}
        where $G=K_C(\node)$ is of course a signed combination of $\l(K_C(\node''')\r)_{\node_1,\ldots,\node_4<''\node'''\in\Node_{<''}(C)}$.
    \end{itemize}
\end{lemma}

\begin{proof}
    The proof is mostly a question of counting.  
    \begin{itemize}
        \item[1.)] Since $\deg(\node)=1$, any $\node''\in S(\node')\setminus\{\node'\}$ satisfies $\node''\notin\Node_{<''}(C)$ and \cref{change of variables} implies that $K_C(\node'')$ is a signed combination of $\l(K_C(\node''')\r)_{\node''<\node'''\in\Node_{<''}(C)}$. The same is true for $K_C(\node)=\sum_{\node'''\in C(\node)}K_C(\node''')$. Denoting the signed coordinate contribution of $K_C(\node')$ to each $\node'''\in C(\node)$ by $s_{\node'''_i}\in\{-1,0,1\}$, we find 
        \begin{equation}
            1+\sum_{i=1}^4s_{\node'''_i}=0
        \end{equation}
        and this equation can only hold true if there exists at least one $i\in\{1,2,3,4\}$ such that $s_{\node'''_i}=0$. The other two signs are automatically $1$ and $-1$. Thence, the vertex we were looking for is $\node''=\node_i'''$.
        \item[2.)] We denote $\node_1,\node_2\in C(\node)\cap\Node_{<''}(C)$ and $\node_1',\node_2',\node_3'\in C(\node)\cap\l(\Node(C)\setminus\Node_{<''}(C)\r)$ so that 
        \begin{equation}
            K_C(\node_1)+K_C(\node_2)+K_C(\node_1')+K_C(\node_2')+K_C(\node_3') = K_C(\node).
        \end{equation}
        As before, we denote by $s_{\node_i}^{\node_j'}\in\{-1,0,1\}$ the signed coordinate contribution of $K_C({\node_i})$ to the signed expansion of $K_C(\node_j')$. We obtain 
        \begin{align}
            1+s_{\node_1}^{\node_1'}+s_{\node_1}^{\node_2'}+s_{\node_1}^{\node_3'}&=0,\label{first condition}\\
            1+s_{\node_2}^{\node_1'}+s_{\node_2}^{\node_2'}+s_{\node_2}^{\node_3'}&=0.\label{second condition}
        \end{align}
        These two equations deliver $36$ possibilities in total. Indeed, \cref{first condition} can only be achieved by $s_{\node_1}^{\node'_i}=-1$ and the remaining signs $s_{\node_1}^{\node'_j}=s_{\node_1}^{\node'_k}=0$ or by $s_{\node_1}^{\node'_i}=1$ and the remaining signs $s_{\node_1}^{\node'_j}=s_{\node_1}^{\node'_k}=-1$. Of course, the same holds true for \cref{second condition}. If the first case applied for \cref{first condition}, there are six different combinations for any fixed $s_{\node_1}^{\node'_i}$. There are three different fixed combinations for the $s_{\node_1}^{\node'_i}$ so that we get $18$ different combinations. Repeating the argument with the second arrangement gives an additional $18$ possibilities, and we end up with $36$ different combinations. 
        
        \item[3.)] As in 2.), we now have 
        \begin{equation}
            K_C(\node_1)+K_C(\node_2)+K_C(\node_3)+K_C(\node_1')+K_C(\node_2')=K_C(\node)
        \end{equation}
        and thus this time 
        \begin{align}
            1+s_{\node_1}^{\node_1'}+s_{\node_1}^{\node_2'}&=0,\label{jo}\\
            1+s_{\node_2}^{\node_1'}+s_{\node_2}^{\node_2'}&=0,\label{transen}\\
            1+s_{\node_3}^{\node_1'}+s_{\node_3}^{\node_2'}&=0.\label{ekelhaft}
        \end{align}
        \Cref{jo,transen,ekelhaft} imply that there exists $i,j,k\in\{1,2\}$ such that $s_{\node_1}^{\node_i'}=s_{\node_2}^{\node_j'}=s_{\node_3}^{\node_k'}=0$. In this case, there are only $8$ combinations to consider. The combinations are $(i,j,k)=(1,1,1)$, $(i,j,k)=(1,1,2)$, $(i,j,k)=(1,2,1)$, $(i,j,k)=(2,1,1)$, $(i,j,k)=(1,2,2)$, $(i,j,k)=(2,1,2)$, $(i,j,k)=(2,2,1)$ and $(i,j,k)=(2,2,2)$.
        \item[4.)] The statement holds trivially by the definition of $K_C$ and \cref{change of variables}.
    \end{itemize}
\end{proof}

\begin{corollary}\label{classifying resonant nodes}
    If $\node\in\Node(C)$ is a resonant node, then $\deg_{<''}(\node)\in\{0,1,2\}$ for all $<''\in\M\l(\Leaf(C)\sqcup\Node(C)\r)$. 
\end{corollary}
\begin{proof}
    Although this is a direct application of \cref{all possibilities}, we may argue directly using \cref{change of variables}: If $\deg_{<''}(\node)=3$, then, due to the linear independence of $\l(K_C(\node')\r)_{\node'\in\Node_{<''}(C)}$, $K_C(\node)=K_C(\node_i)$ for some $i\in\{1,2,3\}$, where we have denoted $C(\node)\cap\Node_{<''}(C)=\{\node_1,\node_2,\node_3\}$. But this contradicts the third point of \cref{change of variables}. The same argument holds if one assumes $\deg_{<''}(\node)=4$.
\end{proof}

\begin{theorem}\label{main technical necessary theorem}

    Let $p\in(1,\infty)$ and $\tilde\gamma>0$. We distinguish the following cases.
    \begin{itemize}
        \item[1.)] For degree $i=1$ nodes. Suppose $A\in\R\setminus\{0\}$, $B,C\in\R$, $\iota\in\{\pm\}$, $\gamma\in\l(0,\tilde\gamma\r)$ and take $P\in\R^d\setminus\{0\}$ with $\abs{P}\geq\epsilon^{\gamma}$. We have
    \begin{align}
        \gamma_\iota(x)\coloneqq\frac{1}{\abs{A+\rmi\l(B+\epsilon^{-\tilde\gamma}\l(\abs{x}^2+\iota\abs{x-P}^2+C\r)\r)}^p},\ \int\limits_{Q_{0,(2n+1)R}}\gamma_\iota\rmd x\leq\Lambda\abs{A}^{1-p}n^{d-1}\epsilon^{\tilde\gamma-\gamma}.
        \end{align}
        \item[2.)] For degree $i\in\llbracket 2,4\rrbracket$ nodes. Let $\l(\iota_j\r)_{j\in\llbracket1,i\rrbracket}\in\{\pm\}^i$, $\l(\iota_j^{j'}\r)_{\substack{j\in\llbracket1,i\rrbracket\\j'\in\llbracket i+1,5\rrbracket}}\in\{\pm,0\}^5$ and write $\iota=\Big(\l(\iota_j\r)_{j\in\llbracket1,i\rrbracket},\l(\iota_j\r)_{j\in\llbracket i+1,5\rrbracket},\l(\iota_j^{j'}\r)_{\substack{j\in\llbracket1,i\rrbracket\\j'\in\llbracket i+1,5\rrbracket}}\Big)$. Further, assume $A\in\R\setminus\{0\}$, $B,C\in\R$ and $P_j\in\R^d$ for $j\in\llbracket i+1,5\rrbracket$. If we define
    \begin{equation}
        \gamma_{\iota}\l(x_1,\ldots,x_i\r)\coloneqq\frac{1}{\abs{A+\rmi\l(B+\epsilon^{-\tilde\gamma}\l(\sum_{j=1}^i\iota_j\abs{x_j}^2+\sum_{j=i+1}^5\iota_j\abs{\sum_{j'=1}^i\iota_j^{j'}x_{j'}-P_j}^2+C\r)\r)}^p},
    \end{equation}
    then
    \begin{equation}
        \int\limits_{Q_{0,(2n+1)R}^i}\gamma_{\iota}\l(x_1,\ldots,x_i\r)\rmd\l(x_1,\ldots,x_i\r)\leq \Lambda\abs{A}^{1-p}n^{id-2}\epsilon^{\tilde\gamma}
    \end{equation}
    
    \end{itemize}
    \end{theorem}
\begin{proof}The proof is as follows.
        \begin{itemize}
            \item[1.)] We first consider $\iota=+$ and find 
        \begin{equation}
        \begin{aligned}\label{already estimated this shit}
            \int\limits_{Q_{0,2(2n+1)R}}\gamma_+^1\rmd x&\leq\abs{A}^{-p}\int\limits_{B_{\sqrt{d}(2n+1)R}(-P/2)}\frac{\rmd x}{\abs{1+\rmi A^{-1}\l(B+\epsilon^{-\tilde\gamma}\l(2\abs{x}^2+\abs{P}^2/2+C\r)\r)}^p}\\&\leq\Lambda\abs{A}^{-p}\int\limits_{0}^{\sqrt{d}(2n+1)R}\frac{r^{d-1}\rmd r}{\abs{1+\rmi A^{-1}\l(B+\epsilon^{-\tilde\gamma}\l(2r^2+\abs{P}^2/2+C\r)\r)}^p}\\&=\frac{\Lambda}{2}\abs{A}^{-p}\int\limits_{0}^{\l(\sqrt{d}(2n+1)R\r)^2}\frac{u^{\frac{d-2}{2}}\rmd u}{\abs{1+\rmi A^{-1}\l(B+\epsilon^{-\tilde\gamma}\l(2u+\abs{P}^2/2+C\r)\r)}^p}\\&\leq\Lambda\abs{A}^{1-p}\l(\sqrt{d}(2n+1)R\r)^{d-2}\epsilon^{\tilde\gamma}\int\limits_\R\frac{\rmd r'}{\abs{1+\rmi r'}^p}\leq\Lambda\abs{A}^{1-p}n^{d-2}\epsilon^{\tilde\gamma}.
        \end{aligned}
        \end{equation}
        In the case of $\iota=-$, we may first write 
        \begin{equation}
            \int\limits_{Q_{0,(2n+1)R}}\gamma_-^1\rmd x =\abs{A}^{-p}\int\limits_{Q_{-P/2,(2n+1)R}}\frac{\rmd x}{\abs{1+\rmi A^{-1}\l(B+\epsilon^{-\tilde\gamma}\l(2xP+C\r)\r)}^p}.
        \end{equation}
        It is now important to recognize that $P\neq0$ so that $\norm{P}_\infty\coloneqq\max_{1\leq i\leq d}\abs{P_i}>0$. Denote by $i_0\in\llbracket1,d\rrbracket$ the index that satisfies $\abs{P_{i_0}}=\norm{P}_\infty$ and apply Fubini's theorem to rewrite
        \begin{equation}
            \begin{gathered}
            \int\limits_{Q_{0,(2n+1)R}}\gamma_-^1\rmd x\\=\abs{A}^{-p}\int\limits_{-\frac{P_1}{2}-(2n+1)R}^{-\frac{P_1}{2}+(2n+1)R}\cdots\int\limits_{-\frac{P_{i_0}}{2}-(2n+1)R}^{-\frac{P_{i_0}}{2}+(2n+1)R}\frac{\rmd x_{i_0}\prod\limits_{i\in\llbracket1,d\rrbracket\setminus\{i_0\}}\rmd x_i}{\abs{1+\rmi A^{-1}\l(B+\epsilon^{-\tilde\gamma}\l(2x_{i_0}P_{i_0}+2\sum_{i\in\llbracket1,d\rrbracket\setminus\{i_0\}}x_iP_i+C\r)\r)}^p},
        \end{gathered}
        \end{equation}
        that is, we would like to carry out the $x_{i_0}$ integration first and substitute the imaginary part in the denominator as before. This leads to 
        \begin{equation}
        \begin{aligned}
            \int\limits_{Q_{0,2(2n+1)R}}\gamma_-^1\rmd x&\leq\Lambda\frac{\abs{A}^{1-p}}{\norm{P}_\infty}\epsilon\int\limits_{-\frac{P_1}{2}-(2n+1)R}^{-\frac{P_1}{2}+(2n+1)R}\cdots\int_\R\frac{\rmd x_{i_0}'}{\abs{1+\rmi x_{i_0}'}^p}\prod_{i\in\llbracket1,d\rrbracket\setminus\{i_0\}}\rmd x_i\\&\leq\Lambda\abs{A}^{1-p}\norm{P}_\infty^{-1}n^{d-1}\epsilon^{\tilde\gamma}\leq\Lambda\abs{A}^{1-p}\abs{P}^{-1}n^{d-1}\epsilon^{\tilde\gamma}\leq\Lambda\abs{A}^{1-p}n^{d-1}\epsilon^{\tilde\gamma-\gamma},
        \end{aligned}
        \end{equation}
        where we used the fact that all norms in finite-dimensional vector spaces are equivalent. 
            \item[2.)]  We consider the appearance of the variable $x_1$ in the sum 
        \begin{equation}            \sum_{j=1}^i\iota_j\abs{x_j}^2+\sum_{j=i+1}^5\iota_j\abs{\sum_{j'=1}^i\iota_j^{j'}x_{j'}-P_j}^2.
        \end{equation}
        We enumerate 
        \begin{equation}
            \l\{j\in\llbracket i+1,5\rrbracket\bigm\vert\iota_j^1\neq0\r\} = \l\{k_1,\ldots,k_q\r\}
        \end{equation}
        and thus have 
        \begin{equation}
        \begin{gathered}
            \sum_{\substack{j\in\llbracket i+1,5\rrbracket\\\iota_j^1\neq0}}\iota_j\abs{\sum_{j'=1}^i\iota_j^{j'}x_{j'}-P_j}^2=\iota_{k_1}\abs{\sum_{j'=1}^i\iota_{k_1}^{j'}x_{j'}-P_{k_1}}^2+\cdots+\iota_{k_q}\abs{\sum_{j'=1}^i\iota_{k_q}^{j'}x_{j'}-P_{k_q}}^2\\
            =\l(\iota_{k_1}+\cdots+\iota_{k_q}\r)\abs{x_1}^2+2\iota_{k_1}\iota_{k_1}^1x_1\l(\sum_{j'=2}^i\iota_{k_1}^{j'}x_{j'}-P_{k_1}\r)+\cdots+2\iota_{k_q}\iota_{k_q}^1x_1\l(\sum_{j'=2}^i\iota_{k_q}^{j'}x_{j'}-P_{k_q}\r).
        \end{gathered}
        \end{equation}
        Thence, the total contribution of $x_1$ reads 
        \begin{equation}
            \l(\iota_1+\iota_{k_1}+\cdots+\iota_{k_q}\r)\abs{x_1}^2+2\iota_{k_1}\iota_{k_1}^1x_1\l(\sum_{j'=2}^i\iota_{k_1}^{j'}x_{j'}-P_{k_1}\r)+\cdots+2\iota_{k_q}\iota_{k_q}^1x_1\l(\sum_{j'=2}^i\iota_{k_q}^{j'}x_{j'}-P_{k_q}\r).\label{x1 contribution}
        \end{equation}
        There are now two qualitatively distinct cases to distinguish. The first case is $c\coloneqq\iota_1+\iota_{k_1}+\cdots+\iota_{k_q}\neq0$. Setting $c_l\coloneqq\iota_{k_l}\iota^1_{k_l}\l(\sum_{j'=2}^i\iota_{k_l}^{j'}x_{j'}-P_{k_l}\r)$, we have 
        \begin{equation}
            c\abs{x_1}^2 + 2(c_1+\cdots+c_q)x_1 = c\l(\abs{x_1}^2+\frac{2}{c}(c_1+\cdots+c_q)x_1\r)=c\abs{x_1+\frac{c_1+\cdots+c_q}{c}}^2-\frac{\abs{c_1+\cdots+c_q}^2}{c}
        \end{equation}
        The integral over $x_1$ now takes the form 
        \begin{equation}
            \begin{aligned}
                &\int\limits_{Q_{0,(2n+1)R}}\gamma_{\iota}\l(x_1,\ldots,x_i\r)\rmd x_1 \\&= \abs{A}^{-p}\int\limits_{Q_{0,(2n+1)R}}\frac{1}{\abs{1+\rmi A^{-1}\l(B+\epsilon^{-\tilde\gamma}\l(c\abs{x_1+\frac{c_1+\cdots+c_q}{c }}^2+R(x_2,\ldots,x_i)\r)\r)}^p}\rmd x_1\\
                &=\abs{A}^{-p}\int_{Q_{\frac{c_1+\cdots+c_{q}}{c},(2n+1)R}}\frac{1}{\abs{1+\rmi A^{-1}\l(B+\epsilon^{-\tilde\gamma}\l(c\abs{x_1}^2+R(x_2,\ldots,x_i)\r)\r)}^p}\rmd x_1,
            \end{aligned}
        \end{equation}
        where $R(x_2,\ldots,x_i)$ is the rest in terms of the remaining variables. One can check that $\frac{\abs{c_1+\cdots+c_q}}{\abs{c}}\leq2(5-i)(i-1)\sqrt{d}(2n+1)R$ and thus $Q_{\frac{c_1+\cdots+c_{q}}{c},(2n+1)R}\subseteq Q_{0,\Lambda n}$ which leads to 
        \begin{equation}
            \label{already done bitch}\int\limits_{Q_{0,(2n+1)R}}\gamma_{\iota}\l(x_1,\ldots,x_i\r)\rmd x_1\leq\abs{A}^{-p}\int_{Q_{0,\Lambda n}}\frac{\rmd x_1}{\abs{1+\rmi A^{-1}\l(B+\epsilon^{-\tilde\gamma}\l(c\abs{x_1}^2+R(x_2,\ldots,x_i)\r)\r)}^p}.
        \end{equation}
        and can be dealt with as in \cref{already estimated this shit}. It leads to 
        \begin{equation}
            \int\limits_{Q_{0,(2n+1)R}}\gamma_{\iota}\l(x_1,\ldots,x_i\r)\rmd x_1\leq\Lambda\abs{A}^{1-p}n^{d-2}\epsilon^{\tilde\gamma}
        \end{equation}
        so that 
        \begin{equation}
            \int\limits_{Q_{0,(2n+1)R}^i}\gamma_{\iota}\l(x_1,\ldots,x_i\r)\rmd\l(x_1,\ldots,x_i\r)\leq\Lambda\abs{A}^{1-p}n^{id-2}\epsilon^{\tilde\gamma}.
        \end{equation}

        In the second case, $c=0$. The sole contribution to the $x_1$ integral comes from 
        \begin{equation}\label{mixed combination}
            2\iota_{k_1}\iota_{k_1}^1x_1\l(\sum_{j'=2}^i\iota_{k_1}^{j'}x_{j'}-P_{k_1}\r)+\cdots+2\iota_{k_q}\iota_{k_q}^1x_1\l(\sum_{j'=2}^i\iota_{k_q}^{j'}x_{j'}-P_{k_q}\r)
        \end{equation}
        according to \cref{x1 contribution}. Since $i\geq2$, there exists some $\tilde c\in\R\setminus\{0\}$ such that one of the terms in \cref{mixed combination} is exactly $\tilde cx_1x_2$. If we substitute $x_2'=x_2+x_1$ for $x_2$, we have to carry out the integration over $x_2'$ over $Q_{x_1,2(2n+1)R}$ before the integration over $x_1$. But we can argue again that $Q_{x_1,2(2n+1)R}\subseteq Q_{0,4(2n+1)R}$, estimate the integral by replacing the integration domain of $x_2'$ with $Q_{0,4(2n+1)R}$ and then use Fubini's theorem to integrate first in the $x_1$ variable. The denominator of the integrand is now quadratic in the $x_1$ variable, so that we can apply \cref{already estimated this shit}. The remaining integrals contribute as $\mathcal O\l(n^{({i-1})d}\r)$. To summarize, 
        \begin{equation}
            \int\limits_{Q_{0,(2n+1)R}^i}\gamma_{\iota}\l(x_1,\ldots,x_i\r)\rmd\l(x_1,\ldots,x_i\r)\leq \Lambda\abs{A}^{1-p}n^{id-2}\epsilon^{\tilde\gamma}.
        \end{equation}
        \end{itemize}    
\end{proof}

\begin{corollary}\label{get gain although loss}
    Let $n\leq\abs{\ln\epsilon}$. Then there exists $\alpha=\alpha(d)>0$ such that for all $i\in\llbracket1,4\rrbracket$
    \begin{equation} \int\limits_{Q_{0,2(2n+1)R}^i}\gamma_{\iota}\l(x_1,\ldots,x_i\r)\rmd\l(x_1,\ldots,x_i\r)\leq\Lambda\abs{A}^{1-p}\epsilon^{\alpha}.
    \end{equation}
\end{corollary}
\begin{proof}
    This is a direct application of \cref{main technical necessary theorem}.
\end{proof}

\begin{lemma}\label{the lemma that fixes beta}
    The function $\gamma_\iota$ from \cref{main technical necessary theorem} satisfies 
    \begin{equation}
        \abs{\nabla\gamma_\iota}\leq\Lambda\abs{A}^{-p-1}\epsilon^{-\tilde\gamma}n.
    \end{equation}
    Furthermore, for all $i\in\llbracket1,4\rrbracket$,
    \begin{equation}
        \frac{1}{L}\sum_{x_1,\ldots,x_i\in Q_{0,2R}^{\Z_L^d}}\sup_{\bigtimes_{j=1}^iQ_{x_j,\frac{1}{2L}}}\abs{\nabla\gamma_\iota}\leq\Lambda L^{id}\abs{A}^{-p-1}n^{id+1}\epsilon^{\beta-\tilde\gamma}.
    \end{equation}
\end{lemma}
\begin{proof}
The calculations for all cases are very similar, and we exemplify this by considering
\begin{equation}
    \gamma_\iota(x,y)\coloneqq\frac{1}{\abs{A+\rmi\l(B+\epsilon^{-\tilde\gamma}\l(\iota_1\abs{x}^2+\iota_2\abs{y}^2-\iota_3\abs{x+y-P}^2+C\r)\r)}^p},
\end{equation}
where $A\in\R\setminus\{0\}$, $B,C\in\R$, $\iota\coloneqq\l(\iota_1,\iota_2,\iota_3\r)\in\{\pm\}^3$ and $P\in B_{(2n+1)R}(0)$. The straightforward calculation

\begin{equation}
    \nabla\gamma_\iota(x,y)=-\frac{2p\epsilon^{-\tilde\gamma}\l(B+\epsilon^{-\tilde\gamma}\l(\iota_1\abs{x}^2+\iota_2\abs{y}^2-\iota_3\abs{x+y-P}^2+C\r)\r)}{\abs{A+\rmi\l(B+\epsilon^{-\tilde\gamma}\l(\iota_1\abs{x}^2+\iota_2\abs{y}^2-\iota_3\abs{x+y-P}^2+C\r)\r)}^{p+2}}\begin{pmatrix}
        \iota_1x-\iota_3(x+y-P)\\\iota_1y-\iota_3(x+y-P)
    \end{pmatrix}
\end{equation}
delivers
\begin{equation}
    \abs{\nabla\gamma_\iota}\leq\Lambda\abs{A}^{-p-1}\epsilon^{-\tilde\gamma}\l(\abs{x}+\abs{y}+\abs{P}\r)
\end{equation}
and since $x,y,P\in B_{(2n+1)R}(0)$, we obtain 
\begin{equation}
    \abs{\nabla_{x,y}\gamma_\iota}\leq\Lambda\abs{A}^{-p-1}\epsilon^{-\tilde\gamma}n
\end{equation}
and therefore 
\begin{gather}
    \frac{1}{L}\sum_{x,y\in Q_{0,2(n+1)R}^{\Z_L^d}}\sup_{\substack{Q_{x,\frac{1}{2L}}\times Q_{y,\frac{1}{2L}}}}\abs{\nabla\gamma_\iota}\leq\Lambda L^{-1}\abs{A}^{-p-1}\epsilon^{-\tilde\gamma}n\abs{Q_{0,2(n+1)R}^{\Z_L^d}}^2\\
    \leq\Lambda L^{2d}\abs{A}^{-p-1}n^{2d+1}\epsilon^{\beta-\tilde\gamma},
\end{gather}
where in the last step we used $L^{-1}=\epsilon^\beta$.
\end{proof}

\begin{corollary}\label{gradients always give epsilon gain}
    If $\beta>\tilde\gamma$, there exists an $\alpha=\alpha(d,\beta)>0$ such that for all $n\leq\abs{\ln\epsilon}$
        \begin{equation}
        \frac{1}{L}\sum_{x_1,\ldots,x_i\in Q_{0,2R}^{\Z_L^d}}\sup_{\bigtimes_{j=1}^iQ_{x_j,\frac{1}{2L}}}\abs{\nabla\gamma_\iota}\leq\Lambda L^{id}\abs{A}^{-p-1}\epsilon^{\alpha}.
    \end{equation}
\end{corollary}
\begin{proof}
    This follows from \cref{the lemma that fixes beta}.
\end{proof}

\subsection{Weighing gains against losses}
\begin{definition}\label{definition of being a resonant node}
    Let $C$ be a couple. We call a node $\node\in\Node(C)$ \textbf{resonant} if $\node_1,\ldots,\node_5$ (unordered) denote its children with $\iota_{\node_1}+\iota_{\node_2}=0$, $\iota_{\node_3}+\iota_{\node_4}=0$, $K_C\l(\node_1\r)+K_C\l(\node_2\r)=0$, $K_C\l(\node_3\r)+K_C\l(\node_4\r)$ and $K_C\l({\node_5}\r)=K_C(\node)$. We call $\node$ \textbf{$1,\ldots,5$-resonant} if the children $\node_1,\ldots,\node_5$ are ordered from left to right and $K_C\l(\node_i\r)=K_C(\node)$ for some $i\in\{1,\ldots,5\}$. We denote $\Node_{\mathrm{res}}(C)\subseteq\Node(C)$ the subset of resonant nodes and $n_{\mathrm{res}}(C)\coloneqq\abs{\Node_{\mathrm{res}}(C)}$. 
\end{definition}

\begin{remark}\label{equivalent definition of being resonant}
     The definition for $\node\in\Node(C)$ to be resonant is equivalent to requiring that $\l\{\leaf\leq\node_1\mid\leaf\in\Leaf(C)\r\}\sqcup\l\{\leaf\leq\node_2\mid\leaf\in\Leaf(C)\r\}$ and $\l\{\leaf\leq\node_3\mid\leaf\in\Leaf(C)\r\}\sqcup\l\{\leaf\leq\node_4\mid\leaf\in\Leaf(C)\r\}$ be self-coupled. 
\end{remark}

We define 

\begin{equation}
    A_\node(\kappa,\xi)\coloneqq\frac{\abs{Q_\node}}{\abs{n\delta^{-1}-\rmi\l(\xi+\epsilon^{-1}\omega_\node(<')\r)}}
\end{equation}

and using \cref{change of variables}, we can consider the linear isomorphism $I\colon\l(\R^d\r)^{\Leaf(C)_+}\to\l(\R^d\r)^{2n(C)+1}$ defined by 
\begin{equation}
    I\l(\kappa\r)\coloneqq\l(K_C(\node_1)(\kappa),\ldots,K_C(\node_{2n})(\kappa),K_C(r_1)(\kappa)\r).
\end{equation}
By definition, $I\l(\l(\Z_L^d\r)^{\Leaf(C)_+}\r)\subseteq\l(\Z_L^d\r)^{2n(C)+1}$ and more precisely \begin{equation}
    I\l(\l(B_R^{\Z_L^d}(0)\r)^{\Leaf(C)_+}\r)\subseteq\l(B_{(2n(C)+1)R}^{\Z_L^d}(0)\r)^{2n(C)+1}
\end{equation} 
and 
\begin{equation}
    I\l(\D_k(C)\r)\subseteq\l(B_{(2n(C)+1)R}^{\Z_L^d}(0)\r)^{2n(C)}\times\{k\}.
\end{equation}
\begin{remark}
    From now on, we will abbreviate $n=n(C)$. As can be seen, the isomorphism $I$ rescales the ball within which our coordinates are taken at worst linearly in the number of nodes, that is, with $\mathcal O(n)$. We thus should be careful in accounting for potential losses that ultimately emerge from breaking mass conservation, see \cref{breaking mass conservation}. One can show that there exist couples that exhibit a factorial loss in the number of nodes.
\end{remark}

We now define 

\begin{equation}
    \tilde A_\node\l(\tilde k_1,\ldots,\tilde k_{2n},k,\xi\r)\coloneqq A_\node\l(I^{-1}\l(\tilde k_1,\ldots,\tilde k_{2n},k\r),\xi\r),
\end{equation}
where the domain is

\begin{equation}
    \l(\tilde k_1,\ldots,\tilde k_{2n}\r)\in\l(B_{(2n(C)+1)R}^{\Z_L^d}(0)\r)^{2n}.
\end{equation}
We identify total order relations $<_\rho\in\M(\Node(C))$ with the underlying bijections $\rho\colon\llbracket1,n\rrbracket\to\Node(C)$.
We may order 

\begin{equation}
    \bigsqcup_{i=1}^4\mathcal P_i=\l\{\node_1<_\rho\cdots<_\rho\node_{n(C)-p_0}\r\}
\end{equation}
and initially estimate 

\begin{equation}
\begin{aligned}
    \prod_{\node\in\Node(C)}A_\node(\kappa,\xi)&\leq\l(\frac{\delta}{n}\r)^{p_0}\prod_{i=1}^4\prod_{\node\in\mathcal P_i}A_\node(\kappa,\xi)\\
    &=\l(\frac{\delta}{n}\r)^{p_0}\prod_{i=1}^{n(C)-p_0}A_{\node_i}(\kappa,\xi)
\end{aligned}
\end{equation}
so that 
\begin{equation}\label{needs structure}
    \sum_{\substack{\kappa\in\D_k(C)\\\kappa\l(\Leaf(C)_+\r)\subseteq B_R(0)}}\prod_{\node\in\Node(C)}A_\node(\kappa,\xi)\leq\l(\frac{\delta}{n}\r)^{p_0}\sum_{\l(x_i\r)_{i=1}^{2n}\in\l(B_{(2n+1)R}^{\Z_L^d}(0)\r)^{2n}}\prod_{i=1}^{n-p_0}\tilde A_{\node_i}(\bm x,\xi),
\end{equation}
where $\bm x=\l(x_1,\ldots,x_{2n}\r)$.
Denoting $\node_i\in\mathcal P_{j_i}$ for $j_i\in\llbracket1,4\rrbracket$ and $i\in\llbracket1,n-p_0\rrbracket$ such that $\sum_{i=1}^{n-p_0}j_i=2n$ (according to \cref{k=0 is included lol}), we may give the sum in \cref{needs structure} additional structure by observing that
\begin{equation}
    \begin{gathered}
        \sum_{\substack{\kappa\in\D_k(C)\\\kappa\l(\Leaf(C)_+\r)\subseteq B_R(0)}}\prod_{\node\in\Node(C)}A_\node(\kappa,\xi)\\\leq\l(\frac{\delta}{n}\r)^{p_0}\sum_{\l(x_i\r)_{i=2n(C)-j_{n(C)-p_0}+1}^{2n(C)}\in\l(B_{(2n+1)R}^{\Z_L^d}(0)\r)^{j_{n(C)-p_0}}}\tilde A_{\node_{n(C)-p_0}}(\bm{x},\xi)\cdots\sum_{\l(x_i\r)_{i=1}^{j_1}\in\l(B_{(2n+1)R}^{\Z_L^d}(0)\r)^{j_{1}}}\tilde A_{\node_1}(\bm{x},\xi).
    \end{gathered}
\end{equation}

Although, by construction, \begin{equation}
    \Node_{<''}(C)\coloneqq\l\{\tilde\node_1<''\cdots<''\tilde\node_{2n(C)}<''\root_1\r\}
\end{equation} 
and thus also $\mathcal P_i$ and $p_i\coloneqq\abs{\mathcal P_i}$ depend on the underlying order $<''\in\M\l(\Leaf(C)\sqcup\Node(C)\r)$, we suppress this dependency as the following two important relations are independent of $<''\in\M(\Leaf(C)\sqcup\Node(C))$:
\begin{align}
    p_0+p_1+p_2+p_3+p_4&=n,\label{first important independent}\\
    p_1+2p_2+3p_3+4p_4&=2n.\label{second important independent}
\end{align}
We denote 
\begin{align}
    \mathcal P_{\mathrm{res},i}&\coloneqq\mathcal P_i\cap\Node_{\mathrm{res}}(C),\\
    p_{\mathrm{res},i}&\coloneqq\abs{\mathcal P_{\mathrm{res},i}}.
\end{align}
and recall from \cref{classifying resonant nodes} that $p_{\mathrm{res},3}=p_{\mathrm{res},4}=0$, so that $p_{\mathrm{res},1}+p_{\mathrm{res},2}=n_{\mathrm{res}}$. We further denote by $\tilde{\mathcal P}_{1}$ the subset of degree $1$ nodes so that the vector $P$ (which represents a signed combination of resonant factors whose wave numbers are not summed over) from point 1.) of \cref{main technical necessary theorem} satisfies $\abs{P}<\epsilon^\gamma$ for some fixed $\gamma\in(0,1)$. Denote $\tilde p_1\coloneqq\abs{\mathcal P_1}$. One can prove $\mathcal P_{\mathrm{res},1}\subsetneq\tilde{\mathcal P}_1$ for our quintic non-linearity which requires us to be careful in accounting for gains in $\epsilon$ and potential losses for degree $1$ nodes.

\begin{lemma}\label{the summation cases}
    Let $i\in\llbracket1,n(C)-p_0\rrbracket$ and $p> 1$. There exists an $\alpha_0>0$ such that for all $n\leq\abs{\ln\epsilon}$, 
    \begin{equation}\label{All summation cases}
    \begin{gathered}
        \sum_{\l(x_j\r)_{j=\sum_{k=1}^{i-1}j_k+1}^{\sum_{k=1}^{i}j_k}\in\l(B_{(2n+1)R}^{\Z_L^d}\r)^{j_i}}\tilde A_{\node_{i}}^p\\\leq\begin{cases}
        \Lambda L^{j_id}\l(\frac{\delta}{n}\r)^p\epsilon^{j_i\alpha_0}&\text{if }j_i\geq2\text{ and }\node_{i}\in\mathcal P_{j_i}\setminus\mathcal P_{\mathrm{res},j_i}\text{ or }j_i=1\text{ and }\node_{i}\in\mathcal P_1\setminus\tilde{\mathcal P}_1,\\
        \Lambda L^{j_id}\l(\frac{\delta}{n}\r)^p&\text{if }j_i\geq2\text{ and }\node_i\in\mathcal P_{\mathrm{res},j_i}\text{ or }j_i=1\text{ and }\node_{i}\in\tilde{\mathcal P}_1.
    \end{cases}        
    \end{gathered}
\end{equation}
\end{lemma}

\begin{proof}
We recall that

\begin{equation}
    \tilde\omega_{\node_{i}}(<')= \iota_{\node_{i}}\tilde\Omega_{\node_{i}}+\sum_{\node_{i}<\node'\in\Node(C)}\iota_{\node'}\tilde\Omega_{\node'},
\end{equation}
where the notation $\tilde\omega$ and $\tilde\Omega$ denotes the application of the coordinate transformation $I$. Now, $\tilde A_{n_{j_i}}$ will be of the form of one of the $\gamma_\iota$ that were defined in \cref{main technical necessary theorem}. More precisely, in the notation of \cref{main technical necessary theorem}, $A=n$, $B=\xi$, $C=\sum_{\node_{j_i}<\node'\in\Node(C)}\iota_{\node'}\tilde\Omega_{\node'}$ and $\tilde\gamma=1$. We may initially convert the sum(s) on the left-hand side of \cref{All summation cases} into sum(s) over gradients and integral(s) over cubes as presented generally in \cref{sums into integrals}. Now, the case $j_i=1$ and $\node_{i}\in\mathcal P_1\setminus\tilde{\mathcal P}_1$ corresponds to point 1.) of \cref{main technical necessary theorem} and the case $j_i\geq2$ and $n_i\in\mathcal P_{j_i}\setminus\mathcal P_{\mathrm{res},j_i}$ corresponds to case 2.) of \cref{main technical necessary theorem}. In these two cases, we apply \cref{get gain although loss,gradients always give epsilon gain} and obtain 

\begin{equation}
    \begin{aligned}
    \sum_{\l(x_j\r)_{j=\sum_{k=1}^{i-1}j_k+1}^{\sum_{k=1}^{i}j_k}\in\l(B_{(2n+1)R}^{\Z_L^d}\r)^{j_i}}\tilde A_{\node_{i}}^p&\leq\Lambda L^{j_id}\l(\frac{\delta}{n}\r)^{p+1}\epsilon^{\alpha}+\Lambda L^{j_id}\l(\frac{\delta}{n}\r)^{p-1}\epsilon^{\alpha}\label{something i gotta talk about}\\
    &\leq\Lambda L^{j_id}\l(\frac{\delta}{n}\r)^p\epsilon^{j_i\alpha_0},
\end{aligned}
\end{equation}
for all $0<\epsilon\leq\delta^{\mathcal W}\leq\delta\ll1$, where $\mathcal W>0$ is large enough and we estimated the first term of \cref{something i gotta talk about} trivially and introduced an additional factor of $n^{-1}$ into the second term while simultaneously exploiting $n\leq\abs{\ln\epsilon}$.

The remaining cases of \cref{All summation cases} are relying on \cref{the assumption i should use}. 
Indeed, 
\begin{equation}
    \sum_{\l(x_j\r)_{j=\sum_{k=1}^{i-1}j_k+1}^{\sum_{k=1}^{i}j_k}\in\l(B_{(2n+1)R}^{\Z_L^d}\r)^{j_i}}\tilde A_{\node_{i}}\leq\frac{\delta}{n}\sum_{\l(x_j\r)_{j=\sum_{k=1}^{i-1}j_k+1}^{\sum_{k=1}^{i}j_k}\in\l(B_{(2n+1)R}^{\Z_L^d}\r)^{j_i}}\abs{\tilde Q_{\node_i}}\leq \Lambda L^{j_id}\frac{\delta}{n}
\end{equation}
and this concludes the proof.
\end{proof}

\begin{remark}\label{observation on cardinality}
    Let $C\in\mathcal C^{\eta,\eta',\iota,-\iota}_{n,n}$, then $C$ contains $4n+1$ positive (and negative) leaves. For this particular fixed couple, there are $(4n+1)!$ many coupling maps. Thence
    \begin{equation}
        \abs{\mathcal C^{\eta,\eta',\iota,-\iota}_{n,n}}\leq\Lambda^n(4n+1)!.
    \end{equation}
    and 
\begin{equation}
    (4n+1)!\leq(4n+1)^{4n+1}\leq(5n)^{5n} = 5^{5n}n^{5n}\leq\Lambda^nn^{5n}.
\end{equation}
\end{remark}

\begin{definition}
    Let $C$ be a couple and $\node\in\Node(C)\sqcup\Leaf(C)$. We call 
    \begin{equation}
        \mathrm{Off}(\node)\coloneqq\{\leaf\leq\node\mid\leaf\in\Leaf(C)\}
    \end{equation}
    the \textbf{offspring set} of $\node$.
\end{definition}
The proof of the cardinality of the set of couples, each of which contains a constant number of resonant nodes, follows the strategy of Lemma 2.4 of \cite{desuzzoni2025waveturbulencesemilinearkleingordon}.
\begin{lemma}\label{cardinality estimate on the number of couples with q resonant nodes}
    We have 
    \begin{equation}
        c_{n,q}\coloneqq\abs{\bigsqcup_{n_1+n_2=n}\l\{C\in\mathcal C_{n_1,n_2}\mid n_{\mathrm{res}}(C)=q\r\}}\leq\Lambda^{n}\l(2(n-q)+2\r)!
    \end{equation}
\end{lemma}
\begin{proof}
    Let $C=\l(T_1,T_2,\sigma\r)\in\mathcal C_{n_1,n_2}$ and $\node\in\Node_{\mathrm{res}}(C)$ be a resonant node (if existent). Its children $C(\node)=\l\{\node_1,\ldots,\node_5\r\}$ (unordered list) satisfy $K_C\l(\node_1\r)+K_C\l(\node_2\r)=0$ and $K_C\l(\node_3\r)+K_C\l(\node_4\r)=0$ which is equivalent to stating that $X_{\node,1}\coloneqq\mathrm{Off}\l(\node_1\r)\sqcup\mathrm{Off}\l(\node_2\r)$ and $X_{\node,2}\coloneqq\mathrm{Off}\l(\node_3\r)\sqcup\mathrm{Off}\l(\node_4\r)$ are self-coupled (see \cref{equivalent definition of being resonant}). We enumerate the resonant nodes $\l(\node_i\r)_{i=1}^q$ and also the associated self-coupled offspring sets $\l(X_{\node_i,1},X_{\node_i,2}\r)_{i=1}^q$. We set $m_{i,j}\coloneqq \abs{\l(X_{\node_i,j}\r)_+}$ to be the cardinality of the positive leaves inside $X_{\node_i,j}$. The number of pairings $\sigma$ for which $X_{\node_i},j$ is self-coupled for all $i\in\llbracket1,q\rrbracket$ and $j=1,2$ is 
    \begin{equation}\label{combinatorial expression}
        \prod_{(i,j)\in\llbracket1,q\rrbracket\times\{1,2\}}m_{i,j}!\times\l(2n+1-\sum_{(i,j)\in\llbracket1,q\rrbracket\times\{1,2\}}m_{i,j}\r)!,
    \end{equation}
    where $2n+1-\sum_{(i,j)\in\llbracket1,q\rrbracket\times\{1,2\}}m_{i,j}$ corresponds to the number of positive leaves inside $\Leaf(C)_+$ that do not belong to any of the $X_{\node_i,j}$. We apply the fact that for any sequence of positive integers $\l(a_i\r)_{i\in\llbracket1,q\rrbracket}$, we have 
    \begin{equation}
        \prod_{i=1}^qa_i!\leq\l(\sum_{i=1}^qa_i-q+1\r)!.
    \end{equation}
    Using this, we obtain
    \begin{equation}\label{estimate on different pairings}
        \abs{\eqref{combinatorial expression}}\leq\l(\sum_{(i,j)\in\llbracket1,q\rrbracket\times\{1,2\}}m_{i,j}-2q+1\r)!\l(2n+1-\sum_{(i,j)\in\llbracket1,q\rrbracket\times\{1,2\}}m_{i,j}\r)!\leq(2(n-q)+2)!
    \end{equation}
    The choice of $2q$ self-coupled offspring sets among the $2n$ possibilities, we obtain 
\begin{equation}
    \abs{\l\{\sigma\text{ pairing such that }n_{\mathrm{res}}(C)=q\r\}}\leq\binom{2n}{2q}(2(n-q)+2)!\leq\Lambda_1^n(2(n-q)+2)!.
\end{equation}
Now 
\begin{equation}\label{estimate on set of trees}
    \abs{\bigsqcup_{\substack{\eta,\eta'\in\{0,1\}\\\iota,\iota'\in\{\pm\}\\n_1+n_2=n}}\T_{n_1}^{\iota,\eta}\times\T_{n_2}^{\iota',\eta'}}\leq\Lambda_2^n
\end{equation}
such that 
\begin{equation}
c_{n,q}\leq\eqref{estimate on different pairings}\cdot\eqref{estimate on set of trees}\leq\Lambda^n(2(n-q)+2)!,    
\end{equation}
where $\Lambda\coloneqq\Lambda^1\Lambda^2$.
\end{proof}

\begin{proposition}\label{key proposition}
Recall

    \begin{equation}\label{to be argued against intensely}
    \sup_{t\in[0,\delta]}\abs{{\J_C^{<'}}\l(\epsilon^{-1}t,k\r)}\leq\Lambda^{n+1}L^{-2dn}\int_\R\frac{\sum\limits_{\substack{\kappa\in\mathcal D_k(C)\\\kappa\l(\Leaf(C)_+\r)\subseteq B_R(0)}}\prod\limits_{\node\in\Node(C)}\frac{\abs{Q_\node}}{\abs{n\delta^{-1}-\rmi\l(\xi+\epsilon^{-1}\omega_\node(<')\r)}}}{\abs{n\delta^{-1}-\rmi\xi}}\rmd\xi.
\end{equation}
    There exist $\mathcal K,\alpha>0$ such that for all $n=n(C),2k\leq\abs{\ln\epsilon}$
    \begin{align}
        \sup_{t\in[0,\delta]}\abs{\J_C^{<'}\l(\epsilon^{-1}t,k\r)}&\leq\Lambda(\Lambda\delta)^n\epsilon^{\alpha\l(n-n_{\mathrm{res}}\r)}\frac{L^{\mathcal K}}{n^{n+1}},\label{the first to be estimated}\\
        \sup_{t\in[0,\delta]}\abs{\J_C\l(\epsilon^{-1}t,k\r)}&\leq\Lambda\l(\Lambda\delta\r)^n\epsilon^{\alpha\l(n-n_{\mathrm{res}}\r)}\frac{L^{\mathcal K}}{n}\label{the second to be estimated},\\
        \sup_{t\in[0,\delta]}\E\l(\abs{\widehat{f_k^\eta}\l(\epsilon^{-1}t,k\r)}\r)&\leq\Lambda(\Lambda\delta)^{2k}\frac{L^{\mathcal K}}{2k} \label{the second estimate}
    \end{align}
\end{proposition}

\begin{proof}

We divide the integral in \cref{to be argued against intensely} into the regions $\l\{\abs{\xi}\leq L^{\mathcal K}\r\}$ and $\l\{\abs{\xi}>L^{\mathcal K} \r\}$ and consider the region $\l\{\abs{\xi}\leq L^{\mathcal K}\r\}$ first. 
Let $p>1$. Using Hölder's inequality for the sum over decorations leads to

\begin{equation}\label{moinsoni}
    \begin{gathered}
        L^{-2dn}\sum\limits_{\substack{\kappa\in\mathcal D_k(C)\\\kappa\l(\Leaf(C)_+\r)\subseteq B_R(0)}}\prod\limits_{\node\in\Node(C)}A_\node\leq\Lambda^n\l(L^{-2dn}\sum\limits_{\substack{\kappa\in\mathcal D_k(C)\\\kappa\l(\Leaf(C)_+\r)\subseteq B_R(0)}}\prod\limits_{\node\in\Node(C)}A_\node^p\r)^{1/p},
    \end{gathered}
\end{equation}

where we have again used \begin{equation}
    \l(\sum_{\substack{\kappa\in\mathcal D_k(C)\\\kappa\l(\Leaf(C)_+\r)\subseteq B_R(0)}}1\r)^{1-1/p}\leq\Lambda^{n}L^{2dn(1-1/p)}.
\end{equation}

We may apply \cref{the summation cases} and \cref{the assumption i have to use} and bound  

\begin{equation}\label{the estimate to be estimated}
    \begin{gathered}
        \sum\limits_{\substack{\kappa\in\mathcal D_k(C)\\\kappa\l(\Leaf(C)_+\r)\subseteq B_R(0)}}\prod\limits_{\node\in\Node(C)}A_\node^p \\\leq\Lambda^nL^{2dn}\l(\frac{\delta}{n}\r)^{pn}\epsilon^{\alpha_0\l(\l(p_1-\tilde p_{1}\r)+2\l(p_2-p_{\mathrm{res},2}\r)+3p_3+4p_4\r)}
    \end{gathered}
\end{equation}
We crucially observe that 
\begin{equation}
\begin{gathered}
    \l(p_1-\tilde p_{1}\r)+2\l(p_2-p_{\mathrm{res},2}\r)+3p_3+4p_4 = 2\l(n-n_{\mathrm{res}}\r)-\l(\tilde p_1-p_{\mathrm{res},1}\r)+p_{\mathrm{res},1}\\\geq n-n_{\mathrm{res}}
\end{gathered}
\end{equation}
having used $p_i-p_{\mathrm{res},i}\leq n-n_{\mathrm{res}}$ and $\tilde p_1-p_{\mathrm{res},1}\leq n-n_{\mathrm{res}}$. Inserting this into \cref{the estimate to be estimated} gives
\begin{equation}
    \abs{\eqref{the estimate to be estimated}}\leq\Lambda^nL^{2dn}\l(\frac{\delta}{n}\r)^{pn}\epsilon^{\alpha_0\l(n-n_{\mathrm{res}}\r)}.
\end{equation}
Combining these estimates with the integration over $\l\{\abs{\xi}\leq L^{\mathcal K}\r\}$ delivers
\begin{equation}\label{what we are about to estimate motherfucker}
\begin{aligned}
    \int_{\abs{\xi}\leq L^{\mathcal K}}\Big[\cdots\Big]\rmd\xi&\leq\Lambda^n \l(\frac{\delta}{n}\r)^n\epsilon^{\frac{\alpha_0}{p}\l(n-n_{\mathrm{res}}\r)}\int_{\abs{\xi}\leq L^{\mathcal K}}\frac{\rmd\xi}{\abs{n-\rmi\xi}}\\&\leq\Lambda^n\frac{\delta^n}{n^{n+1}}\epsilon^{\frac{\alpha_0}{p}\l(n-n_{\mathrm{res}}\r)} L^{\mathcal K},
\end{aligned}
\end{equation}
where the numerator of the integrand in \cref{what we are about to estimate motherfucker} is exactly the right-hand side of \cref{moinsoni}.

For the second integration regime, we take into account $\abs{\Omega_\node}\leq\Lambda n^2$ such that $\abs{\omega_\node(<')}\leq\Lambda n^3\leq\Lambda\abs{\log\epsilon}^3\leq\Lambda{\log(L)}^3$ for all $\node\in\Node(C)$. Thence, $\epsilon^{-1}\abs{\omega_\node\l(<'\r)}\leq\Lambda L^{\frac{1}{\beta}}\log(L)^3$. Consider first the regime $\l\{\xi\leq -L^{\mathcal K}\r\}$, extract the sum $\sum\limits_{\substack{\kappa\in\D_k(C)\\\kappa\l(\Leaf(C)_+\r)\subseteq B_R(0)}}$ out of the integral and write the remaining integral as

\begin{equation}
    \int_{-\infty}^{-L^{\mathcal K}}\Big[\cdots\Big]\rmd\xi = \l(\frac{\delta}{n}\r)^{n+1}\int_{-\infty}^{-L^{\mathcal K}}\frac{1}{\abs{1-\rmi \delta n^{-1}\xi}}\prod\limits_{\node\in\Node(C)}\frac{\abs{Q_\node}}{\abs{1-\rmi \delta n^{-1}\l(\xi+\epsilon^{-1}\omega_\node(<')\r)}}\rmd\xi.
\end{equation}
What is more, 
\begin{equation}
    \begin{gathered}
        \frac{1}{\abs{1-\rmi \delta n^{-1}\l(\xi+\epsilon^{-1}\omega_\node(<')\r)}}\leq\frac{n}{\delta\abs{\xi+\epsilon^{-1}\omega_\node(<')}}\leq\frac{n}{\delta\l(-\xi-\epsilon^{-1}\abs{\omega_\node(<')}\r)}\\\leq\frac{1}{\Lambda L^{\frac{1}{\beta}}\log(L)^3}\frac{n}{\delta\l(-\frac{\xi}{\Lambda L^{\frac{1}{\beta}}\log(L)^3}-1\r)}
    \end{gathered}
\end{equation}
and thus, if we choose $\mathcal K>\frac{1}{\beta}$ and $L\gg1$ large enough, $-\frac{\xi}{\Lambda L^{\frac{1}{\beta}}\log(L)^3}$ is very large so that

\begin{equation}
    \frac{1}{-\frac{\xi}{\Lambda L^{\frac{1}{\beta}}\log(L)^3}-1}\leq\frac{\Lambda L^{\frac{1}{\beta}}\log(L)^3}{\abs{\xi}}.
\end{equation}
We have
\begin{equation}
    \frac{1}{\abs{1-\rmi \delta n^{-1}\l(\xi+\epsilon^{-1}\omega_\node(<')\r)}}\leq\frac{n}{\delta}\frac{\Lambda}{\abs{\xi}}
\end{equation}
and find 
\begin{equation}
    \int_{-\infty}^{-L^{\mathcal K}}\Big[\cdots\Big]\rmd\xi\leq\Lambda^n\int_{-\infty}^{L^{\mathcal K}}\frac{1}{\abs{\xi}^{n+1}}\rmd\xi\leq\frac{\Lambda^n}{n}L^{-\mathcal Kn}.
\end{equation}
The same can be said for the integration regime $\l\{\xi\geq L^{\mathcal K}\r\}$ and so 

\begin{equation}
    \int_{\abs{\xi}\geq L^{\mathcal K}}\Big[\cdots\Big]\rmd\xi\leq\frac{\Lambda^n}{n}L^{-\mathcal Kn}
\end{equation}
which is an arbitrary large power of $L^{-1}$ and thus makes the estimate trivial for this regime. 

We have proven \cref{the first to be estimated}. \Cref{the second to be estimated} follows simply from 
\begin{equation}
    \abs{\J_C\l(\epsilon^{-1}t,k\r)}\leq\abs{\M(\Node(C))}\abs{\J_C^{<'}\l(\epsilon^{-1}t,k\r)}
\end{equation}
and $\abs{\M(\Node(C))}\leq n!$.

Finally, with \cref{cardinality estimate on the number of couples with q resonant nodes},
\begin{equation}
    \E\l(\abs{f_n^\eta\l(\epsilon^{-1}t,k\r)}^2\r)\leq\sum_{C\in\mathcal C_{n,n}^{\eta,\eta,+,-}}\abs{\J_C\l(\epsilon^{-1}t,k\r)}\leq\Lambda(\Lambda\delta)^{2n}\frac{L^{\mathcal K}}{2n}\sum_{q=1}^n\epsilon^{\alpha(n-q)}c_{n,q}\leq\Lambda(\Lambda\delta)^{2n}\frac{L^{\mathcal K}}{2n}\sum_{m=0}^{n-1}\epsilon^{\alpha m}(2m+2)!
\end{equation}
and we may use that $(2m+2)!\leq(4m)!\leq(4m)^{4m}\leq16^n\abs{\ln\epsilon}^{4m}$ such that 
\begin{equation}
    \E\l(\abs{f_n^\eta\l(\epsilon^{-1}t,k\r)}^2\r)\leq\Lambda(\Lambda\delta)^{2n}\frac{L^{\mathcal K}}{2n}\sum_{m=0}^{n-1}\l(\epsilon^{\alpha}\abs{\ln\epsilon}^4\r)^m\leq\Lambda(\Lambda\delta)^{2n}\frac{L^{\mathcal K}}{2n}\sum_{m=0}^{n-1}\epsilon^{\alpha'm}\leq\Lambda(\Lambda\delta)^{2n}\frac{L^{\mathcal K}}{2n},
\end{equation}
for any $\alpha'\in\l(0,\frac{\alpha}{4}\r)$.
\end{proof}

\subsection{The time derivative of $\J_T$}\label{derivative moving}

Recall that 

\begin{equation}
\begin{gathered}
    \widehat{\J_{T}}(t,k) = \l(-\frac{\rmi\epsilon}{L^{2d}}\r)^n\prod_{\node\in\Node(C)}\iota_\node\sum_{\kappa\in\mathcal D_k(T)}\prod_{\node\in\Node(T)}Q_\node^T(\kappa)\\\cdot\int_{I_T(t)}\prod_{\node\in\Node(T)}e^{\rmi\iota_\node\Omega^T_\node(\kappa)t_\node}\rmd t_\node\prod_{\leaf\in\Leaf(T)}\mu^{\eta_\leaf,\iota_\leaf}_{\kappa(\leaf)}
\end{gathered}
\end{equation}
We identify $\M(\Node(A))$ with the set of bijections $\llbracket1,n(A)\rrbracket\to\Node(A)$ and can use as before that

\begin{equation}\label{separation of integration domains}
    I_A(t) = \mathcal Z_A\sqcup\bigsqcup_{\rho\in\M(\Node(A))}\l\{\l(t_\node\r)_{\node\in\Node(A)}\in[0,t]^{n(A)}\bigm\vert0\leq t_{\rho(1)}<\cdots<t_{\rho(n(A))}\leq t\r\},
\end{equation}
where $\mathcal Z_A$ is a subset of $[0,t]^{n(A)}$ of measure zero. It constitutes the boundaries of the right subset in \cref{separation of integration domains}. Then it is generally the case that  

\begin{gather}\label{want to take time derivative}
    \int_{I_A(t)}\prod_{\node\in\Node(A)}f_\node(t_\node)\rmd t_\node = \sum_{\rho\in\M(\Node(A))}\int_{0\leq t_{\rho(1)}<\cdots< t_{\rho(n(A))}\leq t}\prod_{i=1}^{n(A)}f_{\rho(i)}(t_{\rho(i)})\rmd t_{\rho(i)}
\end{gather}
Taking now the time derivative of \cref{want to take time derivative}, we find 
\begin{gather}
    \p_t\int_{I_A(t)}\prod_{\node\in\Node(A)}f_\node(t_\node)\rmd t_\node = \sum_{\rho\in\M(\Node(A))}f_{\rho(n(A))}(t)\int_{0\leq t_{\rho(1)}<\cdots<t_{\rho(n(A)-1)}\leq t}\prod_{i=1}^{n(A)-1}f_{\rho(i)}\l(t_{\rho(i)}\r)\rmd t_{\rho(i)}\\f_{\root_A}(t)\sum_{\rho\M(\Node(A)\setminus\{\root_A\}}\int_{0\leq t_{\rho(1)}<\cdots t_{\rho(n(A)-1)}\leq t}\prod_{i=1}^{n(A)-1}f_{\rho(i)}(t_{\rho(i)})\rmd t_{\rho(i)} = f_{\root_A}(t)\int_{I_A'(t)}\prod_{\node\in\Node(A)\setminus\{\root_A\}}f_\node(t_\node)\rmd t_\node,
\end{gather}
where 

\begin{equation}
    I_A'(t)\coloneqq\l\{\l(t_\node\r)_{\node\in\Node(A)\setminus\{\root_A\}}\in[0,t]^{n(A)-1}\mid\node\leq\node'\text{ implies }t_\node\leq t_{\node'}\r\}.
\end{equation}
This implies 
\begin{equation}
    \begin{gathered}
        \p_t\widehat{\J_{T}}(t,k) = \l(-\frac{\rmi\epsilon}{L^{2d}}\r)^n\prod_{\node\in\Node(C)}\iota_\node\sum_{\kappa\in\mathcal D_k(T)}\prod_{\node\in\Node(T)}Q_\node^T(\kappa)e^{\rmi\iota\Omega_{\root_A}^T(\kappa)t}\\\cdot\int_{I_A'(t)}\prod_{\node\in\Node(A)\setminus\{\root_A\}}e^{\rmi\iota\Omega_\node^T(\kappa)t_\node}\rmd t_\node\prod_{\leaf\in\Leaf(T)}\mu^{\eta_\leaf,\iota_\leaf}_{\kappa(\leaf)}
    \end{gathered}
\end{equation}
Similarly, as before, one proves that 
\begin{gather}
    \E\l(\p_t\widehat{F_{n_1}^{\eta_1,\iota_1}}\l(\epsilon^{-1}t,k\r)\p_t\widehat{F_{n_2}^{\eta_2,\iota_2}}\l(\epsilon^{-1}t,k\r)\r) = \sum_{C\in\mathcal C_{n_1,n_2}^{\eta_1,\eta_2,\iota_1,\iota_2}}{\J_C'}\l(\epsilon^{-1}t,k\r),
\end{gather}
where 
\begin{equation}
    \begin{gathered}
    {\J_C'}\l(\epsilon^{-1}t,k\r)\coloneqq\l(\frac{-\rmi}{L^{2d}}\r)^{n(C)}\prod_{\node\in\Node(C)}\iota_\node\sum_{\kappa\in\mathcal D_k(C)}\l[\prod_{\node\in\Node(T)}Q_\node \r]e^{\rmi\epsilon^{-1}\iota_{\root_1}\Omega_{\root_1}t}e^{\rmi\epsilon^{-1}\iota_{\root_2}\Omega_{\root_2}t}\\\cdot\int_{I_C'(t)}\prod_{\node\in\Node(C)\setminus\mathcal R(C)}e^{\rmi\epsilon^{-1}\iota_\node\Omega_\node t_\node}\rmd t_\node\cdot\prod_{\leaf\in\Leaf(C)_+}M^{\eta_\leaf,\eta_{\sigma(\leaf)}}(\kappa(\leaf))^{\iota_\leaf},
\end{gathered}
\end{equation}
where $I_C'(t)\coloneqq I_{T_1}'(t)\times I_{T_2}'(t)$.
\begin{proposition}\label{key proposition 2}
    There exist $\mathcal K,\alpha>0$ such that for all $n=n(C),2k\leq\abs{\ln\epsilon}$ 
    \begin{align}
        \sup_{t\in[0,\delta]}\abs{\J_C'\l(\epsilon^{-1}t,k\r)}&\leq\Lambda\l(\Lambda\delta\r)^n\epsilon^{\alpha\l(n-n_{\mathrm{res}}\r)}\frac{L^{\mathcal K}}{n},\\
        \sup_{t\in[0,\delta]}\E\l(\abs{\p_t\widehat{f_k^\eta}\l(\epsilon^{-1}t,k\r)}\r)&\leq\Lambda(\Lambda\delta)^{2k}\frac{L^{\mathcal K}}{2k}
    \end{align}
\end{proposition}
\begin{proof}
    The proof is very similar to that of \cref{key proposition}.
\end{proof}

\subsection{A norm estimate on the Dyson iterates}

The previous results may now be collected to prove

\begin{corollary}\label{The estimate the whole world was looking for}
    Let $A>0$. There exist $\Lambda,\mathcal K,\alpha>0$ such that with probability greater than or equal to $1-L^{-A}$, 
    \begin{equation}
        \norm{f_n^\eta}_{\mathcal C\l(\l[0,\delta\epsilon^{-1}\r],H^s\l(\mathbb T_L^d\r)\r)}\leq\Lambda(\Lambda\delta)^{2n}\delta^2{L^{\mathcal K+d+\frac{2}{\beta}+2A}}
    \end{equation}
for all $n\leq\abs{\ln\epsilon}$.
\end{corollary}

\begin{proof}
Using \cref{key proposition,key proposition 2}, we have 
\begin{equation}
    \E\l(\abs{\widehat{f^\eta_n}\l(\epsilon^{-1}t,k\r)}^2\r)+\E\l(\abs{\p_t\widehat{f^\eta_n}\l(\epsilon^{-1}t,k\r)}^2\r)\leq\Lambda(\Lambda\delta)^{2n}\frac{L^{\mathcal K}}{2n}
\end{equation}
and the fact that $\E\l(\abs{\widehat{f_n^\eta}\l(\epsilon^{-1}t,\cdot\r)}^2\r)$ and $\E\l(\abs{\p_t\widehat{f_n^\eta}\l(\epsilon^{-1}t,\cdot\r)}^2\r)$ have compact support inside $B_{(2n+1)R}(0)$ implies further 

\begin{equation}
    \E\l(\norm{f^\eta_n\l(\epsilon^{-1}t,\cdot\r)}^2_{H^s\l(\mathbb T_L^d\r)}\r)+\E\l(\norm{\p_tf^\eta_n\l(\epsilon^{-1}t,\cdot\r)}^2_{H^s\l(\mathbb T_L^d\r)}\r)\leq\Lambda(\Lambda\delta)^{2n}{L^{\mathcal K+d}}n^{2s+d-1}
\end{equation}
We find 

\begin{equation}
\begin{gathered}
    \E\l(\norm{f_n^\eta}^2_{\mathcal C\l([0,\delta\epsilon^{-1}],H^s\l(\mathbb T_L^d\r)\r)}\r)\leq\frac{\delta}{\epsilon}\E\l(\norm{f_n^\eta}^2_{H^1\l([0,\delta\epsilon^{-1}],H^s\l(\mathbb T_L^d\r)\r)}\r) \\= \frac{\delta}{\epsilon}\int_0^{\delta\epsilon^{-1}}\l(\E\l(\norm{f^\eta(t,\cdot)}^2_{H^s\l(\mathbb T_L^d\r)}\r)+\E\l(\norm{\p_tf^\eta(t,\cdot)}^2_{H^s\l(\mathbb T_L^d\r)}\r)\r)\rmd t\\\leq\Lambda(\Lambda\delta)^{2n}\delta^2{L^{\mathcal K+d+\frac{2}{\beta}}}n^{2s+d-1}
\end{gathered}
\end{equation}
If we define 
\begin{equation}
    a_n\coloneqq\Lambda(\Lambda\delta)^{2n}\delta^2{L^{\mathcal K+d+\frac{2}{\beta}+2A}}
\end{equation}
we have with Markov's inequality,
\begin{equation}
    \mathbb P\l(\norm{f_n^\eta}^2_{\mathcal C\l(\l[0,\delta\epsilon^{-1}\r],H^s\l(\mathbb T_L^d\r)\r)}>a_n\r)\leq a_n^{-1}\E\l(\norm{f_n^\eta}^2_{\mathcal C\l([0,\delta\epsilon^{-1}],H^s\l(\mathbb T_L^d\r)\r)}\r)\leq L^{-A}.
\end{equation}

\end{proof}

\subsection{Flower trees and the Fourier representation of $\Leaf^m$}

Recall 
\begin{equation}
    \begin{aligned}
        \widehat{F^\eta}(t,k) =\mu^\eta_k+\int_0^t C^{+}\l(\tau,F^\eta(\tau),\overline{F^\eta}(\tau),\overline{F^{\eta}}(\tau),{F^{\overline\eta}}(\tau),F^{\overline\eta}(\tau)\r)\rmd\tau
    \end{aligned}
\end{equation}
and define 
\begin{equation}
    \tilde C^\eta\l(t,f_1,\ldots,f_5\r)\coloneqq C^+\l(t,f_1^\eta,\overline{f_2^\eta},\overline{f_3^\eta},f_4^{\overline\eta},f_5^{\overline\eta}\r),
\end{equation}
where $f_i$ are considered to be $2$-component functions parametrized by the colour $\eta$.
Further, set
\begin{equation}
    F^\eta_{\leq N}\coloneqq\sum_{n\leq N}F^\eta_n
\end{equation}
and make the ansatz 
\begin{equation}
    F^\eta=F^\eta_{\leq N} + v^\eta
\end{equation}
so that $v^\eta = \mathcal W^\eta+ \Leaf(v)^\eta+\mathcal R^1(v)^\eta+\mathcal R^2(v)^\eta+\mathcal R^3(v)^\eta+\mathcal R^4(v)^\eta$,
where \begin{align}
    \mathcal W^\eta&\coloneqq-F_{\leq N}^\eta+F_0^\eta+\int_0^t C^{+}\l(\tau,F_{\leq N}^\eta(\tau),\overline{F_{\leq N}^\eta}(\tau),\overline{F_{\leq N}^{\eta}}(\tau),{F_{\leq N}^{\overline\eta}}(\tau),F_{\leq N}^{\overline\eta}(\tau)\r)\rmd\tau,\\
    &=\int_0^t\sum_{\substack{0\leq\sum_{i=1}^5n_i\leq N\\\sum_{i=1}^5n_i\geq N}}C^{+}\l(\tau,F_{n_1}^\eta(\tau),\overline{F_{n_2}^\eta}(\tau),\overline{F_{n_3}^{\eta}}(\tau),{F_{n_4}^{\overline\eta}}(\tau),F_{n_5}^{\overline\eta}(\tau)\r)\rmd\tau
    \\
    \mathcal L(v)^\eta&\coloneqq\int_0^t\sum_{\substack{i\in\llbracket1,5\rrbracket\\f_i=v\wedge f_j=F_{\leq N}\forall j\neq i}}\tilde C^\eta\l(\tau,f_1,\ldots,f_i,\ldots,f_5\r)\rmd\tau\\
    \mathcal R^1(v)^\eta&\coloneqq\int_0^t\sum_{\substack{i<j\in\llbracket1,5\rrbracket\\f_i=f_j=v\wedge f_k=F_{\leq N}\forall k\notin\{i,j\}}}\tilde C^\eta\l(f_1,\ldots,f_i,\ldots,f_j,\ldots,f_5\r)\rmd\tau,\\
    \mathcal R^2(v)^\eta&\coloneqq\int_0^t\sum_{\substack{i<j<k\in\llbracket1,5\rrbracket\\f_i=f_j=F_{\leq N}\wedge f_k=v\forall v\notin\{i,j\}}}\tilde C^\eta\l(f_1,\ldots,f_i,\ldots,f_j,\ldots,f_5\r)\rmd\tau,\\
    \mathcal R^3(v)^\eta&\coloneqq\int_0^t\sum_{\substack{i\in\llbracket1,5\rrbracket\\f_i=F_{\leq N}\wedge f_j=v\forall j\neq i}}\tilde C^\eta\l(\tau,f_1,\ldots,f_i,\ldots,f_5\r)\rmd\tau,\\
    \mathcal R^4(v)^\eta&\coloneqq\int_0^t C^{+}\l(\tau,v^\eta(\tau),\overline{v^\eta}(\tau),\overline{v^{\eta}}(\tau),{v^{\overline\eta}}(\tau),v^{\overline\eta}(\tau)\r)\rmd\tau,
\end{align}
where $\mathcal R^i(v)^\eta$ are respectively terms that are of order $i+1$ in $v$.

\begin{definition}

    We define the notion of a \textbf{flower tree} recursively (as the notion of a $5$-ary tree). In a flower tree, one leaf is always specified and called $\flower$. We set for brevity reasons $\T^1\coloneqq\T^{\iota,\eta}$, $\T^2=\T^3\coloneqq\T^{-\iota,\eta}$ and $\T^4=\T^5\coloneqq\T^{\iota,\overline\eta}$ and define recursively 
    \begin{equation}
        \begin{aligned}
            \l(\T_0^{\iota,\eta}\r)^\flower&\coloneqq\l\{\l(\flower,\iota,\eta\r)\r\},\\
             \l(\T_{n+1}^{\iota,\eta}\r)^\flower&\coloneqq\bigsqcup_{\substack{i\in\llbracket1,5\rrbracket}}\l[\bigsqcup_{T_1\in\T^1_{\leq N}}\cdots\bigsqcup_{T_i\in\l(\T^i_n\r)^\flower}\cdots\bigsqcup_{T_5\in\T_{\leq N}^5}\l\{\bullet\l(T_1,\ldots,T_i,\ldots,T_5\r)\r\}\r],
        \end{aligned}
    \end{equation}
    where $\mathcal T_{\leq N}^i\coloneqq\sqcup_{m=0}^N\mathcal T_m^i$. For any $T\in \l(\mathcal T_n^{\iota,\eta}\r)^\flower$, there exists a unique path from the root $\root$ to the flower $\flower$ that we call the \textbf{stem}. We denote the set of branching nodes of the stem by $\mathcal S(T)$. The \textbf{height} of a flower tree $T\in\l(\mathcal T_n^{\iota,\eta}\r)^\flower$ is defined to be $\abs{\mathcal S(T)}$. Iteratively, we have defined that a flower tree of height $n$ is formed by attaching four sub-trees each time of maximal scale $N$, and repeating $n$ times, starting from a single node. 
\end{definition}

\begin{definition}
    For brevity of notation, we denote $\T^1\coloneqq\T^{\iota,\eta}$, $\T^2=\T^3\coloneqq\T^{-\iota,\eta}$ and $\T^4=\T^5\coloneqq\T^{\iota,\overline\eta}$. We define the function
    \begin{equation}
        \J_{T}^\flower(t)\coloneqq\begin{cases}
            \J_{T}(t)&\text{if }T\in\mathcal T_{\leq N}^{\iota,\eta},\\
            v^{\eta,\iota}(t)&\text{if }T\in\l(\mathcal T_0^{\iota,\eta}\r)^\flower,\\
            \int_0^tC^{\iota}\l(\tau,\J_{T_1}^\flower(\tau),\J_{T_2}^\flower(\tau),\J_{T_3}^\flower(\tau),\J_{T_4}^\flower(\tau),\J_{T_5}^\flower(\tau)\r)\rmd\tau &\text{if }T\in\l(\mathcal T_n^{\iota,\eta}\r)^\flower\text{ and } n>0,
        \end{cases}
    \end{equation}
    and we understand $T=\bullet\l(T_1,T_2,T_3,T_4,T_5\r)$ with one $T_i\in\T^i_{n}$ and the rest $T_j\in\T^j_{\leq N}$ for all $j\neq i$.
\end{definition}

\begin{proposition}\label{tree decomposition for flower trees}
    For all $n\in\N$, $\eta\in\{0,1\}$ and $\iota\in\{\pm\}$, we have 
    \begin{equation}
        \mathcal L^n(v)^{\eta,\iota} = \sum_{T\in\l(\mathcal T_n^{\eta,\iota}\r)^\flower}\J_{T}^\flower
    \end{equation}
\end{proposition}
\begin{proof}
    We prove the assertion by the length of the stem of a flower tree. The statement is obviously true in the case $n=0$. Now assume $n>0$. We calculate, using the induction hypothesis in the first line, 
    \begin{equation}
        \begin{aligned}
        \mathcal L^{n+1}(v)^{\eta,\iota} &= \sum_{\substack{i\in\llbracket1,5\rrbracket}}\sum_{T_1\in\T_{\leq N}^1}\cdots\sum_{T_i\in\l(\T^i_n\r)^\flower}\cdots\sum_{T_5\in\T_{\leq N}^5}\int_0^tC^\iota\l(\tau,\J_{T_1}(\tau),\ldots,\J_{T_i}^\flower(\tau),\ldots,\J_{T_5}(\tau)\r)\rmd\tau\\&=\sum_{\substack{i\in\llbracket1,5\rrbracket}}\sum_{T_1\in\T_{\leq N}^1}\cdots\sum_{T_i\in\l(\T^i_n\r)^\flower}\cdots\sum_{T_5\in\T_{\leq N}^5}\J^\flower_{\bullet\l(T_1,\ldots,T_i,\ldots,T_5\r)}\\&=\sum_{T\in\l(\T_{n+1}^{\eta,\iota}\r)^\flower}\J_T^\flower
    \end{aligned}
    \end{equation}
\end{proof}

For $T\in\l(\T_n^{\iota,\eta}\r)^\flower$, we define $\mathcal D_{k,k_\flower}(T)$ as the subset of $\mathcal D_k(T)$ such that any $k$-decoration $\kappa$ of $T$ shall additionally satisfy $\kappa(\flower)=k_\flower$. From this definition, it is clear that 
\begin{equation}\label{obvious decomposition}
    \mathcal D_k(T) =\bigsqcup_{k_\flower\in\Z_L^d}\mathcal D_{k,k_\flower}(T)\text{ and thus }\sum_{\kappa\in\mathcal D(T)}=\sum_{k_\flower\in\Z_L^d}\sum_{\kappa\in\mathcal D_{k,k_\flower}(T)}.
\end{equation} 

\begin{remark}\label{trivial remark i should use somehow}
    Let $\tilde\node\in\Node(T)$ and define  
    \begin{equation}
        \begin{gathered}
            I_T\l(t,t_{\tilde\node}\r)\coloneqq\Big\{\l(t_\node\r)_{\node\in\Node(T)\setminus\{\tilde\node\}}\in[0,t]^{n(T)-1}\bigm\vert\node<\node'\text{ implies }t_\node< t_{\node'},\ \tilde\node<\node\\\text{ implies }t_{\tilde\node}< t_\node\text{ and }\node< \tilde\node\text{ implies }t_\node< t_{\tilde\node}\Big\},
    \end{gathered}
    \end{equation}
    then \begin{equation}
        I_T(t)=\l\{\l(t_\node\r)_{\node\in\Node(T)}\mid t_{\tilde\node}\in[0,t]\text{ and }\l(t_\node\r)_{\node\in\Node(T)\setminus\{\tilde\node\}}\in I_{T}(t,t_{\tilde\node})\r\}.
    \end{equation} 
\end{remark}

\begin{proposition}\label{distinction in stem}
    For any $T\in\l(\T_n^{\eta,\iota}\r)^\flower$, the Fourier transform of $\J_T^\flower$ reads 
    \begin{equation}\label{the fourier decomposition that we were about to prove}
        \widehat{\J_T^\flower}(t,k) =\frac{1}{L^d}\sum_{k_\flower\in\Z_L^d}\int_0^t\widehat{v^{\eta_\flower,\iota_\flower}}(t_\flower,k_\flower){\mathcal G_T}(t,t_\flower,k,k_\flower)\rmd t_\flower,
    \end{equation}
    where \begin{equation}
    \begin{gathered}
        {\mathcal G_T}\l(t,t_\flower,k,k_\flower\r)\coloneqq\frac{\l(-\rmi\epsilon\r)^{n(T)}}{L^{d(2n(T)-1)}}\prod_{\node\in\Node(T)}\iota_\node\sum_{\kappa\in\D_{k,k_\flower}(T)}e^{\rmi\iota_{P(\flower)}\Omega_{P(\flower)}^Tt_\flower}\prod_{\node\in\Node(T)}Q_\node^T(\kappa)\\\cdot\int_{I_T\l(t,t_\flower\r)}\prod_{\node\in\Node(T)\setminus\{P(\flower)\}}e^{\rmi\iota_\node\Omega_\node^T t_\node}\rmd t_\node\prod_{\leaf\in\Leaf(T)\setminus\{\flower\}}\mu_{\kappa(\leaf)}^{\eta_\leaf,\iota_\leaf}.
    \end{gathered}
    \end{equation}
\end{proposition}

\begin{proof}

    As done before, one can prove by induction over $\abs{\Node(T)}$ and find
    \begin{equation}\label{analogous fourier transform}
        \begin{gathered}
        \widehat{\J_T^\flower}(t,k) = \l(\frac{-\rmi\epsilon}{L^{2d}}\r)^{n(T)}\prod_{\node\in\Node(T)}\iota_\node\sum_{\kappa\in\D_k(T)}\prod_{\node\in\Node(T)}Q_\node^T(\kappa)\\\cdot\int_{I_T(t)}\widehat{v^{\eta_\flower,\iota_\flower}}(t_{P(\flower)},\kappa(\flower))\prod_{\node\in\Node(T)}e^{\rmi\iota_\node\Omega_\node^T(\kappa)t_\node}\rmd t_\node\prod_{\leaf\in\Leaf(T)\setminus\{\flower\}}\mu_{\kappa(\leaf)}^{\eta_\leaf,\iota_\leaf}.
    \end{gathered}
    \end{equation}
    Using \cref{trivial remark i should use somehow}, we may rewrite 

    \begin{equation}\label{application of something particular}
        \begin{gathered}
        \int_{I_T(t)}\widehat{v^{\eta_\flower,\iota_\flower}}(t_{P(\flower)},\kappa(\flower))\prod_{\node\in\Node(T)}e^{\rmi\iota\Omega_\node^T t_\node}\rmd t_\node = \int_0^t\widehat{v^{\eta_\flower,\iota_\flower}}(t_{P(\flower)},\kappa(\flower))e^{\rmi\iota\Omega_{P(\flower)}t_{P(\flower)}}\\\cdot\int_{I_T(t,t_{P(\flower)})}\prod_{\node\in\Node(T)\setminus\{P(\flower)\}}e^{\rmi\iota\Omega_\node^T t_\node}\rmd t_\node\rmd t_{P(\flower)}.
    \end{gathered}
    \end{equation}
    and again decompose the sum via all decorations as in \cref{obvious decomposition}. Putting \cref{analogous fourier transform,application of something particular,obvious decomposition} together delivers \cref{the fourier decomposition that we were about to prove}.
\end{proof}

We now define 
\begin{equation}
    \l(\T_{m,n}^{\iota,\eta}\r)^\flower\coloneqq\l\{T\in\l(\T_m^{\iota,\eta}\r)^\flower\bigm\vert\abs{\Leaf(T)}=n+1\r\}.
\end{equation}
One can prove quite quickly that

\begin{align}
    \min\l\{\abs{\Leaf(T)}\mid T\in\l(\T_m^{\iota,\eta}\r)^\flower\r\} &= 4m,\\
    \max\l\{\abs{\Leaf(T)}\mid T\in\l(\T_m^{\iota,\eta}\r)^\flower\r\} &= 4m(4N+1).
\end{align}
Further denote by $\l(\T_{m,n}^{\iota,\eta,\iota_\flower,\eta_\flower}\r)^\flower\subseteq\l(\T_{m,n}^{\iota,\eta}\r)^\flower$ the subset of flower trees whose flower has sign $\iota_\flower$ and $\eta_\flower$. 
\begin{remark}
    Note that if $\iota_\flower$ or $\eta_\flower$ does not equal the sign or colour of the flower dictated by the sign and colour at the root, we automatically have $\l(\T_{m,n}^{\iota,\eta,\iota_\flower,\eta_\flower}\r)^\flower=\emptyset$.
\end{remark}
We may decompose 
\begin{equation}
    \l({\T}_m^{\iota,\eta}\r)^\flower = \bigsqcup_{\substack{\iota_\flower\in\{\pm\}\\\eta_\flower\in\{0,1\}}}\bigsqcup_{n=4m}^{4m(4N+1)}\l(\T_{m,n}^{\iota,\eta,\iota_\flower,\eta_\flower}\r)^\flower.
\end{equation}
In that way, we may use \cref{tree decomposition for flower trees,distinction in stem} and decompose

\begin{equation}
    \begin{gathered}
    \reallywidehat{\Leaf^m(v)^{\eta,\iota}}\l(\epsilon^{-1}t,k\r) 
    =\sum_{\substack{k_\flower\in\Z_L^d\\\iota_\flower\in\{\pm\}\\\eta_\flower\{0,1\}}}\sum_{n=4m}^{4m(4N+1)}\frac{1}{L^d}\int_0^{t}\widehat{v^{\eta_\flower,\iota_\flower}}\l(\epsilon^{-1}t_\flower,k_\flower\r){\mathcal Y_{m,n}^{\eta,\iota,\eta_\flower,\iota_\flower}}\l(\epsilon^{-1}t,\epsilon^{-1}t_\flower,k,k_\flower\r)\rmd t_\flower,
\end{gathered}
\end{equation}
where \begin{align}
    {\mathcal Y_{m,n}^{\eta,\iota,\eta_\flower,\iota_\flower}}&\coloneqq \sum_{T\in\l(\T_{m,n}^{\iota,\eta,\iota_\flower,\eta_\flower}\r)^\flower}{\mathcal G_T}.
\end{align}

\subsection{Flower couples}

\begin{definition}
    A \textbf{flower couple} is a couple formed by two flower trees such that the two flowers are paired. In particular, they have opposite signs. We define for $n_i\in\llbracket4m_i,4m_i(4N+1)\rrbracket$
    \begin{equation}
        \tilde{\mathcal C}_{m_1,n_1,m_2,n_2}^{\eta,\eta',\iota,\iota'}(\iota_\flower,\eta_\flower)\coloneqq\l\{\l(T_0,T_1,\sigma\r)\in\mathcal C_{m_1,m_2}^{\eta,\eta',\iota,\iota'}\bigm\vert T_0\in\l(\T_{m_1,n_1}^{\iota,\eta,\iota_\flower,\eta_\flower}\r)^\flower,\ T_1\in\l(\T_{m_2,n_2}^{\iota,\eta,\iota_\flower,\eta_\flower}\r)^\flower\text{ and }\sigma(\flower_0)=\flower_1\r\},
    \end{equation}
    where $\flower_0$ and $\flower_1$ are respectively the flowers of $T_0$ and $T_1$.
    For each flower couple $C\in\tilde{\mathcal C}^{\eta,\eta',\iota,\iota'}_{m_1,n_1,m_2,n_2}\l(\iota_\flower,\eta_\flower\r)$, we define $\mathcal D_{k,k_\flower}(C)$ to be the subset of all $k$-decorations such that $\kappa\l(\flower_0\r) = k_\flower$ and $\kappa\l(\flower_1\r) = -k_\flower$.
\end{definition}
\begin{remark}
     Note that as soon as the sign or colour of the root flips, the sign or colour of the flower flips too.
\end{remark}
\begin{proposition}\label{decomposition on flower trees}
    We can decompose 
    \begin{equation}
    \begin{gathered}
        \E\l({\mathcal Y_{m_1,n_1}^{\eta,\iota,\eta_\flower,\iota_\flower}\l(\epsilon^{-1}t,\epsilon^{-1}t_\flower,k,k_\flower\r)}{\mathcal Y_{m_2,n_2}^{\eta',\iota',\eta_\flower',\iota_\flower'}\l(\epsilon^{-1}t,\epsilon^{-1}t_\flower,-k,-k_\flower\r)}\r)\\=\sum_{C\in\tilde{\mathcal C}^{\eta,\eta',\iota,\iota'}_{m_1,n_1,m_2,n_2}\l(\iota_\flower,\eta_\flower\r)}\mathcal G_C\l(\epsilon^{-1}t,\epsilon^{-1}t_\flower,k,k_\flower\r),
    \end{gathered}
    \end{equation}
    where 
    \begin{equation}
    \begin{gathered}
        \mathcal G_C\l(\epsilon^{-1}t,\epsilon^{-1}t_\flower,k,k_\flower\r)\coloneqq\frac{(-\rmi)^{n(C)}}{L^{2d(n(C)-1)}}\prod_{\node\in\Node(C)}\iota_\node\sum_{\kappa\in\D_{k,k_\flower}(C)}e^{\rmi\epsilon^{-1}\l(\iota_{P(\flower_0)}\Omega_{P(\flower_0)}+\iota_{P(\flower_1)}\Omega_{P(\flower_1)}\r)t_\flower}\prod_{\node\in\Node(C)} Q_\node\\\cdot\int_{I_C(t,t_\flower)}\prod_{\node\in\Node(C)\setminus\{P(\flower_0),P(\flower_1)\}}e^{\rmi\epsilon^{-1}\iota_\node\Omega_\node t_\node}\rmd t_\node\prod_{\leaf\in\Leaf(C)_+\setminus\{\flower_0,\flower_1\}}M^{\eta_\leaf,\eta_{\sigma(\leaf)}}(\kappa(\leaf))^{\iota_\leaf}
    \end{gathered}
    \end{equation}
    and  
    \begin{equation}
        I_C\l(t,t_\flower\r)\coloneqq I_{T_0}\l(t,t_\flower\r)\times I_{T_1}\l(t,t_\flower\r)
    \end{equation} 
    for $C=\l(T_0,T_1,\sigma\r)\in\tilde{\mathcal C}_{m_1,n_1,m_2,n_2}^{\eta,\eta',\iota,\iota'}\l(\iota_\flower,\eta_\flower\r)$.
\end{proposition}
\begin{proof}
We may decompose
\begin{equation}
\begin{gathered}
    \E\l({\mathcal Y_{m_1,n_1}^{\eta,\iota,\eta_\flower,\iota_\flower}\l(\epsilon^{-1}t,\epsilon^{-1}t_\flower,k,k_\flower\r)}{\mathcal Y_{m_2,n_2}^{\eta',\iota',\eta_\flower',\iota_\flower'}\l(\epsilon^{-1}t,\epsilon^{-1}t_\flower,-k,-k_\flower\r)}\r) \\= \sum_{\substack{T_0\in\l(\T_{m_1,n_1}^{\iota,\eta,\iota_\flower,\eta_\flower}\r)^\flower\\T_1\in\l(\T_{m_2,n_2}^{\iota',\eta',\iota_\flower',\eta_\flower'}\r)^\flower}}\E\l(\mathcal G_{T_0}\l(\epsilon^{-1}t,\epsilon^{-1}t_\flower,k,k_\flower\r)\mathcal G_{T_1}\l(\epsilon^{-1}t,\epsilon^{-1}t_\flower,-k,-k_\flower\r)\r)\\
    =\sum_{\substack{T_0\in\l(\T_{m_1,n_1}^{\iota,\eta,\iota_\flower,\eta_\flower}\r)^\flower\\T_1\in\l(\T_{m_2,n_2}^{\iota',\eta',\iota_\flower',\eta_\flower'}\r)^\flower\\\sigma\text{ pairing of }\Leaf(T_0)\sqcup\Leaf(T_1)}}\frac{(-\rmi)^{m_1+m_2}}{L^{2d\l(m_1+m_2-1\r)}}\prod_{\node\in\Node(T_0)\sqcup\Node(T_1)}\iota_\node\\\cdot\sum_{\substack{\kappa_0\in\D_{k,k_\flower}(T_0)\\\kappa_1\in\D_{-k,-k_\flower}(T_1)\\\kappa(\leaf)+\kappa(\sigma(\leaf))=0\forall\leaf\in\l(\Leaf(T_0)\sqcup\Leaf(T_1)\r)_+}}e^{\rmi\epsilon^{-1}\l(\iota_{P(\flower_0)}\Omega_{P(\flower_1)}+\iota_{P(\flower_1)}\Omega_{P(\flower_1)}\r)t_\flower}\prod_{\node\in\Node(T_0)\sqcup\Node(T_1)} Q_\node\\\cdot\int_{I_{T_0}(t,t_\flower)\times I_{T_1}(t,t_\flower)}\prod_{\node\in\l(\Node(T_0)\sqcup\Node(T_1)\r)\setminus\{P(\flower_0),P(\flower_1)\}}e^{-\rmi\iota_\node\epsilon^{-1}\Omega_\node t_\node}\rmd t_\node\\\cdot\prod_{\leaf\in\l(\Leaf(T_0)\sqcup\Leaf(T_1)\r)_+\setminus\{\flower_0,\flower_1\}}\delta_{\iota_\leaf+\iota_{\sigma(\leaf)}}M^{\eta_\leaf,\eta_{\sigma(\leaf)}}(\kappa(\leaf))^{\iota_\leaf},
\end{gathered}
\end{equation}
where by $\kappa$ we mean the unique element in $\l(\Z_L^d\r)^{\Leaf(T_0)\sqcup\Leaf(T_1)}$ such that $\left.\kappa\right|_{\Leaf(T_i)}\coloneqq\kappa_i$. Using the fact that we may identify $\{(\kappa_0,\kappa_1)\in\D_{k,k_\flower}(T_0)\times\D_{-k,-k_\flower}(T_1)\mid\kappa(\leaf)+\kappa(\sigma(\leaf))=0\forall\leaf\in\l(\Leaf(T_0)\sqcup\Leaf(T_1)\r)_+\}\cong\D_{k,k_\flower}(C)$, we may rewrite

\begin{equation}
    \begin{gathered}
        \E\l({\mathcal Y_{m_1,n_1}^{\eta,\iota,\eta_\flower,\iota_\flower}\l(\epsilon^{-1}t,\epsilon^{-1}t_\flower,k,k_\flower\r)}{\mathcal Y_{m_2,n_2}^{\eta',\iota',\eta'_\flower,\iota'_\flower}\l(\epsilon^{-1}t,\epsilon^{-1}t_\flower,-k,-k_\flower\r)}\r)\\=\frac{(-\rmi)^{n(C)}}{L^{2d(n(C)-1)}}\sum_{C\in\tilde{\mathcal C}_{m_1,n_1,m_2,n_2}^{\eta,\eta',\iota,\iota'}}\prod_{\node\in\Node(C)}\iota_\node\sum_{\kappa\in\D_{k,k_\flower}(C)}e^{\rmi\epsilon^{-1}\l(\iota_{P(\flower_0)}\Omega_{P(\flower_0)}+\iota_{P(\flower_1)}\Omega_{P(\flower_1)}\r)t_\flower}\prod_{\node\in\Node(C)} Q_\node\\\cdot\int_{I_C(t,t_\flower)}\prod_{\node\in\Node(C)\setminus\{P(\flower_0),P(\flower_1)\}}e^{\rmi\epsilon^{-1}\iota_\node\Omega_\node t_\node}\rmd t_\node\prod_{\leaf\in\Leaf(C)_+\setminus\{\flower_0,\flower_1\}}M^{\eta_\leaf,\eta_{\sigma(\leaf)}}(\kappa(\leaf))^{\iota_\leaf}
    \end{gathered}
\end{equation}
which completes the proof.
\end{proof}

\subsection{Estimates on the kernels of $\Leaf^m$}

Recall, $I_C(t,t_\flower)=I_{T_1}(t,t_\flower)\times I_{T_2}(t,t_\flower)$ and decompose

\begin{equation}
    I_{T_i}(t,t_\flower) = \mathcal Z_{T_i}\sqcup\bigsqcup_{\rho\in\M(\Node(T_{i}))}I_{T_i}^\rho(t,t_\flower),
\end{equation}
where 
\begin{equation}
\begin{gathered}
    I_{T_i}^\rho\l(t,t_\flower\r)\coloneqq\Big\{\l(t_\node\r)_{\node\in\Node(T_i)\setminus\{P(\flower_i)\}}\in[0,t]^{n(T_i)-1},\ 0\leq t_{\rho(1)}<\cdots<t_{\rho(\rho^{-1}(P(\flower_i))-1)}\\\leq t_\flower\leq t_{\rho(\rho^{-1}(P(\flower_i))+1)}<\cdots<t_{\rho(n(T_i))}\leq t\Big\}
\end{gathered}
\end{equation}
and we denote 
\begin{align}
    \Node^{>_\rho}(T_i)&\coloneqq\rho\l(\llbracket\rho^{-1}(P(\flower_i))+1,n(T_i)\rrbracket\r),\\
    \Node^{<_\rho}(T_i)&\coloneqq\rho\l(\llbracket1,\rho^{-1}(P(\flower_i))-1\rrbracket\r),\\
    I_{T_i}^{>_\rho}(t,t_\flower)&\coloneqq\l\{\l(t_\node\r)_{\node\in\Node^{>_\rho}(T_i)}\in[t_\flower,t]^{n^{>_\rho}(T_i)}\bigm\vert t_\flower\leq t_{\rho(\rho^{-1}(P(\flower_i))+1)}<\cdots<t_{\rho(n(T_i))}\leq t\r\},\\
    I_{T_i}^{<_\rho}(t_\flower)&\coloneqq\l\{(t_\node)_{\node\in\Node^{<_\rho}(T_i)}\in[0,t_\flower]^{n^{<_\rho}(T_i)}\bigm\vert0\leq t_{\rho(1)}<\cdots<t_{\rho(\rho^{-1}(P(\flower_i))-1)}\leq t_\flower\r\},\\
    n^{>_{\rho}}(T_i)&\coloneqq\abs{\Node^{>_\rho}(T_i)},\\
    n^{<_{\rho}}(T_i)&\coloneqq\abs{\Node^{<_\rho}(T_i)},
\end{align}
and notice that $n^{<_\rho}(T_i)+n^{>_\rho}(T_i)=n(T_i)-1$.
With these definitions, we have
\begin{equation}
    I_{T_i}^\rho(t,t_\flower)=I_{T_i}^{<_\rho}(t_\flower)\times I_{T_i}^{>_\rho}(t,t_\flower)
\end{equation}
and of course, $I_{T_i}^{>_\rho}(t,t_\flower)=(t_\flower,\cdots,t_\flower)+I_{T_i}^{>_\rho}(t-t_\flower,0)$ so that we may rewrite
\begin{equation}
    \begin{gathered}
    \int_{I_{T_i}(t,t_\flower)}\prod_{\node\in\Node(T_i)\setminus\{P(\flower_i)\}}e^{\rmi\epsilon^{-1}\iota_\node\Omega_\node t_\node}\rmd t_\node \\= \sum_{\rho\in\M(\Node(T_i))}\int_{I_{T_i}^{>_\rho}(t,t_\flower)}\prod_{\node\in\Node^{>_\rho}(T_i)}e^{\rmi\epsilon^{-1}\iota_\node\Omega_\node t_\node}\rmd t_\node\int_{I_{T_i}^{<_\rho}(t_\flower)}\prod_{\node\in\Node^{<_\rho}(T_i)}e^{\rmi\epsilon^{-1}\iota_\node\Omega_\node t_\node}\rmd t_\node\\
    =\sum_{\rho\in\M(\Node(T_i))}\prod_{\node\in\Node^{>_\rho}(T_i)}e^{\rmi\epsilon^{-1}\iota_\node\Omega_\node t_\flower}\int_{I_{T_i}^{>_\rho}(t-t_\flower,0)}\prod_{\node\in\Node^{>_\rho}(T_i)}e^{\rmi\epsilon^{-1}\iota_\node\Omega_\node t_\node}\rmd t_\node\int_{I_{T_i}^{<_\rho}(t_\flower)}\prod_{\node\in\Node^{<_\rho}(T_i)}e^{\rmi\epsilon^{-1}\iota_\node\Omega_\node t_\node}\rmd t_\node.
\end{gathered}
\end{equation}
Applying the resolvent identity (\cref{resolvent identity}) twice leads to 

\begin{equation}\label{resolvent identity applied two times}
    \begin{gathered}
    \int_{I_{T_i}(t,t_\flower)}\prod_{\node\in\Node(T_i)\setminus\{P\l(\flower_i\r)\}}e^{\rmi\epsilon^{-1}\iota_\node\Omega_\node t_\node}\rmd t_\node = \sum_{\rho\in\M(\Node(T_i))}e^{\rmi\epsilon^{-1}\omega^\rho_{P(\flower_i)}t_\flower}\frac{e^{(n(T_i)-1)t}}{4\pi^2}\\\cdot\int_{\R\times\R}\frac{e^{-\rmi\xi_1(t-t_\flower)}e^{-\rmi\xi_2 t_\flower}}{\l(n^{>_\rho}(T_i)-\rmi\xi_1\r)\l(n^{<_\rho}(T_i)-\rmi\xi_2\r)}\prod_{\substack{\node_1\in\Node^{>_\rho}(T_i)}}\frac{1}{n^{>_\rho}(T_i)-\rmi\l(\xi_1+\omega_{\node_1}^\rho\r)}\\\cdot\prod_{\node_2\in\Node^{<_\rho}(T_i)}\frac{1}{n^{<_\rho}(T_i)-\rmi\l(\xi_2+\omega_{\node_2}^\rho\r)}\rmd\l(\xi_1,\xi_2\r)
\end{gathered}
\end{equation}
where 

\begin{align}
    \omega_{\flower}^\rho&\coloneqq\sum_{\node\in\Node^{>_\rho}(T_i)}\iota_\node\Omega_\node,\\
    \omega_\node^\rho&\coloneqq\begin{cases}
        \sum\limits_{\Node^{<_\rho}(T_i)\ni\node'\geq_\rho\node}\iota_{\node'}\Omega_{\node'}&\text{if }\node\in\Node^{<_\rho}(T_i),\\
        \sum\limits_{\Node^{>_\rho}(T_i)\ni\node'\geq_\rho\node}\iota_{\node'}\Omega_{\node'}&\text{if }\node\in\Node^{>_\rho}(T_i).
    \end{cases}
\end{align}
By convention, if one of the sets $\Node^{>_\rho}(T_i)$ or $\Node^{<_\rho}(T_i)$ is empty, then the corresponding integral in \cref{resolvent identity applied two times} is omitted. 

We can then decompose 

\begin{equation}\label{order decomposition}
    {\mathcal G_C}\l(\epsilon^{-1}t,\epsilon^{-1}t_\flower,k,k_\flower\r)=\sum_{\substack{\rho_1\in\M(\Node(T_1))\\\rho_2\in\M(\Node(T_2))}}{\mathcal G_C^{\rho_1,\rho_2}}\l(\epsilon^{-1}t,\epsilon^{-1}t_\flower,k,k_\flower\r)
\end{equation}
with 

\begin{equation}\label{all integrals to be calculated}
    \begin{gathered}
    \abs{{\mathcal G_C^{\rho_1,\rho_2}}\l(\epsilon^{-1}t,\epsilon^{-1}t_\flower,k,k_\flower\r)}\\\leq\Lambda^{n+1}L^{-2d(n-1)}\int_{\R^4}\frac{A_C^{\rho_1,\rho_2}(\xi_1,\xi_2,\xi_3,\xi_4,k,k_\flower)\rmd\l(\xi_1,\xi_2,\xi_3,\xi_4\r)}{\abs{n^{>_{\rho_1}}(T_1)-\rmi\xi_1}\abs{n^{<_{\rho_1}}(T_1)-\rmi\xi_2}\abs{n^{>_{\rho_2}}(T_2)-\rmi\xi_3}\abs{n^{<_{\rho_2}}(T_2)-\rmi\xi_4}}
    \end{gathered}
\end{equation}
and 
\begin{equation}
    \begin{gathered}
        A_C^{\rho_1,\rho_2}(\xi_1,\xi_2,\xi_3,\xi_4,k,k_\flower)\coloneqq \sum_{\substack{\kappa\in\D_{k,k_\flower}(C)\\\kappa\l(\Leaf(C)_+\r)\subseteq B_R(0)}}\prod_{\node\in\Node^{>_{\rho_1}}(T_1)}\frac{\abs{Q_\node}}{\abs{n^{>_{\rho_1}}(T_1)-\rmi\l(\xi_1+\omega_\node^{\rho_1}\r)}}\\\cdot\prod_{\node\in\Node^{<_{\rho_1}}(T_1)}\frac{\abs{Q_\node}}{\abs{n^{<_{\rho_1}}(T_1)-\rmi\l(\xi_2+\omega_\node^{\rho_1}\r)}}
    \\\cdot\prod_{\node\in\Node^{>_{\rho_2}}(T_2)}\frac{\abs{Q_\node}}{\abs{n^{>_{\rho_2}}(T_2)-\rmi\l(\xi_3+\omega_\node^{\rho_2}\r)}}\prod_{\node\in\Node^{<_{\rho_2}}(T_2)}\frac{\abs{Q_\node}}{\abs{n^{<_{\rho_2}}(T_2)-\rmi\l(\xi_4+\omega_\node^{\rho_2}\r)}}.
    \end{gathered}
\end{equation}
Again, if one of the sets $\Node^{>_{\rho_i}}(T_i)$ or $\Node^{<_{\rho_i}}(T_i)$ is empty, we omit the corresponding integral over $\R$ in \cref{all integrals to be calculated} and ${A_C^{\rho_1,\rho_2}}$ depends on $3$ or less $\xi_j$.

Given $\rho_i\in\M(\Node(T_i))$, we may define $\rho\in\M(\Node(C))$ by setting the restriction of $\rho$ on $\Node(T_i)$ to be $\rho_i$ and defining the nodes of the second tree to be always below the ones of the first tree.
For any $\rho\in\M(\Node(C))$ we can construct a total order relation $<''$ on $\l(\Node(C)\sqcup\Leaf(C)\r)\setminus\{\flower_0,\flower_1\}$ as required in \cref{change of variables} but this time $\Node_{<''}(C)$ will have $\abs{\Node_{<''}(C)}=2n(C)$ since $\flower_0$, $\flower_1$ and the corresponding edges adjacent to these flowers are excluded from the momentum graph and subsequent spanning tree construction. In particular, $\sum_{i=1}^4ip_i=2n(C)-1$ but still $\sum_{i=0}^4p_i=n(C)$. 

\begin{proposition}\label{estimates on the kernels}
Let $C=(T_0,T_1,\sigma)\in\tilde C_{m_1,n_1,m_2,n_2}^{\eta,\eta',\iota,\iota'}$. Then there exist $\Lambda,\mathcal K,\alpha>0$ such that for all $n\leq\abs{\ln\epsilon}$, 
\begin{equation}\label{what has to be calculated now}
    \sup_{0\leq t_\flower\leq t\leq\delta}\abs{{\mathcal G_C}\l(\epsilon^{-1}t,\epsilon^{-1}t_\flower,k,k_\flower\r)}\leq\Lambda\l(\Lambda\delta\r)^{n(C)}\epsilon^{\alpha\l(n(C)-n_{\mathrm{res}}(C)\r)}\frac{L^{\mathcal K}}{n(C)}
\end{equation}
for all $k,k_\flower\in\Z_L^d$ (note, $n(C)=m_1+m_2$).
\end{proposition}
\begin{proof}
    The proof is similar to the one given in \cref{key proposition}. The main difference is the cardinality $\abs{\Node_{<''}(C)}=2n(C)$ (which is merely of a numerical nature) and the fact that one has to apply the proof strategy of \cref{key proposition} to all four integrals in \cref{all integrals to be calculated}. That is, one may proceed by a domain separation $\l\{\abs{\xi_i}\leq L^{\mathcal K}\r\}$ and $\l\{\abs{\xi_i}> L^{\mathcal K}\r\}$ for each variable $\xi_i$. In that way, one obtains \cref{the first to be estimated} for $\mathcal G_C^{\rho_1,\rho_2}$. Using the decomposition \eqref{order decomposition},
    \begin{equation}
        \abs{\M(\Node(T_0))}\abs{\M(\Node(T_1))}\leq n(C)!,
    \end{equation}
    and the assumption $n(C)\leq\abs{\ln\epsilon}$ gives the desired estimate. 
\end{proof}

\subsection{The time derivative of $\mathcal Y_{m,n}^{\eta,\iota}$}

Let $f_\node$ be a function of $t$.

\begin{lemma}
    Taking the time derivative leads to 
    \begin{equation}
        \p_t\l(\int_{I_T(t,t_\flower)}\prod_{\node\in\Node(T)\setminus\{P(\flower)\}}f_\node(t_\node)\rmd t_\node\r)=           f_\root(t)\int_{I_T'(t,t_\flower)}\prod\limits_{\node\in\Node(T)\setminus\{\root,P(\flower)\}}f_\node(t_\node)\rmd t_\node,
    \end{equation}
    where
    \begin{equation}
\begin{gathered}
    I_T'(t,t_\flower)\coloneqq\Big\{\l(t_\node\r)_{\node\in\Node(T)\setminus\{\root,P(\flower)\}}\in[0,t]^{n(T)-2}\bigm\vert\node<\node'\text{ implies }t_\node<t_{\node'},\ \node<P(\flower)\\\text{ implies }t_\node<t_\flower\text{ and }P(\flower)<\node\text{ implies }t_\flower<t_\node\Big\}.
\end{gathered}
\end{equation}
\end{lemma}

\begin{proof}
First rewrite
\begin{equation}
\begin{gathered}
    \int_{I_T(t,t_\flower)}\prod_{\node\in\Node(T)\setminus\{P(\flower)\}}f_\node(t_\node)\rmd t_\node=\sum_{\substack{\rho\in\M(\Node(T))\\\rho^{-1}\l(P(\flower)\r)\leq n(T)-1}}\int_{I_T^{<_\rho}(t_\flower)}\prod_{\node\in\Node^{<_\rho}(T)}f_\node(t_\node)\rmd t_\node\\\cdot\int_{t_\flower\leq t_{\rho(\rho^{-1}(P(\flower))+1)}<\cdots<t_{\rho(n(T))}\leq t}\prod_{i=\rho(\rho^{-1}(P(\flower))+1)}^{n(T)}f_{\rho(i)}(t_{\rho(i)})\rmd t_{\rho(i)}\\
    +\sum_{\substack{\rho\in\M(\Node(T))\\\rho^{-1}(P(\flower))=n(T)}}\int_{0\leq t_{\rho(1)}<\cdots<t_{\rho(n(T)-1)}\leq t_\flower}\prod_{i=1}^{n(T)-1}f_{\rho(i)}(t_{\rho(i)})\rmd t_{\rho(i)}.
\end{gathered}
\end{equation}
since $I_T^\rho(t,t_\flower)=I_T^{<_\rho}(t_\flower)\times I_T^{>_\rho}(t,t_\flower)$.

\begin{equation}
    \begin{gathered}
    \p_t\l(\int_{I_T(t,t_\flower)}\prod_{\node\in\Node(T)\setminus\{P(\flower)\}}f_\node(t_\node)\rmd t_\node\r)=\sum_{\substack{\rho\in\M(\Node(T))\\\rho^{-1}\l(P(\flower)\r)\leq n(T)-1}}\int_{I_T^{<_\rho}(t_\flower)}\prod_{\node\in\Node^{<_\rho}(T)}f_\node(t_\node)\rmd t_\node \\\cdot f_{\root}(t)\int_{t_\flower\leq t_{\rho(\rho^{-1}(P(\flower))+1)}<\cdots<t_{\rho(n(T)-1)}\leq t}\prod_{\node\in\Node^{>_\rho}(T)\setminus\{\root\}}f_{\node}(t_\node)\rmd t_\node\\=f_\root(t)\sum_{\substack{\rho\in\M(\Node(T))\\\rho^{-1}(P(\flower))\leq n(T)-1}}\int_{\l(I_T^\rho\r)'(t,t_\flower)}\prod_{\node\in\Node(T)\setminus\{\root,P(\flower)\}}f_\node(t_\node)\rmd t_\node \\=f_\root(t)\sum_{\substack{\rho\in\M(\Node(T))}}\int_{\l(I_T^\rho\r)'(t,t_\flower)}\prod_{\node\in\Node(T)\setminus\{\root,P(\flower)\}}f_\node(t_\node)\rmd t_\node\\= f_\root(t)\int_{I_T'(t,t_\flower)}\prod_{\node\in\Node(T)\setminus\{\root,P(\flower)\}}f_\node(t_\node)\rmd t_\node,
\end{gathered}
\end{equation}
where
\begin{equation}
\begin{gathered}
    \l(I_T^\rho\r)'(t,t_\flower)\coloneqq \Big\{\l(t_\node\r)_{\node\in\Node(T)\setminus\{\root,P(\flower)\}}\in[0,t]^{n(T)-1},\ 0\leq t_{\rho(1)}<\cdots<t_{\rho(\rho^{-1}(P(\flower))-1)}\\\leq t_\flower\leq t_{\rho(\rho^{-1}(P(\flower))+1)}<\cdots<t_{\rho(n(T)-1)}\leq t\Big\}
\end{gathered}
\end{equation}
and we used $\Node^{<_\rho}(T)\sqcup\Node^{>_\rho}(T)=\Node(T)\setminus\{P(\flower)\}$. Note that if $\abs{\stem(T)}=1$, then $P(\flower)=\root$ and if $\abs{\stem(T)}\geq2$, then $P(\flower)<\root$ and in particular $\rho^{-1}(P(\flower))\leq n(T)-1$ for all $\rho\in\M(\Node(T))$.
\end{proof}

Using these formulae for $f_\node(t_\node)\coloneqq e^{\rmi\epsilon^{-1}\iota_\node\Omega_\node t_\node}$, we obtain 

\begin{gather}
    {\p_t\mathcal Y_{m,n}^{\eta,\iota}}  =\sum_{T\in\l(\T_{m,n}^{\iota,\eta}\r)^\flower}{\p_t\mathcal G_T}
\end{gather}
with 
\begin{equation}\label{the derivative on flower trees}
    \begin{gathered}
    {\p_t\mathcal G_T}\l(\epsilon^{-1}t,\epsilon^{-1}t_\flower,k,k_\flower\r) = \frac{(-\rmi)^{n(T)}}{L^{d(2n(T)-1)}}\prod_{\node\in\Node(T)}\iota_\node\sum_{\kappa\in\D_{k,k_\flower}(T)}e^{\rmi\epsilon^{-1}\iota_{P(\flower)}\Omega_{P(\flower)}t_\flower}e^{\rmi\epsilon^{-1}\iota\Omega_\root t}\\\cdot\prod_{\node\in\Node(T)}Q_\node^T(\kappa)\int_{I_T'(t,t_\flower)}\prod_{\node\in\Node(T)\setminus\{\root,P(\flower)\}}e^{\rmi\iota\epsilon^{-1}\Omega_\node t_\node}\rmd t_\node\prod_{\leaf\in\Leaf(T)\setminus\{\flower\}}\mu_{\kappa(\leaf)}^{\eta_\leaf,\iota_\leaf}
\end{gathered}
\end{equation}
One may repeat the analysis as before and obtain the following result. 
\begin{proposition}
We may decompose
\begin{equation}
\begin{gathered}
    \E\l({\p_t\mathcal Y_{m_1,n_1}^{\eta,\iota,\eta_\flower,\iota_\flower}\l(\epsilon^{-1}t,\epsilon^{-1}t_\flower,k,k_\flower\r)}{\p_t\mathcal Y_{m_2,n_2}^{\eta',\iota',\eta'_\flower,\iota'_\flower}\l(\epsilon^{-1}t,\epsilon^{-1}t_\flower,-k,-k_\flower\r)}\r)\\=\sum_{C\in\tilde{\mathcal C}^{\eta,\eta',\iota,\iota'}_{m_1,n_1,m_2,n_2}\l(\iota_\flower,\eta_\flower\r)}\mathcal G_C'\l(\epsilon^{-1}t,\epsilon^{-1}t_\flower,k,k_\flower\r),
\end{gathered}
\end{equation}
where 
\begin{equation}
\begin{gathered}
    \mathcal G_C'\l(\epsilon^{-1}t,\epsilon^{-1}t_\flower,k,k_\flower\r)\coloneqq\frac{(-\rmi)^{n(C)}}{L^{2d(n(C)-1)}}\prod_{\node\in\Node(C)}\iota_\node\sum_{\kappa\in\D_{k,k_\flower}(C)}e^{\rmi\epsilon^{-1}\l(\iota_{P(\flower_0)}\Omega_{P(\flower_0)}+\iota_{P(\flower_1)}\Omega_{P(\flower_1)}\r)t_\flower}\\\cdot e^{\rmi\epsilon^{-1}\l(\iota\Omega_\root+\iota'\Omega_{\root'}\r)t}\prod_{\node\in\Node(C)} Q_\node\int_{I_C'(t,t_\flower)}\prod_{\node\in\Node(C)\setminus\{\root,\root',P(\flower_0),P(\flower_1)\}}e^{\rmi\epsilon^{-1}\iota_\node\Omega_\node t_\node}\rmd t_\node\\\cdot\prod_{\leaf\in\Leaf(C)_+\setminus\{\flower_0,\flower_1\}}M^{\eta_\leaf,\eta_{\sigma(\leaf)}}(\kappa(\leaf))^{\iota_\leaf},
\end{gathered}
\end{equation}
where 
\begin{equation}
    I_C'(t,t_\flower)\coloneqq I_{T_0}'(t,t_\flower)\times I_{T_1}'(t,t_\flower)
\end{equation}
for $C=\l(T_0,T_1,\sigma\r)\in\tilde{\mathcal C}_{m_1,n_1,m_2,n_2}^{\eta,\eta',\iota,\iota'}$
\end{proposition}
\begin{proof}
    Using \cref{the derivative on flower trees}, one may repeat the proof of \cref{the derivative on flower trees} and obtain the result. 
\end{proof}
Finally, we find 
\begin{proposition}\label{estimate on the time derivative of the kernel}
Let $C\in\tilde C_{m_1,n_1,m_2,n_2}^{\eta,\eta',\iota,\iota'}\l(\iota_\flower,\eta_\flower\r)$. There exist $\Lambda,\mathcal K,\alpha>0$ such that for all $n(C)\leq\abs{\ln\epsilon}$, 
\begin{equation}\label{what has to be calculated now}
    \sup_{0\leq t_\flower\leq t\leq\delta}\abs{{\p_t\mathcal G_C}\l(\epsilon^{-1}t,\epsilon^{-1}t_\flower,k,k_\flower\r)}\leq\Lambda\l(\Lambda\delta\r)^{n(C)}\epsilon^{\alpha\l(n(C)-n_{\mathrm{res}}(C)\r)}\frac{L^{\mathcal K}}{n(C)}
\end{equation}
for all $k,k_\flower\in\Z_L^d$.
\end{proposition}
\begin{proof}
    The proof is completely analogous to that of \cref{estimates on the kernels}.
\end{proof}

\subsection{A norm estimate on $\mathcal Y_{m,n}^{\eta,\iota}$}
We may now combine the results \cref{estimates on the kernels,estimate on the time derivative of the kernel} into

\begin{corollary}\label{norm estimate on y}
There exists $\Lambda,\mathcal K>0$ such that for any fixed $A>0$, with porbability $\geq1-L^{-A}$ and for all $n\leq\abs{\ln\epsilon}$, we have 

    \begin{equation}
    \begin{gathered}
        \norm{{\mathcal Y_{m,n}^{\eta,\iota,\eta_\flower,\iota_\flower}}}_{L_t^\infty L_{t_\flower}^1L_{k,k_\flower}^\infty}\leq \Lambda(\Lambda\delta)^{\frac{n(C)}{2}}{L^{{\mathcal K}+\frac{4}{\beta}+{d}+2A}}
    \end{gathered}
    \end{equation}
\end{corollary}
\begin{proof}
Combining \cref{estimates on the kernels,estimate on the time derivative of the kernel}, we first find 
that there exist $\Lambda,\mathcal K$ such that for all $n(C)\leq\abs{\ln\epsilon}$, we have  
    \begin{equation}
    \begin{gathered}
        \sup_{0\leq t_\flower\leq t\leq\lambda}\l(\E\l(\abs{{\mathcal Y_{m,n}^{\eta,\iota,\eta_\flower,\iota_\flower}}\l(\epsilon^{-1}t,\epsilon^{-1}t_\flower k,k_\flower\r)}^2\r)+\E\l(\abs{{\p_t\mathcal Y_{m,n}^{\eta,\iota,\eta_\flower,\iota_\flower}}\l(\epsilon^{-1}t,\epsilon^{-1}t_\flower k,k_\flower\r)}^2\r)\r)\\\leq\Lambda(\Lambda\delta)^{n(C)}\frac{L^{\mathcal K}}{n(C)}
    \end{gathered}
    \end{equation}
    for all $k_\flower,k\in\Z_L^d$, and with Gaussian hypercontractivity (\cref{hypercontractivity}) and any $p\geq2$ there exists $c(p)>0$ such that 
\begin{equation}
    \begin{gathered}
    \E\l(\abs{{\mathcal Y_{m,n}^{\eta,\iota,\eta_\flower,\iota_\flower}}\l(\epsilon^{-1}t,\epsilon^{-1}t_\flower k,k_\flower\r)}^p\r)+\E\l(\abs{{\p_t\mathcal Y_{m,n}^{\eta,\iota,\eta_\flower,\iota_\flower}}\l(\epsilon^{-1}t,\epsilon^{-1}t_\flower k,k_\flower\r)}^p\r)\\\leq c(p)^{pn}\l[\l(\E\l(\abs{{\mathcal Y_{m,n}^{\eta,\iota,\eta_\flower,\iota_\flower}}\l(\epsilon^{-1}t,\epsilon^{-1}t_\flower k,k_\flower\r)}^2\r)\r)^{p/2}+\l(\E\l(\abs{{\p_t\mathcal Y_{m,n}^{\eta,\iota,\eta_\flower,\iota_\flower}}\l(\epsilon^{-1}t,\epsilon^{-1}t_\flower k,k_\flower\r)}^2\r)\r)^{p/2}\r]\\\leq c(p)^{pn}\Lambda^{\frac{p}{2}}(\Lambda\delta)^{\frac{pn(C)}{2}}\frac{L^{\frac{p\mathcal K}{2}}}{n(C)^{\frac{p}{2}}}
\end{gathered}
\end{equation}
By integrating over the time interval $[0,\delta\epsilon^{-1}]$ we find

\begin{equation}
\begin{gathered}
    \E\l(\norm{{\mathcal Y_{m,n}^{\eta,\iota,\eta_\flower,\iota_\flower}}(\cdot,\cdot,k,k_\flower)}^p_{L_{t,t_\flower}^p}\r)+\E\l(\norm{{\p_t\mathcal Y_{m,n}^{\eta,\iota,\eta_\flower,\iota_\flower}}(\cdot,\cdot,k,k_\flower)}^p_{L_{t,t_\flower}^p}\r)\\\leq c(p)^{pn}\Lambda^{\frac{p}{2}}(\Lambda\delta)^{\frac{pn(C)}{2}}\delta^2\frac{L^{\frac{p\mathcal K}{2}+\frac{2}{\beta}}}{n(C)^{\frac{p}{2}}}
\end{gathered}
\end{equation}
We estimate

\begin{equation}
    \begin{gathered}
    \E\l(\sup_{t\in[0,\delta\epsilon^{-1}]}\sup_{k,k_\flower\in\Z_L^d}\norm{{\mathcal Y_{m,n}^{\eta,\iota,\eta_\flower,\iota_\flower}}(t,\cdot,k,k_\flower)}_{L_{t_\flower}^1}^p\r)\leq\Lambda\epsilon^{1-p}\E\l(\sup_{t\in[0,\delta\epsilon^{-1}]}\sup_{k,k_\flower\in\Z_L^d}\norm{{\mathcal Y_{m,n}^{\eta,\iota,\eta_\flower,\iota_\flower}}(t,\cdot,k,k_\flower)}_{L_{t_\flower}^p}^p\r)\\
    \leq\epsilon^{1-p}\sum_{k,k_\flower\in\Z_L^d}\E\l(\sup_{t\in\l[0,\delta\epsilon^{-1}\r]}\norm{{\mathcal Y_{m,n}^{\eta,\iota,\eta_\flower,\iota_\flower}}\l(t,\cdot,k,k_\flower\r)}_{L_{t_\flower}^p}^p\r)
\end{gathered}
\end{equation}
We now use the Gagliardo-Nirenberg inequality and Hölder's inequality to obtain 
\begin{equation}
    \begin{gathered}
        \E\l(\sup_{t\in[0,\delta\epsilon^{-1}]}\sup_{k,k_\flower\in\Z_L^d}\norm{{\mathcal Y_{m,n}^{\eta,\iota,\eta_\flower,\iota_\flower}}(t,\cdot,k,k_\flower)}_{L_{t_\flower}^1}^p\r)\\\leq\sum_{k,k_\flower\in\Z_L^d}\Bigg(\epsilon^{1-p}\E\l(\norm{{\p_t\mathcal Y_{m,n}^{\eta,\iota,\eta_\flower,\iota_\flower}}\l(\cdot,\cdot,k,k_\flower\r)}_{L_{t,t_\flower}^p}^p\r)^{1/p}\E\l(\norm{{\mathcal Y_{m,n}^{\eta,\iota,\eta_\flower,\iota_\flower}}\l(\cdot,\cdot,k,k_\flower\r)}_{L_{t,t_\flower}^p}^p\r)^{1-\frac{1}{p}}\\+\epsilon^{2-p}\E\l(\norm{{\mathcal Y_{m,n}^{\eta,\iota,\eta_\flower,\iota_\flower}}\l(\cdot,\cdot,k,k_\flower\r)}_{L_{t,t_\flower}^p}^p\r)\Bigg)\\\leq c(p)^{pn}\Lambda^{\frac{p}{2}}(\Lambda\delta)^{\frac{pn(C)}{2}}\delta^2\frac{L^{\frac{p\mathcal K}{2}+\frac{2}{\beta}}}{n(C)^{\frac{p}{2}}}(nL)^d\l(\epsilon^{-1-p}+\epsilon^{-p}\r)\\\leq c(p)^{pn}\Lambda^{\frac{p}{2}}(\Lambda\delta)^{\frac{pn(C)}{2}}\delta^2{L^{\frac{p\mathcal K}{2}+\frac{2}{\beta}+d+\frac{1+p}{\beta}+2pA}}L^{-pA}
    \end{gathered}
\end{equation}
where we used $\epsilon=L^{-\frac{1}{\beta}}$ and $n\leq\Lambda\abs{\ln\epsilon}$.
We now set 
\begin{equation}
    a_n\coloneqq c(p)^{n}\Lambda^{\frac{1}{2}}(\Lambda\delta)^{\frac{n(C)}{2}}\delta^{\frac{2}{p}}{L^{\frac{\mathcal K}{2}+\frac{3}{p\beta}+\frac{d}{p}+\frac{1}{\beta}+2A}}
\end{equation}
and obtain with Markov's inequality
\begin{equation}
\begin{gathered}
    \mathbb P\l(\sup_{t\in\l[0,\delta\epsilon^{-1}\r]}\sup_{k,k_\flower\in\Z_L^d}\norm{\mathcal Y_{m,n}^{\eta,\iota,\eta_\flower,\iota_\flower}\l(t,t_\flower,k,k_\flower\r)}_{L^1_{t_\flower}}> a_n\r)\\\leq a_n^{-1}\E\l(\sup_{t\in\l[0,\delta\epsilon^{-1}\r]}\sup_{k,k_\flower\in\Z_L^d}\norm{\mathcal Y_{m,n}^{\eta,\iota,\eta_\flower,\iota_\flower}\l(t,\cdot,k,k_\flower\r)}_{L_{t_\flower}^1}\r)\leq L^{-A}
\end{gathered}
\end{equation}
which completes the proof.
\end{proof}

\subsection{A norm estimate on $\Leaf^m$}

We are finally in the position to formulate a norm estimate on $\Leaf^m$. For that, we set $\mathcal X\coloneqq\mathcal C\l([0,\delta\epsilon^{-1}],H^s\l(\mathbb T_L^d\r)\r)$ and $I_m^N\coloneqq\llbracket4m,4m(4N+1)\rrbracket$.

\begin{remark}
        If $B$ is a Banach space with norm $\norm{\cdot}$, we define the norm $\norm{\cdot}_n$ on $B^n$ to be $\norm{(b_1,\ldots,b_n)}_n\coloneqq\max_{1\leq i\leq n}\norm{b_i}$.
    \end{remark}

\begin{theorem}\label{norm estimates on operator}
    We may estimate 
    \begin{equation}
        \norm{\Leaf^m}_{\mathcal X^2\to\mathcal X^2}\leq \Lambda\l(\Lambda\sqrt{\delta}\r)^mN(L)^{2(s+d)}L^{\frac{\mathcal K}{2}+\frac{1}{\beta}+\frac{d}{2}+A}
    \end{equation}
\end{theorem}

\begin{proof}
    Recall that $Y_{m,n}^{\tilde\eta,\iota'}\l(\eta,\iota,\eta_\flower,\iota_\flower\r)=Y_{m,n}^{\tilde\eta}\l(\eta,\iota'\iota,\eta_\flower,\iota_\flower\r)$.
    We have 
    \begin{equation}
    \begin{gathered}
        \norm{\Leaf^m(v)(t,\cdot)}_{H^s\l(T_L^d\r)}^2 =\sum_{k\in\Z_L^d}\langle k\rangle^{2s}\abs{\widehat{\Leaf^m(v)^\eta}(t,k)}^2
        \\=\frac{\epsilon^2}{L^{2d}}\sum_{\substack{k,k_\flower,k_\flower'\in\Z_L^d\\n,n'\in I_m^N\\\iota_\flower\in\{\pm\}\\\eta_\flower\in\{0,1\}}}\langle k\rangle^{2s}\int_{[0,t]^2}\Big(\widehat{v^{\eta_\flower,\iota_\flower}}\l(t_\flower,k_\flower\r){\mathcal Y_{m,n}^{\eta,+,\eta_\flower,\iota_\flower}}\l(t,t_\flower,k,k_\flower\r)\\\cdot\widehat{v^{\eta_\flower',-\iota_\flower'}}\l(t_\flower',-k_\flower'\r){\mathcal Y_{m,n'}^{\eta,-,\eta_\flower,-\iota_\flower}}\l(t,t_\flower',-k,-k_\flower'\r)\Big)\rmd t_\flower\rmd t_\flower'. 
    \end{gathered}       
    \end{equation}
    It is easy to prove that $\langle k\rangle^{2s}\leq \Lambda\l\langle k-k_\flower\r\rangle^s\l\langle k_\flower\r\rangle^s\l\langle k-k_\flower'\r\rangle^s\l\langle k_\flower'\r\rangle^s$ for all $k,k_\flower,k_\flower'\in\Z_L^d$. It is now important to notice that $k-k_\flower$ and $k-k_\flower'$ are supported in a ball of radius $\mathcal O(n)$, respectively $\mathcal O(n')$ around the origin. Thence, if we perform the coordinate transformation $k_\flower\mapsto k_\flower-k$ and $k_\flower'\mapsto k_\flower'-k$, we obtain 
    \begin{equation}
        \begin{gathered}
            \norm{\Leaf^m(v)^\eta(t,\cdot)}_{H^s\l(T_L^d\r)}^2\leq \Lambda\frac{\epsilon^2}{L^{2d}}\sum_{\substack{n,n'\in I_m^N\\\iota_\flower\in\{\pm\}\\\eta_\flower\in\{0,1\}}}\sum_{k\in\Z_L^d}\sum_{\substack{\abs{k_\flower}\lesssim n\\\abs{k_\flower'}\lesssim n'}}\l\langle k-k_\flower\r\rangle^s\l\langle k_\flower\r\rangle^s\l\langle k-k_\flower'\r\rangle^s\l\langle k_\flower'\r\rangle^s\\\cdot\int_{[0,t]^2}\l(\abs{\widehat{v^{\eta_\flower,\iota_\flower}}\l(t_\flower,k_\flower-k\r)}\abs{{\mathcal Y_{m,n}^{\eta,+,\eta_\flower,\iota_\flower}}\l(t,t_\flower,k,k_\flower-k\r)}\abs{\widehat{v^{\eta_\flower',-\iota_\flower'}}\l(t_\flower',k-k_\flower'
            \r)}\abs{{\mathcal Y_{m,n'}^{\eta,-,\eta_\flower,-\iota_\flower}}\l(t,t_\flower',-k,k-k_\flower'\r)}\r)\rmd t_\flower\rmd t_\flower'.\\ \leq \Lambda\frac{\epsilon^2}{L^{2d}}\sum_{\substack{n,n'\in I_m^N\\\iota_\flower\in\{\pm\}\\\eta_\flower\in\{0,1\}}}n^s\l(n'\r)^s\norm{{Y_{m,n}^{\eta,+,\eta_\flower,\iota_\flower}}}_{{L^\infty_tL^1_{t_\flower}L^\infty_{k,k_\flower}}}\norm{{Y_{m,n'}^{\eta,-,\eta_\flower',-\iota_\flower'}}}_{L^\infty_tL^1_{t_\flower}L^\infty_{k,k_\flower}}\\\cdot\sum_{k\in\Z_L^d}\sum_{\substack{\abs{k_\flower}\lesssim n\\\abs{k_\flower'}\lesssim n'}}\l\langle k-k_\flower\r\rangle^s\l\langle k-k_\flower'\r\rangle^s            
            \norm{\widehat{v^{\eta_\flower,\iota_\flower}}(\cdot,k_\flower-k)}_{L^\infty_{t_\flower}}\norm{\widehat{v^{\eta_\flower',-\iota_\flower'}}(\cdot,k-k_\flower')}_{L^\infty_{t_\flower}}.
        \end{gathered}
    \end{equation}
Using the Cauchy-Schwarz inequality, 
\begin{equation}
    \begin{gathered}
        \sum_{k\in\Z_L^d}\sum_{\substack{\abs{k_\flower}\lesssim n\\\abs{k_\flower'}\lesssim n'}}\langle k-k_\flower\rangle^s\langle k-k_\flower'\rangle^s\norm{\widehat{v^{\eta_\flower,\iota_\flower}}\l(\cdot,k_\flower-k\r)}_{L^\infty_{t_\flower}}\norm{\widehat{v^{\eta_\flower',-\iota_\flower'}}\l(t,k-k_\flower'\r)}_{L^\infty_{t_\flower}}\\\leq\Lambda\sum_{\substack{\abs{k_\flower}\lesssim n\\\abs{k_\flower'}\lesssim n'}}\norm{v^{\eta_\flower}(t,\cdot)}_{H^s\l(T_L^d\r)}\norm{v^{\eta_\flower'}(t,\cdot)}_{H^s\l(T_L^d\r)}\\\leq\Lambda n^d \l(n'\r)^dL^{2d}\norm{v(t,\cdot)}_{H^s\l(T_L^d\r)^2}^2
    \end{gathered}
\end{equation}
and thus 
\begin{equation}
    \begin{gathered}
        \norm{\Leaf^m(v)^\eta(t,\cdot)}_{H^s\l(T_L^d\r)}^2\leq\Lambda\norm{v}^2_{\mathcal X^2}\epsilon^2\\\cdot\sum_{\substack{n,n'\in I_m^N\\\iota_\flower\in\{\pm\}\\\eta_\flower\in\{0,1\}}}n^{s+d}\l(n'\r)^{s+d}\norm{\mathcal Y_{m,n}^{\eta,+,\eta_\flower,\iota_\flower}}_{L_t^\infty L_{t_\flower}^1 L_{k,k_\flower}^\infty}\norm{\mathcal Y_{m,n}^{\eta,-,\eta_\flower',-\iota_\flower'}}_{L_t^\infty L_{t_\flower}^1 L_{k,k_\flower}^\infty}
    \end{gathered}
\end{equation}
\Cref{norm estimate on y} then delivers 
\begin{equation}
    \begin{gathered}
        \norm{\Leaf^m(v)^\eta(t,\cdot)}_{H^s\l(T_L^d\r)}^2\leq\Lambda\norm{v}^2_{\mathcal X^2}\epsilon^2\l(\sum_{n\in I_m^N}n^{s+d}\l(\Lambda\sqrt{\delta}\r)^{\frac{n}{4}-m}\r)^2\l(\Lambda\delta\r)^{m}{L^{{\mathcal K}+\frac{4}{\beta}+{d}+2A}}.
    \end{gathered}
\end{equation}
Now using that $n^{s+d}\leq\Lambda N^{2(s+d)}$ and choosing $\delta$ small enough so that $\Lambda\sqrt{\delta}<1$ 
\begin{equation}
    \sum_{n\in I_m^N}n^{s+d}\l(\Lambda\sqrt{\delta}\r)^{\frac{n}{4}-m}\leq\Lambda N^{2(s+d)}
\end{equation}
so that 
\begin{equation}
    \begin{gathered}
        \norm{\Leaf^m(v)}_{\mathcal X^2}\leq\Lambda\l(\Lambda\sqrt{\delta}\r)^mN(L)^{2(s+d)}L^{\frac{\mathcal K}{2}+\frac{1}{\beta}+\frac{d}{2}+A}\norm{v}_{\mathcal X^2}
    \end{gathered}
\end{equation}

\end{proof}

\subsection{The contraction argument}\label{the fixed point argument}

The following proposition establishes the existence and uniqueness of a solution $f^\eta$ to \eqref{system}.

\begin{proposition}
    There exists $\Lambda_0,\tilde\alpha>0$ such that the operator 
    \begin{equation}
        \mathcal X^2\ni v\mapsto\l(\Id-\Leaf\r)^{-1}\l(\mathcal W+\mathcal R^1(v)+\mathcal R^2(v)+\mathcal R^3(v)+\mathcal R^4(v)\r)\in\mathcal X^2,
    \end{equation}
    restricts to a contraction map on the closed ball
    \begin{equation}
        \mathcal Z\coloneqq\l\{\l(v_1,v_2\r)\in \mathcal X^2\mid\norm{\l(v_1,v_2\r)}_{\mathcal X^2}\leq\Lambda_0L^{-\tilde\alpha}\r\}.
    \end{equation}
\end{proposition}

\begin{proof}
According to \cref{norm estimate on y}, we have 
\begin{equation}\label{the norm estimate we wanted}
        \norm{\Leaf^m}_{\mathcal X^2\to\mathcal X^2}\leq \Lambda\l(\Lambda\sqrt{\delta}\r)^mN(L)^{2(s+d)}L^{\frac{\mathcal K}{2}+\frac{1}{\beta}+\frac{d}{2}+A}\\\leq L^{\frac{\log(\Lambda)}{\log(L)}+m\frac{\log(\Lambda)}{\log(L)}-\frac{m}{2}\frac{\log(\delta)}{\log(L)}+\frac{\mathcal K}{2}+\frac{1}{\beta}+\frac{d}{2}+2A},
    \end{equation}
where we set $N=N(L)\coloneqq\lfloor\log(L)\rfloor$. For $m=N(L)$, we obtain
\begin{equation}
    \norm{\Leaf^{N(L)}}_{\mathcal X^2\to\mathcal X^2}\leq L^{\frac{\log(\Lambda)}{\log(L)}+\log(\Lambda)-\frac{1}{2}\log(\delta)+\frac{\mathcal K}{2}+\frac{1}{\beta}+\frac{d}{2}+2A}.
\end{equation}
For $\delta$ small enough, we achieve 
\begin{equation}
    \frac{\log(\Lambda)}{\log(L)}+\log(\Lambda)-\frac{1}{2}\log(\delta)+\frac{\mathcal K}{2}+\frac{1}{\beta}+\frac{d}{2}+2A<0
\end{equation}
and for $L$ large enough, 
\begin{equation}
    \norm{\Leaf^{N(L)}}_{\mathcal X^2\to\mathcal X^2}<1.
\end{equation}
In this case, the operator $\Id-\Leaf^{N(L)}$ becomes invertible. We obtain the invertibility of $\Id-\Leaf$ since
\begin{equation}\label{operator inversion identity}
        \l(\Id-\Leaf^{N(L)}\r)^{-1}\sum_{n\leq N(L)-1}\Leaf^n=\l(\Id-\Leaf\r)^{-1}.
\end{equation}
Using 
\begin{equation}
        \norm{\l(\Id-\Leaf^{N(L)}\r)^{-1}}_{\mathcal X^2\to\mathcal X^2}\leq\frac{1}{1-\norm{\Leaf^{N(L)}}_{\mathcal X^2\to\mathcal X^2}}\leq\Lambda,
\end{equation}
and \cref{the norm estimate we wanted}, we get
\begin{equation}
\begin{gathered}
    \norm{\l(\Id-\Leaf\r)^{-1}}_{\mathcal X^2\to\mathcal X^2}\\\leq\norm{\l(\Id-\Leaf^{N(L)}\r)^{-1}}_{\mathcal X^2\to\mathcal X^2}\sum_{0\leq m\leq N(L)-1}\norm{\Leaf^m}_{\mathcal X^2\to\mathcal X^2}\leq\Lambda N(L)^{2(s+d)}L^{\frac{\mathcal K}{2}+\frac{1}{\beta}+\frac{d}{2}+A}\\\leq\Lambda L^{\frac{\mathcal K}{2}+\frac{1}{\beta}+\frac{d}{2}+2A}.
\end{gathered}
\end{equation}

 We start by estimating $\mathcal W^\eta$ and recall
        \begin{gather}
            \mathcal W^\eta=\int_0^t\sum_{\substack{0\leq\sum_{i=1}^5n_i\leq N(L)\\\sum_{i=1}^5n_i\geq N(L)}}C^{+}\l(\tau,F_{n_1}^\eta(\tau),\overline{F_{n_2}^\eta}(\tau),\overline{F_{n_3}^{\eta}}(\tau),{F_{n_4}^{\overline\eta}}(\tau),F_{n_5}^{\overline\eta}(\tau)\r)\rmd\tau.
        \end{gather}
        We have, since $s>\frac{d}{2}$, \begin{equation}
            \begin{gathered}
            \norm{C^{+}(t,f,g,h,u,v)}_{H^s\l(T_L^d\r)}^2\leq\frac{\epsilon^2}{L^{4d}}\sum_{k\in\Z_L^d}\langle k\rangle^{2s}\abs{\sum_{\sum_{i=1}^5k_i=k}\hat f(k_1)\hat g(k_2)\hat h(k_3)\hat u(k_4)\hat v(k_5)}^2\\
            =\epsilon^2\sum_{k\in\Z_L^d}\langle k\rangle^{2s}\abs{\widehat{fghuv}(k)}^2=\epsilon^2\norm{fghuv}_{H^s\l(T_L^d\r)}^2\lesssim {\epsilon^2}\norm{f}_{H^s\l(T_L^d\r)}^2\norm{g}_{H^s\l(T_L^d\r)}^2\norm{h}_{H^s\l(T_L^d\r)}^2\norm{u}_{H^s\l(T_L^d\r)}^2\norm{v}_{H^s\l(T_L^d\r)}^2.
        \end{gathered}
        \end{equation}
        It thus follows with \cref{The estimate the whole world was looking for}
        \begin{equation}\label{the constant term i have to estimate}
            \begin{gathered}
            \norm{\mathcal W^\eta}_{\mathcal X}\leq\int_0^t\sum_{\substack{0\leq n_1,n_2,n_3,n_4,n_5\leq N(L)\\n_1+n_2+n_3+n_4+n_5\geq N(L)}}\norm{C^{+}\l(\tau,F_{n_1}^\eta(\tau),\overline{F_{n_2}^\eta}(\tau),F_{n_3}^{\overline\eta}(\tau),\overline{F_{n_4}^{\overline\eta}}(\tau),F_{n_5}^{\overline\eta}(\tau)\r)}_{H^s\l(T_L^d\r)}\rmd\tau\\
            \lesssim\epsilon\int_0^t\sum_{\substack{0\leq n_1,n_2,n_3,n_4,n_5\leq N(L)\\n_1+n_2+n_3+n_4+n_5\geq N(L)}}\norm{f_{n_1}^\eta(\tau)}_{H^s\l(T_L^d\r)}\norm{f_{n_2}^\eta(\tau)}_{H^s\l(T_L^d\r)}\norm{f_{n_3}^{\overline{\eta}}(\tau)}_{H^s\l(T_L^d\r)}\norm{f_{n_4}^{\overline\eta}(\tau)}_{H^s\l(T_L^d\r)}\norm{f_{n_5}^{\overline\eta}(\tau)}_{H^s\l(T_L^d\r)}\\
            \lesssim\Lambda\delta^{10}L^{5\mathcal K+5d+\frac{9}{\beta}+10A}\sum_{\substack{0\leq n_1,n_2,n_3,n_4,n_5\leq N(L)\\n_1+n_2+n_3+n_4+n_5\geq N(L)}}\l(\Lambda\delta\r)^{2\l(n_1+n_2+n_3+n_4+n_5\r)}\\\leq\Lambda(\Lambda\delta)^{N(L)}\delta^{10}L^{5\mathcal K+5d+\frac{9}{\beta}+10A}\l(\sum_{n\geq0}(\Lambda\delta)^{2n}\r)^5\\\leq\Lambda(\Lambda\delta)^{N(L)}\delta^{10}L^{5\mathcal K+5d+\frac{9}{\beta}+10A}.
        \end{gathered}
        \end{equation}
Furthermore, 
\begin{equation}
\begin{aligned}
\norm{\l(\Id-\Leaf\r)^{-1}\mathcal W}_{\mathcal X^2}&\leq \Lambda\delta^{10}(\Lambda\delta)^{N(L)}\delta^{10} L^{\frac{\mathcal K}{2}+\frac{1}{\beta}+\frac{d}{2}+2A+5\mathcal K+5d+\frac{9}{\beta}+10A}\\&\leq\Lambda\delta^{10}L^{\log(\Lambda)+\log(\delta)+\frac{\mathcal K}{2}+\frac{1}{\beta}+\frac{d}{2}+2A+5\mathcal K+5d+\frac{9}{\beta}+10A}
\end{aligned}
\end{equation}

We denote $\mathcal R(v)^\eta\coloneqq\mathcal R^1(v)^\eta+\mathcal R^2(v)^\eta+\mathcal R^3(v)^\eta+\mathcal R^4(v)^\eta$ for all the terms that are at least quadratic in $v$. For them, we calculate (using the boundedness of the operator $\tilde C^{\eta,+}$)
        \begin{equation}
            \begin{gathered}
            \norm{\mathcal R(v)}_{\mathcal X^2}\leq\Lambda\epsilon\norm{v}_{\mathcal X}^2\Bigg(\sum_{n_1,n_2,n_3\leq N}\norm{f_{n_1}}_{\mathcal X^2}\norm{f_{n_2}}_{\mathcal X^2}\norm{f_{n_3}}_{\mathcal X^2}\\+\norm{v}_{\mathcal X^2}\sum_{n_1,n_2\leq N}\norm{f_{n_1}}_{\mathcal X^2}\norm{f_{n_2}}_{\mathcal X^2}+\norm{v}_{\mathcal X^2}^2\sum_{n\leq N}\norm{f_n}_{\mathcal{X}^2}+\norm{v}_{\mathcal{X}^2}^3\Bigg)\\\leq\Lambda \epsilon\norm{v}_{\mathcal X^2}^2\Bigg(L^{3\mathcal K+3d+\frac{6}{\beta}+6A}\l(\sum_{n\leq N(L)}(\Lambda\delta)^{2n}\r)^3+\norm{v}_{\mathcal{X}^2}L^{2\mathcal K+ 2d+\frac{4}{\beta}+4A}\l(\sum_{n\leq N(L)}(\Lambda\delta)^{2n}\r)^2\\+\norm{v}_{\mathcal{X}^2}^2L^{\mathcal K+d\frac{2}{\beta}+2A}\sum_{n\leq N(L)}(\Lambda\delta)^{2n}+\norm{v}_{\mathcal{X}^2}^3\Bigg)\\
            \leq\Lambda\norm{v}_{\mathcal X^2}^2\l(L^{3\mathcal K+3d+\frac{5}{\beta}+6A}+\norm{v}_{\mathcal X^2}L^{2\mathcal K+ 2d+\frac{3}{\beta}+4A}+\norm{v}_{\mathcal X^2}^2L^{\mathcal K+d\frac{1}{\beta}+2A}+\norm{v}_{\mathcal X^2}^3\r).
        \end{gathered}
        \end{equation}
This implies 
        \begin{equation}
            \begin{gathered}
                \norm{\l(\Id-\mathcal L\r)^{-1}\mathcal R(v)}_{\mathcal{X}^2}\\\leq\Lambda\norm{v}_{\mathcal X^2}^2\l(L^{\frac{7}{2}\mathcal K+\frac{7}{2}d+\frac{6}{\beta}+8A}+\norm{v}_{\mathcal X^2}L^{\frac{5}{2}\mathcal K+ \frac{5}{2}d+\frac{4}{\beta}+6A}+\norm{v}_{\mathcal X^2}^2L^{\frac{3}{2}\mathcal K+\frac{3}{2}d+\frac{2}{\beta}+4A}+\norm{v}_{\mathcal X^2}^3\r).
            \end{gathered}
        \end{equation}
        
We now assumed that $\norm{v}_{\mathcal{X}^2}\leq\Lambda_0L^{-\tilde\alpha}$ for some $\Lambda_0,\tilde\alpha>0$ to be determined in the following. Simultaneously, we would like the above expression to be bounded by $\Lambda_0L^{-\alpha}$. 
        We may choose
        \begin{gather}\label{our choice of alpha}
            \tilde\alpha\coloneqq2\l(\frac{7}{2}\mathcal K+\frac{7}{2}d+\frac{6}{\beta}+8A\r)
        \end{gather}       
        and first estimate

        \begin{equation}\label{what i am working on right now}
            \begin{gathered}
                \norm{\l(\Id-\Leaf\r)^{-1}\l(\W+\mathcal R(v)\r)}_{\mathcal{X}^2}\leq\norm{\l(\Id-\Leaf\r)^{-1}\mathcal W}_{\mathcal X^2}+\norm{\l(\Id-\Leaf\r)^{-1}\mathcal R(v)}_{\mathcal X^2}\\
                \leq \Lambda\delta^{10}L^{\log(\Lambda)+\log(\delta)+\frac{\mathcal K}{2}+\frac{1}{\beta}+\frac{d}{2}+2A+5\mathcal K+5d+\frac{9}{\beta}+10A} +4\Lambda\Lambda_0^2L^{-\frac{3}{2}\tilde\alpha}
            \end{gathered}
        \end{equation}
        Thence, we may choose $\Lambda_0\coloneqq\Lambda$ and $\delta$ small enough so that 
        \begin{equation}
            \log(\Lambda)+\log(\delta)+\frac{\mathcal K}{2}+\frac{1}{\beta}+\frac{d}{2}+2A+5\mathcal K+5d+\frac{9}{\beta}+10A\leq-\frac{3}{2}\tilde\alpha
        \end{equation}
        and if $L$ is large enough, we may bound \cref{what i am working on right now} by 
        \begin{equation}
            \norm{\l(\Id-\Leaf\r)^{-1}\l(\W+\mathcal R(v)\r)}_{\mathcal{X}^2}\leq\frac{\Lambda}{2}L^{-\tilde\alpha}+\frac{\Lambda}{2}L^{-\tilde\alpha}=\Lambda L^{-\tilde\alpha}.
        \end{equation}
        We have proven 

        \begin{equation}
            \l(\Id-\Leaf\r)^{-1}\l(\mathcal W+\mathcal R\r)\l(\mathcal Z\r)\subseteq\mathcal Z.
        \end{equation}
        Now we prove contractibility. It should be noted generally that we will be dealing with terms of the form
        \begin{equation}
            \begin{gathered}
            \abs{\sum_{\sum_{i=1}^5k_i=k}\l(\widehat{v_i^{\iota_1}}(k_1)\widehat{v_j^{\iota_2}}(k_2)-\widehat{w_i^{\iota_1}}(k_1)\widehat{w_j^{\iota_2}}(k_2)\r)\hat f(k_3)\hat g(k_4)\hat h(k_5)}\\
            =\abs{\l(\l(\widehat{v_i^{\iota_1}}*\widehat{v_j^{\iota_2}}-\widehat{w_i^{\iota_1}}*\widehat{w_j^{\iota_2}}\r)*\hat f*\hat g*\hat h\r)(k)}\\\leq\abs{\l(\reallywidehat{v_i^{\iota_1}}*\widehat{v_j^{\iota_2}-w_j^{\iota_2}}*\hat f*\hat g*\hat h\r)(k)}+\abs{\l(\widehat{v_i^{\iota_1}-w_i^{\iota_1}}*\widehat{w_j^{\iota_2}}*\hat f*\hat g*\hat h\r)(k)}\\
            =L^{2d}\abs{\reallywidehat{v_i^{\iota_1}\l(v_j^{\iota_2}-w_j^{\iota_2}\r)fgh}(k)}+L^{2d}\abs{\reallywidehat{\l(v_i^{\iota_1}-w_i^{\iota_1}\r)w_j^{\iota_2}fgh}(k)},
        \end{gathered}
        \end{equation}
        where $i,j\in\{1,2\}$ and $\iota_1,\iota_2\in\{\pm\}$. The following estimate follows. 
        \begin{equation}
            \begin{gathered}
            \norm{\mathcal R^k(v_1,v_2)^\eta-\mathcal R^k(w_1,w_2)^\eta}_{\mathcal X}^2\\\leq\frac{\epsilon^2}{L^{4d}}\sum_{k\in\Z_L^d}\langle k\rangle^{2s}\abs{\sum_{\sum_{i=1}^5k_i=k}\l(\widehat{v_i^{\iota_1}}(k_1)\widehat{v_j^{\iota_2}}(k_2)-\widehat{w_i^{\iota_1}}(k_1)\widehat{w_j^{\iota_2}}(k_2)\r)\hat f(k_3)\hat g(k_4)\hat h(k_5)}^2\\\lesssim 
            \epsilon^2\norm{v_i^{\iota_1}\l(v_j^{\iota_2}-w_j^{\iota_2}\r)fgh}_{\mathcal X}^2+\epsilon^2\norm{\l(v_i^{\iota_1}-w_i^{\iota_1}\r)w_j^{\iota_2}fgh}_{\mathcal X}^2\\\lesssim \epsilon^2\norm{f}_{\mathcal X}^2\norm{g}_{\mathcal X}^2\norm{h}_{\mathcal X}^2\l(\norm{v_i^{\iota_1}}_{\mathcal X}^2\norm{v_j^{\iota_2}-w_j^{\iota_2}}_{\mathcal X}^2+\norm{w_j^{\iota_2}}_{\mathcal X}^2\norm{v_i^{\iota_1}-w_i^{\iota_1}}_{\mathcal X}^2\r)\\
            \lesssim\epsilon^2L^{-2\tilde\alpha}\norm{f}_{\mathcal X}^2\norm{g}_{\mathcal X}^2\norm{h}_{\mathcal X}^2\l(\norm{v_j^{\iota_2}-w_j^{\iota_2}}_{\mathcal X}+\norm{v_i^{\iota_1}-w_i^{\iota_1}}_{\mathcal X}\r)^2\\\lesssim\epsilon^2L^{-2\tilde\alpha}\norm{f}_{\mathcal X}^2\norm{g}_{\mathcal X}^2\norm{h}_{\mathcal X}^2\norm{\l(v_1-w_1,v_2-w_2\r)}_{\mathcal X^2}^2,
        \end{gathered}
        \end{equation}
        where $f,g,h\in\bigsqcup_{\substack{\iota_i\in\{\pm\}\\1\leq i\leq6}}\l\{v_1^{\iota_1},v_2^{\iota_2},w_1^{\iota_3},w_2^{\iota_4},F_{\leq N}^{\eta,\iota_5},F_{\leq N}^{\overline\eta,\iota_6}\r\}$. More precisely, $k-1$ functions among $f$, $g$ and $h$ are of type $v_i^{\iota}$ or $w_j^{\iota}$ and the remaining ones are of type $F_{\leq N}^{\eta,\iota}$ or $F_{\leq N}^{\overline\eta,\iota}$.
        Finally,

        \begin{equation}
            \norm{\mathcal R^k(v_1,v_2)^\eta-\mathcal R^k(w_1,w_2)^\eta}_{\mathcal X}\lesssim\epsilon L^{-\tilde\alpha}\norm{f}_{\mathcal X}\norm{g}_{\mathcal X}\norm{h}_{\mathcal X}\norm{\l(v_1-w_1,v_2-w_2\r)}_{\mathcal X^2}.
        \end{equation}

        We estimate further

\begin{equation}
    \begin{gathered}
            \norm{\l(\Id-\Leaf\r)^{-1}\l(\mathcal W+\mathcal R(v)\r)-\l(\Id-\Leaf\r)^{-1}\l(\mathcal W+\mathcal R(w)\r)}_{\mathcal{X}^2}=\norm{\l(\Id-\Leaf\r)^{-1}\l(\mathcal R(v)-\mathcal R(w)\r)}_{\mathcal{X}^2}\\\leq\norm{\l(\Id-\Leaf\r)^{-1}}_{\mathcal X^2\to\mathcal X^2}\norm{\mathcal R(v)-\mathcal R(w)}_{\mathcal{X}^2}\\
            \lesssim  L^{\frac{\mathcal K}{2}+\frac{1}{\beta}+\frac{d}{2}+2A}\sum_{i=1}^4\norm{\mathcal R^i(v)-\mathcal R^i(w)}_{\mathcal{X}^2}\\      
            \lesssim \epsilon L^{\frac{\mathcal K}{2}+\frac{1}{\beta}+\frac{d}{2}+2A-\tilde\alpha}\sum_{i=0}^3\l(L^{\mathcal K+d+\frac{2}{\beta}+2A}\sum_{n\leq N(L)}(\Lambda\delta)^{2n}\r)^i\l(\Lambda L^{-\tilde\alpha}\r)^{3-i}\norm{v-w}_{\mathcal X^2}\\  
            \leq L^{-\frac{\tilde\alpha}{2}}\sum_{i=0}^3\l(\Lambda L^{-\l(\tilde\alpha + \mathcal K+d+\frac{2}{\beta}+2A\r)}\r)^i\norm{v-w}_{\mathcal X^2}
        \end{gathered}
\end{equation}   
and as we can see, if $L$ is large enough, we obtain that 
\begin{equation}\label{whose fixed points we are after}
    \mathcal Z\ni v\mapsto\l(\Id-\Leaf\r)^{-1}\l(\mathcal W+\mathcal R(v)\r)\in\mathcal Z
\end{equation}
is indeed a contraction. Finally, applying the Banach fixed point theorem to the map in \cref{whose fixed points we are after}, finishes the proof.
\end{proof}

\section{Second part of the proof of \cref{this is the main theorem}}\label{fourth section}
\begin{definition}
    For convenience reasons, we define 
    \begin{equation}
        \begin{gathered}
            \tilde\zeta_{1,\iota}^\eta(\xi_1,\xi_2,\xi)\coloneqq Q_{\iota}^\eta\l(\xi,\xi_1,\xi_2,-\xi_1,-\xi_2\r),\\
            \tilde\zeta_{2,\iota}^\eta(\xi_1,\xi_2,\xi)\coloneqq Q_{\iota}^\eta\l(\xi,\xi_1,\xi_2,-\xi_2,-\xi_1\r),\\
            \tilde\zeta_{3,\iota}^\eta(\xi_1,\xi_2,\xi)\coloneqq Q_{\iota}^\eta\l(\xi_1,-\xi_1,\xi_2,\xi,-\xi_2\r),\\
            \tilde\zeta_{4,\iota}^\eta(\xi_1,\xi_2,\xi)\coloneqq Q_{\iota}^\eta\l(\xi_1,\xi_2,-\xi_1,\xi,-\xi_2\r),\\
            \tilde\zeta_{5,\iota}^\eta(\xi_1,\xi_2,\xi)\coloneqq Q_{\iota}^\eta\l(\xi_1,-\xi_1,\xi_2,-\xi_2,\xi\r),\\
            \tilde\zeta_{6,\iota}^\eta(\xi_1,\xi_2,\xi)\coloneqq Q_{\iota}^\eta\l(\xi_1,\xi_2,-\xi_1,-\xi_2,\xi\r).
        \end{gathered}
    \end{equation}
\end{definition}

We define the following sequence of maps 
\begin{equation}
\begin{gathered}
    \rho_0^\eta\coloneqq M^{\eta,\eta},\ \rho_0^\times\coloneqq M^{0,1},\\
    \rho^\eta_n(t,\xi)\coloneqq2\sum_{n_1+n_2+n_3=n-1}\Bigg[\sum_{l=1}^2\int_0^t\rho_{n_3}^\eta(s,\xi)\int_{\R^d}\int_{\R^d}\Im\l(\rho_{n_1}^{\times,\overline\eta}(s,-\xi_1)\tilde\zeta_{l,+}^\eta\l(\xi_1,\xi_2,\xi\r)\rho_{n_2}^{\times,\overline\eta}(s,-\xi_2)\r)\rmd\xi_1\rmd\xi_1\rmd s\\+\sum_{l=3}^6\int_0^t\int_{\R^d}\int_{\R^d}\Im\l(\rho_{n_3}^{\times,\overline\eta}(s,\xi)\tilde\zeta_{l,+}^\eta\l(\xi_1,\xi_2,\xi\r)\rho_{n_2}^{\times,\overline\eta}(s,-\xi_2)\r)\rho_{n_1}^\eta(s,\xi_1)\rmd\xi_1\rmd\xi_2\rmd s\Bigg].
\end{gathered}
\end{equation}
and 

\begin{equation}
    \begin{gathered}
       \rho_n^\times(t,\xi) =-\rmi\tau\sum_{n_1+n_2+n_3=n-1}\\\Bigg[\sum_{l=1}^2\int_0^t\rho_{n_3}^\times(s,\xi)\int_{\R^d}\int_{\R^d}\l(\tilde\zeta_{l,+}^0\l(\xi_1,\xi_2,\xi\r)-\overline{\tilde\zeta_{l,+}^1\l(\xi_1,\xi_2,\xi\r)}\r)\overline{\rho_{n_1}^\times(s,-\xi_1)}\overline{\rho_{n_2}^\times(s,-\xi_2)}\rmd\xi_1\rmd\xi_2\rmd s\\
    +\sum_{l=3}^6\int_0^t\int_{\R^d}\int_{\R^d}\l(\rho_{n_3}^1(s,\xi)\tilde\zeta_{l,+}^0\l(\xi_1,\xi_2,\xi\r)\rho_{n_1}^0(s,\xi_1)-\rho_{n_3}^0(s,\xi)\overline{\tilde\zeta_{l,+}^1\l(\xi_1,\xi_2,\xi\r)}\rho_{n_1}^1(s,\xi_1)\r)\overline{\rho_{n_2}^\times(s,-\xi_2)}\rmd\xi_1\rmd\xi_2\rmd s\Bigg].
    \end{gathered}
\end{equation}
Since $M^{\eta,\eta}$ is real-values, $\rho^\eta$ is real-valued as well. This is shown by induction.

\subsection{The effective system}\label{where local well-posedness is proven}
    We would like to converge to a solution of \cref{dee} that should look like 
    \begin{equation}\label{the effective system I want to converge to}
        \begin{gathered}    \rho^\eta(t,\xi)=M^{\eta,\eta}(\xi)+2\sum_{l=1}^2\int_0^t\rho^\eta(s,\xi)\int_{\R^d}\int_{\R^d}\Im\l(\tilde\zeta_{l,\iota}^\eta\l(\xi_1,\xi_2,\xi\r)\rho^{\times,\overline\eta}(s,-\xi_1)\rho^{\times,\overline\eta}(s,-\xi_2)\r)\rmd\xi_1\rmd\xi_1\rmd s\\+2\sum_{l=3}^6\int_0^t\int_{\R^d}\int_{\R^d}\Im\l(\rho^{\times,\overline\eta}(s,\xi)\tilde\zeta_{l,\iota}^\eta\l(\xi_1,\xi_2,\xi\r)\rho^{\times,\overline\eta}(s,-\xi_2)\r)\rho^\eta(s,\xi_1)\rmd\xi_1\rmd\xi_2\rmd s
        \end{gathered}
    \end{equation}
    and
    \begin{equation}
    \begin{gathered}  
        \rho^\times(t,\xi)=M^{0,1}(\xi)-\rmi\sum_{l=1}^2\int_0^t\rho^\times(s,\xi)\int_{\R^d}\int_{\R^d}\l(\tilde\zeta_{l,+}^0\l(\xi_1,\xi_2,\xi\r)-\overline{\tilde\zeta_{l,+}^1\l(\xi_1,\xi_2,\xi\r)}\r)\overline{\rho^\times(s,-\xi_1)}\overline{\rho^\times(s,-\xi_2)}\rmd\xi_1\rmd\xi_2\rmd s\\-\rmi\sum_{l=3}^6\int_0^t\int_{\R^d}\int_{\R^d}\l(\rho^1(s,\xi)\tilde\zeta_{l,+}^0\l(\xi_1,\xi_2,\xi\r)\rho^0(s,\xi_1)-\rho^0(s,\xi)\overline{\tilde\zeta_{l,+}^1\l(\xi_1,\xi_2,\xi\r)}\rho^1(s,\xi_1)\r)\overline{\rho^\times(s,-\xi_2)}\rmd\xi_1\rmd\xi_2\rmd s.
        \end{gathered}
    \end{equation}
If we define the vector field $V\colon\l(W^{1,\infty}\l(\R^d\r)\cap W^{1,1}\l(\R^d\r)\r)^3\to\l(W^{1,\infty}\l(\R^d\r)\cap W^{1,1}\l(\R^d\r)\r)^3$ so that $V(\rho^0,\rho^1,\rho^\times)$ is the right-hand side of \cref{dee} (ignoring the time-dependency), one can quite quickly see that $V$ is locally Lipschitz continuous. 
For this reason, we have local well-posedness of the system \cref{the effective system I want to converge to} on a finite time interval $[0,\delta]$ for some $\delta>0$ sufficiently small.

The components of $V$ are each estimated to be 

\begin{equation}
    \begin{gathered}
        \norm{V\l(\rho^0,\rho^1,\rho^\times\r)^\eta}_{W^{1,\infty}\cap W^{1,1}}\leq4\norm{\rho^\eta}_{W^{1,\infty}\cap W^{1,1}}\norm{Q_+^\eta}_{W^{1,\infty}\cap W^{1,1}}\norm{\rho^\times}_{L^1}^2\\+8\norm{\rho^\times}_{W^{1,\infty}\cap W^{1,1}}\norm{Q_+^\eta}_{W^{1,\infty}\cap W^{1,1}}\norm{\rho^\times}_{L^1}\norm{\rho^\eta}_{L^1}
    \end{gathered}
\end{equation}
and 
\begin{equation}
    \begin{gathered}
        \norm{V\l(\rho^0,\rho^1,\rho^\times\r)^\times}_{W^{1,\infty}\cap W^{1,1}}\leq 4\norm{\rho^\times}_{W^{1,\infty}\cap W^{1,1}}\max_{\eta}\norm{Q_+^\eta}_{W^{1,\infty}\cap W^{1,1}}\norm{\rho^\times}_{L^1}^2\\+4\norm{\rho^1}_{W^{1,\infty}\cap W^{1,1}}\max_{\eta}\norm{Q_+^\eta}_{W^{1,\infty}\cap W^{1,1}}\norm{\rho^0}_{L^1}\norm{\rho^\times}_{L^1}\\
        +4\norm{\rho^0}_{W^{1,\infty}\cap W^{1,1}}\max_{\eta}\norm{Q_+^\eta}_{W^{1,\infty}\cap W^{1,1}}\norm{\rho^1}_{L^1}\norm{\rho^\times}_{L^1}
    \end{gathered}
\end{equation}

\begin{definition}
    We define Catalan $3$-fold convolution recursively by $c_0\coloneqq1$ and for all $n>0$,
    \begin{equation}
        c_n\coloneqq\sum_{n_1+n_2+n_3=n-1}c_{n_1}c_{n_2}c_{n_3}=\frac{3}{2n+3}\binom{2n+3}{n}\leq(5e)^n.
    \end{equation}
\end{definition}

\begin{remark}
    For higher non-linearities of degree $2k+1$, one should use a $k$-Catalan convolution.
\end{remark}
\begin{proposition}\label{another trivial bound}
    It holds
    \begin{equation}\label{convergence estimate}
        \max_\star\norm{\rho_n^\star}_{L^\infty\l(W^{1,\infty}\cap W^{1,1}\r)}\leq \Lambda(\Lambda\delta)^nc_n.
    \end{equation}
    In particular, the series $\sum_{n\geq0}\rho_n^\star$ converges 
    absolutely in $L^\infty\l(W^{1,\infty}\cap W^{1,1}\r)$.
\end{proposition}
\begin{proof}
We prove \cref{convergence estimate} by induction, and the start $n=0$ is obviously true because $M^\star$ is bounded. Now assume $n>0$ and that the statement holds for all $0\leq m\leq n-1$. We find,

\begin{equation}
    \begin{gathered}
        \norm{\rho_n^\eta}_{L^\infty(W^{1,\infty}\cap W^{1,1})}\leq\delta\sum_{n_1+n_2+n_3=n-1}\Bigg[4\norm{\rho^\eta_{n_3}}_{L^\infty(W^{1,\infty}\cap W^{1,1})}\norm{Q_+^\eta}_{W^{1,\infty}\cap W^{1,1}}\norm{\rho_{n_1}^\times}_{L^\infty L^1}\norm{\rho_{n_2}^\times}_{L^\infty L^1} \\+ 4\norm{\rho^\times_{n_3}}_{L^\infty\l(W^{1,\infty}\cap W^{1,1}\r)}\norm{Q_+^\eta}_{W^{1,\infty}\cap W^{1,1}}\norm{\rho^\times_{n_2}}_{L^\infty L^1}\norm{\rho^\eta_{n_1}}_{L^\infty L^1}\\+ 4\norm{\rho^\times_{n_3}}_{L^\infty\l(W^{1,\infty}\cap W^{1,1}\r)}\norm{Q_+^\eta}_{W^{1,\infty}\cap W^{1,1}}\norm{\rho^\times_{n_2}}_{L^\infty L^1}\norm{\rho^\eta_{n_1}}_{L^\infty L^1}\Bigg]\\\leq\Lambda^2\delta\l(\Lambda\delta\r)^{n-1}\sum_{n_1+n_2+n_3=n-1}c_{n_1}c_{n_2}c_{n_3}\\=\Lambda(\Lambda\delta)^nc_n,
    \end{gathered}
\end{equation}
where we used the boundedness of $Q^\eta_\iota$ and its gradient. Similarly
\begin{equation}
    \begin{gathered}
        \norm{\rho^\times_n}_{L^\infty\l(W^{1,\infty}\cap W^{1,1}\r)}\leq4\max_\eta\norm{Q_+^\eta}_{W^{1,\infty}\cap W^{1,1}}\delta\sum_{n_1+n_2+n_3=n-1}\Bigg[\norm{\rho^\times_{n_3}}_{L^\infty\l(W^{1,\infty}\cap W^{1,1}\r)}\norm{\rho^\times_{n_1}}_{L^\infty L^1}\norm{\rho^\times_{n_2}}_{L^\infty L^1}\\+\norm{\rho^1_{n_3}}_{L^\infty\l(W^{1,\infty}\cap W^{1,1}\r)}\norm{\rho_{n_1}^0}_{L^\infty L^1}\norm{\rho_{n_2}^\times}_{L^\infty L^1}+\norm{\rho_{n_3}^0}_{L^\infty\l(W^{1,\infty}\cap W^{1,1}\r)}\norm{\rho_{n_1}^1}_{L^\infty L^1}\norm{\rho_{n_2}^\times}_{L^\infty L^1}\Bigg]\\\leq\Lambda(\Lambda\delta)^nc_n
    \end{gathered}
\end{equation}
    Since $c_n\leq(5e)^n$, we have the cruder estimate $\max_\star\norm{\rho_n^\star}_{L^\infty\l(W^{1,\infty}\cap W^{1,1}\r)}\leq \Lambda(\Lambda\delta)^n$ by increasing $\Lambda$. Thence, 
    \begin{equation}
        \sum_{n\geq0}\norm{\rho_n^\star}_{L^\infty\l(W^{1,\infty}\cap W^{1,1}\r)}\leq\Lambda\sum_{n\geq0}(\Lambda\delta)^n<\infty
    \end{equation}
    which becomes the geometric series upon decreasing $\delta$ so that $\Lambda\delta<1$.
\end{proof}

\subsection{Solution to the effective system as a fixed point}
We recall $N=N(L)=\lfloor\log(L)\rfloor$ and define 
\begin{equation}
    v^\star\coloneqq\rho^\star-\rho_{\leq N}^\star,
\end{equation}
where 
\begin{equation}
    \rho_{\leq N}^\star\coloneqq\sum_{n\leq N}\rho_n^\star.
\end{equation}
Then obviously, 

\begin{equation}
    v^\star = \Sigma^\star + \Leaf(v)^\star+ Q(v,v)^\star+R(v,v,v)^\star,
\end{equation}
where 
\begin{equation}
    \begin{aligned}
        \Sigma^\eta&\coloneqq-\rho_{\leq N}^\eta + M^{\eta,\eta}+2\sum_{l=1}^2\int_0^t\rho_{\leq N}^\eta(s,\xi)\int_{\R^d}\int_{\R^d}\Im\l(\tilde\zeta_{l,+}^\eta(\xi_1,\xi_2,\xi)\rho_{\leq N}^{\times,\overline\eta}(s,-\xi_1)\rho_{\leq N}^{\times,\overline\eta}(s,-\xi_2)\r)\rmd\xi_1\rmd\xi_2\rmd s\\&+2\sum_{l=3}^6\int_0^t\int_{\R^d}\int_{\R^d}\Im\l(\rho_{\leq N}^{\times,\overline\eta}(s,\xi)\tilde\zeta_{l,+}^\eta(\xi_1,\xi_2,\xi)\rho_{\leq N}^{\times,\overline\eta}(s,-\xi_2)\r)\rho_{\leq N}^\eta(s,\xi_1)\rmd\xi_1\rmd\xi_2\rmd s\\
        &=2\sum_{\substack{0\leq n_1,n_2,n_3\leq N\\N\leq n_1+n_2+n_3\leq3N}}\Bigg[\sum_{l=1}^2\int_0^t\rho_{n_3}^\eta(s,\xi)\int_{\R^d}\int_{\R^d}\Im\l(\tilde\zeta_{l,+}^\eta(\xi_1,\xi_2,\xi)\rho_{n_1}^{\times,\overline\eta}(s,-\xi_1)\rho_{n_2}^{\times,\overline\eta}(s,-\xi_2)\r)\rmd\xi_1\rmd\xi_2\rmd s\\&+\sum_{l=3}^6\int_0^t\int_{\R^d}\int_{\R^d}\Im\l(\rho_{n_3}^{\times,\overline\eta}(s,\xi)\tilde\zeta_{l,+}^\eta(\xi_1,\xi_2,\xi)\rho_{n_1}^{\times,\overline\eta}(s,-\xi_2)\r)\rho_{n_2}^\eta(s,\xi_1)\rmd\xi_1\rmd\xi_2\rmd s\Bigg],
        \end{aligned}
        \end{equation}
\begin{equation}
\begin{gathered}    
        \Sigma^\times\coloneqq-\rho_{\leq N}^\times+M^{0,1}-\rmi\sum_{l=1}^2\int_0^t\rho_{\leq N}^\times(s,\xi)\int_{\R^d}\int_{\R^d}\l(\tilde\zeta_{l,+}^0(\xi_1,\xi_2,\xi)-\overline{\tilde\zeta_{l,+}^1(\xi_1,\xi_2,\xi)}\r)\overline{\rho_{\leq N}^\times(s,-\xi_1)}\overline{\rho_{\leq N}^\times(s,-\xi_2)}\rmd\xi_1\rmd\xi_2\rmd s\\
        -\rmi\sum_{l=3}^6\int_0^t\int_{\R^d}\int_{\R^d}\l(\rho_{\leq N}^1(s,\xi)\tilde\zeta_{l,+}^0(\xi_1,\xi_2,\xi)\rho_{\leq N}^0(s,\xi_1)-\rho_{\leq N}^0(s,\xi)\overline{\tilde\zeta_{l,+}^1(\xi_1,\xi_2,\xi)}\rho_{\leq N}^1(s,\xi_1)\r)\overline{\rho_{\leq N}^\times(s,-\xi_2)}\rmd\xi_1\rmd\xi_2\rmd s\\
        =-\rmi\sum_{\substack{0\leq n_1,n_2,n_3\leq N\\N\leq n_1+n_2+n_3\leq3N}}\Bigg[\sum_{l=1}^2\int_0^t\rho_{n_3}^\times(s,\xi)\int_{\R^d}\int_{\R^d}\l(\tilde\zeta_{l,+}^0(\xi_1,\xi_2,\xi)-\overline{\tilde\zeta_{l,+}^1(\xi_1,\xi_2,\xi)}\r)\overline{\rho_{n_1}^\times(s,-\xi_1)}\overline{\rho_{n_2}^\times(s,-\xi_2)}\rmd\xi_1\rmd\xi_2\rmd s\\
        -\rmi\sum_{l=3}^6\int_0^t\int_{\R^d}\int_{\R^d}\l(\rho_{n_3}^1(s,\xi)\tilde\zeta_{l,+}^0(\xi_1,\xi_2,\xi)\rho_{n_1}^0(s,\xi_1)-\rho_{n_3}^0(s,\xi)\overline{\tilde\zeta_{l,+}^1(\xi_1,\xi_2,\xi)}\rho_{n_1}^1(s,\xi_1)\r)\overline{\rho_{n_2}^\times(s,-\xi_2)}\rmd\xi_1\rmd\xi_2\rmd s\Bigg],
    \end{gathered}
\end{equation}

\begin{equation}
    \begin{gathered}
        \Leaf^\eta(v)\coloneqq 2\sum_{l=1}^2\sum_{i=1}^3\sum_{\substack{f_i=v\\f_j=\rho_{\leq N}\forall j\neq i}}\int_0^tf_1^\eta(s,\xi)\int_{\R^d}\int_{\R^d}\Im\l(\tilde\zeta_{l,+}^\eta(\xi_1,\xi_2,\xi)f_2^{\times,\overline\eta}(s,-\xi_1)f_3^{\times,\overline\eta}(s,-\xi_2)\r)\rmd\xi_1\rmd\xi_2\rmd s\\
        +2\sum_{l=3}^6\sum_{i=1}^3\sum_{\substack{f_i=v\\f_j=\rho_{\leq N}\forall j\neq i}}\int_0^t\int_{\R^d}\int_{\R^d}\Im\l(f_1^{\times,\overline\eta}(s,\xi)\tilde\zeta_{l,+}^\eta(\xi_1,\xi_2,\xi)f_2^{\times,\overline\eta}(s,-\xi_2)\r)f_3^\eta(s,\xi_1)\rmd\xi_1\rmd\xi_2\rmd s,
    \end{gathered}
\end{equation}

\begin{equation}
    \begin{gathered}
        \Leaf^\times(v)\coloneqq-\rmi\sum_{l=1}^2\sum_{i=1}^3\sum_{\substack{f_i=v\\f_j=\rho_{\leq N}\forall j\neq i}}\int_0^tf_1^\times(s,\xi)\int_{\R^d}\int_{\R^d}\l(\tilde\zeta_{l,+}^0(\xi_1,\xi_2,\xi)-\overline{\tilde\zeta_{l,+}^1(\xi_1,\xi_2,\xi)}\r)\overline{f_2^\times(s,-\xi_1)}\overline{f_3^\times(s,-\xi_2)}\rmd\xi_1\rmd\xi_2\rmd s\\
        -\rmi\sum_{l=3}^6\sum_{i=1}^3\sum_{\substack{f_i=v\\f_j=\rho_{\leq N}\forall j\neq i}}\int_0^t\int_{\R^d}\int_{\R^d}\l(f_1^1(s,\xi)\tilde\zeta_{l,+}^0(\xi_1,\xi_2,\xi)f_2^0(s,\xi_1)-f_1^0(s,\xi)\tilde\zeta_{l,+}^1(\xi_1,\xi_2,\xi)f_2^1(s,\xi_1)\r)\overline{f_3^\times(s,-\xi_2)}\rmd\xi_1\rmd\xi_2\rmd s,
    \end{gathered}
\end{equation}

\begin{equation}
    \begin{gathered}
        Q^\eta(v)\coloneqq 2\sum_{l=1}^2\sum_{i=1}^3\sum_{\substack{f_i=\rho\\f_j=v\forall j\neq i}}\int_0^tf_1^\eta(s,\xi)\int_{\R^d}\int_{\R^d}\Im\l(\tilde\zeta_{l,+}^\eta(\xi_1,\xi_2,\xi)f_2^{\times,\overline\eta}(s,-\xi_1)f_3^{\times,\overline\eta}(s,-\xi_2)\r)\rmd\xi_1\rmd\xi_2\rmd s\\
        +2\sum_{l=3}^6\sum_{i=1}^3\sum_{\substack{f_i=\rho\\f_j=v\forall j\neq i}}\int_0^t\int_{\R^d}\int_{\R^d}\Im\l(f_1^{\times,\overline\eta}(s,\xi)\tilde\zeta_{l,+}^\eta(\xi_1,\xi_2,\xi)f_2^{\times,\overline\eta}(s,-\xi_2)\r)f_3^\eta(s,\xi_1)\rmd\xi_1\rmd\xi_2\rmd s,
    \end{gathered}
\end{equation}

\begin{equation}
    \begin{gathered}
        Q^\times(v)\coloneqq-\rmi\sum_{l=1}^2\sum_{i=1}^3\sum_{\substack{f_i=\rho\\f_j=v\forall j\neq i}}\int_0^tf_1^\times(s,\xi)\int_{\R^d}\int_{\R^d}\l(\tilde\zeta_{l,+}^0(\xi_1,\xi_2,\xi)-\overline{\tilde\zeta_{l,+}^1(\xi_1,\xi_2,\xi)}\r)\overline{f_2^\times(s,-\xi_1)}\overline{f_3^\times(s,-\xi_2)}\rmd\xi_1\rmd\xi_2\rmd s\\
        -\rmi\sum_{l=3}^6\sum_{i=1}^3\sum_{\substack{f_i=\rho\\f_j=v\forall j\neq i}}\int_0^t\int_{\R^d}\int_{\R^d}\l(f_1^1(s,\xi)\tilde\zeta_{l,+}^0(\xi_1,\xi_2,\xi)f_2^0(s,\xi_1)-f_1^0(s,\xi)\tilde\zeta_{l,+}^1(\xi_1,\xi_2,\xi)f_2^1(s,\xi_1)\r)\overline{f_3^\times(s,-\xi_2)}\rmd\xi_1\rmd\xi_2\rmd s,
    \end{gathered}
\end{equation}

\begin{equation}
    \begin{gathered}
        R^\eta(v)\coloneqq 2\sum_{l=1}^2\int_0^tv^\eta(s,\xi)\int_{\R^d}\int_{\R^d}\Im\l(\tilde\zeta_{l,+}^\eta(\xi_1,\xi_2,\xi)v^{\times,\overline\eta}(s,-\xi_1)v^{\times,\overline\eta}(s,-\xi_2)\r)\rmd\xi_1\rmd\xi_2\rmd s\\+2\sum_{l=3}^6\int_0^t\int_{\R^d}\int_{\R^d}\Im\l(v^{\times,\overline\eta}(s,\xi)\tilde\zeta_{l,+}^\eta(\xi_1,\xi_2,\xi)v{\times,\overline\eta}(s,-\xi_2)\r)v^\eta(s,\xi_1)\rmd\xi_1\rmd\xi_2\rmd s
    \end{gathered}
\end{equation}
and
\begin{equation}
    \begin{gathered}
        R^\times(v)\coloneqq-\rmi\sum_{l=1}^2\int_0^tv^\times(s,\xi)\int_{\R^d}\int_{\R^d}\l(\tilde\zeta_{l,+}^0(\xi_1,\xi_2,\xi)-\overline{\tilde\zeta_{l,+}^1(\xi_1,\xi_2,\xi)}\r)\overline{v^\times(s,-\xi_1)}\overline{v^\times(s,-\xi_2)}\rmd\xi_1\rmd\xi_2\rmd s\\
        -\rmi\sum_{l=3}^6\int_0^t\int_{\R^d}\int_{\R^d}\l(v^1(s,\xi)\tilde\zeta_{l,+}^0(\xi_1,\xi_2,\xi)v^0(s,\xi_1)-v^0(s,\xi)\overline{\tilde\zeta_{l,+}^1(\xi_1,\xi_2,\xi)}v^1(s,\xi_1)\r)\overline{v^\times(s,-\xi_2)}\rmd\xi_1\rmd\xi_2\rmd s.
    \end{gathered}
\end{equation}
Using \cref{another trivial bound},

\begin{equation}\label{1. relevant estimate}
    \begin{gathered}
        \norm{\Sigma^\eta}_{L^\infty\l(W^{1,\infty}\cap W^{1,1}\r)}+\norm{\Sigma^\times}_{L^\infty\l(W^{1,\infty}\cap W^{1,1}\r)}\\\leq \delta\max_\eta\norm{Q_+^\eta}_{W^{1,\infty}\cap W^{1,1}}\sum_{\substack{0\leq n_1,n_2,n_3\leq N\\ N\leq n_1+n_2+n_3\leq3N}}\Big[2\norm{\rho^\eta_{n_3}}_{L^{\infty}\l(W^{1,\infty}\cap W^{1,1}\r)}\norm{\rho_{n_1}^\times}_{L^1}\norm{\rho_{n_2}^\times}_{L^1}\\+4\norm{\rho^\times_{n_3}}_{L^\infty\l(W^{1,\infty}\cap W^{1,1}\r)}\norm{\rho^\times_{n_1}}_{L^1}\norm{\rho^\eta_{n_2}}_{L^1}+4\norm{\rho^\times_{n_3}}_{L^\infty\l(W^{1,\infty}\cap W^{1,1}\r)}\norm{\rho^\times_{n_1}}_{L^1}\norm{\rho^\times_{n_2}}_{L^1}\\+4\norm{\rho^1_{n_3}}_{L^\infty\l(W^{1,\infty}\cap W^{1,1}\r)}\norm{\rho^0_{n_1}}_{L^1}\norm{\rho^\times_{n_2}}_{L^1}+4\norm{\rho^0_{n_3}}_{L^\infty\l(W^{1,\infty}\cap W^{1,1}\r)}\norm{\rho^1_{n_1}}_{L^1}\norm{\rho^\times_{n_2}}_{L^1}\Big]\\\leq 14\Lambda\delta\max_\eta\norm{Q_+^\eta}_{W^{1,\infty}\cap W^{1,1}}\sum_{\substack{0\leq n_1,n_2,n_3\leq N\\ N\leq n_1+n_2+n_3\leq 3N}}\l(\Lambda\delta\r)^{n_1+n_2+n_3}\\\leq\Lambda\delta(\Lambda\delta)^N\sum_{\substack{0\leq n_1,n_2,n_3\leq N\\0\leq n_1+n_2+n_3\leq 2N}}(\Lambda\delta)^{n_1+n_2+n_3}\\\leq(\Lambda\delta)^{N+1}\l(\sum_{n\geq0}(\Lambda\delta)^n\r)^3\leq(\Lambda\delta)^{N+1}
    \end{gathered}
\end{equation}
by increasing $\Lambda$ and simultaneously decreasing $\delta$ in every step necessary. 
Similarly, we obtain 

\begin{equation}\label{invertibility bound}
    \begin{gathered}
        \norm{\Leaf^\eta(v)}_{L^\infty\l(W^{1,\infty}\cap W^{1,1}\r)}+\norm{\Leaf^\times(v)}_{L^\infty\l(W^{1,\infty}\cap W^{1,1}\r)}\\\leq 72\max_\eta\norm{Q_+^\eta}_{W^{1,\infty}\cap W^{1,1}}\Lambda\delta\l(\sum_{n=0}^N(\Lambda\delta)^{n}\r)^2\max_\star\norm{v^\star}_{L^\infty\l(W^{1,\infty}\cap W^{1,1}\r)}\\\leq\Lambda\delta\norm{v}_{\l(L^\infty\l(W^{1,\infty}\cap W^{1,1}\r)\r)^3}.
    \end{gathered}
\end{equation}
The quadratic term is estimated similarly by 

\begin{equation}\label{thrid relevant estimate}
    \begin{gathered}
        \norm{Q(v,v)^\eta}_{L^\infty\l(W^{1,\infty}\cap W^{1,1}\r)}+\norm{Q(v,v)^\times}_{L^\infty\l(W^{1,\infty}\cap W^{1,1}\r)}\\
        72\max_\eta\norm{Q_+^\eta}_{W^{1,\infty}\cap W^{1,1}}\Lambda\delta\l[\sum_{n=0}(\Lambda\delta)^n\r]\norm{v}_{\l(L^\infty\l(W^{1,\infty}\cap W^{1,1}\r)\r)^3}^2
    \end{gathered}
\end{equation}
and finally, the cubic terms deliver 

\begin{equation}\label{fourth relevant estimate}
    \begin{gathered}
        \norm{R^\eta(v,v,v)}_{L^\infty\l(W^{1,\infty}\cap W^{1,1}\r)}+\norm{R^\times(v,v,v)}_{L^\infty\l(W^{1,\infty}\cap W^{1,1}\r)}\\
        \leq24\max_\eta\norm{Q_+^\eta}_{W^{1,\infty}\cap W^{1,1}}\delta\norm{v}_{\l(L^\infty\l(W^{1,\infty}\cap W^{1,1}\r)\r)^3}^3.
    \end{gathered}
\end{equation}
Upon choosing $\delta$ small enough, we find with \cref{invertibility bound} that 
\begin{equation}
    \norm{\Leaf}_{\l(L^\infty\l(W^{1,\infty}\cap W^{1,1}\r)\r)^3\to\l(L^\infty\l(W^{1,\infty}\cap W^{1,1}\r)\r)^3}\leq\frac{1}{3}
\end{equation}
so that $\Id-\Leaf$ is invertible using the Neumann series of inversion, and so that 
\begin{equation}\label{fifth relevant estimate}
    \norm{\l(\Id-\Leaf\r)^{-1}}_{\mathrm{op}}\leq\frac{1}{1-\norm{\Leaf}_{\mathrm{op}}}\leq\frac{3}{2}\leq\Lambda. 
\end{equation}
We just proved that the following reformulation of the fixed-point equation 

\begin{equation}
    v = \l(\Id-\Leaf\r)^{-1}\l(\Sigma+Q(v,v)+R(v,v,v)\r)
\end{equation}
is well-formulated.
\begin{proposition}\label{having a solution to the resonant system as a fixed point solution}
    The operator
    \begin{equation}
        \l(L^\infty\l(W^{1,\infty}\cap W^{1,1}\r)\r)^3\ni v\mapsto\l(\Id-\Leaf\r)^{-1}\l(\Sigma+Q(v,v)+R(v,v,v)\r)\in\l(L^\infty\l(W^{1,\infty}\cap W^{1,1}\r)\r)^3
    \end{equation}
    restricts to a contraction on the closed ball
    \begin{equation}
        \l\{v\in\l(L^\infty\l(W^{1,\infty}\cap W^{1,1}\r)\r)^3\bigm\vert\norm{v}_{\l(L^\infty\l(W^{1,\infty}\cap W^{1,1}\r)\r)^3}\leq(\Lambda\delta)^N\r\}.
    \end{equation}
\end{proposition}
\begin{proof}
    The fact that this operator maps the ball onto itself is easily proven from the above estimates \eqref{invertibility bound}, \eqref{thrid relevant estimate}, \eqref{fourth relevant estimate}, and \eqref{fifth relevant estimate}. The contractibility property is proven by simply calculating 
    \begin{equation}
        \begin{gathered}
        \norm{\l(\Id-\Leaf\r)^{-1}\l(\Sigma+Q(v,v)+R(v,v,v)\r)+\l(\Id-\Leaf\r)^{-1}\l(\Sigma+Q(w,w)+R(w,w,w)\r)}_{\l(L^\infty\l(W^{1,\infty}\cap W^{1,1}\r)\r)^3}\\\leq\Lambda\l(\norm{Q(v,v)-Q(w,w)}_{\l(L^\infty\l(W^{1,\infty}\cap W^{1,1}\r)\r)^3}+\norm{R(v,v,v)-R(w,w,w)}_{\l(L^\infty\l(W^{1,\infty}\cap W^{1,1}\r)\r)^3}\r).
    \end{gathered}
    \end{equation}
    One may bound 
    \begin{equation}
        \begin{gathered}
            \norm{Q(v,v)-Q(w,w)}_{\l(L^\infty\l(W^{1,\infty}\cap W^{1,1}\r)\r)^3}\leq\Lambda\delta\sum_{n=0}^N(\Lambda\delta)^n(\Lambda\delta)^N\norm{v-w}_{\l(L^\infty\l(W^{1,\infty}\cap W^{1,1}\r)\r)^3}\\\leq(\Lambda\delta)^{N+1}\norm{v-w}_{\l(L^\infty\l(W^{1,\infty}\cap W^{1,1}\r)\r)^3}
        \end{gathered}
    \end{equation}
    and 
    \begin{equation}
        \begin{gathered}
            \norm{R(v,v,v)-R(w,w,w)}_{\l(L^\infty\l(W^{1,\infty}\cap W^{1,1}\r)\r)^3}\leq\Lambda\delta(\Lambda\delta)^{2N}\norm{v-w}_{\l(L^\infty\l(W^{1,\infty}\cap W^{1,1}\r)\r)^3}\\(\Lambda\delta)^{2N+1}\norm{v-w}_{\l(L^\infty\l(W^{1,\infty}\cap W^{1,1}\r)\r)^3},
        \end{gathered}
    \end{equation}
    so that
    \begin{equation}
        \begin{gathered}
        \norm{\l(\Id-\Leaf\r)^{-1}\l(\Sigma+Q(v,v)+R(v,v,v)\r)+\l(\Id-\Leaf\r)^{-1}\l(\Sigma+Q(w,w)+R(w,w,w)\r)}_{\l(L^\infty\l(W^{1,\infty}\cap W^{1,1}\r)\r)^3}\\\leq\Lambda(\Lambda\delta)^{N+1}\l(1+(\Lambda\delta)^N\r)\norm{v-w}_{\l(L^\infty\l(W^{1,\infty}\cap W^{1,1}\r)\r)^3}
    \end{gathered}
    \end{equation}
    and for $\delta$ small enough and $L$ large enough, we find $\Lambda(\Lambda\delta)^{N(L)+1}\l(1+(\Lambda\delta)^{N(L)}\r)<1$ and the contraction property follows. 
\end{proof}

\section{Third part of the proof of \cref{this is the main theorem}}\label{fifth section}

\subsection{On resonant nodes}

\begin{lemma}\label{properties of resonant nodes lemma}
    If $\node\in\Node(C)$ is resonant, then If $\node\in\Node(C)$ is resonant, then it can only be $1,4$- or $5$-resonant. If furthermore $\node\in\Node(C)$ is such that $\l\{\node'\leq\node\mid\node'\in\Node(C)\r\}$ only contains resonant nodes and if we denote by $C(\node)=\l\{\node_1,\ldots,\node_5\r\}$ the children of $\node$ (as in \cref{definition of being a resonant node}), then there exists $\Leaf(C)\ni\leaf<\node$ such that $K_C(\node)=K_C\l(\node_5\r)=K_C(\leaf)$, and $\l(K_C(\node_i),K_C(\node_5)\r)$ are linearly independent for all $i\in\llbracket1,4\rrbracket$.
  
\end{lemma}

\begin{proof}
    The first statement is clear from the nonlinearity of the quintic Schrödinger equation: If node $\node$ has sign and colour $(\iota,\eta)$, then its children have sign and colour $(\iota,\eta)$, $(-\iota,\eta)$, $(-\iota,\eta)$, $(\iota,\overline\eta)$ and $(\iota,\overline\eta)$.

    We prove the first assertion first by induction on the height of $\node$. For $h(\node)=1$, the statement is clear as the children of $\node$ are all leaves. Now suppose $h(\node)>1$. Now, if $\node_5$ was a leaf, we are done and if not, then it is resonant, and the induction assumption implies that there exists $\Leaf(C)\ni\leaf<\node_5$ such that $K_C(\node_5)=K_C(\leaf)$.

    We may now use the first statement to prove the second. There exist $\leaf_i\in\Leaf(C)$ such that $K_C(\node_i)=K_C(\leaf_i)$ and thus, $K_C(\leaf_1)+K_C(\leaf_2)=0$, $K_C(\leaf_3)+K_C(\leaf_4)=0$ and $K_C(\node)=K_C(\leaf_5)$. This implies $\sigma(\leaf_1)=\leaf_2$ and $\sigma(\leaf_3)=\leaf_4$. Assuming the contrary of being linearly independent is equivalent to assuming $\sigma(\leaf_5)\in\{\leaf_1,\leaf_2\}\sqcup\{\leaf_3,\leaf_4\}$. In that case, we would be led to the contradictory statement $K_C(\leaf_i)=0$ for some $i\in\llbracket1,4\rrbracket$.
\end{proof}

\begin{lemma} 

    Let $C$ be a couple and suppose $\node\in\Node(C)$ is such that all elements of $\Node_\node(C)\coloneqq\{\node'\leq\node\mid\node'\in\Node(C)\}$ are resonant nodes. Then, $\Omega_{\node'}=0$ for all $\node'\in\Node_\node(C)$ and furthermore 
    \begin{equation}
        \int_{I_\node(t)}\prod_{\node'\in\Node_\node(C)}e^{\rmi\epsilon^{-1}\iota_{\node'}\Omega_{\node'}t_{\node'}}\rmd t_{\node'} = \abs{\M(\Node_\node(C))}\frac{t^{\abs{\Node_\node(C)}}}{\abs{\Node_\node(C)}!},
    \end{equation}
    where 
    \begin{equation}
        I_\node(t)\coloneqq\l\{\l(t_{\Node(C)\ni\node'\leq\node}\r)\in[0,t]^{\abs{\Node_\node(C)}}\bigm\vert\node_1<\node_2\Rightarrow t_{\node_1}<t_{\node_2}\r\}.
    \end{equation}
\end{lemma}
\begin{proof}
    The assertion $\Omega_{\node'}=0$ for all $\node'\in\Node_\node(C)$ follows simply from the definition of being a resonant node and thus 
    \begin{equation}
        \int_{I_\node(t)}\prod_{\node'\in\Node_\node(C)}e^{\rmi\iota\epsilon^{-1}\Omega_{\node'}t_{\node'}}\rmd t_{\node'} =\sum_{\rho\in\M(\Node_\node(C))}\int_{0\leq t_{\rho(1)}<\cdots<t_{\rho(n_\node)}\leq t}\rmd\boldsymbol{t} = \abs{\M(\Node_\node(C))}\frac{t^{n_\node}}{n_\node!},
    \end{equation}
    where we simply denoted $n_\node\coloneqq\abs{\Node_\node(C)}$.
\end{proof}

\begin{lemma}\label{ein kekiges lemma}
    Let $C$ be a couple and $\node\in\Node(C)$ be such that $\Node(C)\ni\node'<\node$ implies $\node'$ is resonant. We have 
    \begin{gather}
        \abs{\int_{I_\node(t)}\prod_{\node'\leq\node}e^{\rmi\epsilon^{-1}\Omega_{\node'}t_{\node'}}\rmd t_{\node'}}\leq\begin{cases}
            \frac{t^{n_\node}}{n_\node}&\text{if }\abs{\Omega_\node}\leq3\epsilon\delta^{-1},\\
            3t^{n_\node-1}\frac{\epsilon}{\abs{\Omega_\node}}&\text{otherwise.}
        \end{cases}
    \end{gather}
\end{lemma}
\begin{proof}
See Lemma 5.3 in \cite{desuzzoni2025waveturbulencesemilinearkleingordon}.    
\end{proof}

\begin{lemma}\label{The lemma that I gotta apply now}
    Let $C\in\mathcal C_{n_1,n_2}$ be such that there exists $\node\in\Node(C)$ with the property that if $\Node(C)\ni\node'<\node$, then $\node'$ is resonant. Then we have for all $t\in[0,\delta]$,
    \begin{equation}
        \abs{\int_{I_C(t)}\prod_{\node'\in\Node(C)}e^{\rmi\epsilon^{-1}\iota_{\node'}\Omega_{\node'}t_{\node'}\rmd t_{\node'}}}\leq\begin{cases}
            {t^{n(C)}}&\text{if }\abs{\Omega_\node}\leq3\epsilon\delta^{-1},\\
            3t^{n(C)-1}\frac{\epsilon}{\abs{\Omega_\node}}&\text{otherwise.}
        \end{cases}
    \end{equation}
\end{lemma}
\begin{proof}
    This is a minor adaptation of Lemma 5.4 in \cite{desuzzoni2025waveturbulencesemilinearkleingordon}.
\end{proof}

\subsection{On non-resonant nodes}

\begin{theorem} \label{theorem: every non-resonant couple contributes with decay in box size}

    There exist $\Lambda,\alpha>0$ such that for all couples $C$ that contain at least one non-resonant node, we have 
    \begin{equation}
        \sup_{t\in[0,\lambda]}\abs{{\mathcal J_C}\l(\epsilon^{-1}t,k\r)}\leq\Lambda^{n(C)+1}\epsilon^\alpha
    \end{equation}
    for all $k\in\Z_L^d$.
\end{theorem}
\begin{proof}
    We suppose that $\node\in\Node(C)$ is non-resonant and suppose without loss of generality that $\node\in\Node(A)$ and that for all $\Node(A)\ni\node'<\node$, we have that $\node'$ is resonant.
    We apply \cref{The lemma that I gotta apply now} and also use $\min\l(1,\frac{a}{b}\r)\leq\frac{\sqrt{2}a}{\abs{a+\rmi b}}$ for all $a\geq0$ and $b>0$ applied to $a\coloneqq3\delta^{-1}$ and $b\coloneqq \epsilon^{-1}\abs{\Omega_\node}$ so that 

    \begin{gather}
        \abs{\int_{I_C(t)}\prod_{\node'\in\Node(C)}e^{\rmi\iota\epsilon^{-1}\Omega_{\node'}t_{\node'}}\rmd t_{\node'}}\leq\frac{3\sqrt{2}\delta^{n(C)-1}}{\abs{3\delta^{-1}+\rmi\epsilon^{-1}\Omega_\node}}.
    \end{gather}
    With the usual bounds, we thus get 
    \begin{equation}
    \begin{gathered}
        \abs{{\mathcal J_C}\l(\epsilon^{-1}t,k\r)}\leq\Lambda^{2n(C)+1}\delta^{n(C)-1}L^{-2n(C)d}\sum_{\substack{\kappa\in\D_k(C)\\\kappa\l(\Leaf(C)_+\r)\subseteq B_R(0)}}\frac{1}{\abs{3\delta^{-1}+\rmi\epsilon^{-1}\Omega_\node}},
    \end{gathered}
    \end{equation}
    where we have used the boundedness of $\abs{Q_{\node'}}$.
    Denote $C(\node)=\l\{\node_1,\ldots,\node_5\r\}$ and recall that the set $C(\node)\cap\Node(C)$ contains only resonant nodes or leaves. \Cref{properties of resonant nodes lemma} implies there exist $\leaf_1,\ldots,\leaf_5\in\Leaf(C)$ such that $K_C(\node_i)=K_C(\leaf_i)$ for all $i$. Since $\node$ is not resonant, we may assume that $K_C(\node_1)+K_C(\node_2)$ does not vanish identically. We may complete $\l(K_C(\node_1),K_C(\node_2)\r)$ into a basis of $V(C)$ and obtain 
        \begin{equation}
            \abs{{\mathcal J_C}\l(\epsilon^{-1}t,k\r)}\leq\Lambda\Lambda^{n(C)}\delta^{n(C)-1}L^{-2d}\sum_{k_1,k_2\in B_R^{\Z_L^d}(0)}\frac{1}{\abs{3\delta^{-1}+\rmi\epsilon^{-1}\l(\abs{\sum_{i=1}^2k_i}^2-\sum_{i=1}^2\iota_{\node_i}\abs{k_i}^2\r)}}.
        \end{equation}
        Using \cref{main technical necessary theorem} and $n(C)\leq\abs{\ln\epsilon}$, we find $\alpha>0$ such that 
        \begin{equation}
            \sum_{k_1,k_2,k_3\in B_R^{\Z_L^d}(0)}\frac{1}{\abs{3\delta^{-1}+\rmi\epsilon^{-1}\l(\abs{\sum_{i=1}^3k_i}^2-\sum_{i=1}^3\iota_{\node_i}\abs{k_i}^2\r)}}\leq\Lambda L^{3d}\epsilon^\alpha
        \end{equation}
        so that indeed 
        \begin{equation}
            \abs{{\mathcal J_C}\l(\epsilon^{-1}t,k\r)}\leq\Lambda\Lambda^{n(C)}\delta^{n(C)-1}\epsilon^\alpha
        \end{equation}
        and this gives the desired estimate by increasing $\Lambda$ and decreasing $\delta$.
\end{proof}

\subsection{Types of resonant couples}
We denote $\res_{n_1+n_2}(\iota,\iota',\eta,\eta')\subseteq\mathcal C_{n_1,n_2}^{\iota,\iota',\eta,\eta'}$ as the subset of signed and coloured couples of scale $n_1+n_2$ whose nodes are all resonant.

Using that 
\begin{equation}
\abs{I_C(t)}=\abs{\M(\Node(C))}\frac{t^n}{n!},
\end{equation}
we find for any $C\in\res_n(\iota,\iota',\eta,\eta')$,

\begin{equation}
    \J_C\l(\epsilon^{-1}t,k\r) = \abs{\M(\Node(C))}\frac{t^n}{n!}(-\rmi)^nL^{-2dn}\sum_{\kappa\in\D_k(C)}\prod_{\node\in\Node(C)}\iota_\node Q_\node\prod_{\leaf\in\Leaf(C)_+}M^{\eta_\leaf,\eta_{\sigma(\leaf)}}(\kappa(\leaf)).
\end{equation}

\begin{lemma}\label{all types of resonant nodes}
    Let $n>0$ and $C\in\res_n\l(\iota,\iota',\eta,\eta'\r)$. There exist couples  $C_i\in\res_{n_i}\l(\iota_i,\iota_i',\eta_i,\eta_i'\r)$, $i=1,2,3$, such that if we write $C_i\eqqcolon\l(A_i,A_i',\sigma_i\r)$ and define $\sigma\colon\Leaf(C_1)\sqcup\Leaf(C_2)\sqcup\Leaf(C_3)\to\Leaf(C_1)\sqcup\Leaf(C_2)\sqcup\Leaf(C_3)$ by $\left.\sigma\right|_{\Leaf(C_i)}\coloneqq\sigma_i$ then there can only be the following configurations (see \cref{all possible resonant roots} for an illustration of all possible resonant roots).

    \begin{equation}\label{need to code this}
        \begin{aligned}
        C&=\l(\bullet\l(((A_3,\iota,\eta),(A_1',-\iota,\eta),(A_2',-\iota,\eta),(A_1,\iota,\overline\eta),(A_2,\iota,\overline\eta)),\iota,\eta\r),(A_3',-\iota,\eta'),\sigma\r),\\
        C&=\l(\bullet\l(((A_3,\iota,\eta),(A_2',-\iota,\eta),(A_1',-\iota,\eta),(A_1,\iota,\overline\eta),(A_2,\iota,\overline\eta)),\iota,\eta\r),(A_3',-\iota,\eta'),\sigma\r),\\
        C&=\l(\bullet\l(((A_1,\iota,\eta),(A_1',-\iota,\eta),(A_2',-\iota,\eta),(A_3,\iota,\overline\eta),(A_2,\iota,\overline\eta)),\iota,\eta\r),(A_3',-\iota,\eta'),\sigma\r),\\
        C&=\l(\bullet\l(((A_1,\iota,\eta),(A_2',-\iota,\eta),(A_1',-\iota,\eta),(A_3,\iota,\overline\eta),(A_2,\iota,\overline\eta)),\iota,\eta\r),(A_3',-\iota,\eta'),\sigma\r),\\
        C&=\l(\bullet\l(((A_1,\iota,\eta),(A_1',-\iota,\eta),(A_2',-\iota,\eta),(A_2,\iota,\overline\eta),(A_3,\iota,\overline\eta)),\iota,\eta\r),(A_3',-\iota,\eta'),\sigma\r),\\
        C&=\l(\bullet\l(((A_1,\iota,\eta),(A_2',-\iota,\eta),(A_1',-\iota,\eta),(A_2,\iota,\overline\eta),(A_3,\iota,\overline\eta)),\iota,\eta\r),(A_3',-\iota,\eta'),\sigma\r),\\
        C&=\l((A_3',\iota,\eta),\bullet\l(((A_3,-\iota,\eta'),(A_1',\iota,\eta'),(A_2',\iota,\eta'),(A_1,-\iota,\overline\eta'),(A_2,-\iota,\overline\eta')),-\iota,\eta'\r),\sigma\r),\\
        C&=\l((A_3',\iota,\eta),\bullet\l(((A_3,-\iota,\eta'),(A_2',\iota,\eta'),(A_1',\iota,\eta'),(A_1,-\iota,\overline\eta'),(A_2,-\iota,\overline\eta')),-\iota,\eta'\r),\sigma\r),\\
        C&=\l((A_3',\iota,\eta),\bullet\l(((A_1,-\iota,\eta'),(A_1',\iota,\eta'),(A_2',\iota,\eta'),(A_3,-\iota,\overline\eta'),(A_2,-\iota,\overline\eta')),-\iota,\eta'\r),\sigma\r),\\
        C&=\l((A_3',\iota,\eta),\bullet\l(((A_1,-\iota,\eta'),(A_2',\iota,\eta'),(A_1',\iota,\eta'),(A_3,-\iota,\overline\eta'),(A_2,-\iota,\overline\eta')),-\iota,\eta'\r),\sigma\r),\\
        C&=\l((A_3',\iota,\eta),\bullet\l(((A_1,-\iota,\eta'),(A_1',\iota,\eta'),(A_2',\iota,\eta'),(A_2,-\iota,\overline\eta'),(A_3,-\iota,\overline\eta')),-\iota,\eta'\r),\sigma\r),\\
        C&=\l((A_3',\iota,\eta),\bullet\l(((A_1,-\iota,\eta'),(A_2',\iota,\eta'),(A_1',\iota,\eta'),(A_2,-\iota,\overline\eta'),(A_3,-\iota,\overline\eta')),-\iota,\eta'\r),\sigma\r)
    \end{aligned}
    \end{equation}
    If we are in any of the first six cases, we say that the \textnormal{resonant branching} occurs at the left tree of the couple $C$; otherwise we say that the resonant branching occurs at the right tree. 
    In the notation, we understand $\l(A_i,A_i',\sigma_i\r)$ as sub-couples with $\sigma_i\coloneqq\left.\sigma\right|_{\Leaf\l(A_i\r)\sqcup\Leaf\l(A_i'\r)}$. Furthermore, $\res_n(\iota,\iota,\eta,\eta')=\emptyset$. 
\end{lemma}

\begin{proof}
Since $n>0$, $\mathcal R(C)\cap\Node(C)\neq\emptyset$. According to \cref{properties of resonant nodes lemma}, each node can only be a $1,4$ or $5$-resonant node, and in each case, a sub-tree of the left or right tree may couple to the right or left tree of $C$. That is, if an element of $\mathcal R(C)\cap\Node(C)$ is an $i$-resonant node (for $i\in\{1,4,5\}$), then there are four possibilities. The resonant branching might happen at the left or right tree of $C$, and wherever it happens, the remaining sup-couples have two ways to couple with each other (see \cref{all possible resonant roots}). The first six cases of \cref{need to code this} are all the possibilities in the situation when the resonant branching happens at the left tree of $C$ and the remaining ones when the resonant branching happens at the right tree of $C$. The first, second, seventh, and eighth line of \cref{need to code this} describe $1$-resonant roots. The third, fourth, ninth, and tenth lines describe $4$-resonant roots, and the remaining four lines describe $5$-resonant roots. 

Finally suppose $\res_n(\iota,\iota',\eta,\eta')\neq\emptyset$. If $n=0$ then obviously $\iota=-\iota'$. Now assume $n>0$ and $C\in\res_n(\iota,\iota',\eta,\eta')$. The previous claim implies that $C_1\in\res_{n_1+n_1'}(\iota,-\iota,\overline\eta,\eta)$, $C_2\in\res_{n_2+n_2'}(\iota,-\iota,\overline\eta,\eta)$ and $C_3\in\res_{n_3+n_3'}(\iota,\iota',\eta,\eta')$ and since $n_3+n_3'<n$, we have $\iota'=-\iota$ which completes the induction step. 
\end{proof}

\begin{figure}[H]
    \centering
    \includegraphics[width=0.8\linewidth]{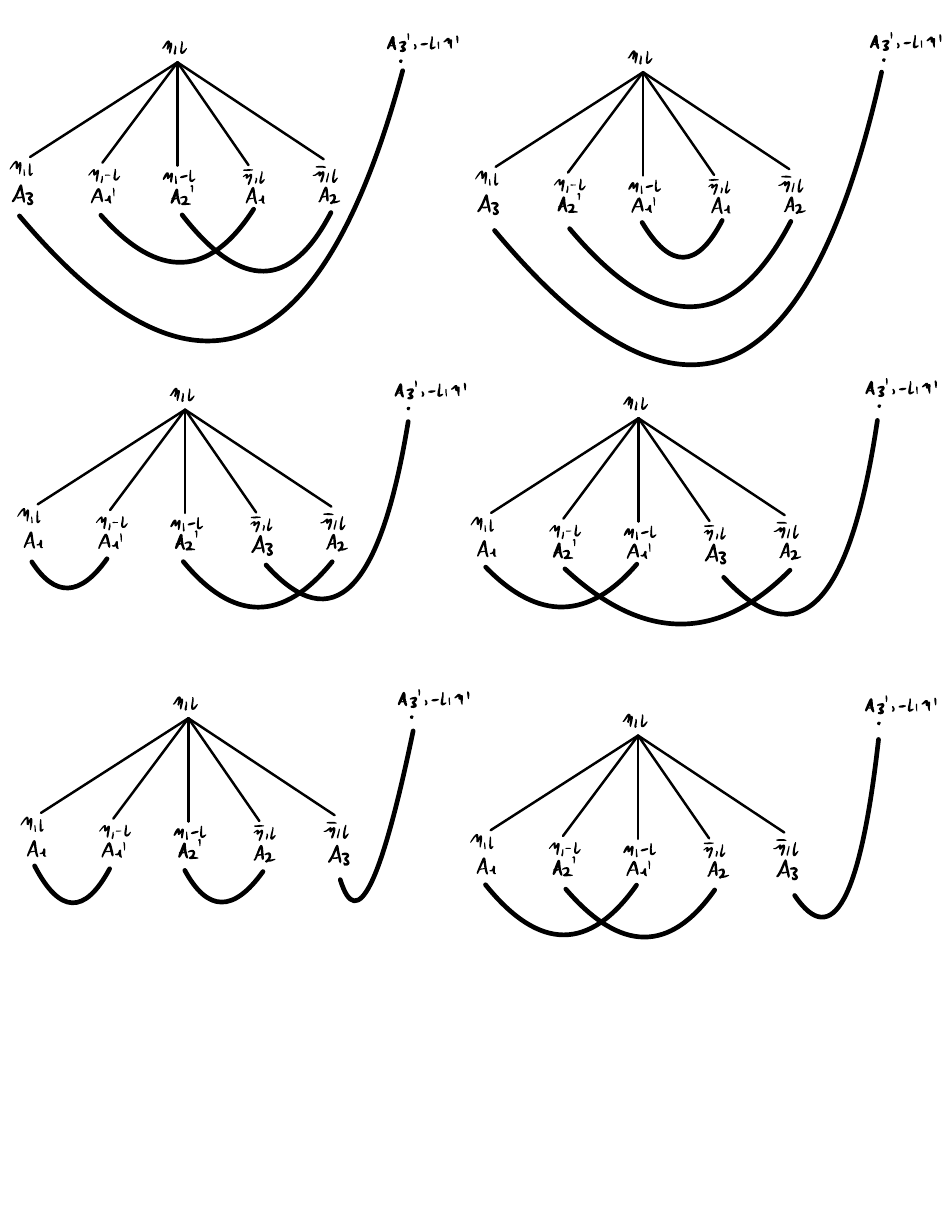}
    \vspace{-2cm}
    \caption{These are the first six cases of \eqref{need to code this}, where the resonant branching occurs for the left tree. The remaining six cases can be drawn by switching the left with the right tree and exchanging the sign $\iota\leftrightarrow-\iota$ and the colour $\eta\leftrightarrow\eta'$.}
    \label{all possible resonant roots}
\end{figure}
From now on, we write $\res_n(\iota,\eta,\eta')=\res_n(\iota,-\iota,\eta,\eta')$
We label each case of \cref{need to code this} with a code by looking at the indices of the subtrees. The first six couples are encoded by $\l(3,1',2',1,2\r)$, $\l(3,2',1',1,2\r)$, $\l(1,1',2',3,2\r)$, $\l(1,2',1',3,2\r)$, $\l(1,1',2',2,3\r)$ and $\l(1,2',1',2,3\r)$. Using the index set $\{l,r\}$ to differentiate whether the resonant branching happens in the left ($l$) or right ($r$) tree, one can uniquely name all possibilities in \cref{need to code this}.

\begin{lemma}
    For any resonant couple (see \cref{all types of resonant nodes} for all the cases), there exists a canonical bijection 
    \begin{equation}
        \D_k(C)\cong\l\{\l(k_1,k_2,\kappa_1,\kappa_2,\kappa_3\r)\mid(k_1,k_2)\in\Z_L^d\times\Z_L^d,\ (\kappa_1,\kappa_2,\kappa_3)\in\D_{-k_1}(C_1)\times\D_{-k_2}(C_2)\times\D_k(C_3)\r\}\eqqcolon\tilde\D_k(C).
    \end{equation}
    Furthermore, 
    \begin{equation}
        \J_C\l(\epsilon^{-1}t,k\r)=\abs{\M(\Node(C))}\frac{t^n}{n!}\tilde\J_C(k),
    \end{equation}
    where \begin{equation}\label{the recursive definition}
        \tilde\J_C(k)\coloneqq (-\rmi)^nL^{-2dn(C)}\sum_{\kappa\in\D_k(C)}\prod_{\node\in\Node(C)}\iota_\node Q_\node\prod_{\leaf\in\Leaf(C)_+}M^{\eta_\leaf,\eta_{\sigma(\leaf)}}(\kappa(\leaf))
    \end{equation}
    can be recursively written as 
    \begin{equation}
        \tilde\J_C(k)=-\rmi\iota_{\root_A} L^{-2d}\l[\sum_{k_1,k_2\in\Z_L^d}Q_{\root_A}\tilde\J_{C_1}(-k_1)\tilde\J_{C_2}(-k_2)\r]\tilde\J_{C_3}(k).
    \end{equation}
\end{lemma}
\begin{proof}
       We define the map $\psi\colon\D_k(C)\to\tilde\D_k(C)$ by 
       \begin{equation}
           \psi(\kappa)\coloneqq\l(K_C(r_{A_1})(\kappa),K_C(r_{A_2})(\kappa),\left.\kappa\right|_{\Leaf(C_1)_+},\left.\kappa\right|_{\Leaf(C_2)_+},\left.\kappa\right|_{\Leaf(C_3)_+}\r).
       \end{equation}
       Since $\Leaf(C_1)_+\sqcup\Leaf(C_2)_+\sqcup\Leaf(C_3)_+=\Leaf(C)_+$, the map $\psi$ is a well-defined injective map. Now take $(k_1,k_2,\kappa_1,\kappa_2,\kappa_3)\in\tilde\D_k(C)$ and define $\kappa\in\l(\R^d\r)^{\Leaf(C)_+}$ by 
       \begin{equation}\label{eqn that makes the bijection canonical}
           \kappa(\leaf)\coloneqq\begin{cases}
               \kappa_1(\leaf)&\text{if }\leaf\in\Leaf(C_1),\\
               \kappa_2(\leaf)&\text{if }\leaf\in\Leaf(C_2),\\
               \kappa_3(\leaf)&\text{if }\leaf\in\Leaf(C_3).
           \end{cases}
       \end{equation}
       Thus, by definition, $\Im\kappa\subseteq\Z_L^d$ and $K_C(\root_A)(\kappa)=k_1-k_1+k_2-k_2+k=k$. This implies that indeed $\kappa\in\D_K(C)$ and $\psi(\kappa)=(k_1,k_2,\kappa_1,\kappa_2,\kappa_3)$ by construction of $\kappa$. \Cref{eqn that makes the bijection canonical} is the reason that makes $\psi$ canonical. 
       We can now use the bijection $\psi$ to rewrite 

       \begin{equation}\label{what concludes the proof}
           \begin{gathered}
           \sum_{\kappa\in\D_k(C)}\prod_{\node\in\Node(C)}\iota_\node Q_\node\prod_{\leaf\in\Leaf(C)_+}M^{\eta_\leaf,\eta_{\sigma(\leaf)}}(\kappa(\leaf)) \\= \iota_{\root_A}\sum_{k_1,k_2\in\Z_L^d}Q_{\root_A}\sum_{\kappa_1\in\D_{-k_1}(C_1)}\prod_{\node\in\Node(C_1)}\iota_\node Q_\node\prod_{\leaf\in\Leaf(C_1)_+}M^{\eta_\leaf,\eta_{\sigma(\leaf)}}(\kappa_1(\leaf))\\\cdot\sum_{\kappa_2\in\D_{-k_2}(C_2)}\prod_{\node\in\Node(C_2)}\iota_\node Q_\node\prod_{\leaf\in\Leaf(C_2)_+}M^{\eta_\leaf,\eta_{\sigma(\leaf)}}(\kappa_2(\leaf))\\\cdot\sum_{\kappa_3\in\D_{k}(C_3)}\prod_{\node\in\Node(C_3)}\iota_\node Q_\node\prod_{\leaf\in\Leaf(C_3)_+}M^{\eta_\leaf,\eta_{\sigma(\leaf)}}(\kappa_3(\leaf)).
       \end{gathered}
       \end{equation}
       Plugging \cref{what concludes the proof} into the definition \cref{the recursive definition} concludes the proof.
\end{proof}

\subsection{From resonant couples to ternary trees}\label{recursive structure for subcouples}

    We denote 
    \begin{equation}
        \begin{aligned}
        \res_n^{\mathrm{map}}(\iota,\eta,\eta')&\coloneqq\l\{(C,\rho)\mid C\in\res_n(\iota,\eta,\eta')\text{ and }\rho\colon\Node(C)\to\N^*\text{ strictly increasing and injective}\r\},\\
        \res_{n}^{\mathrm{map}}(\iota,\eta,\eta',M)&\coloneqq\{(C,\rho)\in\res_n^{\mathrm{map}}(\iota,\eta,\eta')\mid\max\Im\rho=M\},\\
        \res_n^{\mathrm{ord}}(\iota,\eta,\eta')&\coloneqq\l\{(C,\rho)\mid C\in\res_n(\iota,\eta,\eta')\text{ and }\rho\in\M(\Node(C))\r\},\\
        \res_{n}^{\mathrm{ord}}(\iota,\eta,\eta',M)&\coloneqq\{(C,\rho)\in\res_n^{\mathrm{ord}}(\iota,\eta,\eta')\mid\max\Im\rho=M\}.
    \end{aligned}
    \end{equation}

\begin{definition}

    We define our ternary trees. 
    \begin{gather}
        \tilde\G_0\coloneqq\{\perp\},\\
        \tilde \G_{n+1}\coloneqq\bigsqcup_{\substack{n_1+n_2+n_3=n\\i,j\in\{1,\ldots,6\}}}\l\{\bullet_{il}(G_1,G_2,G_3),\bullet_{jl}(G_1,G_2,G_3)\mid(G_1,G_2,G_3)\in\tilde\G_{n_1}\times\tilde\G_{n_2}\times\tilde\G_{n_3}\r\},
    \end{gather}
    where the nodes $\bullet_{il}$ indicate the first six possibilities in \cref{all types of resonant nodes} and $\bullet_{ir}$ indicate the remaining six. The subscripts $l$ and $r$ stand for left and right, respectively and whether the left or the right tree in a resonant couple branches into subtrees in the sense of \cref{all types of resonant nodes}. Denote
\begin{equation}
    \begin{aligned}
        \G_n^{\mathrm{map}}(\iota,(\eta,\eta'))&\coloneqq\l\{(G,\rho)\mid G\in\G_n(\iota,(\eta,\eta'))\text{ and }\rho\mid\Node(C)\to\N^*\text{ strictly increasing and injective}\r\},\\
        \G_n^{\mathrm{ord}}(\iota,(\eta,\eta'))&\coloneqq\l\{(G,\rho)\mid G\in\G_n(\iota,(\eta,\eta'))\text{ and }\rho\in\M(\Node(G))\r\},\\
        \G_{n}^{\mathrm{map}}(\iota,(\eta,\eta'),M)&\coloneqq\{(B,\rho)\in\G_n^{\mathrm{map}}(\iota,(\eta,\eta'))\mid\max\Im\rho=M\}.
    \end{aligned}
\end{equation}
\end{definition}
    
\begin{remark}
    Observe, $\res_n^{\mathrm{ord}}(\iota,\eta,\eta')=\res_n^{\mathrm{map}}(\iota,\eta,\eta',n)$ and $\G_n^{\mathrm{ord}}(\iota,(\eta,\eta'))=\G_n^{\mathrm{map}}(\iota,(\eta,\eta'),n)$.
\end{remark}

\begin{definition}\label{from resonant couples to ternary trees}
    Define the map $\Upsilon\colon\G_n^{\mathrm{map}}(\iota,(\eta,\eta'),M)\to\res_n^{\mathrm{map}}\l(\iota,\eta,\eta',M\r)$ recursively. If $n=0$ then $G\in\G_0^{\mathrm{map}}\l(\iota,\l(\eta,\eta'\r),M\r)$ takes the form $G=(\perp,\iota,(\eta,\eta'))$ and we set $\Upsilon(G)\coloneqq\l((\perp,\iota,\eta),\l(\perp,-\iota,\eta'\r),\sigma\r)$, where $\sigma$ is the obvious pairing. If $n>0$ then any $(G,\rho)\in\G_n^{\mathrm{map}}\l(\iota,\l(\eta,\eta'\r),M\r)$ can be uniquely decomposed as $G=(\bullet_\star(G_1,G_2,G_3),\iota,(\eta,\eta'))$, where $\l(G_i,\rho_i\r)\in\G_{n_i}^{\mathrm{map}}\l(\iota_i,\l(\eta_i,\eta'_i\r)\r)$ with $n_1+n_2+n_3=n-1$, $\star\in\{1l,\ldots,6l,1r,\ldots,6r\}$ and $\rho_i\coloneqq\left.\rho\right|_{\Node(G_i)}$. The sign and colour rules for our ternary trees are 
    \begin{align}
        \iota_i=\begin{cases}
            \iota&\text{if }\star\in\{1l,\ldots,6l\},\\
            -\iota&\text{if }\star\in\{1r,\ldots,6r\},
        \end{cases}\ 
        (\eta_i,\eta_i')=\begin{cases}
            (\overline\eta,\eta)&\text{if }i=1,2,\ \star\in\{1l,2l\},\\
            (\eta,\eta')&\text{if }i=3,\ \star\in\{1l,2l\},\\
            (\eta,\eta)&\text{if }i=1,\ \star\in\{3l,\ldots,6l\},\\
            (\overline\eta,\eta)&\text{if }i=2,\ \star\in\{3l,\ldots,6l\},\\
            (\overline\eta,\eta')&\text{if }i=3,\ \star\in\{3l,\ldots,6l\},\\
            (\overline\eta',\eta')&\text{if }i=1,2,\ \star\in\{1r,2r\},\\
            (\eta',\eta)&\text{if }i=3,\ \star\in\{1r,2r\},\\
            (\eta',\eta')&\text{if }i=1,\ \star\in\{3r,\ldots,6r\},\\
            (\overline\eta',\eta')&\text{if }i=2,\ \star\in\{3r,\ldots,6r\},\\
            (\overline\eta',\eta)&\text{if }i=3,\ \star\in\{3r,\ldots,6r\}
        \end{cases}\label{all sign and colour possibilities}
    \end{align}
    Setting $\rho_i\coloneqq\left.\rho\right|_{\Node(G_i)}$, we may assume that $(C_i,\tilde\rho_i)\coloneqq\Upsilon\l(G_i,\rho_i\r)\in\res_{n_i}\l(\iota_i,\eta_i,\eta'_i\r)$ (from left to right) are already defined and set $C_i=\l(A_i,A_i',\sigma_i\r)$ so that we may define
    \begin{equation}
        \Upsilon(G,\rho)\coloneqq\begin{cases}
            \l(\l(\bullet\l(((A_3,\iota,\eta),(A_1',-\iota,\eta),(A_2',-\iota,\eta),(A_1,\iota,\overline\eta),(A_2,\iota,\overline\eta)),\iota,\eta\r),(A_3',-\iota,\eta'),\sigma\r),\tilde\rho\r)&\text{if }\star=1l,\label{the first set}\\
        \l(\l(\bullet\l(((A_3,\iota,\eta),(A_2',-\iota,\eta),(A_1',-\iota,\eta),(A_1,\iota,\overline\eta),(A_2,\iota,\overline\eta)),\iota,\eta\r),(A_3',-\iota,\eta'),\sigma\r),\tilde\rho\r)&\text{if }\star=2l,\\
        \l(\l(\bullet\l(((A_1,\iota,\eta),(A_1',-\iota,\eta),(A_2',-\iota,\eta),(A_3,\iota,\overline\eta),(A_2,\iota,\overline\eta)),\iota,\eta\r),(A_3',-\iota,\eta'),\sigma\r),\tilde\rho\r)&\text{if }\star=3l,\\
        \l(\l(\bullet\l(((A_1,\iota,\eta),(A_2',-\iota,\eta),(A_1',-\iota,\eta),(A_3,\iota,\overline\eta),(A_2,\iota,\overline\eta)),\iota,\eta\r),(A_3',-\iota,\eta'),\sigma\r),\tilde\rho\r)&\text{if }\star=4l,\\
        \l(\l(\bullet\l(((A_1,\iota,\eta),(A_1',-\iota,\eta),(A_2',-\iota,\eta),(A_2,\iota,\overline\eta),(A_3,\iota,\overline\eta)),\iota,\eta\r),(A_3',-\iota,\eta'),\sigma\r),\tilde\rho\r)&\text{if }\star=5l,\\
        \l(\l(\bullet\l(((A_1,\iota,\eta),(A_2',-\iota,\eta),(A_1',-\iota,\eta),(A_2,\iota,\overline\eta),(A_3,\iota,\overline\eta)),\iota,\eta\r),(A_3',-\iota,\eta'),\sigma\r),\tilde\rho\r)&\text{if }\star=6l,\\
        \l(\l((A_3',\iota,\eta),\bullet\l(((A_3,-\iota,\eta'),(A_1',\iota,\eta'),(A_2',\iota,\eta'),(A_1,-\iota,\overline\eta'),(A_2,-\iota,\overline\eta')),-\iota,\eta'\r),\sigma\r),\tilde\rho\r)&\text{if }\star=1r,\\
        \l(\l((A_3',\iota,\eta),\bullet\l(((A_3,-\iota,\eta'),(A_2',\iota,\eta'),(A_1',\iota,\eta'),(A_1,-\iota,\overline\eta'),(A_2,-\iota,\overline\eta')),-\iota,\eta'\r),\sigma\r),\tilde\rho\r)&\text{if }\star=2r,\\
        \l(\l((A_3',\iota,\eta),\bullet\l(((A_1,-\iota,\eta'),(A_1',\iota,\eta'),(A_2',\iota,\eta'),(A_3,-\iota,\overline\eta'),(A_2,-\iota,\overline\eta')),-\iota,\eta'\r),\sigma\r),\tilde\rho\r)&\text{if }\star=3r,\\
        \l(\l((A_3',\iota,\eta),\bullet\l(((A_1,-\iota,\eta'),(A_2',\iota,\eta'),(A_1',\iota,\eta'),(A_3,-\iota,\overline\eta'),(A_2,-\iota,\overline\eta')),-\iota,\eta'\r),\sigma\r),\tilde\rho\r)&\text{if }\star=4r,\\
        \l(\l((A_3',\iota,\eta),\bullet\l(((A_1,-\iota,\eta'),(A_1',\iota,\eta'),(A_2',\iota,\eta'),(A_2,-\iota,\overline\eta'),(A_3,-\iota,\overline\eta')),-\iota,\eta'\r),\sigma\r),\tilde\rho\r)&\text{if }\star=5r,\\
        \l(\l((A_3',\iota,\eta),\bullet\l(((A_1,-\iota,\eta'),(A_2',\iota,\eta'),(A_1',\iota,\eta'),(A_2,-\iota,\overline\eta'),(A_3,-\iota,\overline\eta')),-\iota,\eta'\r),\sigma\r),\tilde\rho\r)&\text{if }\star=6r,
        \end{cases}
    \end{equation}
    where $\left.\tilde\rho\right|_{\Node(C_i)}\coloneqq\tilde\rho_i$ and $\tilde\rho(\root_A)\coloneqq M$. The map $\Upsilon$ is indeed well-defined.
\end{definition}

\begin{definition}
    We now define a map $\tilde\Upsilon\colon\res_n^{\mathrm{map}}(\iota,\eta,\eta',M)\to\G_n^{\mathrm{map}}(\iota,(\eta,\eta'),M)$ recursively. If $n=0$ then $C\in\res_n^{\mathrm{map}}(\iota,\eta,\eta',M)$ will have the only possible form $C=((\perp,\iota,\eta),(\perp,-\iota,\eta'),\sigma)$, where $\sigma$ is the obvious pairing. We set in this case $\tilde\Upsilon(C)\coloneqq(\perp,\iota,(\eta,\eta'))$. If $n>0$ and $(C,\rho)\in\res_n^{\mathrm{map}}(\iota,\eta,\eta')$ then we may write $C=(A,A',\sigma)$ and extract $C_1$, $C_2$ and $C_3$ as written in \cref{all types of resonant nodes} and denote $\rho_i\coloneqq\left.\rho\right|_{\Node(C_i)}$. More precisely,

    \begin{align}
        C_1&=\begin{cases}
            ((A_1,\iota,\overline\eta),(A_1',-\iota,\eta),\sigma_1)&\text{if }\texttt{line}=1,2\ \&\ \max\Im\rho=\rho(\root_A),\\
            ((A_1,\iota,\eta),(A_1',-\iota,\eta),\sigma_1)&\text{if }\texttt{line}=3,4,5,6\ \&\ \max\Im\rho=\rho(\root_A),\\
            ((A_1,-\iota,\overline\eta'),(A_1',\iota,\eta'),\sigma_1)&\text{if }\texttt{line}=7,8\ \&\ \max\Im\rho=\rho(\root_{A'}),\\
            ((A_1,-\iota,\eta'),(A_1',\iota,\eta'),\sigma_1)&\text{if }\texttt{line}=9,10,11,12\ \&\ \max\Im\rho=\rho(\root_{A'}).
        \end{cases}\\
        C_2&=\begin{cases}
            ((A_2,\iota,\overline\eta),(A_2',-\iota,\eta),\sigma_2)&\text{if }\texttt{line}=1,2\ \&\ \max\Im\rho=\rho(\root_A),\\
            ((A_2,\iota,\overline\eta),(A_2',-\iota,\eta),\sigma_2)&\text{if }\texttt{line}=3,4,5,6\ \&\ \max\Im\rho=\rho(\root_A),\\
            ((A_2,-\iota,\overline\eta'),(A_2',\iota,\eta'),\sigma_2)&\text{if }\texttt{line}=7,8\ \&\ \max\Im\rho=\rho(\root_{A'}),\\
            ((A_2,-\iota,\overline\eta'),(A_2',\iota,\eta'),\sigma_2)&\text{if }\texttt{line}=9,10,11,12\ \&\ \max\Im\rho=\rho(\root_{A'}).
        \end{cases}\\
        C_3&=\begin{cases}
            ((A_3,\iota,\eta),(A_3',-\iota,\eta'),\sigma_3)&\text{if }\texttt{line}=1,2\ \&\ \max\Im\rho=\rho(\root_A),\\
            ((A_3,\iota,\overline\eta),(A_3',-\iota,\eta'),\sigma_3)&\text{if }\texttt{line}=9,10,11,12\ \&\ \max\Im\rho=\rho(\root_A),\\
            ((A_3,-\iota,\eta'),(A_3',\iota,\eta),\sigma_3)&\text{if }\texttt{line}=1,2\ \&\ \max\Im\rho=\rho(\root_{A'}),\\
            ((A_3,-\iota,\overline\eta'),(A_3',\iota,\eta),\sigma_3)&\text{if }\texttt{line}=9,10,11,12\ \&\ \max\Im\rho=\rho(\root_{A'})
        \end{cases}
    \end{align}
    where \texttt{line} stands for the lines in \cref{need to code this} and $\sigma_i\coloneqq\left.\sigma\right|_{\Leaf(C_i)}$.   
    We may assume $(G_i,\tilde\rho_i)\coloneqq\tilde\Upsilon(C_i,\rho_i)$ have been defined and set 
    \begin{equation}
        \tilde\Upsilon(C,\rho)\coloneqq\begin{cases}
            \l(\l(\bullet_{1l}\l(G_1,G_2,G_3\r),\iota,(\eta,\eta')\r),\tilde\rho\r)&\text{if }\max\Im\rho=\rho(\root_A)\ \&\ \mathrm{code}=(31'2'12),\\
            \l(\l(\bullet_{2l}\l(G_1,G_2,G_3\r),\iota,(\eta,\eta')\r),\tilde\rho\r)&\text{if }\max\Im\rho=\rho(\root_A)\ \&\ \mathrm{code}=(32'1'12),\\
            \l(\l(\bullet_{3l}\l(G_1,G_2,G_3\r),\iota,(\eta,\eta')\r),\tilde\rho\r)&\text{if }\max\Im\rho=\rho(\root_A)\ \&\ \mathrm{code}=(11'2'32),\\
            \l(\l(\bullet_{4l}\l(G_1,G_2,G_3\r),\iota,(\eta,\eta')\r),\tilde\rho\r)&\text{if }\max\Im\rho=\rho(\root_A)\ \&\ \mathrm{code}=(12'1'32),\\
            \l(\l(\bullet_{5l}\l(G_1,G_2,G_3\r),\iota,(\eta,\eta')\r),\tilde\rho\r)&\text{if }\max\Im\rho=\rho(\root_A)\ \&\ \mathrm{code}=(11'2'23),\\
            \l(\l(\bullet_{6l}\l(G_1,G_2,G_3\r),\iota,(\eta,\eta')\r),\tilde\rho\r)&\text{if }\max\Im\rho=\rho(\root_A)\ \&\ \mathrm{code}=(12'1'23),\\
            \l(\l(\bullet_{1r}\l(G_1,G_2,G_3\r),\iota,(\eta,\eta')\r),\tilde\rho\r)&\text{if }\max\Im\rho=\rho(\root_{A'})\ \&\ \mathrm{code}=(31'2'12),\\
            \l(\l(\bullet_{2r}\l(G_1,G_2,G_3\r),\iota,(\eta,\eta')\r),\tilde\rho\r)&\text{if }\max\Im\rho=\rho(\root_{A'})\ \&\ \mathrm{code}=(32'1'12),\\
            \l(\l(\bullet_{3r}\l(G_1,G_2,G_3\r),\iota,(\eta,\eta')\r),\tilde\rho\r)&\text{if }\max\Im\rho=\rho(\root_{A'})\ \&\ \mathrm{code}=(11'2'32),\\
            \l(\l(\bullet_{4r}\l(G_1,G_2,G_3\r),\iota,(\eta,\eta')\r),\tilde\rho\r)&\text{if }\max\Im\rho=\rho(\root_{A'})\ \&\ \mathrm{code}=(12'1'32),\\
            \l(\l(\bullet_{5r}\l(G_1,G_2,G_3\r),\iota,(\eta,\eta')\r),\tilde\rho\r)&\text{if }\max\Im\rho=\rho(\root_{A'})\ \&\ \mathrm{code}=(11'2'23),\\
            \l(\l(\bullet_{6r}\l(G_1,G_2,G_3\r),\iota,(\eta,\eta')\r),\tilde\rho\r)&\text{if }\max\Im\rho=\rho(\root_{A'})\ \&\ \mathrm{code}=(12'1'23),
        \end{cases}
    \end{equation}
    where $\left.\tilde\rho\right|_{\Node(G_i)}\coloneqq\tilde\rho_i$ and $\tilde\rho(\root_G)\coloneqq M$.
\end{definition}

\begin{proposition}
    The maps $\Upsilon$ and $\tilde\Upsilon$ are actually bijections and $\Upsilon^{-1}=\tilde\Upsilon$.
\end{proposition}
\begin{proof}
Due to the recursive structure in all the definitions, we may show $\Upsilon\circ\tilde\Upsilon=\Id_{\res_n^{\mathrm{map}}(\iota,\eta,\eta', M)}$ by induction. Proving $\tilde\Upsilon\circ\Upsilon=\Id_{\G_n^{\mathrm{map}}(\iota,(\eta,\eta'),M)}$ is completely analogous. For $n=0$, the statement is clear. Now assume $n>0$ and suppose $(C,\rho)\in\res_n^{\mathrm{map}}(\iota,\eta,\eta',M)$ with $C=(A,A',\sigma)$ and assume first $\rho(\root_A)=\max\Im\rho=M$. This takes us into the first six cases listed in \cref{all types of resonant nodes}. We assume further, to be in the first case, that is 
\begin{equation}
    C=\l(\bullet\l(((A_3,\iota,\eta),(A_1',-\iota,\eta),(A_2',-\iota,\eta),(A_1,\iota,\overline\eta),(A_2,\iota,\overline\eta)),\iota,\eta\r),(A_3',-\iota,\eta'),\sigma\r).
\end{equation}
Of course, the construction of $\tilde\Upsilon$ tells us to set $C_i\coloneqq(A_i,A_i',\sigma_i)$, $\sigma_i\coloneqq\left.\sigma\right|_{\Leaf(A_i)\sqcup\Leaf(A_i')}$ and $\rho_i\coloneqq\left.\rho\right|_{\Node(C_i)}$ so that $(C_1,\rho_1)\in\res_{n_1}^{\mathrm{map}}(\iota,\overline\eta,\eta,\max\Im\rho_1)$, $(C_2,\rho_2)\in\res_{n_2}^{\mathrm{map}}(\iota,\overline\eta,\eta,\max\Im\rho_2)$ and $(C_3,\rho_3)\in\res_{n_3}^{\mathrm{map}}(\iota,\eta,\eta',\max\Im\rho_3)$ where $n_1+n_2+n_3=n-1$. The construction further instructs us to set $(B_1,\tilde\rho_1)\coloneqq\tilde\Upsilon(C_1,\rho_1)\in\G_{n_1}(\iota,(\overline\eta,\eta),\max\Im\rho_1)$, $(B_2,\tilde\rho_2)\coloneqq\tilde\Upsilon(C_2,\rho_2)\in\G_{n_2}(\iota,(\overline\eta,\eta),\max\Im\rho_2)$ and $(B_3,\tilde\rho_3)\coloneqq\tilde\Upsilon(C_3,\rho_3)\in\G_{n_3}(\iota,(\eta,\eta'),\max\Im\rho_1)$ so that $\tilde\Upsilon(C,\rho)=\l(\l(\bullet_{1l}(G_1,G_2,G_3),\iota,(\eta,\eta'),M\r),\tilde\rho\r)$, where $\left.\tilde\rho\right|_{\Node(G_i)}\coloneqq\tilde\rho_i$ and $\tilde\rho(\root_G)\coloneqq M$. Due to the induction hypothesis, we have $\Upsilon\l(G_i,\tilde\rho_i\r)=\Upsilon\circ\tilde\Upsilon\l(C_i,\rho_i\r)=\l(C_i,\rho_i\r)$. This implies precisely $\Upsilon\circ\tilde\Upsilon(C,\rho)=(C,\rho)$. This inductive argument of course works also for the remaining $11$ cases. 
\end{proof}

\subsection{Ternary trees recapture the recursive structure}

\begin{definition}
    We define for $\l(C,\rho_C\r)\in\res_n^{\mathrm{map}}(\iota,\eta,\eta')$, 
    \begin{equation}
        \J_{\l(C,\rho_C\r)}(t,k)\coloneqq(-\rmi)^nL^{-2dn}\sum_{\kappa\in\D_k(C)}\prod_{\node\in\Node(C)}\iota_\node Q_\node\prod_{\leaf\in\Leaf(C)_+}M^{\eta_\leaf,\eta_{\sigma(\leaf)}}(\kappa(\leaf))\frac{t^n}{n!}.
    \end{equation}
    For any $(G,\rho_G)\in\G_n^{\mathrm{map}}(\iota,(\eta,\eta'))$, we set $(C,\rho_C)=\Upsilon\l(G,\rho_G\r)$ and $\mathcal H_G^{\rho_G}\coloneqq\J_{\Upsilon\l(G,\rho_G\r)}$ and set for any $G\in\G_n(\iota,(\eta,\eta'))$,
    \begin{equation}
        \mathcal H_G\coloneqq\sum_{\rho_G\in\M(\Node(G))}\mathcal H_G^{\rho_G}.
    \end{equation}
\end{definition}

\begin{remark}\label{coka remark}
    Note that $\J_C(\epsilon^{-1}t,k)=\abs{\M(\Node(C))}\J_{\l(C,\rho_C\r)}(t,k)$ for all $\rho_C\in\M(\Node(C))$ such that \begin{equation}
        \J_C(\epsilon^{-1}t,k)=\sum_{\rho_C\in\M(\Node(C))}\J_{\l(C,\rho_C\r)}(t,k).
    \end{equation}
    Furthermore, if $G\in\G_n(\iota,(\eta,\eta'))$ and $\rho_1,\rho_2\in\M(\Node(G))$ and if we set $(C_i,\tilde\rho_i)\coloneqq\Upsilon(G,\rho_i)$, we actually have $C_1=C_2$. The analogous statement for $\tilde\Upsilon$ is actually untrue due to the fact that we may have the {resonant branching} (in the sense of \cref{all types of resonant nodes}) at the left or right tree of a resonant couple.
\end{remark}

\begin{lemma}\label{multinomial lemma}
    Suppose $n>0$ and $G\in\G_n(\iota,(\eta,\eta'))$ writes as $G=\bullet_\star(G_1,G_2,G_3)$, where $G_i$ have scale $n_i$ respectively so that $n_1+n_2+n_3=n-1$. Then 
    \begin{equation}
        \abs{\M(\Node(G))} = \frac{(n_1+n_2+n_3)!}{n_1!n_2!n_3!}\abs{\M(\Node(G_1))}\abs{\M(\Node(G_2))}\abs{\M(\Node(G_3))}
    \end{equation}
\end{lemma}
\begin{proof}
    If $\l(\rho_1,\rho_2,\rho_3\r)\in\M(\Node(G_1))\times\M(\Node(G_2))\times\M(\Node(G_3))$ then we choose some strictly increasing injective function $\mu_1\colon\llbracket1,n_1\rrbracket\to\llbracket1,n_1+n_2+n_3\rrbracket$ and some other strictly increasing injective function $\mu_2\colon\llbracket1,n_2\rrbracket\to\llbracket1,n_1+n_2+n_3\rrbracket$ such that $\Im\mu_1\cap\Im\mu_2=\emptyset$. Then we let $\mu_3\colon\llbracket1,n_3\rrbracket\to\llbracket1,n_1+n_2+n_3\rrbracket$ be the unique strictly increasing injective function such that $\Im\mu_3\cap\Im\mu_i=\emptyset$ for both $i=1,2$. Any order $\rho\in\M(\Node(G))$ may now be represented in the form 
    \begin{equation}
        \rho(\node)\coloneqq\begin{cases}
            \mu_1(\rho_1(\node))&\text{if }\node\in\Node(G_1),\\
            \mu_3(\rho_2(\node))&\text{if }\node\in\Node(G_2),\\
            \mu_3(\rho_3(\node))&\text{if }\node\in\Node(G_3). 
        \end{cases}
    \end{equation}
    It is thus clear that 
    \begin{equation}
        \abs{\M(\Node(G))}=c\abs{\M(\Node(G_1))}\abs{\M(\Node(G_2))}\abs{\M(\Node(G_3))},
    \end{equation}
    where $c$ encodes the different possibilities, $\mu_1$ and $\mu_2$ can be defined. Initially counting the number of choices to choose $\mu_1$ is $(n_1+n_2+n_3)\cdots(n_2+n_3+1)=\frac{(n_1+n_2+n_3)!}{(n_2+n_3)!}$ but this quantity includes internal permutations of the image of $\mu_1$ which is wrong since $\mu_1$ is supposed to be strictly increasing. Thus, we must divide this fraction by $n_1!$ to obtain that the number of choices for $\mu_1$ is $\binom{n_1+n_2+n_3}{n_1}$. After $\mu_1$ has been chosen, we choose $\mu_2$ and start counting the possibilities again naively $(n_2+n_3)\cdots(n_3+1)=\frac{(n_2+n_3)!}{n_3!}$ and this number again includes all possible permutations of the image points of $\mu_2$ so they have to be divided out which leads us to $\binom{n_2+n_3}{n_2}$. Thence, 
    \begin{equation}
        c=\binom{n_1+n_2+n_3}{n_1}\binom{n_2+n_3}{n_2}=\frac{(n_1+n_2+n_3)!}{n_1!n_2!n_3!}=\binom{n_1+n_2+n_3}{n_1,n_2,n_3}.
    \end{equation}
\end{proof}

\begin{proposition}\label{insertment proposition}
    Let $G\in\G_n(\iota,(\eta,\eta'))$. Then if $n=0$, $\Q_G(t,k)=M^{\eta,\eta',\iota}(k)$ and if $n>0$, we may decompose uniquely $G=(\bullet_\star(G_1,G_2,G_3),\iota,(\eta,\eta'))$, where $G_i\in\G_{n_i}(\iota_i,(\eta_i,\eta_i'))$ with $n_1+n_2+n_3=n-1$ and $\iota_i,\eta_i,\eta_i'$ depend on the nature of $\star\in\{1l,\ldots,6l,1r,\ldots,6r\}$. All possibilities have been written in \cref{all sign and colour possibilities}. We then have the recursive structure 
    \begin{align}
        \Q_G(t,k)=\begin{cases}
            -\rmi\iota\int_0^t\Q_{G_3}(s,k)L^{-2d}\sum\limits_{k_1,k_2}\tilde\zeta_{j,\iota}^\eta(k_1,k_2,k)\Q_{G_1}(s,-k_1)\Q_{G_2}(s,-k_2)\rmd s&\text{if }\star\in\{1l,2l\}\\
            -\rmi\iota\int_0^t\Q_{G_3}(s,k)L^{-2d}\sum\limits_{k_1,k_2}\tilde\zeta_{3,\iota}^{\eta}(k_1,k_2,k)\Q_{G_1}(s,k_1)\Q_{G_2}(s,-k_2)\rmd s&\text{if }\star=\in\{3l,\ldots,6l\}\\
            \rmi\iota\int_0^t\Q_{G_3}(s,-k)L^{-2d}\sum\limits_{k_1,k_2}\tilde\zeta_{j,-\iota}^{\eta'}(k_1,k_2,-k)\Q_{G_1}(s,-k_1)\Q_{G_2}(s,-k_2)\rmd s&\text{if }\star\in\{1r,2r\},\\
            \rmi\iota\int_0^t\Q_{G_3}(s,-k)L^{-2d}\sum\limits_{k_1,k_2}\tilde\zeta_{j,-\iota}^{\eta'}(k_1,k_2,-k)\Q_{G_1}(s,k_1)\Q_{G_2}(s,-k_2)\rmd s&\text{if }\star\in\{3r,\ldots,6r\}.
        \end{cases}
    \end{align}
\end{proposition}

\begin{proof}
The case $n=0$ follows by definition. Now suppose $n>0$. If $\rho\in\M(\Node(G))$ and $(C,\rho_C)=\Upsilon(G,\rho)$. Of course, 
\begin{equation}
    \Q_G=\abs{\M(\Node(G))}\Q_G^\rho=\abs{\M(\Node(G))}\J_C^{\rho_C} = \abs{\M(\Node(G))}\frac{t^n}{n!}\tilde\J_C\eqqcolon\abs{\M(\Node(G))}\frac{t^n}{n!}\tilde\Q_G
\end{equation}
 so that $\tilde\Q_G\coloneqq\tilde\J_C$ and this definition is well-defined because if $(C_i,\rho_i)=\Upsilon(G,\tilde\rho_i)$ for $\tilde\rho_i\in\M(\Node(G))$, then $C_1=C_2$ due to the construction of $\Upsilon$, see \cref{coka remark}. 

 Let $\star=1l$. In this case, $G_i\in\G_{n_i}(\iota,(\overline\eta,\eta))$ for $i=1,2$ and $G_3\in\G_{n_3}(\iota,(\eta,\eta'))$. Furthermore, \cref{the first set} implies 
        \begin{equation}
            \begin{aligned}
                \tilde\Q_G(k)=\tilde\J_C(k)&=-\rmi\iota_{\root_A}\tilde\J_{C_3}(k)L^{-2d}\sum_{k_1,k_2\in\Z_L^d}Q_{\root_A}\tilde\J_{C_1}(-k_1)\tilde\J_{C_2}(-k_2)\\&=-\rmi\iota_{\root_A}\tilde\J_{C_3}(k)L^{-2d}\sum_{k_1,k_2\in\Z_L^d}\tilde\zeta_{1,\iota}^\eta(k_1,k_2,k)\tilde\J_{C_1}(-k_1)\tilde\J_{C_2}(-k_2)\\&=-\rmi\iota_{\root_A}\tilde\Q_{G_3}(k)L^{-2d}\sum_{k_1,k_2\in\Z_L^d}\tilde\zeta_{1,\iota}^\eta(k_1,k_2,k)\tilde\Q_{G_1}(-k_1)\tilde\Q_{G_2}(-k_2).
            \end{aligned}
        \end{equation}
        One can now prove via induction that 
        \begin{equation}
            \Q_G(t,k)=-\rmi\iota_{\root_A}\int_0^t\Q_{G_3}(s,k)L^{-2d}\sum_{k_1,k_2}\tilde\zeta_{1,\iota}^\eta(k_1,k_2,k)\Q_{G_1}(s,-k_1)\Q_{G_2}(s,-k_2)\rmd s.
        \end{equation}
        The induction goes as follows. If $n=0$, then there is nothing to prove. If $n>0$, we find with \cref{multinomial lemma},
        \begin{equation}
            \begin{gathered}
                -\rmi\iota_{\root_A}\int_0^t\Q_{G_3}(s,k)L^{-2d}\sum_{k_2,k_3}\tilde\zeta_{1,\iota}^\eta(k_1,k_2,k)\Q_{G_1}(s,-k_1)\Q_{G_2}(s,-k_2)\rmd s\\
                =\tilde\Q_G(k)\int_0^t\frac{s^{n_1+n_2+n_3}}{n_1!n_2!n_3!}\rmd s\abs{\M(\Node(G_1))}\abs{\M(\Node(G_2))}\abs{\M(\Node(G_3))} \\= \tilde\Q_G(k)\frac{t^n}{nn_1!n_2!n_3!}\abs{\M(\Node(G_1))}\abs{\M(\Node(G_2))}\abs{\M(\Node(G_3))}\\=\abs{\M(\Node(G))}\tilde\Q_G(k)\frac{t^n}{n!}=\Q_G(t,k).
            \end{gathered}
        \end{equation}
        Proving the remaining $11$ cases works analogously. 
\end{proof}

We now set 
\begin{equation}
    \begin{gathered}
    \rho_{L,n}^{(\eta,\eta'),\iota}(t)\coloneqq\sum_{C\in\res_n(\iota,\eta,\eta')}\J_C\l(\epsilon^{-1}t\r)=\sum_{(C,\rho)\in\res_n^{\mathrm{ord}}(\iota,\eta,\eta')}\J_{(C,\rho)}(t) = \sum_{(G,\rho)\in\G_n^{\mathrm{ord}}(\iota,(\eta,\eta'))}\Q_G^\rho(t)\\
    \sum_{G\in\G_n(\iota,(\eta,\eta'))}\sum_{\rho\in\M(\Node(G))}\Q_G^\rho(t) = \sum_{G\in\G_n(\iota,(\eta,\eta'))}\Q_G(t).
\end{gathered}
\end{equation}
\begin{remark}\label{for general non-linearity and resonant structure}
    If \eqref{system} had a non-linearity of degree $2k+1$, then, for every resonant couple, $k+1$ sub-couples would form. One may analogously construct a bijection $\Upsilon$ between the set of resonant couples and $(k+1)$-ary trees. 
\end{remark}

We may use the fact that if $n>0$, we can decompose $G=(\bullet_\star(G_1,G_2,G_3),\iota,(\eta,\eta'))$ for some 
\begin{gather}
    G_1\in\begin{cases}
        \G_{n_1}(\iota,(\overline\eta,\eta))&\text{if }\star=1l,2l,\\
        \G_{n_1}(\iota,(\eta,\eta))&\text{if }\star=3l,\ldots,6l,\\
        \G_{n_1}(-\iota,(\overline\eta',\eta'))&\text{if }\star=1r,2r,\\
        \G_{n_1}(-\iota,(\eta',\eta'))&\text{if }\star=3r,\ldots,6r,
    \end{cases}\ 
    G_2\in\begin{cases}
        \G_{n_2}(\iota,(\overline\eta,\eta))&\text{if }\star=1l,\ldots,6l,\\
        \G_{n_2}(-\iota,(\overline\eta',\eta'))&\text{if }\star=1r,\ldots,6r,
    \end{cases}\\ 
    G_3\in\begin{cases}
        \G_{n_3}(\iota,(\eta,\eta'))&\text{if }\star=1l,2l,\\
        \G_{n_3}(\iota,(\overline\eta,\eta'))&\text{if }\star=3l,\ldots,6l,\\
        \G_{n_3}(-\iota,(\eta',\eta))&\text{if }\star=1r,2r,\\
        \G_{n_3}(-\iota,(\overline\eta',\eta))&\text{if }\star=3r,\ldots,6r,
    \end{cases}
\end{gather}
where $n_1+n_2+n_3=n-1$.

We introduce the notation 
\begin{equation}
    \iota^\dagger\coloneqq\begin{cases}
        \iota&\text{if }\dagger=r,\\
        -\iota&\text{if }\dagger=l.
    \end{cases},\ \eta^\dagger\coloneqq\begin{cases}
        \eta&\text{if }\dagger=r,\\
        \eta'&\text{if }\dagger=l
    \end{cases},\ \overline\dagger\coloneqq\begin{cases}
            l&\text{if }\dagger=r,\\
            r&\text{if }\dagger=l.
        \end{cases}\text{ and }\tilde\iota^\dagger\coloneqq\begin{cases}
            +&\text{if }\dagger=r,\\
            -&\text{if }\dagger=l.
        \end{cases}
\end{equation}
Let us decompose 
\begin{equation}
    \begin{gathered}
    \rho_{L,n}^{(\eta,\eta'),\iota}=\sum_{n_1+n_2+n_3=n-1}\sum_{\dagger\in\{l,r\}}\Bigg(\sum_{\substack{G_1\in\G_{n_1}(\iota^\dagger,(\overline{\eta^\dagger},\eta^\dagger))\\G_2\in\G_{n_2}(\iota^\dagger,(\overline{\eta^\dagger},\eta^\dagger))\\G_3\in\G_{n_3}(\iota^\dagger,(\eta^\dagger,\eta^{\overline\dagger}))}}\l(\Q_{\bullet_{1\dagger}(G_1,G_2,G_3)}+\Q_{\bullet_{2\dagger}(G_1,G_2,G_3)}\r)\\+\sum_{\substack{G_1\in\G_{n_1}(\iota^\dagger,(\eta^\dagger,\eta^\dagger))\\G_2\in\G_{n_2}(\iota^\dagger,(\overline{\eta^\dagger},\eta^\dagger))\\G_3\in\G_{n_3}(\iota^\dagger,(\overline{\eta^\dagger},\eta^{\overline\dagger}))}}\l(\Q_{\bullet_{3\dagger}(G_1,G_2,G_3)}+\Q_{\bullet_{4\dagger}(G_1,G_2,G_3)}+\Q_{\bullet_{5\dagger}(G_1,G_2,G_3)}+\Q_{\bullet_{6\dagger}(G_1,G_2,G_3)}\r)\Bigg)
\end{gathered}
\end{equation}
Using \cref{insertment proposition}, we find 

\begin{equation}
    \begin{gathered}
        \rho_{L,n}^{(\eta,\eta'),\iota}(t,k)=\sum_{G\in\G_n(\iota,(\eta,\eta'))}\Q_G=-\rmi\sum_{n_1+n_2+n_3=n-1}\sum_{\dagger\in\{l,r\}}\iota^\dagger\\\Bigg[\sum_{\substack{\substack{G_1\in\G_{n_1}(\iota^\dagger,(\overline{\eta^\dagger},\eta^\dagger))\\G_2\in\G_{n_2}(\iota^\dagger,(\overline{\eta^\dagger},\eta^\dagger))\\G_3\in\G_{n_3}(\iota^\dagger,(\eta^\dagger,\eta^{\overline\dagger}))}}}\sum_{l=1}^2\int_0^t\Q_{G_3}(s,\tilde\iota^\dagger k)L^{-2d}\sum_{k_1,k_2}\tilde\zeta_{l,\iota^\dagger}^{\eta^\dagger}(k_1,k_2,\tilde\iota^\dagger k)\Q_{G_1}(s,-k_1)\Q_{G_2}(s,-k_2)\rmd s\\\sum_{\substack{G_1\in\G_{n_1}(\iota^\dagger,(\eta^\dagger,\eta^\dagger))\\G_2\in\G_{n_2}(\iota^\dagger,(\overline{\eta^\dagger},\eta^\dagger))\\G_3\in\G_{n_3}(\iota^\dagger,(\overline{\eta^\dagger},\eta^{\overline\dagger}))}}\sum_{l=3}^6\int_0^t\Q_{G_3}(s,\tilde\iota^\dagger k)L^{-2d}\sum_{k_1,k_2}\tilde\zeta_{l,\iota^\dagger}^{\eta^\dagger}(k_1,k_2,\tilde\iota^\dagger k)\Q_{G_1}(s,k_1)\Q_{G_2}(s,-k_2)\rmd s\Bigg]
    \end{gathered}
\end{equation}

\begin{equation}\label{here I have to carry out the sum over dagger}
\begin{gathered}
    =-\rmi\sum_{n_1+n_2+n_3=n-1}\sum_{\dagger\in\{l,r\}}\iota^\dagger
    \Bigg[\sum_{l=1}^2\int_0^t\rho_{L,n_3}^{(\eta^\dagger,\eta^{\overline\dagger}),\iota^\dagger}(s,\tilde\iota^\dagger k)L^{-2d}\sum_{k_1,k_2}\tilde\zeta_{l,\iota^\dagger}^{\eta^\dagger}(k_1,k_2,\tilde\iota^\dagger k)\rho_{L,n_1}^{(\overline{\eta^\dagger},\eta^\dagger),\iota^\dagger}(s,-k_1)\rho_{L,n_2}^{(\overline{\eta^\dagger},\eta^\dagger),\iota^\dagger}(s,-k_2)\rmd s\\+\sum_{l=3}^6\int_0^t\rho_{L,n_3}^{(\overline{\eta^\dagger},\eta^{\overline\dagger}),\iota^\dagger}(s,\tilde\iota^\dagger k)L^{-2d}\sum_{k_1,k_2}\tilde\zeta_{l,\iota^\dagger}^{\eta^\dagger}(k_1,k_2,\tilde\iota^\dagger k)\rho_{L,n_1}^{(\eta^\dagger,\eta^\dagger),\iota^\dagger}(s,k_1)\rho_{L,n_2}^{(\overline{\eta^\dagger},\eta^\dagger),\iota^\dagger}(s,-k_2)\rmd s
    \Bigg]
\end{gathered}
\end{equation}

\begin{lemma}
    It holds that $\rho_{L,n}^{(\eta,\eta'),-\iota}(t,k)=\overline{\rho_{L,n}^{(\eta,\eta'),\iota}(t,-k)}$.
\end{lemma}
\begin{proof}
    The assertion holds for $n=0$ because in this case $\rho_{L,0}^{(\eta,\eta'),\iota}=M^{\eta,\eta',\iota}$ and $M^{\eta,\eta',-\iota}(k)=\overline{M^{\eta,\eta',\iota}(-k)}$. Now suppose $n>0$ and the statement holds for all $0\leq m<n$. The induction hypothesis and the fact that $Q_{-\iota}^\eta\l(k_1,k_2,k_3,k_4,k_5\r)=\overline{Q_{\iota}^\eta(k_1,k_2,k_3,k_4,k_5)}$ comlete the proof by induction.
\end{proof}

\begin{lemma}
    It further holds,
    \begin{equation}
        \rho_{L,n}^{(\eta',\eta),\iota} = \overline{\rho_{L,n}^{(\eta,\eta'),\iota}}.
    \end{equation}
    In particular, $\rho_{L,n}^{(\eta,\eta),\iota}$ is real-valued. 
\end{lemma}
\begin{proof}
    We prove the statement via induction and the induction start $n=0$ holds by the properties of $M^{\eta,\eta'}$ because $M^{\eta',\eta}=\overline{M^{\eta,\eta'}}$. Now assume $n>0$. Let $\eta'=\eta$. Indeed, $M^{\eta,\eta}$ is real-valued, and since $n>0$, we find, using the induction hypothesis and carrying out the sum over $\dagger\in\{l,r\}$ in \cref{here I have to carry out the sum over dagger}:
    \begin{equation}
        \begin{gathered}
            \rho_{L,n}^{(\eta,\eta),\iota}(t,k)=2\iota\sum_{n_1+n_2+n_3=n-1}\Bigg[\sum_{l=1}^2\int_0^t\rho_{L,n_3}^{(\eta,\eta),\iota}(s,k)\frac{1}{L^{2d}}\sum_{k_1,k_2}\Im\l(\rho_{L,n_1}^{(\overline\eta,\eta),\iota}(s,-k_1)\tilde\zeta_{l,\iota}^\eta(k_1,k_2,k)\rho_{L,n_2}^{(\overline\eta,\eta),\iota}(s,-k_2)\r)\rmd s\\+\sum_{l=3}^6\int_0^tL^{-2d}\sum_{k_1,k_2}\Im\l(\rho_{L,n_3}^{(\overline\eta,\eta),\iota}(s,k)\tilde\zeta_{l,\iota}^\eta(k_1,k_2,k)\rho_{L,n_2}^{(\overline\eta,\eta),\iota}(s,-k_2)\r)\rho_{L,n_1}^{(\eta,\eta),\iota}(s,k_1)\rmd s\Bigg]
        \end{gathered}
    \end{equation}
    which is a real-valued expression. 
    In the case $\eta'=\overline\eta$, we find 

    \begin{equation}
        \begin{gathered}
            \rho_{L,n}^{(\eta,\overline\eta),\iota}(t,k)=-\iota\rmi\sum_{n_1+n_2+n_3=n-1}\\\Bigg[\sum_{l=1}^2\int_0^t\rho_{L,n_3}^{(\eta,\overline\eta),\iota}(s,k)L^{-2d}\sum_{k_1,k_2}\l(\tilde\zeta_{l,\iota}^\eta(k_1,k_2,k)-\overline{\tilde\zeta_{l,\iota}^{\overline\eta}(k_1,k_2,k)}\r)\rho_{L,n_1}^{(\overline\eta,\eta),\iota}(s,-k_1)\rho_{L,n_2}^{(\overline\eta,\eta),\iota}(s,-k_2)\rmd s\\+\sum_{l=3}^6\int_0^tL^{-2d}\sum_{k_1,k_2}\bigg(\rho_{L,n_3}^{(\overline\eta,\overline\eta),\iota}(s,k)\tilde\zeta_{l,\iota}^\eta(k_1,k_2,k)\rho_{L,n_1}^{(\eta,\eta),\iota}(s,k_1)\\-\rho_{L,n_3}^{(\eta,\eta),\iota}(s,k)\overline{\tilde\zeta_{l,\iota}^{\overline\eta}(k_1,k_2,k)}\rho_{L,n_1}^{(\overline\eta,\overline\eta),\iota}(s,k_1)\bigg)\rho_{L,n_2}^{(\overline\eta,\eta),\iota}(s,-k_2)\rmd s\Bigg]\\=\overline{\rho_{L,n}^{(\overline\eta,\eta),\iota}(t,k)},
        \end{gathered}
    \end{equation}
    where in the last equality sign we used the induction hypothesis. 
\end{proof}
The previous claim implies that the only relevant quantities are 
\begin{equation}
    \rho_{L,n}^\eta\coloneqq \rho_{L,n}^{(\eta,\eta),+}\text{ and }\rho_{L,n}^\times\coloneqq \rho_{L,n}^{(0,1),+}
\end{equation}
and all other quantities can be deduced from these three quantities. 

\subsection{The large-box limit}\label{the large box limit section}

For $A,L>0$, we denote by $\mathcal E_{L,A}$ the probability set of measure $\mathbb P\l(\mathcal E_{L,A}\r)\geq1-L^{-A}$ on which \cref{The estimate the whole world was looking for} holds. Restricting ourselves to this set, the solution $f^\eta$ exists uniquely on the time interval $[0,\delta\epsilon^{-1}]$.

We may now compare 
\begin{equation}
    \E\l(\vmathbb{1}_{\mathcal E_{L,A}}\abs{\widehat{f^\eta}\l(\epsilon^{-1}t,k\r)}^2\r)\text{ and }\E\l(\vmathbb{1}_{\mathcal E_{L,A}}{\widehat{f^0}\l(\epsilon^{-1}t,k\r)\overline{\widehat{f^1}\l(\epsilon^{-1}t,k\r)}}\r)
\end{equation}
to \begin{equation}
    \rho^{\eta}(t,k)\text{ and }\rho^\times(t,k).
\end{equation}
and denote \begin{equation}
    [\eta,\eta']\coloneqq\begin{cases}
        \eta&\text{if }\eta=\eta',\\
        \times&\text{if }(\eta,\eta')=(0,1).
    \end{cases}
\end{equation}
so that

\begin{gather}
    \E\l(\vmathbb{1}_{\mathcal E_{L,A}}{\widehat{f^\eta}\l(\epsilon^{-1}t,k\r)}\overline{\widehat{f^{\eta'}}\l(\epsilon^{-1}t,k\r)}\r)-\rho^{[\eta,\eta']}(t,k)\\\label{first line}=\E\l(\vmathbb{1}_{\mathcal E_{L,A}}{\widehat{f^\eta}\l(\epsilon^{-1}t,k\r)}\overline{\widehat{f^{\eta'}}\l(\epsilon^{-1}t,k\r)}\r)-\E\l(\vmathbb{1}_{\mathcal E_{L,A}}{\widehat{f^\eta_{\leq N(L)}}\l(\epsilon^{-1}t,k\r)}\overline{\widehat{f^{\eta'}_{\leq N(L)}}\l(\epsilon^{-1}t,k\r)}\r)\\\label{second line}+\E\l(\vmathbb{1}_{\mathcal E_{L,A}}{\widehat{f^\eta_{\leq N(L)}}\l(\epsilon^{-1}t,k\r)}\overline{\widehat{f^{\eta'}_{\leq N(L)}}\l(\epsilon^{-1}t,k\r)}\r)-\E\l({\widehat{f^\eta_{\leq N(L)}}\l(\epsilon^{-1}t,k\r)}\overline{\widehat{f^{\eta'}}\l(\epsilon^{-1}t,k\r)}\r)\\\label{third line}+\E\l({\widehat{f^\eta_{\leq N(L)}}\l(\epsilon^{-1}t,k\r)}\overline{\widehat{f^{\eta'}_{\leq N(L)}}\l(\epsilon^{-1}t,k\r)}\r)-\sum_{n_1,n_2\leq N(L)}\rho_{L,n_1+n_2}^{[\eta,\eta']}(t,k)\\\label{fourth line}+\sum_{n_1,n_2\leq N(L)}\rho_{L,n_1+n_2}^{[\eta,\eta']}(t,k)-\sum_{n_1,n_2\leq N(L)}\rho_{n_1+n_2}^{[\eta,\eta']}(t,k)\\\label{fifth line}+\sum_{n_1,n_2\leq N(L)}\rho_{n_1+n_2}^{[\eta,\eta']}(t,k)-\rho^{[\eta,\eta']}(t,k).
\end{gather}

Inserting $F^\eta=F_{\leq N}^\eta+v^\eta$ into \cref{first line}, we obtain that (with the fact that $H^s\l(T_L^d\r)$ is a Banach algebra)

\begin{equation}
\begin{gathered}
    \abs{\cref{first line}}\\\leq\E\l(\vmathbb{1}_{\mathcal E_{L,A}}\abs{\widehat{f_{\leq N(L)}^\eta}\l(\epsilon^{-1}t,k\r)}\abs{\widehat{v^{\eta'}}\l(\epsilon^{-1}t,k\r)}\r)+\E\l(\vmathbb{1}_{\mathcal{E}_{L,A}}\abs{\widehat{v^\eta}\l(\epsilon^{-1}t,k\r)}\abs{\widehat{f^{\eta'}_{\leq N(L)}}\l(\epsilon^{-1}t,k\r)}\r)\\+\E\l(\vmathbb{1}_{\mathcal E_{L,A}}\abs{\widehat{v^\eta}\l(\epsilon^{-1}t,k\r)}\abs{\widehat{v^{\eta'}}\l(\epsilon^{-1}t,k\r)}\r)\\\leq\E\l(\vmathbb{1}_{\mathcal E_{L,A}}\norm{f_{\leq N(L)}^\eta v^{\eta'}}_{L^\infty H^s\l(\mathbb T_L^d\r)}\r) + \E\l(\vmathbb{1}_{\mathcal E_{L,A}}\norm{v^\eta f^{\eta'}_{\leq N(L)}}_{L^\infty H^s\l(\mathbb T_L^d\r)}\r)+\E\l(\vmathbb{1}_{\mathcal E_{L,A}}\norm{v^\eta v^{\eta'}}_{L^\infty H^s\l(H^s\l(\mathbb T_L^d\r)\r)}\r)\\\leq\sum_{n=0}^{N(L)}\E\l(\vmathbb{1}_{\mathcal E_{L,A}}\norm{f_{n}^\eta}_{L^\infty H^s\l(T_L^d\r)}\norm{v^{\eta'}}_{L^\infty H^s\l(\mathbb T_L^d\r)}\r) + \sum_{n=0}^N\E\l(\vmathbb{1}_{\mathcal E_{L,A}}\norm{f^{\eta'}_{n}}_{L^\infty H^s\l(\mathbb T_L^d\r)}\norm{v^\eta}_{L^\infty H^s\l(\mathbb T_L^d\r)}\r)\\+\E\l(\vmathbb{1}_{\mathcal E_{L,A}}\norm{v^\eta}_{L^\infty H^s\l(H^s\l(\mathbb T_L^d\r)\r)}\norm{v^{\eta'}}_{L^\infty H^s\l(H^s\l(\mathbb T_L^d\r)\r)}\r)\\\leq2\l(1-L^{-A}\r)\Lambda L^{2\l(\mathcal K+d+\frac{2}{\beta}+2A-\tilde\alpha\r)}\sum_{n\leq N(L)}(\Lambda\delta)^{2n}\\+\l(1-L^{-A}\r)\Lambda^2L^{-2\tilde\alpha}\xlongrightarrow{L\longrightarrow\infty}0.
\end{gathered}
\end{equation}

For \cref{second line}, we will use Hölder inequality two times and Gaussian hypercontractivity to get 

\begin{equation}
    \begin{gathered}
        \abs{\cref{second line}}\\\leq\sum_{0\leq n_1,n_2\leq N(L)}\E\l(\vmathbb{1}_{\mathcal E_{L,A}^c}\abs{\widehat{f^\eta_{n_1}}\l(\epsilon^{-1}t,k\r)}\abs{\widehat{f_{n_2}^{\eta'}}\l(\epsilon^{-1}t,k\r)}\r)\\\leq\E\l(\vmathbb{1}_{\mathcal E_{L,A}^c}\r)\sum_{0\leq n_1,n_2\leq N(L)}\E\l(\abs{\widehat{f^\eta_{n_1}}\l(\epsilon^{-1}t,k\r)}^2\abs{\widehat{f_{n_2}^{\eta'}}\l(\epsilon^{-1}t,k\r)}^2\r)^{1/2}\\\leq L^{-A/2}\sum_{0\leq n_1,n_2\leq N(L)}\E\l(\abs{\widehat{f^\eta_{n_1}}\l(\epsilon^{-1}t,k\r)}^4\r)^{1/4}\E\l(\abs{\widehat{f^{\eta'}_{n_2}}\l(\epsilon^{-1}t,k\r)}^4\r)^{1/4}\\\leq\Lambda L^{-\frac{A}{2}}\sum_{n_1,n_2=0}^{N(L)}c^{n_1}c^{n_2}\E\l(\abs{\widehat{f^\eta_{n_1}}\l(\epsilon^{-1}t,k\r)}^2\r)^{1/2}\E\l(\abs{\widehat{f^{\eta'}_{n_2}}\l(\epsilon^{-1}t,k\r)}^2\r)^{1/2}\\\Lambda L^{-\frac{A}{2}+\mathcal K}\l(\sum_{n\leq N(L)}\l(\Lambda\sqrt{\delta}\r)^{2n}\r)^2\\\leq\Lambda L^{-\frac{A}{2}+\mathcal K}\xlongrightarrow{L\longrightarrow\infty}0,
    \end{gathered}
\end{equation}
for all $A>A_0\geq2\mathcal K$, where we used \cref{key proposition}.

We now rewrite \cref{third line} as 
\begin{equation}
    \begin{gathered}
        {\cref{third line}}\\=\sum_{n_1,n_2\leq N(L)}\sum_{C\in\mathcal C_{n_1,n_2}^{\eta,\eta',+,-}}{\J_C}\l(\epsilon^{-1}t,k\r)-\sum_{n_1,n_2\leq N(L)}\sum_{C\in\res_{n_1+n_2}(+,\eta,\eta')}{\J_C}\l(\epsilon^{-1}t,k\r)\\
        =\sum_{n_1,n_2\leq N(L)}\sum_{\substack{C\in\mathcal C_{n_1,n_2}^{\eta,\eta',+,-}\\n_{\mathrm{res}}(C)\geq1}}{\J_C}\l(\epsilon^{-1}t,k\r)
    \end{gathered}
\end{equation}
Let $\mathcal W>0$ to be determined. We may now estimate, using \cref{key proposition,theorem: every non-resonant couple contributes with decay in box size},

\begin{equation}
\begin{gathered}
    \abs{\sum_{n_1,n_2\leq N(L)}\sum_{\substack{C\in\mathcal C_{n_1,n_2}^{\eta,\eta',-,+}\\n_{\mathrm{res}}(C)\geq1}}{\J_C}\l(\epsilon^{-1}t,k\r)}\\\leq\abs{\sum_{\substack{n_1,n_2\leq N(L)\\n_1+n_2\leq\mathcal W}}\sum_{\substack{C\in\mathcal C_{n_1,n_2}^{\eta,\eta',-,+}\\n_{\mathrm{res}}(C)\geq1}}{\J_C}\l(\epsilon^{-1}t,k\r)}\\+\abs{\sum_{\substack{n_1,n_2\leq N(L)\\\mathcal W<n_1+n_2\leq\mathcal N(L)}}\sum_{0<q\leq n_1+n_2-\mathcal W}\sum_{\substack{C\in\mathcal C_{n_1,n_2}^{\eta,\eta',-,+}\\n_{\mathrm{res}}(C)=q\geq1}}{\J_C}\l(\epsilon^{-1}t,k\r)}\\+\abs{\sum_{\substack{n_1,n_2\leq N(L)\\\mathcal W<n_1+n_2\leq\mathcal N(L)}}\sum_{n_1+n_2-\mathcal W<q\leq n_1+n_2}\sum_{\substack{C\in\mathcal C_{n_1,n_2}^{\eta,\eta',-,+}\\n_{\mathrm{res}(C)}=q\geq1}}{\J_C}\l(\epsilon^{-1}t,k\r)}\\+\abs{\sum_{\substack{n_1,n_2\leq N(L)\\n_1+n_2>N(L)}}\sum_{\substack{C\in\mathcal C_{n_1,n_2}^{\eta,\eta',-,+}\\n_{\mathrm{res}}(C)\geq1}}{\J_C}\l(\epsilon^{-1}t,k\r)}\\
    \leq\Lambda(\Lambda\delta)^{\mathcal W}L^{-\frac{\alpha}{\beta}}\\+\Lambda L^{K-\frac{\alpha'}{\beta}\mathcal W}\sum_{\mathcal W<n_1+n_2\leq N(L)}(\Lambda\delta)^{2(n_1+n_2)}\sum_{0<q\leq n_1+n_2-\mathcal W}\epsilon^{\alpha'\l(n_1+n_2-q-\mathcal W\r)}\\+\Lambda L^{-\frac{\alpha}{\beta}}\sum_{\substack{n_1,n_2\leq N(L)\\\mathcal W<n_1+n_2\leq N(L)}}(\Lambda\delta)^{2\l(n_1+n_2\r)}\\+\Lambda(\Lambda\delta)^{2N(L)}L^{\mathcal K}\sum_{\substack{n_1,n_2\leq N(L)\\n_1+n_2>N(L)}}(\Lambda\delta)^{2\l(n_1+n_2-N(L)\r)}\xlongrightarrow{L\longrightarrow\infty}0,
\end{gathered}
\end{equation}
where we may freely choose $\alpha'\in\l(0,\frac{\alpha}{4}\r)$ and see that we should require \begin{equation}
    \mathcal W>\frac{\beta}{\alpha'}\mathcal K.
\end{equation}
Of course, the observation $(\Lambda\delta)^{2N(L)}L^{\mathcal K}\xlongrightarrow{L\to\infty}0$ requires $\delta$ to be small enough, depending only on $\Lambda$ and $\mathcal K$. 

\begin{remark}\label{critical remark}
    We also have for all $n\in\N$ and $L>0$
\begin{equation}
    \sup_{t\in[0,\delta]}\sup_{k\in\Z_L^d}\abs{\rho_{L,n}^{[\eta,\eta']}(t,k)}\leq\Lambda(\Lambda\delta)^nc_n
\end{equation}
and the proof is again of an inductive nature as done in \cref{another trivial bound}. The main difference is that this time the compact support of $\rho_{L,n}^{[\eta,\eta']}(t,\cdot)$ is used. 
\end{remark}

\begin{lemma}\label{abc lemma}
    The maps $\xi\mapsto\rho_{L,n}^{[\eta,\eta']}(t,\xi)$ and $\xi\mapsto\rho_n^{[\eta,\eta']}(t,\xi)$ are compactly supported in $B_R(0)$. What is more, there exist $\Lambda,\alpha>0$ such that for all $n\leq N$:
    \begin{equation}
        \sup_{t\in[0,\delta]}\sup_{k\in\Z_L^d}\abs{\rho_{L,n}^{[\eta,\eta']}(t,k)-\rho_n^{[\eta,\eta']}(t,k)}\leq\Lambda L^{-\alpha}c_n(\Lambda\delta)^n
    \end{equation}
\end{lemma}

\begin{proof}
    We prove this via induction on $n$. If $n=0$, the left-hand side is actually $0$ and if $n>0$ this is simply due to the recursive structure of these functions.
We recall
\begin{equation}
\begin{gathered}
    \rho^{[\eta,\eta']}_{n}(t,k)\coloneqq2\sum_{n_1+n_2+n_3=n-1}\int_0^t\int_{\R^d}\int_{\R^d}\sum_{l=1}^2\Big[\Im\l(\rho_{n_3}^{[\eta,\eta']}(s,k)\tilde\zeta_{l,+}^\eta(\xi_1,\xi_2,\xi)\rho_{n_1}^{[\overline\eta,\eta]}(s,-\xi_1)\rho_{n_2}^{[\overline\eta,\eta]}(s,-\xi_2)\r)\\+
    \sum_{l=3}^6\Im\l(\rho_{n_3}^{[\overline\eta,\eta']}(s,k)\tilde\zeta_{l,+}^\eta(\xi_1,\xi_2,\xi)\rho_{n_2}^{[\overline\eta,\eta]}(s,-\xi_2)\r)\rho_{n_1}^{[\eta,\eta]}(s,\xi_1)\Big]\rmd\xi_1\rmd\xi_2\rmd s
\end{gathered}
\end{equation}
and estimate 
\begin{equation}
    \begin{gathered}
        \abs{\rho_{L,n}^{[\eta,\eta']}(t,k)-\rho_n^{[\eta,\eta']}(t,k)} \\=\Bigg|\sum_{n_1+n_2+n_3=n-1}\Bigg[\int_0^tL^{-2d}\sum_{k_1,k_2}\Bigg[\sum_{l=1}^2\Big[\Im\l(\l(\rho_{L,n_3}^{[\eta,\eta']}(s,k)-\rho_{n_3}^{[\eta,\eta']}(s,k)\r)\tilde\zeta_{l,+}^\eta(k_1,k_2,k)\rho_{L,n_1}^{[\overline\eta,\eta]}(s,-k_1)\rho_{L,n_2}^{[\overline\eta,\eta]}(s,-k_2)\r)\\+\Im\l(\rho_{n_3}^{[\eta,\eta']}(s,k)\tilde\zeta_{l,+}^\eta(k_1,k_2,k)\l(\rho_{L,n_1}^{[\overline\eta,\eta]}(s,-k_1)-\rho_{n_1}^{[\overline\eta,\eta]}(s,-k_1)\r)\rho_{L,n_2}^{[\overline\eta,\eta]}(s,-k_2)\r)\\+\Im\l(\rho_{n_3}^{[\eta,\eta']}(s,k)\tilde\zeta_{l,+}^\eta(k_1,k_2,k)\rho_{n_1}^{[\overline\eta,\eta]}(s,-k_1)\l(\rho_{L,n_2}^{[\overline\eta,\eta]}(s,-k_2)-\rho_{n_2}^{[\overline\eta,\eta]}(s,-k_2)\r)\r)\Big]\\+\sum_{l=3}^6\Big[\Im\l(\l(\rho_{L,n_3}^{[\overline\eta,\eta']}(s,k)-\rho_{n_3}^{[\overline\eta,\eta']}(s,k)\r)\tilde\zeta_{l,+}^\eta(k_1,k_2,k)\rho_{L,n_2}^{[\overline\eta,\eta]}(s,-k_2)\r)\rho_{L,n_1}^{[\eta,\eta]}(s,k_1)\\+\Im\l(\rho_{n_3}^{[\overline\eta,\eta']}(s,k)\tilde\zeta_{l,+}^\eta(k_1,k_2,k)\l(\rho_{L,n_2}^{[\overline\eta,\eta]}(s,-k_2)-\rho_{n_2}^{[\overline\eta,\eta]}(s,-k_2)\r)\r)\rho_{L,n_1}^{[\eta,\eta]}(s,k_1)\\+\Im\l(\rho_{n_3}^{[\overline\eta,\eta']}(s,k)\tilde\zeta_{l,+}^\eta(k_1,k_2,k)\rho_{n_2}^{[\overline\eta,\eta]}(s,-k_2)\r)\l(\rho_{L,n_1}^{[\eta,\eta]}(s,k_1)-\rho_{n_1}^{[\eta,\eta]}(s,k_1)\r)\Big]\\+\sum_{l=1}^2\Bigg[\int_0^tL^{-2d}\sum_{k_1,k_2}\Im\l(\rho_{n_3}^{[\eta,\eta']}(s,k)\tilde\zeta_{l,+}^\eta(k_1,k_2,k)\rho_{n_1}^{[\overline\eta,\eta]}(s,-k_1)\rho_{n_2}^{[\overline\eta,\eta]}(s,-k_2)\r)\\-\int_0^t\int_{\R^d}\int_{\R^d}\Im\l(\rho_{n_3}^{[\eta,\eta']}(s,k)\tilde\zeta_{l,+}^\eta(\xi_1,\xi_2,k)\rho_{n_1}^{[\overline\eta,\eta]}(s,-\xi_1)\rho_{n_2}^{[\overline\eta,\eta]}(s,-\xi_2)\r)\Bigg]\\+\sum_{l=3}^6\Bigg[\int_0^tL^{-2d}\sum_{k_1,k_2}\Im\l(\rho_{n_3}^{[\overline\eta,\eta']}(s,k)\tilde\zeta_{l,+}^\eta(k_1,k_2,k)\rho_{n_2}^{[\overline\eta,\eta]}(s,-k_2)\r)\rho_{n_1}^{[\eta,\eta]}(s,k_1)\\-\int_0^t\int_{\R^d}\int_{\R^d}\Im\l(\rho_{n_3}^{[\overline\eta,\eta']}(s,k)\tilde\zeta_{l,+}^\eta(\xi_1,\xi_2,k)\rho_{n_2}^{[\overline\eta,\eta]}(s,-\xi_2)\r)\rho_{n_1}^{[\eta,\eta]}(s,\xi_1)\rmd\xi_1\rmd\xi_2\rmd s\Bigg]\Bigg]\Bigg|.   
    \end{gathered}
\end{equation}
Terms like 

\begin{equation}\label{another thing i am referring to}
    \begin{gathered}
        \abs{\sum_{n_1+n_2+n_3=n-1}\int_0^tL^{-2d}\sum_{k_1,k_2}\Im\l(\l(\rho_{L,n_3}^{[\eta,\eta']}(s,k)-\rho_{n_3}^{[\eta,\eta']}(s,k)\r)\tilde\zeta_{l,+}^\eta(k_1,k_2,k)\rho_{L,n_1}^{[\overline\eta,\eta]}(s,-k_1)\rho_{L,n_2}^{[\overline\eta,\eta]}(s,-k_2)\r)}\\\leq\sum_{n_1+n_2+n_3=n-1}\int_0^tL^{-2d}\sum_{k_1,k_2}\abs{\rho_{L,n_3}^{[\eta,\eta']}(s,k)-\rho_{n_3}^{[\eta,\eta']}(s,k)}\abs{\tilde\zeta_{l,+}^\eta(k_1,k_2,k)}\abs{\rho_{L,n_1}^{[\overline\eta,\eta]}(s,-k_1)}\abs{\rho_{L,n_2}^{[\overline\eta,\eta]}(s,-k_2)}
    \end{gathered}
\end{equation}
are estimated by using the induction hypothesis 
\begin{equation}
    \abs{\rho_{L,n_3}^{[\eta,\eta']}(s,k)-\rho_{n_3}^{[\eta,\eta']}(s,k)}\leq\Lambda L^{-\alpha}c_{n_3}(\Lambda\delta)^{n_3}
\end{equation}
and \cref{critical remark}. The remaining term 
\begin{equation}
    \begin{gathered}
        L^{-2d}\sum_{k_1,k_2}\abs{\tilde\zeta_{l,+}^\eta(k_1,k_2,k)}\abs{\rho_{L,n_1}^{[\overline\eta,\eta]}(s,-k_1)}\abs{\rho_{L,n_2}^{[\overline\eta,\eta]}(s,-k_2)}\\\leq\Lambda(\Lambda\delta)^{n_1+n_2}c_{n_1}c_{n_2} L^{-2d}\sum_{k_1,k_2\in B_R^{\Z_L^d}(0)}1\leq\Lambda(\Lambda\delta)^{n_1+n_2}c_{n_1}c_{n_2}
    \end{gathered}
\end{equation}
is estimated by using \cref{another trivial bound} and the compact support of the functions $\rho_{L,n}^\star$. Finally,

\begin{equation}
    \abs{\cref{another thing i am referring to}}\leq(\Lambda\delta)^{n}L^{-\alpha}\sum_{n_1+n_2+n_3=n-1}c_{n_1}c_{n_2}c_{n_3}=(\Lambda\delta)^{n}L^{-\alpha}c_n.
\end{equation}
Terms of the form 

\begin{equation}\label{What I am estimating as of now}
    \begin{gathered}
        \Bigg|\sum_{n_1+n_2+n_3=n-1}\Bigg[\int_0^tL^{-2d}\sum_{k_1,k_2}\Im\l(\rho_{n_3}^{[\eta,\eta']}(s,k)\tilde\zeta_{l,+}^\eta(k_1,k_2,k)\rho_{n_1}^{[\overline\eta,\eta]}(s,-k_1)\rho_{n_2}^{[\overline\eta,\eta]}(s,-k_2)\r)\\-\int_0^t\int_{\R^d}\int_{\R^d}\Im\l(\rho_{n_3}^{[\eta,\eta']}(s,k)\tilde\zeta_{l,+}^\eta(\xi_1,\xi_2,k)\rho_{n_1}^{[\overline\eta,\eta]}(s,-\xi_1)\rho_{n_2}^{[\overline\eta,\eta]}(s,-\xi_2)\r)\Bigg]\Bigg|\\=\Bigg|\sum_{n_1+n_2+n_3=n-1}\Bigg[\int_0^t\Im\Bigg[L^{-2d}\sum_{k_1,k_2}\rho_{n_3}^{[\eta,\eta']}(s,k)\tilde\zeta_{l,+}^\eta(k_1,k_2,k)\rho_{n_1}^{[\overline\eta,\eta]}(s,-k_1)\rho_{n_2}^{[\overline\eta,\eta]}(s,-k_2)\\-\int_{\R^d}\int_{\R^d}\rho_{n_3}^{[\eta,\eta']}(s,k)\tilde\zeta_{l,+}^\eta(\xi_1,\xi_2,k)\rho_{n_1}^{[\overline\eta,\eta]}(s,-\xi_1)\rho_{n_2}^{[\overline\eta,\eta]}(s,-\xi_2)\Bigg]\rmd s\Bigg]\Bigg|\\\leq\sum_{n_1+n_2+n_3=n-1}\Bigg|\int_0^t\Im\Bigg[L^{-2d}\sum_{k_1,k_2}\rho_{n_3}^{[\eta,\eta']}(s,k)\tilde\zeta_{l,+}^\eta(k_1,k_2,k)\rho_{n_1}^{[\overline\eta,\eta]}(s,-k_1)\rho_{n_2}^{[\overline\eta,\eta]}(s,-k_2)\\-\int_{\R^d}\int_{\R^d}\rho_{n_3}^{[\eta,\eta']}(s,k)\tilde\zeta_{l,+}^\eta(\xi_1,\xi_2,k)\rho_{n_1}^{[\overline\eta,\eta]}(s,-\xi_1)\rho_{n_2}^{[\overline\eta,\eta]}(s,-\xi_2)\Bigg]\rmd s\Bigg|\\\leq\sum_{n_1+n_2+n_3=n-1}\int_0^t\abs{\rho_{n_3}^{[\eta,\eta']}(s,k)}\Bigg|L^{-2d}\sum_{k_1,k_2}\tilde\zeta_{l,+}^\eta(k_1,k_2,k)\rho_{n_1}^{[\overline\eta,\eta]}(s,-k_1)\rho_{n_2}^{[\overline\eta,\eta]}(s,-k_2)\\-\int_{\R^d}\int_{\R^d}\tilde\zeta_{l,+}^\eta(k,\xi_1,\xi_2)\rho_{n_1}^{[\overline\eta,\eta]}(s,-\xi_1)\rho_{n_2}^{[\overline\eta,\eta]}(s,-\xi_2)\Bigg|\rmd s.
    \end{gathered}
\end{equation}
Now using \cref{sums into integrals} and the fact that the functions $\rho_n^\star$ have compact support, we obtain that 

\begin{equation}
    \begin{gathered}
        \Bigg|L^{-2d}\sum_{k_1,k_2}\tilde\zeta_{l,+}^\eta(k_1,k_2,k)\rho_{n_1}^{[\overline\eta,\eta]}(s,-k_1)\rho_{n_2}^{[\overline\eta,\eta]}(s,-k_2)\\-\int_{\R^d}\int_{\R^d}\tilde\zeta_{l,+}^\eta(\xi_1,\xi_2,k)\rho_{n_1}^{[\overline\eta,\eta]}(s,-\xi_1)\rho_{n_2}^{[\overline\eta,\eta]}(s,-\xi_2)\Bigg|\\\leq\frac{\Lambda}{L^{2d+1}}\sum_{k_1,k_2\in B_R^{\Z_L^d}(0)}\sup_{Q_{k_1,\frac{1}{2L}}\times Q_{k_2,\frac{1}{2L}}}\abs{\nabla_{k_1,k_2}\l(\tilde\zeta_{l,+}^\eta(k_1,k_2,k)\rho_{n_1}^{[\overline\eta,\eta]}(s,-k_1)\rho_{n_2}^{[\overline\eta,\eta]}(s,-k_2)\r)}\\\leq\frac{\Lambda}{L^{2d+1}}\norm{Q_+^\eta}_{\l(W^{1,\infty}\r)^5}(\Lambda\delta)^{n_1+n_2}c_{n_1}c_{n_2}\Lambda L^{2d}\leq(\Lambda\delta)^{n_1+n_2}c_{n_1}c_{n_2}L^{-1}
    \end{gathered}
\end{equation}
by increasing $\Lambda$ and decreasing $\delta$. Finally,

\begin{equation}
    \begin{gathered}
        \abs{\cref{What I am estimating as of now}}\leq\Lambda\delta\sum_{n_1+n_2+n_3=n-1}(\Lambda\delta)^{n_1+n_2+n_3}c_{n_1}c_{n_2}c_{n_3}L^{-1} = (\Lambda\delta)^nL^{-1}c_n
    \end{gathered}
\end{equation}
and this completes the proof.
\end{proof}
Now, \cref{abc lemma} leads to 
\begin{equation}
        \abs{\cref{fourth line}}\leq\sum_{n\leq 2N(L)}\abs{\rho_{L,n}^{[\eta,\eta']}(t,k)-\rho_n^{[\eta,\eta']}(t,k)}\leq\Lambda L^{-\alpha}\sum_{n=0}^{2N(L)}(\Lambda\delta)^nc_n\xlongrightarrow{L\longrightarrow\infty}0.
    \end{equation}

For the final line \eqref{fifth line}, we use \cref{having a solution to the resonant system as a fixed point solution} to observe 
\begin{equation}
    \abs{\cref{fifth line}}\leq\l(\Lambda\delta\r)^{N(L)}\xlongrightarrow{L\longrightarrow\infty}0.
\end{equation}

\appendix
\section{Preliminary lemmas}

\begin{lemma}[Isserlis' Theorem]\label{Isserlis}
Let $k_j\in\Z_L^d$, $\iota_j\in\{\pm\}$ for $1\leq j\leq n$, then 
\begin{equation}
    \E\l[\prod_{j=1}^n\mu_{k_j}^{\eta_j,\iota_j}\r] = \sum_{\mathcal P}\prod_{\{j,j'\}\in\mathcal P}\delta_{\iota_j+\iota_{j'}}\delta_{k_j-k_{j'}}M^{\eta_j,\eta_{j'}}(k_j)^{\iota_j}
\end{equation}
\end{lemma}
\begin{proof}
    This is a minor adaptation of \cite[Lemma A.2]{Deng2021FullDO}.
\end{proof}

\begin{lemma}[Gaussian Hypercontractivity]\label{hypercontractivity}
    Let $\{g_k\}$ be i.i.d. Gaussians or random phases, $\iota_j\in\{\pm\}$ and X be the random variable 
    \begin{equation}
        X(\omega)\coloneqq \sum_{k_1,\ldots,k_n}a_{k_1,\ldots,k_n}\prod_{j=1}^ng_{k_j}^{\zeta_j}(\omega).
    \end{equation}
    Then for $p\geq2$,
    \begin{equation}
        \E\abs{X}^p\leq(p-1)^{\frac{np}{2}}\E\l(\abs{X}^2\r)^{\frac{p}{2}}.
    \end{equation}
\end{lemma}
\begin{proof}
    See Lemma 2.6 of \cite{hyper}.
\end{proof}

 

\paragraph{Conflict of interest statement.}

The author has no conflicts of interest to declare.

\paragraph{Data availability statement.} 

Data sharing is not applicable to this article as no new data were created or analyzed in this study.
\printbibliography

\raggedright \textsc{Centre de math\'ematiques Laurent-Schwartz, \'Ecole Polytechnique, 91120 Palaiseau, France}\\
\textit{Email address:} \texttt{shayan.zahedi@polytechnique.edu}
\end{document}